\documentclass[11pt]{amsart}
\usepackage[T1]{fontenc}
\usepackage[utf8]{inputenc}
\usepackage{amsmath,amssymb,amsthm,mathrsfs}
\usepackage[margin=1in]{geometry}
\usepackage[colorlinks=true,linkcolor=blue,citecolor=blue,urlcolor=blue]{hyperref}

\newtheorem{theorem}{Theorem}[section]
\newtheorem{proposition}[theorem]{Proposition}
\newtheorem{lemma}[theorem]{Lemma}
\newtheorem{corollary}[theorem]{Corollary}
\theoremstyle{definition}
\newtheorem{definition}[theorem]{Definition}
\theoremstyle{remark}
\newtheorem{remark}[theorem]{Remark}

\newcommand{\R}{\mathbb R}
\newcommand{\Z}{\mathbb Z}

\newcommand{\dd}{\,\mathrm{d}}
\newcommand{\groupid}{\mathbf{e}}
\newcommand{\latticemin}{\lambda_{\min}}
\newcommand{\extmin}{\lambda_{\mathrm{extmin}}}
\newcommand{\qform}{\mathfrak{q}}
\newcommand{\norm}[1]{\left\lVert #1\right\rVert}
\newcommand{\abs}[1]{\left\lvert #1\right\rvert}
\newcommand{\SY}{\mathcal S_{\mathrm Y}}
\newcommand{\Q}{\mathcal Q}
\newcommand{\etaY}{\eta_{\mathrm Y}}
\newcommand{\SR}{\mathcal S_{\mathrm R}}
\newcommand{\Qall}{\mathcal Q_{\mathrm{all}}}
\DeclareMathOperator{\SL}{SL}
\DeclareMathOperator{\SO}{SO}
\DeclareMathOperator{\Lie}{Lie}
\DeclareMathOperator{\Ad}{Ad}
\DeclareMathOperator{\diag}{diag}
\DeclareMathOperator{\Leb}{Leb}
\DeclareMathOperator{\rank}{rank}
\DeclareMathOperator{\End}{End}
\DeclareMathOperator{\Sym}{Sym}
\DeclareMathOperator{\inj}{inj}
\DeclareMathOperator{\vol}{vol}
\DeclareMathOperator{\Height}{ht}

\title{Effective equidistribution of polynomial trajectories in
  $\operatorname{SL}_3(\mathbb R)/\operatorname{SL}_3(\mathbb Z)$}

\author{Hao Xing}
\address{CUNY Graduate Center, New York, NY 10016, USA}
\email{thinkhowxh@gmail.com}

\begin{document}

\begin{abstract}
For a fixed matrix-polynomial trajectory in
$\operatorname{SL}_3(\mathbb R)/\operatorname{SL}_3(\mathbb Z)$,
we give a necessary and sufficient intrinsic Diophantine condition
for its time averages to equidistribute at a power rate on its
homogeneous orbit closure. The condition uses a lattice in the
space of polynomial functions on the normalized orbit-closure
group, each admitting a representative of total degree at most
two in the matrix entries. These polynomial tests have a common
degree bound of two, independent of the curve degree and optimal
in general. The proof begins with a classification of the possible
orbit-closure groups and derives lower bounds for polynomial
coefficient vectors from the intrinsic condition. It combines
quantitative nilmanifold equidistribution, unipotent tangent
approximation, and effective flow estimates, with arithmetic
separation and quantitative nondivergence controlling cusp
excursions and periodic-orbit obstructions.
\end{abstract}

\maketitle

\section{Introduction and main theorem}
\label{sec:mainresults}

Let $G=\SL_3(\R)$, $\Gamma=\SL_3(\Z)$, and $X=G/\Gamma$. Let $\Phi\in\SL_3(\R[t])$ be a nonconstant matrix-polynomial
curve with $\Phi(0)=\groupid$. For $x\in X$ and $f\in C_c(X)$,
consider the time average
\begin{equation}\label{eq:average}
 A_T^\Phi f(x)=\frac1T\int_0^T f(\Phi(t)x)\dd t.
\end{equation}

Ratner \cite{Ratner1991} proved that orbits of one-parameter
$\Ad$-unipotent subgroups on homogeneous spaces of finite volume
equidistribute with respect to invariant probability measures
on their homogeneous orbit closures.
Shah \cite[Theorem 1.1]{Shah1994} established the corresponding
conclusion for trajectories defined by regular algebraic maps
from $\R$, including matrix-polynomial curves. In the present
setting, his theorem gives
\[
 A_T^\Phi f(x)\longrightarrow\mu_Y(f)
 \qquad(T\longrightarrow\infty),
 \qquad
 Y=\overline{\{\Phi(t)x:t\ge0\}},
\]
where $Y$ is a homogeneous orbit of finite volume and $\mu_Y$
is its invariant probability measure.

Qualitative equidistribution alone does not imply a power estimate
in a fixed derivative norm of finite order. For example, a linear
flow on a two-dimensional torus with a Liouville slope is
equidistributed but admits no such estimate; see
Section~\ref{subsec:liouvilleobstruction}.

In this paper, we give a necessary and sufficient intrinsic
Diophantine criterion for a fixed pair $(\Phi,x)$ to admit a power
estimate uniform over smooth functions supported in each fixed
compact subset of $X$. The criterion is formulated using a fixed
lattice of polynomial functions on the normalized orbit-closure
group. These functions admit representatives of total degree at
most two in the matrix entries. The common degree bound of two
is independent of the curve degree and optimal in general.

\subsection{The intrinsic condition and the main theorem}

For the base point $x=g\Gamma\in G/\Gamma$, left translation
by $g^{-1}$ sends $x$ to the identity coset $\Gamma$ and replaces
$\Phi(t)$ by $g^{-1}\Phi(t)g$. Let $J_g$ be the closed subgroup
generated by this normalized curve. Let $\mathbf H_g$ be the smallest
$\mathbb Q$-algebraic subgroup containing $J_g$
(Lemma~\ref{app:rational-envelope}), and put
$H_g=\mathbf H_g(\R)^\circ$, where $\circ$ denotes the identity
component in the real topology. The normalized homogeneous orbit
closure is $H_g\Gamma/\Gamma$, as stated in
\eqref{eq:actualsupport}. We call $H_g$ the normalized
orbit-closure group. The subgroup $J_g$ may be proper in $H_g$.
The intrinsic condition below requires a polynomial lower bound
for the distance of nonzero integral polynomial tests on $H_g$
from the subspace of right-$J_g$-invariant functions.

\begin{definition}[Normalized polynomial test spaces]
\label{def:polynomialtestspaces}
Let $\mathbf H\le\SL_3$ be a connected algebraic subgroup defined
over $\mathbb Q$ (Section~\ref{app:basic}), and put
$H=\mathbf H(\R)^\circ$. For each integer $m\ge1$, define
\[
 \mathscr P_{\mathbf H}^{(m)}
 =\{P|_H:
     P\in\R[x_{ij}:1\le i,j\le3],\
     \deg P\le m,\ P(\groupid)=0\}.
\]
Here $\deg P$ denotes total degree in the matrix entries.
The degree bound applies to an ambient representative;
representatives agreeing on $H$ define the same test.

By Lemma~\ref{lem:polynomial-functions}, restrictions of polynomials
with rational coefficients give $\mathscr P_{\mathbf H}^{(m)}$
a $\mathbb Q$-structure. For each $m$, fix a rational basis
$p_{m,1},\ldots,p_{m,r_m}$ and equip
$\mathscr P_{\mathbf H}^{(m)}$ with the inner product for which
this basis is orthonormal. Define
\[
 \mathscr L_{\mathbf H}^{(m)}
 =\bigoplus_{i=1}^{r_m}\Z p_{m,i}.
\]
We call the elements of $\mathscr L_{\mathbf H}^{(m)}$
\emph{integral polynomial tests} with degree cutoff $m$.
For a closed subgroup $J\le H$, let
$\mathscr K_{J,\groupid}^{(m)}$ denote the subspace of
right-$J$-invariant functions:
\begin{equation}\label{eq:intrinsickernelidentity}
 \mathscr K_{J,\groupid}^{(m)}
 =\{p\in\mathscr P_{\mathbf H}^{(m)}:
       p(bj)=p(b)\ \text{for all }b\in H,\ j\in J\}.
\end{equation}
For degree cutoff $2$, abbreviate
\[
 \mathscr P_{\mathbf H}=\mathscr P_{\mathbf H}^{(2)},
 \qquad
 \mathscr L_{\mathbf H}=\mathscr L_{\mathbf H}^{(2)},
 \qquad
 \mathscr K_{J,\groupid}=\mathscr K_{J,\groupid}^{(2)}.
\]
\end{definition}

\begin{definition}[Intrinsic Diophantine condition]\label{def:intrinsicdiophantine}
The pair $(J,\groupid)$ has type $(c_0,\tau)$ if $0<c_0\le1$, $\tau\ge0$,
and
\begin{equation}\label{eq:intrinsicnormalizedcondition}
 \operatorname{dist}(p,\mathscr K_{J,\groupid})
       \ge c_0\norm p^{-\tau}
       \qquad(0\ne p\in\mathscr L_{\mathbf H}).
\end{equation}
We say it has \emph{finite intrinsic type} if such $c_0,\tau$ exist.
\end{definition}

Finite intrinsic type and the admissible exponent $\tau$ are
independent of the chosen lift $g$ and rational basis, although
$c_0$ may change; see Lemma~\ref{lem:intrinsicconjugacy}.

To specify the derivative norms, fix a basis
$Z_1,\ldots,Z_8$ of $\mathfrak{sl}_3(\R)$. For
$Z\in\mathfrak{sl}_3(\R)$ and integers $s\ge0$, put
\[
D_Zf(y)=\left.\frac{\mathrm d}{\mathrm dr}\right|_{r=0}
 f(\exp(rZ)y),\qquad
 \norm f_{C_G^s}
 =\sum_{j=0}^s\ \sum_{i_1,\ldots,i_j=1}^8
       \norm{D_{Z_{i_1}}\cdots D_{Z_{i_j}}f}_\infty.
\]
The term $j=0$ is $\norm f_\infty$. The degree of a matrix-polynomial
curve is the maximum degree of its entries.

\begin{theorem}
\label{thm:intrinsic}
Let $\Phi\in\SL_3(\R[t])$ be nonconstant, with $\Phi(0)=\groupid$ and
$\deg\Phi\le d$ for an integer $d\ge1$, and let $x\in X$.
Choose $g\in G$ with $x=g\Gamma$ and define
\begin{equation}\label{eq:normalizedcurve}
 \Psi_g(t)=g^{-1}\Phi(t)g,\qquad
 J_g=\overline{\langle\Psi_g(\R)\rangle}.
\end{equation}
Let $\mathbf H_g$ be the smallest $\mathbb Q$-algebraic subgroup
of $\SL_3$ containing $J_g$, and set $H_g=\mathbf H_g(\R)^\circ$.
Then $\mathbf H_g$ is connected, its radical is unipotent,
$H_g\Gamma/\Gamma$ is closed with finite invariant volume, and
\begin{equation}\label{eq:actualsupport}
 Y:=\overline{\{\Phi(t)x:t\ge0\}}
   =g\,\overline{J_g\Gamma/\Gamma}
   =gH_g\Gamma/\Gamma.
\end{equation}
Let $\mu_Y$ be the unique $gH_gg^{-1}$-invariant probability
measure on $Y$. Suppose that $(J_g,\groupid)$ has intrinsic type
$(c_0,\tau)$ in $\mathscr P_{\mathbf H_g}$ as in
Definition~\ref{def:intrinsicdiophantine}.

There are an integer $s=s(\mathbf H_g,d)\ge1$ and a number
$\vartheta_{\mathbf H_g,d}>0$ such that for every compact
$\Omega\subset X$ there is
$C=C(\Phi,g,c_0,\tau,\Omega,\mathscr L_{\mathbf H_g})>0$ with
\begin{equation}\label{eq:intrinsicmain}
 \left|\frac1T\int_0^T f(\Phi(t)x)\dd t-\mu_Y(f)\right|
 \le C\norm f_{C_G^s}T^{-\vartheta_{\mathbf H_g,d}/(\tau+1)}
 \qquad(T\ge2)
\end{equation}
for every $f\in C_c^\infty(X)$ supported in $\Omega$.
\end{theorem}

Conversely, a power estimate in a fixed finite-order derivative
norm forces finite intrinsic type.

\begin{proposition}
\label{prop:intrinsicnecessity}
Let $\Phi,x,g,\Psi_g,J_g,\mathbf H_g,Y$ be as in
Theorem~\ref{thm:intrinsic}, without assuming intrinsic type.
Suppose there are an integer $s\ge1$ and $\delta>0$ such that for
every compact $\Omega\subset X$ some $C_\Omega>0$ satisfies
\begin{equation}\label{eq:conversepower}
 \left|\frac1T\int_0^T f(\Phi(t)x)\dd t-\mu_Y(f)\right|
 \le C_\Omega\norm f_{C_G^s}T^{-\delta}\qquad(T\ge2)
\end{equation}
for every $f\in C_c^\infty(X)$ supported in $\Omega$.
Then for every fixed integer $m\ge1$ there exist $c_m>0$ and
$\tau_m\ge0$ such that
\begin{equation}\label{eq:converseintrinsic}
 \operatorname{dist}(p,\mathscr K_{J_g,\groupid}^{(m)})
 \ge c_m\norm p^{-\tau_m}
 \qquad(0\ne p\in\mathscr L_{\mathbf H_g}^{(m)}).
\end{equation}
Consequently finite intrinsic type with degree cutoff $2$ is
equivalent to the existence of a power estimate of the form
\eqref{eq:conversepower}.
\end{proposition}

If the normalized curve generates $H_g$ as a closed subgroup, then
$J_g=H_g$ and $\mathscr K_{J_g,\groupid}=\{0\}$.
Thus the intrinsic condition holds automatically with
$(c_0,\tau)=(1,0)$, giving a power rate; see
Corollary~\ref{cor:intrinsicfullgeneration}. In particular, if the
image of $\Phi$ generates $G$ as a closed subgroup, then $Y=X$
and the theorem applies at every $x\in X$. Normalized curves whose
coefficients lie in a number field give another sufficient
criterion; see Corollary~\ref{cor:intrinsicalgebraic}.

\begin{remark}[Optimality of the degree cutoff]
\label{rem:intrinsicoptimality}
Degree cutoff $2$ is the smallest common degree cutoff for all
normalized orbit-closure groups in this setting.
Section~\ref{subsec:degreeonecounterexample} gives a toral trajectory
satisfying finite intrinsic type with degree cutoff $1$ but
admitting no power estimate uniform over smooth test functions
in any fixed derivative norm of finite order.
\end{remark}

\subsection{Related work}
\label{subsec:previousresults}

Effective equidistribution for horocycle flows and expanding
closed horocycles has a long history; see
\cite{Sarnak1981,Burger1990,FlaminioForni2003,Strombergsson2013}.
For effective estimates for expanding horospherical averages
on finite-volume semisimple quotients, see
\cite{KleinbockMargulis2012,Edwards2021} and the preprint
appendix to \cite{Zheng2015}. In the nilpotent setting,
Green and Tao \cite{GreenTao2012} establish a quantitative
Leibman theorem for discrete polynomial sequences on compact
nilmanifolds.

Effective results on the distribution of closed semisimple
orbit measures and adelic homogeneous subsets are developed
in \cite{EMV2009,EMMV2020}. Effective density in semisimple
quotients is treated in
\cite{LindenstraussMargulis2014,LindenstraussMohammadi2023}.
Arithmetic avoidance for unipotent trajectories is studied
by Lindenstrauss, Margulis, Mohammadi, and Shah
\cite{LMMS2024}, building on the Dani--Margulis linearization
method \cite{DaniMargulis1993}.

More recent work concerns polynomially effective
equidistribution for non-horospherical unipotent dynamics.
Chow and Yang \cite{ChowYang2024} treat a special family
of expanding translates in $\SL_3(\R)/\SL_3(\Z)$, and
Yang \cite{Yang2025} treats singular unipotent orbit
segments under quantitative control of diagonal excursions.
Lindenstrauss, Mohammadi, and Wang \cite{LMW2026}
establish equidistribution alternatives on arithmetic
quotients of $\SL_2(\mathbb C)$ and
$\SL_2(\R)\times\SL_2(\R)$; Lindenstrauss, Mohammadi,
Wang, and Yang \cite{LMWY2025} treat regular unipotent
dynamics in arithmetic quotients of $\R$-quasi-split
semisimple groups of absolute rank two. Lin
\cite{Lin2026} establishes equidistribution alternatives for expanding translates for
certain unipotent subgroups in arithmetic quotients of
perfect groups, under representation hypotheses. For related
effective equidistribution results for random walks on
simple homogeneous spaces, see B\'enard and He
\cite{BenardHe2025}.

On affine homogeneous spaces, effective results for several
families of expanding translates, horocycle lifts, and
unipotent orbits are obtained in
\cite{Strombergsson2015,BV2016,Vinogradov2016,
StrombergssonVishe2020,Kim2024,SSV2026}.

\vspace{3mm}
The present paper gives a necessary and sufficient intrinsic
Diophantine criterion for a power rate in the equidistribution of
fixed matrix-polynomial trajectories on their homogeneous orbit
closures, using polynomial tests of degree at most two on the
normalized orbit-closure group. We also establish, in
Proposition~\ref{prop:regularaffineflow}, a regular affine flow
estimate uniform in the initial translation and varying base
matrices under the stated lower bound for the shortest vector
of the renormalized base lattice.

\subsection{Main ideas and organization of the proof}
\label{subsec:readingguide}

The proof first normalizes the homogeneous orbit closure using
Shah's results \cite{Shah1994,Shah1991}, as summarized in
Lemma~\ref{lem:orbitnormalization}. The support classification
in Proposition~\ref{prop:supportclassification} reduces the
problem to unipotent supports, proper semisimple supports,
the ambient quotient, and rational conjugates of the affine
group $\SL_2(\R)\ltimes\R^2$ and its contragredient embedding.
The algebraic and lattice facts used in these reductions are
collected in Appendix~\ref{app:algebraic-background}.

The common arithmetic step converts intrinsic type into lower
bounds for the nonconstant coefficients of polynomial tests.
Lemma~\ref{lem:intrinsiccoefficients} constructs a finite
coefficient map whose kernel is the invariant subspace in
Definition~\ref{def:intrinsicdiophantine}. For each test, distance
from the kernel is comparable to the norm of its coefficient vector.
Lemma~\ref{lem:intrinsictransfer} transfers these bounds to the
relevant representations. The support classification yields
the common degree bound for the tests.

On unipotent supports, which are compact nilmanifolds,
Lemma~\ref{lem:continuousleibman} derives a continuous-time form
of Green and Tao's quantitative Leibman theorem, stated in
Section~\ref{subsec:nilinput}. The coefficient bounds exclude
its horizontal-character obstruction. For the remaining supports,
the leading coefficient of the right logarithmic derivative
$\Psi'\Psi^{-1}$ is a nonzero nilpotent matrix $N$.
Lemma~\ref{lem:leading} gives the matrix-polynomial form of
Shah's unipotent increment limit. Tangent approximation reduces
the polynomial average to growing segments of $\exp(sN)$ at
moving starting points. Quantitative nondivergence
\cite[Theorem~5.2]{KM1998} controls the exceptional times.
We average the flow estimates over the remaining times and sum
the resulting bounds over dyadic intervals.

For the ambient quotient, Yang's estimate
(Theorem~\ref{thm:yang}) handles singular tangents ($N^2=0$),
and the rank-two expanding-translate estimate
(Theorem~\ref{prop:regularinput}) handles regular tangents
($N^2\ne0$). Arithmetic separation, nondivergence, and cusp
compression control the compactness and periodic-orbit
alternatives and the required starting thresholds.

Proper semisimple supports reduce through rational coordinates
and finite covers to horocycle averages on fixed rank-one
quotients. Zheng's theorem supplies the fixed-lattice estimate
recorded in Section~\ref{subsec:congruencecovers}, which applies
to both compact and noncompact $\SL_2(\R)$ quotients.

The affine argument distinguishes translation, block, and
regular tangents. Translation tangents are treated by Fourier
analysis in the torus fibers. The block estimate is the
specialization of Lin's theorem justified in
Section~\ref{subsec:affineblock}; regular tangents require
Proposition~\ref{prop:regularaffineflow}. Its proof uses Fourier
expansion in the fibers: the zero mode is controlled by the base
horocycle estimate, and oscillatory integrals and exponential
sums give uniform cancellation for the nonzero modes.
The uniformity in the initial translation and varying base
matrices permits the polynomial transfer in
Proposition~\ref{prop:regularaffinereduction}.

For necessity, an integral polynomial test close to the invariant
subspace remains almost invariant along the curve for a long
time. Smooth tests supported in a fixed compact set detect this
behavior, while polynomial sublevel estimates control their
invariant integrals. A power bound therefore forces finite
intrinsic type. This argument is independent of the effective
flow estimates.

Section~\ref{sec:refinement} establishes normalization and the
intrinsic coefficient bounds. Sections~\ref{sec:tools}
and~\ref{sec:inputs} contain the common averaging arguments and
quantitative inputs. The estimates for each support type are
organized as follows.

\begin{center}
\renewcommand{\arraystretch}{1.18}
\begin{tabular}{@{}p{0.29\textwidth}p{0.66\textwidth}@{}}
\hline
Support type & Estimate used in the final proof\\
\hline
Unipotent & Proposition~\ref{prop:separatedunipotent}
(Section~\ref{sec:unipotentproof}).\\
Proper semisimple & Propositions~\ref{prop:congruencemodular}
and~\ref{prop:compactpolynomial}
(Section~\ref{sec:modularproof}).\\
Ambient quotient $X$ & Propositions~\ref{prop:separatedsingular}
and~\ref{prop:separatedregular}
(Section~\ref{sec:ambientproof}).\\
Affine & Proposition~\ref{prop:affineintrinsic}, combining the three
tangent cases in Section~\ref{sec:leviaffine}. The regular-flow proof
is in Section~\ref{sec:regularaffine}.\\
\hline
\end{tabular}
\end{center}

Section~\ref{subsec:mainassembly} assembles these
estimates and compares norms;
Section~\ref{subsec:intrinsicnecessity} proves
the converse.
Section~\ref{sec:applications} gives examples of polynomial curves
generating their normalized orbit-closure groups, the number-field
criterion, examples concerning necessity and the degree cutoff,
and an application to compact rank-one quotients.

\par\medskip
\noindent\textbf{Statement of AI use and acknowledgment.}
The author began this project by studying special cases of
Theorem~\ref{thm:intrinsic}. Since mid-2026, ChatGPT models from GPT-5.5 to GPT 6 Astra were subsequently used in formulating the main theorem, developing
and reviewing arguments, assisting with the literature review,
and revising the exposition. This work would not be possible without the foundational works of pioneering mathematicians in homogeneous dynamics and number theory. The author is responsible for the mathematical statements, proofs, and citations in the paper.

\subsection*{Notation and conventions}

We use $\SL_n$ to denote the special linear algebraic group,
with base field specified by the context. We write $\groupid$ for
the group identity and $I_n$ for the $n$ by $n$ identity matrix;
thus $\groupid=I_n$ in $\SL_n$.

In addition to $G,\Gamma,X$ from above, put
$G_2=\SL_2(\R)$, $\Gamma_2=\SL_2(\Z)$, and $X_2=G_2/\Gamma_2$.
Write $m_X,m_{X_2}$ for the invariant probability measures induced
by Haar measure. We use the continuous averages $A_T^\Phi$ defined
in \eqref{eq:average}.
We normalize $\Phi(0)=\groupid$; a general
initial value is absorbed by replacing the curve with
$\Phi(t)\Phi(0)^{-1}$ and the starting point with $\Phi(0)x$.

For a fixed group element $b$, let $L_b(y)=by$ denote left
translation on the quotient, and write $F_b=L_b^*F=F\circ L_b$
for its pullback on functions. Thus $F_b(y)=F(by)$.

An orbit is called \emph{periodic} when it is closed and carries
a finite invariant measure. Orbit volumes are measured using fixed
Haar measures on the acting groups, with the normalization specified
for each cited estimate. On each periodic orbit, the invariant
measure used in averages is normalized to have total mass one.

Write $E_{ij}$ for the matrix unit with entry one in position
$(i,j)$ and zero elsewhere. Unless a two-dimensional representation
is specified, $E_{ij}$ denotes a $3\times3$ matrix unit. Closures refer to
the real Lie-group or quotient topology unless Zariski closure
is specified. For a subfield $F\subset\R$,
$\mathbf H(F[t])$ denotes matrices over $F[t]$ satisfying the
defining equations of $\mathbf H$ (Section~\ref{app:basic}).
The radical of an algebraic group is its largest connected normal
solvable subgroup, and $R_u$ denotes its largest connected normal
unipotent subgroup (Section~\ref{app:levi}).

For a lattice $\Lambda\subset\R^n$, write
$\latticemin(\Lambda)=\min_{0\ne v\in\Lambda}\norm v$.
We also write $\latticemin(b\Gamma_n)=\latticemin(b\Z^n)$, where
$\Gamma_n=\SL_n(\Z)$. The capped exterior systole $\extmin(y)$ on $X$
is defined in \eqref{eq:systole}.

All vector norms are Euclidean in the specified bases, exterior
powers have the induced Euclidean norms, and matrix norms are
operator norms unless otherwise stated. For an invertible map $A$,
$s_{\min}(A)=\inf_{\norm v=1}\norm{Av}$. The notation $U\ll_d V$
means $|U|\le C_dV$, with dependence only on $d$ and fixed conventions.
The constant may increase from line to line. We keep dependence on
the curve, lift, and arithmetic separation explicit. The Lie algebra
$\Lie(H)$ is the tangent space at the identity, and $\Ad(g)Z=gZg^{-1}$.

\section{Normalized supports and intrinsic coefficient bounds}
\label{sec:refinement}

\subsection{Normalization and independence of the lift}

The following two reductions justify the intrinsic formulation
used in the introduction. They use only qualitative orbit theory
and finite-dimensional linear algebra.
The rationality assertion concerns the normalized group
$\mathbf H_g$; the acting group $gH_gg^{-1}$ need not be defined
over $\mathbb Q$.

\iffalse
Finite volume of a prescribed real-algebraic orbit alone would
not give this conclusion. The support identification is qualitative
and does not provide an algorithm for recovering $\mathbf H_g$
or $c_0$ from finite-precision coefficients.
\fi

\begin{lemma}
\label{lem:orbitnormalization}
For the data in \eqref{eq:normalizedcurve}, the group $\mathbf H_g$
is connected, its radical is unipotent, and $H_g\Gamma/\Gamma$
is closed with finite invariant volume. The orbit-closure identity
\eqref{eq:actualsupport} holds. In particular,
$\mathbf H_g/R_u(\mathbf H_g)$ is semisimple.
\end{lemma}

\begin{proof}
A determinant-one matrix-polynomial map is regular algebraic.
The subgroup generated by $\Psi_g(\mathbb R)$ is connected:
each set of words of a fixed length in the curve and its inverse
is a continuous image of a product of connected sets and contains
$\groupid$. Their union is connected, and so is its closure $J_g$.
The group $J_g$ is generated by algebraic unipotent one-parameter
subgroups by \cite[Proposition 2.1(1)]{Shah1994}
(Proposition~\ref{app:shah-normalization}).
Shah's polynomial equidistribution theorem
\cite[Theorem 1.1 and Proposition 2.1(2)]{Shah1994}
gives a closed subgroup $F$ containing $\Psi_g(\R)$ such that
$F\Gamma/\Gamma$ is closed and carries an $F$-invariant probability
measure $\mu_F$, to which the trajectory averages converge.
The support of $\mu_F$ is a nonempty $F$-invariant subset of
$F\Gamma/\Gamma$. Since $F$ acts transitively on this orbit,
$\mu_F$ has full support there.

We may take $F$ connected. Indeed, $\Psi_g(\R)$ is connected
and contains $\groupid$, so it lies in $F^\circ$.
The subset $C=F^\circ\Gamma/\Gamma$ is open and closed in
$F\Gamma/\Gamma$: its preimage under the quotient map from $F$
is $F^\circ(F\cap\Gamma)$, a union of connected components of $F$.
All trajectory averages are supported on $C$, so closedness
of $C$ implies that their limiting probability $\mu_F$ is
also supported there. Full support of $\mu_F$ therefore gives
$C=F\Gamma/\Gamma$, equivalently $F^\circ\Gamma=F\Gamma$.
Replacing $F$ by $F^\circ$ preserves both the orbit and
the containment $\Psi_g(\R)\subset F$.

Since $J_g\subset F$, convergence to a measure with full support
gives
\[
 \overline{\{\Psi_g(t)\Gamma:t\ge0\}}
 =\overline{J_g\Gamma/\Gamma}=F\Gamma/\Gamma.
\]
The arithmetic identification
\cite[Proposition 3.2]{Shah1991} now gives
$F=\mathbf H_g(\R)^\circ$ and shows that the radical of
$\mathbf H_g$ is unipotent.

Let $\mathbf H_g^\circ$ denote the algebraic identity component
of $\mathbf H_g$. By Lemma~\ref{app:components}, it is defined
over $\mathbb Q$. Since $J_g$ is connected, it lies in
$\mathbf H_g^\circ(\R)$. The minimality of $\mathbf H_g$
among $\mathbb Q$-algebraic subgroups containing $J_g$ therefore
gives $\mathbf H_g=\mathbf H_g^\circ$, proving that
$\mathbf H_g$ is algebraically connected.

Since the radical of $\mathbf H_g$ is unipotent, it equals
$R_u(\mathbf H_g)$. The quotient
$\mathbf H_g/R_u(\mathbf H_g)$ therefore has trivial radical
and is semisimple; see Proposition~\ref{app:levi-decomposition}.

Left translation by $g$ proves \eqref{eq:actualsupport}.
\end{proof}

\begin{lemma}
\label{lem:polynomial-functions}
Let $\mathbf H\le\SL_n$ be connected and defined over $\mathbb Q$,
and put $H=\mathbf H(\R)^\circ$. For an integer $m\ge0$, define
\[
 W_m=\{P|_H:
       P\in\R[x_{ij}:1\le i,j\le n],\ \deg P\le m\}.
\]
Thus $W_m$ consists of functions on $H$ represented by polynomials
of total degree at most $m$ in the matrix entries. Polynomials
agreeing on $H$ represent the same element of $W_m$; the degree
bound is imposed on an ambient representative.

This space has a rational structure given by the restrictions of
polynomials with rational coefficients. Right translation
$R(g)p(b)=p(bg)$ preserves $W_m$. In a rational basis, its matrix
entries are polynomials with rational coefficients of degree at
most $m$ in the entries of $g$. The integer span $M_m$ of the
restrictions to $H$ of matrix-entry monomials of degree at most
$m$ is a full lattice in $W_m$, preserved by
$R(H\cap\SL_n(\mathbb Z))$. Any two lattices obtained as integer
spans of rational bases of $W_m$ are commensurable, meaning that
their intersection has finite index in each.
\end{lemma}
\begin{proof}
By Lemma~\ref{app:components}, $H$ is Zariski dense in $\mathbf H$.
Consequently, the kernel of restriction to $H$ consists exactly
of the ambient polynomials of degree at most $m$ that vanish on
$\mathbf H$. Since $\mathbf H$ is defined over $\mathbb Q$, these
relations are the real span of relations with rational coefficients.
The quotient by this kernel therefore gives $W_m$ its rational
structure.

Substitution $b\mapsto bg$ preserves degree in the entries of $b$.
Expanding a monomial of degree at most $m$ gives coefficients
that are polynomials of degree at most $m$ in the entries of $g$.
Expressing the resulting monomials in a fixed rational basis of
$W_m$ uses rational linear combinations, which proves the assertion
about the translation matrices.

The monomial functions defining $M_m$ are a finite rational spanning
set. In a rational basis, their coordinates have a common denominator,
so their integer span is discrete. Since they span $W_m$, this is a
full lattice. If $g\in H\cap\SL_n(\mathbb Z)$, both $g$ and $g^{-1}$
have integral entries. Substitution by either matrix expresses each
monomial as an integer combination of monomials of degree at most
$m$, so $R(g)M_m=M_m$.

Finally, let $L$ and $L'$ be the integer spans of two rational bases.
The change-of-basis matrix and its inverse have rational entries.
Clearing their denominators gives integers $a,a'\ge1$ such that
$aL'\subset L$ and $a'L\subset L'$. Thus $L\cap L'$ contains
$aL'$ and $a'L$, which have finite index in $L'$ and $L$,
respectively. This proves commensurability.
\end{proof}

For $\mathbf H\le\SL_3$ and $m\ge1$, the normalized test space
of Definition~\ref{def:polynomialtestspaces} is
\[
 \mathscr P_{\mathbf H}^{(m)}
 =\{p\in W_m:p(\groupid)=0\},\qquad
 W_m=\R1\oplus\mathscr P_{\mathbf H}^{(m)}.
\]
Right translation need not preserve $\mathscr P_{\mathbf H}^{(m)}$.
In particular, the coefficient functions for degree cutoff $2$ in
Lemma~\ref{lem:intrinsiccoefficients} are taken in the full
space $W_2=\R1\oplus\mathscr P_{\mathbf H}$.
If a function $p$ on $H$ is
invariant under every right translation, evaluation at $\groupid$
gives $p(b)=p(\groupid)$ for every $b\in H$. Hence a normalized
invariant function is zero, as used in
Corollary~\ref{cor:intrinsicfullgeneration}.

Conjugation by an integral matrix is linear in the matrix entries
and fixes $\groupid$. It therefore preserves the degree bound,
rational coefficients, and normalization, giving the rational
isomorphism of test spaces used in the following lemma.

\begin{lemma}
\label{lem:intrinsicconjugacy}
The existence of finite intrinsic type for $(\Phi,x)$ and its
exponent $\tau$ do not depend on the lift $g$ of $x$ or the rational
basis used for $\mathscr P_{\mathbf H_g}$. The constant $c_0$ may change.
\end{lemma}
\begin{proof}
Let $g'=g\gamma$ be another lift of $x$, with $\gamma\in\Gamma$.
Then $\Psi_{g'}(t)=\gamma^{-1}\Psi_g(t)\gamma$, so
\[
 J_{g'}=\gamma^{-1}J_g\gamma,\qquad
 \mathbf H_{g'}=\gamma^{-1}\mathbf H_g\gamma.
\]
The second equality follows from minimality over $\mathbb Q$,
since conjugation by $\gamma$ is defined over $\mathbb Q$.

Define the pullback
\[
 T_\gamma:\mathscr P_{\mathbf H_{g'}}
          \longrightarrow\mathscr P_{\mathbf H_g},
 \qquad
 (T_\gamma p)(b)=p(\gamma^{-1}b\gamma).
\]
Conjugation is linear in the matrix entries and fixes $\groupid$,
so this map preserves the degree cutoff and the normalization
$p(\groupid)=0$. It is a rational linear isomorphism, with inverse
given by pullback under $b\mapsto\gamma b\gamma^{-1}$.
Conjugation carries right multiplication by $j\in J_g$ to
right multiplication by $\gamma^{-1}j\gamma\in J_{g'}$.
Thus the pullback and its inverse give
$T_\gamma\mathscr K_{J_{g'},\groupid}
       =\mathscr K_{J_g,\groupid}.$

Let $\mathscr L$ and $\mathscr L'$ be the lattices chosen in
$\mathscr P_{\mathbf H_g}$ and $\mathscr P_{\mathbf H_{g'}}$,
respectively. Write $\norm{\cdot}_g$ and $\norm{\cdot}_{g'}$
for their chosen Euclidean norms, and
$\operatorname{dist}_g$ and $\operatorname{dist}_{g'}$
for the corresponding distances.
The matrix of $T_\gamma$ in the chosen rational bases has
rational entries. Clearing its denominators gives an integer
$m\ge1$ such that
$mT_\gamma\mathscr L'\subset\mathscr L.$
This is the denominator-clearing argument at the end of the proof
of Lemma~\ref{lem:polynomial-functions}.
Choose $C_\gamma\ge1$ such that
$\norm{T_\gamma p}_g\le C_\gamma\norm{p}_{g'}$ for all $p$.

Suppose that $(J_g,\groupid)$ has type $(c_0,\tau)$.
For $0\ne p\in\mathscr L'$, the vector $mT_\gamma p$ is a
nonzero element of $\mathscr L$. Consequently,
\[
\begin{aligned}
 \operatorname{dist}_{g'}(p,\mathscr K_{J_{g'},\groupid})
 \ge
 \frac{\operatorname{dist}_g
       (mT_\gamma p,\mathscr K_{J_g,\groupid})}
      {mC_\gamma}\ge
 \frac{c_0}{mC_\gamma}\norm{mT_\gamma p}_g^{-\tau}\ge
 c_0(mC_\gamma)^{-(\tau+1)}\norm{p}_{g'}^{-\tau}.
\end{aligned}
\]
To see the first inequality, for every $k\in\mathscr K_{J_{g'},\groupid}$, the norm bound for
$T_\gamma$ gives $\norm{mT_\gamma p-mT_\gamma k}_g
 \le mC_\gamma\norm{p-k}_{g'}.$
Since $T_\gamma\mathscr K_{J_{g'},\groupid}
=\mathscr K_{J_g,\groupid}$ and these are linear subspaces,
$mT_\gamma\mathscr K_{J_{g'},\groupid}
 =\mathscr K_{J_g,\groupid}.$
Taking infima over $k\in\mathscr K_{J_{g'},\groupid}$ therefore yields $\operatorname{dist}_g
   (mT_\gamma p,\mathscr K_{J_g,\groupid})
 \le mC_\gamma
   \operatorname{dist}_{g'}
   (p,\mathscr K_{J_{g'},\groupid}).$

Hence $(J_{g'},\groupid)$ has the same
admissible exponent $\tau$, with a possibly smaller constant.
Applying the same calculation to $T_\gamma^{-1}$ proves
the converse.

To compare two rational bases of $\mathscr P_{\mathbf H_g}$,
write $\mathscr L_1,\mathscr L_2$ for their integer spans and
$\norm{\cdot}_1,\norm{\cdot}_2$ for the corresponding Euclidean
norms. The subspace
$\mathscr K=\mathscr K_{J_g,\groupid}$ is the same for both choices.
Since the change-of-basis matrix has rational entries, there is
an integer $m\ge1$ such that $m\mathscr L_2\subset\mathscr L_1$.
Choose $C\ge1$ such that
$\norm v_1\le C\norm v_2$ for every $v\in\mathscr P_{\mathbf H_g}$.

Suppose the intrinsic condition holds for the first basis:
\[
 \operatorname{dist}_1(q,\mathscr K)
 \ge c_0\norm q_1^{-\tau}
 \qquad(0\ne q\in\mathscr L_1).
\]
For $0\ne p\in\mathscr L_2$, we have
$0\ne mp\in\mathscr L_1$. Norm comparison and the identity
$m\mathscr K=\mathscr K$ therefore give
\[
 \operatorname{dist}_2(p,\mathscr K)
 \ge \frac{\operatorname{dist}_1(mp,\mathscr K)}{mC}
 \ge \frac{c_0\norm{mp}_1^{-\tau}}{mC}
 \ge c_0(mC)^{-(\tau+1)}\norm p_2^{-\tau}.
\]
Thus the intrinsic condition holds for the second basis with
the same exponent $\tau$. The inverse change-of-basis matrix
also has rational entries, so interchanging the two bases proves
the converse. Hence the admissible exponent $\tau$ is independent
of the rational basis, although the constant $c_0$ may change.
\end{proof}

\subsection{From intrinsic type to coefficient bounds}
\label{subsec:intrinsicproof}

For use at a representative $h\in H$, we extend the notation of
Definition~\ref{def:intrinsicdiophantine}. Given a connected
rational group $\mathbf H$, put $H=\mathbf H(\R)^\circ$.
For a closed subgroup $J\le H$ and $h\in H$, define
\begin{equation}\label{eq:intrinsickernel}
 \mathscr K_{J,h}
 =\{p\in\mathscr P_{\mathbf H}:
     p(bh^{-1}jh)=p(b)\ \text{for all }b\in H,\ j\in J\}.
\end{equation}
We say that $(J,h)$ has type $(c_0,\tau)$ if $0<c_0\le1$,
$\tau\ge0$, and
\begin{equation}\label{eq:intrinsicdiophantine}
 \operatorname{dist}(p,\mathscr K_{J,h})
       \ge c_0\norm p^{-\tau}
       \qquad(0\ne p\in\mathscr L_{\mathbf H}).
\end{equation}
This agrees with Definition~\ref{def:intrinsicdiophantine} when
$h=\groupid$. For a general degree cutoff $m\ge1$, the same definition
in $\mathscr P_{\mathbf H}^{(m)}$ gives
$\mathscr K_{J,h}^{(m)}$, consistently with
Definition~\ref{def:polynomialtestspaces} when $h=\groupid$.

We express the intrinsic condition using a finite coefficient map.
Its kernel is the invariant subspace in \eqref{eq:intrinsickernel},
so a lower bound for distance to the kernel gives lower bounds for
polynomial coefficients. The second lemma transfers these bounds to
representations without fixed vectors.
Both lemmas apply to any connected rational $\mathbf H\le\SL_3$.
Throughout this subsection, $H=\mathbf H(\R)^\circ$,
$\Psi\in\mathbf H(\R[t])$ is nonconstant and normalized with
$\deg\Psi\le d$, $h\in H$, and
$J=\overline{\langle\Psi(\R)\rangle}$.
For a representation $\rho$ on $V$, write
$V^H=\{v\in V:\rho(b)v=v\text{ for all }b\in H\}$.

\begin{lemma}[Coefficient criterion for intrinsic type]
\label{lem:intrinsiccoefficients}
Put $\mathscr W_{\mathbf H}=\R1\oplus\mathscr P_{\mathbf H}$, with
the orthonormal basis $1$ followed by the chosen basis of
$\mathscr P_{\mathbf H}$. Right translation defines a representation
\[
 (R(g)p)(b)=p(bg)\qquad(g,b\in H)
\]
on $\mathscr W_{\mathbf H}$. Its matrix entries are rational
polynomial functions of the entries of $g$, of degree at most $2$.
For fixed $\Psi$ and $h$ as above, let
$B_\nu:\mathscr P_{\mathbf H}\to\mathscr W_{\mathbf H}$,
$1\le\nu\le2d$, be the uniquely determined linear maps in the
polynomial expansion
\[
 R(\Psi(t)h)p-R(h)p
 =\sum_{\nu=1}^{2d}t^\nu B_\nu(p)
 \qquad(p\in\mathscr P_{\mathbf H}).
\]
Equivalently, for every $b\in H$,
 $p(b\Psi(t)h)-p(bh)
 =\sum_{\nu=1}^{2d}t^\nu\bigl(B_\nu(p)\bigr)(b).$
Thus $B_\nu(p)$ is the function on $H$ giving the coefficient of
$t^\nu$ in this pointwise expansion (the dependence of $B_\nu$
on $\Psi$ and $h$ is suppressed in the notation). Define the
finite coefficient map by
\begin{equation}\label{eq:intrinsiccoefficientmap}
 \mathscr C_{\Psi,h}:\mathscr P_{\mathbf H}
       \longrightarrow\mathscr W_{\mathbf H}^{2d},
 \qquad
 \mathscr C_{\Psi,h}(p)=(B_1(p),\ldots,B_{2d}(p)).
\end{equation}
Then $\ker\mathscr C_{\Psi,h}=\mathscr K_{J,h}$. If
\eqref{eq:intrinsicdiophantine} holds, this map is nonzero and
\begin{equation}\label{eq:intrinsiccoefficientbound}
 \norm{\mathscr C_{\Psi,h}(p)}\ge c_0\sigma\norm p^{-\tau}
       \qquad(0\ne p\in\mathscr L_{\mathbf H}),
\end{equation}
where $\sigma>0$ is the least positive singular value of
$\mathscr C_{\Psi,h}$, and its target is equipped with the orthogonal
direct-sum norm. Conversely,
a coefficient bound of this form implies the distance condition
with the same exponent and a changed constant.
If $\Psi,h$ have entries in a subfield $F\subset\R$, the matrices
representing $B_\nu$ in the chosen rational bases of
$\mathscr P_{\mathbf H}$ and $\mathscr W_{\mathbf H}$ have entries
in $F$.
\end{lemma}
\begin{proof}
Apply Lemma~\ref{lem:polynomial-functions} with $m=2$ and
$W_2=\mathscr W_{\mathbf H}$. It gives the asserted rational
polynomial entries of $R(g)$ and their degree bound. Direct
substitution gives $R(g_1)R(g_2)=R(g_1g_2)$.
Substituting $\Psi(t)h$ into the entries of $R(g)$ gives a
polynomial family of operators of degree at most $2d$ in $t$,
with constant term $R(h)$ because $\Psi(0)=\groupid$. Restricting
its positive-degree coefficient operators to $\mathscr P_{\mathbf H}$
defines the maps $B_\nu$ and hence \eqref{eq:intrinsiccoefficientmap}.
Evaluation at $b\in H$ gives the stated pointwise expansion.
If $\Psi$ and $h$ have entries in $F$, the matrices representing
these coefficient operators also have entries in $F$.

The equality $\mathscr C_{\Psi,h}(p)=0$ holds precisely when
$R(\Psi(t)h)p=R(h)p$ for every $t$. By the representation law,
this is equivalent to $R(h)p$ being fixed by every $R(\Psi(t))$,
hence by $R(J)$. Explicitly, this condition says
$p(bjh)=p(bh)$ for every $b\in H$ and $j\in J$.
Replacing $b$ by $bh^{-1}$ gives
$p(bh^{-1}jh)=p(b)$, which is exactly
\eqref{eq:intrinsickernel}.
Singular-value decomposition now gives
\[
 \sigma\operatorname{dist}(p,\mathscr K_{J,h})
 \le\norm{\mathscr C_{\Psi,h}(p)}
 \le\norm{\mathscr C_{\Psi,h}}\operatorname{dist}(p,\mathscr K_{J,h}).
\]
The condition excludes a zero map, since the nonconstant curve
ensures $\mathscr P_{\mathbf H}\ne0$. This proves both conversions.
\end{proof}

The constant terms of the coefficient functions $B_\nu(p)$ in
\eqref{eq:intrinsiccoefficientmap} are needed for the
horizontal-character tests. For exact coefficients, the lemma
determines $\mathscr K_{J,h}$ as the kernel of a finite matrix,
without first determining $J$. It does not give a procedure for
deciding exact zero relations from numerical approximations.
Rational changes of basis give equivalent norms and commensurable
lattices; clearing
a fixed denominator shows that the exponent $\tau$ is unchanged.
Likewise, replacing $h$ by $h\gamma$, $\gamma\in H\cap\Gamma$,
maps the old invariant subspace to the new one by the rational
invertible map $p\mapsto p(\,\cdot\,\gamma^{-1})-p(\gamma^{-1})$.
Thus existence of finite type and its exponent depend on the
starting point $h\Gamma$, independently of its chosen lift.

\begin{lemma}[Transfer to representation coefficient bounds]
\label{lem:intrinsictransfer}
Let $V_{\mathbb Q}$ be a finite-dimensional $\mathbb Q$-vector
space, put $V=V_{\mathbb Q}\otimes_{\mathbb Q}\R$, and fix a
$\mathbb Q$-basis $v_1,\ldots,v_r$ of $V_{\mathbb Q}$.
Let $\rho:H\to\mathrm{GL}(V)$ be a representation such that,
in this basis, each matrix entry of $\rho$ is the restriction
to $H$ of a polynomial with rational coefficients and total
degree at most two in the ambient matrix entries.

Equip $V$ with the inner product for which $v_1,\ldots,v_r$
is orthonormal, and define $V_\Z=\bigoplus_{i=1}^r\Z v_i.$ Assume that $\rho$ has no nonzero $H$-invariant vectors. Then the matrix entries of $\rho(\Psi(t)h)$ are
polynomials in $t$ of degree at most $2d$. Specifically, there exist uniquely determined coefficient operators
$A_1,\ldots,A_{2d}\in\End(V)$  such that we have the expansion
\[
 \rho(\Psi(t)h)=\rho(h)+\sum_{\nu=1}^{2d}t^\nu A_\nu,
\]
whose constant term is $\rho(h)$ because $\Psi(0)=\groupid$.
The operators $A_\nu$ depend on $\rho$, $\Psi$, and $h$.
Under \eqref{eq:intrinsicdiophantine}, there is $c_\rho>0$ such that
\begin{equation}\label{eq:intrinsicrepresentationbound}
 \left(\sum_{\nu=1}^{2d}\norm{A_\nu w}^2\right)^{1/2}
       \ge c_\rho\norm w^{-\tau}
       \qquad(0\ne w\in V_\Z).
\end{equation}
The constant is obtained from $c_0,\tau$, the coefficient map in
Lemma~\ref{lem:intrinsiccoefficients}, and fixed linear maps.
\end{lemma}
\begin{proof}
Let $\ell_i$ be the dual coordinate functionals and put
\[
 p_{i,w}(b)=\ell_i(\rho(b)w)-\ell_i(w).
\]
By the rational structure in Lemma~\ref{lem:polynomial-functions},
a fixed integer $Q\ge1$ makes $Qp_{i,w}\in\mathscr L_{\mathbf H}$
for every integral $w$, and $\norm{p_{i,w}}\le C\norm w$.
For each $w\ne0$ at least one $p_{i,w}$ is nonzero, since otherwise
$w$ would be $H$-fixed (Section~\ref{app:lie}). Choose such an
index $i$. The homomorphism property of $\rho$ gives
\[
 p_{i,w}(b\Psi(t)h)-p_{i,w}(bh)
 =\ell_i\bigl(\rho(b)(\rho(\Psi(t)h)-\rho(h))w\bigr).
\]
Comparing the coefficient functions in Lemma~\ref{lem:intrinsiccoefficients}
with the operator expansion above yields
\[
 \bigl(B_\nu(p_{i,w})\bigr)(b)
 =\ell_i(\rho(b)A_\nu w)
 \qquad(b\in H,\ 1\le\nu\le2d).
\]
The norm of this function in $\mathscr W_{\mathbf H}$ is at most
a fixed constant times $\norm{A_\nu w}$. Applying
\eqref{eq:intrinsiccoefficientbound} to $Qp_{i,w}$ therefore gives
\[
 C Q\left(\sum_\nu\norm{A_\nu w}^2\right)^{1/2}
 \ge c_0\sigma\norm{Qp_{i,w}}^{-\tau}
 \ge c_0\sigma(QC\norm w)^{-\tau},
\]
which proves the assertion.
\end{proof}

\begin{remark}
\label{rem:coefficientlift}
In the notation of Lemma~\ref{lem:intrinsictransfer}, let
$\gamma\in H\cap\Gamma$ satisfy $\rho(\gamma)V_\Z=V_\Z$.
The matrices $h$ and $h\gamma$ represent the same point of
$H\Gamma/\Gamma$. The coefficient operators for the latter
representative are given by
\[
 \rho(\Psi(t)h\gamma)
 =\rho(h\gamma)+\sum_{\nu=1}^{2d}t^\nu A'_\nu,
 \qquad A'_\nu=A_\nu\rho(\gamma).
\]
For $0\ne w\in V_\Z$, the vector $\rho(\gamma)w$ is again a
nonzero element of $V_\Z$. Applying the original coefficient bound
to this vector and using
$\norm{\rho(\gamma)w}\le\norm{\rho(\gamma)}_{\mathrm{op}}\norm w$
gives the same exponent $\tau$, with constant
$c_\rho\max\{1,\norm{\rho(\gamma)}_{\mathrm{op}}\}^{-\tau}$.
For a finite nonempty family $\mathcal R$ satisfying this lattice condition
and sharing a separation constant $c$, one may therefore use
\[
 c\min_{\rho\in\mathcal R}
       \max\{1,\norm{\rho(\gamma)}_{\mathrm{op}}\}^{-\tau}.
\]
The calculation with $\gamma^{-1}$ gives the converse, so the
admissible exponent is independent of the representative under
these lattice-preservation hypotheses.
\end{remark}

The individual support estimates below are stated using coefficient
bounds. Lemmas~\ref{lem:intrinsiccoefficients} and
\ref{lem:intrinsictransfer} show that these bounds follow from the
intrinsic hypothesis of Theorem~\ref{thm:intrinsic}.
The representation degrees are checked where each representation is
introduced. The degree bounds for the ambient representations are
verified in Lemma~\ref{lem:degreecutoff}.

\subsection{Derivative norms}

The chain rule gives $D_ZF_b=(D_{\Ad(b)Z}F)_b$,
so each fixed-order derivative norm of $F_b$ is bounded by
a constant depending on $b$ and the order times the
corresponding norm of $F$.

For a closed orbit $Y$ and directions
$W=(W_1,\ldots,W_r)$, define
\[
 \mathcal N_{W,s,Y}(f)
 =\sum_{j=0}^s\ \sum_{i_1,\ldots,i_j=1}^r
       \sup_{y\in Y}|D_{W_{i_1}}\cdots D_{W_{i_j}}f(y)|.
\]
The term $j=0$ is $\sup_Y|f|$.
If $W_i=\sum_a w_{ia}Z_a$ and $b_W=1+\sum_{i,a}|w_{ia}|$, then
\begin{equation}\label{eq:intrinsicpointwise}
 \mathcal N_{W,s,Y}(f)\le(s+1)b_W^s\norm f_{C_G^s}.
\end{equation}
This follows by expanding each ordered derivative. Supremum norms
make restriction to a proper orbit explicit.

The space $C_b^s$ consists of functions whose derivatives of orders
at most $s$, in the indicated Lie algebra directions, are continuous
and bounded; $C_c^\infty$ denotes smooth compactly supported functions.

Let $\SY$ and $\SR$ be the fixed Sobolev norms in
Theorem~\ref{thm:yang} and Theorem~\ref{prop:regularinput}, and put
\begin{equation}\label{eq:allnorm}
 \Qall(f)=\SY(f)+\SR(f)+\norm f_{C_G^1},\qquad
 \Q(f)=\SY(f)+\norm f_{C_G^1}.
\end{equation}
The norms $\SY,\SR$ and their derivative orders are the ones in
Theorem~\ref{thm:yang} and Theorem~\ref{prop:regularinput}. They do
not vary with $\Phi,g,c,\tau,T$. In each norm, a finite
collection of Lie derivatives is measured in weighted $L^2$ or
supremum norms, with weights bounded on compact sets. Thus there is
a fixed integer $s_*$ such that, for every compact $\Omega\subset X$,
\begin{equation}\label{eq:localsobolev}
 \Qall(f)\le C_\Omega\norm f_{C_G^{s_*}}
 \qquad\bigl(f\in C_c^\infty(X),\ \operatorname{supp}f\subseteq\Omega\bigr).
\end{equation}
Indeed the weights are bounded on $\Omega$ and each $L^2$ norm is
at most its supremum times the square root of the volume of $\Omega$.
Adding lower-order derivative terms to either norm preserves every
estimate in which that norm appears on the right-hand side.

\section{Algebraic lemmas and averaging estimates}
\label{sec:tools}

\subsection{The leading tangent and its normalization}

The logarithmic derivative $B(t)=\Psi'(t)\Psi(t)^{-1}$ records the
velocity after translating $\Psi(t)$ back to the identity. Its
highest coefficient controls the increment on the shrinking scale
$t^{-q}$, which is why we call $N$ the leading tangent. In dimension
three, a nonzero nilpotent is \emph{singular} when $N^2=0$ (rank one)
and \emph{regular} when $N^2\ne0$ (one Jordan block of size three),
as in Lemma~\ref{app:nilpotent-charts}.

The following is the matrix-polynomial form of Shah's lemma
\cite[Lemma 4.1]{Shah1994}.

\begin{lemma}[Unipotent tangent lemma]
\label{lem:leading}
Let $n\in\{2,3\}$ and let $\Psi\in\SL_n(\R[t])$ be nonconstant,
with $\Psi(0)=\groupid$ and $\deg\Psi\le d$. Write
\begin{equation}\label{eq:logderivative}
 B(t)=\Psi'(t)\Psi(t)^{-1}
      =t^qN+\sum_{j=0}^{q-1}t^jN_j,\qquad N\ne0.
\end{equation}
Then $0\le q\le nd-1$ and $N^{d+1}=0$. For $q\ge1$,
\begin{equation}\label{eq:tangentlimit}
 \lim_{t\to+\infty}\Psi(t+st^{-q})\Psi(t)^{-1}=\exp(sN).
\end{equation}
For $q=0$ one has $\Psi(t)=\exp(tN)$ and the increment identity
is exact. In particular, full generation of $\SL_n(\R)$ excludes
$q=0$.
\end{lemma}

\begin{proof}
The cofactor formula gives $\deg\Psi^{-1}\le(n-1)d$, hence
$\deg B\le nd-1$. For $F_j=\Psi^{(j)}\Psi^{-1}$,
\begin{equation}\label{eq:recurrence}
 F_1=B,\qquad F_{j+1}=F_j'+F_jB.
\end{equation}
If $q\ge1$, induction gives $\deg F_j\le jq$ with coefficient
$N^j$ at $t^{jq}$. Since $F_{d+1}=0$, this gives $N^{d+1}=0$.
The exact Taylor formula
\begin{equation}\label{eq:taylor}
 \Psi(t+st^{-q})\Psi(t)^{-1}
      =I_n+\sum_{j=1}^d\frac{s^j}{j!}t^{-jq}F_j(t)
\end{equation}
proves the asserted limit. If $q=0$, solve $\Psi'=N\Psi$ with
$\Psi(0)=\groupid$. Polynomiality of $\exp(tN)$ again gives
$N^{d+1}=0$, and its image is a single unipotent one-parameter
subgroup.
\end{proof}

We shall use the following normalization of a nonzero nilpotent matrix in later sections:
\begin{lemma}
\label{lem:normalization}
Let $n\in\{2,3\}$ and let
$0\ne N\in\mathfrak{sl}_n(\R)$ be nilpotent.
There exists $b_0\in\SL_n(\R)$ such that
\begin{equation}\label{eq:mainnormalizers}
 b_0^{-1}Nb_0=
 \begin{cases}
 \varepsilon E_{12},\quad\varepsilon\in\{1,-1\},&n=2,\\
 E_{23},&n=3,\ N^2=0,\\
 E_{12}+E_{23},&n=3,\ N^2\ne0.
 \end{cases}
\end{equation}
\end{lemma}
\begin{proof}
For $n=2$, choose $v=Nw\ne0$, let
$\varepsilon=\operatorname{sgn}\det(v,w)$, and take
$b_0=|\det(v,w)|^{-1/2}(\varepsilon v,w)$.
Then $\det b_0=1$ and $b_0^{-1}Nb_0=\varepsilon E_{12}$.
For $n=3$ and $N^2=0$, choose $v_3$ with $Nv_3\ne0$, put
$v_2=Nv_3$, and choose $v_1\in\ker N\setminus\R v_2$.
Rescale $v_1$ to give $b_0=(v_1,v_2,v_3)$ determinant one.
For $N^2\ne0$, choose $v$ with $N^2v\ne0$ and put
$b_*=(N^2v,Nv,v)$ and $b_0=(\det b_*)^{-1/3}b_*$, using the
real cube root. The respective conjugates are $E_{23}$ and
$E_{12}+E_{23}$.
\end{proof}

\subsection{Coefficient growth and tangent normalization}

For an integer $D\ge1$, define the coefficient-comparison constant
\begin{equation}\label{eq:vandermondeconstant}
 \kappa_D=\min\{1,[(D+1)^{1/2}\norm{V_D^{-1}}_{\mathrm{op}}]^{-1}\},
 \qquad V_D=((1+i/D)^j)_{0\le i,j\le D}.
\end{equation}
Applying $V_D^{-1}$ to vector-valued evaluations on $[1,2]$
shows that, for arbitrary vectors $v_0,\ldots,v_D$,
\[
 \sup_{1\le r\le2}\left\|\sum_{j=0}^D r^jv_j\right\|
 \ge\kappa_D\left(\sum_{j=0}^D\|v_j\|^2\right)^{1/2}.
\]
The same comparison holds in every Euclidean dimension. The
following lemma converts coefficient separation into growth
uniformly over nonzero lattice vectors.

\begin{lemma}
\label{lem:separatedgrowth}
Let $D\ge1$ and let $F(t)=R+\sum_{j=1}^D t^j A_j$ be a polynomial endomorphism of
a Euclidean space with specified lattice $\Z^n$, where $R$ is
invertible. Suppose $c>0$, $\tau\ge0$, and
\[
 \left(\sum_{j=1}^D\norm{A_jw}^2\right)^{1/2}
      \ge c\norm w^{-\tau}\qquad(0\ne w\in\Z^n).
\]
Put $\alpha=(\tau+1)^{-1}$ and $m=s_{\min}(R)>0$.
With $\kappa_D$ defined in \eqref{eq:vandermondeconstant},
\begin{equation}\label{eq:separatedgrowth}
 \sup_{t\in[T,2T]}\norm{F(t)w}
 \ge \kappa_D m^{1-\alpha}c^\alpha T^\alpha
       \qquad(T\ge1,\ 0\ne w\in\Z^n).
\end{equation}
\end{lemma}
\begin{proof}
Apply polynomial norm comparison to $F(Tr)w$ on $1\le r\le2$,
retaining both the constant and nonconstant coefficients. With
$u=\norm w$, its supremum is at least
\[
 \kappa_D\max\{mu,Tcu^{-\tau}\}
 \ge\kappa_D(mu)^{1-\alpha}(Tcu^{-\tau})^\alpha
 =\kappa_Dm^{1-\alpha}c^\alpha T^\alpha.
\]
The constants in this bound are independent of $T$ and $w$.
\end{proof}

The following lemma records how coefficient bounds change when
the leading tangent is put into the normal form of
Lemma~\ref{lem:normalization}.

\begin{lemma}
\label{lem:coefficientnormalization}
Let $L\le\SL_n(\R)$ be a matrix group, let
$\Psi:\R\to L$ be a matrix-polynomial curve with
$\Psi(0)=\groupid$, and let $g,b_0\in L$.
Let $\mathcal R$ be a finite nonempty family of representations
$\rho:L\to\mathrm{GL}(V_\rho)$ on nonzero finite-dimensional
Euclidean spaces, with specified full lattices $V_{\rho,\Z}$.
For each $\rho$, fix an integer $D_\rho\ge1$ and suppose
$\rho(\Psi(t)g)$ is polynomial in $t$ of degree at most $D_\rho$. Put
$P(t)=b_0^{-1}\Psi(t)b_0$ and $h=b_0^{-1}g$, and define the
coefficient operators by
\[
\begin{aligned}
 \rho(\Psi(t)g)
 &=\rho(g)+\sum_{j=1}^{D_\rho}t^jA_{\rho,j},\\
 \rho(P(t)h)
 &=\rho(h)+\sum_{j=1}^{D_\rho}t^j\widetilde A_{\rho,j}.
\end{aligned}
\]
Then
\begin{equation}\label{eq:coefficientnormalization}
 P(t)h=b_0^{-1}\Psi(t)g,\qquad
 \widetilde A_{\rho,j}=\rho(b_0^{-1})A_{\rho,j}.
\end{equation}
Suppose $c>0$ and $\tau\ge0$ satisfy
\[
 \left(\sum_{j=1}^{D_\rho}\norm{A_{\rho,j}w}^2\right)^{1/2}
 \ge c\norm w^{-\tau}
 \qquad(\rho\in\mathcal R,\ 0\ne w\in V_{\rho,\Z}).
\]
The normalized coefficients $\widetilde A_{\rho,j}$ satisfy the same bound with $c$
replaced by
\[
 c_*=c\min\left\{1,\min_{\rho\in\mathcal R}
                s_{\min}(\rho(b_0^{-1}))\right\}>0.
\]
In particular, the admissible separation exponent $\tau$ is
preserved by this normalization.
\end{lemma}
\begin{proof}
The first identity in \eqref{eq:coefficientnormalization} follows
from the definitions of $P$ and $h$. Applying $\rho$ gives
$\rho(P(t)h)=\rho(b_0^{-1})\rho(\Psi(t)g)$, which proves
polynomiality of the normalized family and the coefficient identity.
For every $\rho\in\mathcal R$ and $0\ne w\in V_{\rho,\Z}$,
\[
 \left(\sum_j\norm{\widetilde A_{\rho,j}w}^2\right)^{1/2}
 \ge s_{\min}(\rho(b_0^{-1}))
       \left(\sum_j\norm{A_{\rho,j}w}^2\right)^{1/2}
 \ge c_*\norm w^{-\tau}.
\]
%Applying the same argument to the inverse coordinate change proves the converse with a possibly different constant.
\end{proof}

\subsection{Averaging the tangent and summing dyadic intervals}

For a normalized curve $P$ in either dimension, let $q$ be its
logarithmic-derivative degree and define
\begin{equation}\label{eq:incrementM}
 M(P)=1+\sum_{j=1}^d\frac1{j!}
       \sum_{k=0}^{jq-1}
       \norm{[t^k]\bigl(P^{(j)}(t)P(t)^{-1}\bigr)}_{\mathrm{op}}.
\end{equation}
The symbol $[t^k]$ extracts a coefficient. Empty inner sums are zero,
so $M(P)=1$ when $q=0$.

We first estimate polynomial averages on intervals $[T,2T]$,
where the tangent approximation and coefficient growth estimates
apply. Lemma~\ref{lem:smoothing} compares the polynomial average
on such an interval with an average of shorter unipotent flow
segments based at the moving points $P(t)y$, with an explicit
error bound. Once the subsequent arguments give a power estimate
on these intervals above an initial threshold,
Lemma~\ref{lem:dyadic} converts it into the desired estimate on
$[0,T]$, preserving the decay exponent and making the dependence
on that threshold explicit.

\begin{lemma}[Unipotent smoothing on a matrix subgroup]
\label{lem:smoothing}
Let $\mathcal G\le\SL_n(\R)$ be a fixed connected closed matrix
subgroup, $n\in\{2,3\}$, and fix a basis of its Lie algebra.
Let $P(t)\in\mathcal G$ be a normalized matrix-polynomial curve
of degree at most $d$, with leading logarithmic-derivative
coefficient a fixed nonzero nilpotent $N_*\in\Lie(\mathcal G)$.
Put $p=1$ if $N_*^2=0$ and $p=2$ otherwise.
For any discrete $\Lambda\le\mathcal G$, define
\[
 \mathcal U_LF(y)=\frac1{|I_L|}\int_{I_L}
          F(\exp(sN_*)y)\dd s,\qquad
 I_L\in\{[0,L],[-L/2,L/2],[-L,L]\}.
\]
For $T\ge T_d:=(6d+6)^2$, $1\le L\le T$, $y\in\mathcal G/\Lambda$,
and any bounded $C^1$ function with bounded first derivatives,
\begin{equation}\label{eq:smoothing}
 \left|\frac1T\int_T^{2T}F(P(t)y)\dd t -\frac1T\int_T^{2T}\mathcal U_LF(P(t)y)\dd t\right| 
       \le C_d\norm F_{C^1} \left(M(P)T^{-1}(1+L)^{d+p}+L/T\right).
\end{equation}
The $C^1$ norm uses the fixed basis of $\Lie(\mathcal G)$;
$C_d$ may also depend on that basis, $\mathcal G$, and $N_*$,
but not on $\Lambda,P,y,L,T$. In particular the statement
applies inside the affine group and to either sign of a fixed tangent.
If $L=T^{\alpha/[2(d+p)]}$ with $0<\alpha\le1$, the right side
is at most $C_dM(P)\norm F_{C^1}T^{-1/2}$.
\end{lemma}

\begin{proof}
For $q\ge1$, \eqref{eq:recurrence} and \eqref{eq:taylor} give
\[
 \norm{P(t+st^{-q})P(t)^{-1}-\exp(sN_*)}_{\mathrm{op}}
       \le M(P)T^{-1}(1+L)^d
\]
for $t\in[T,2T]$ and $|s|\le L$. The multiplicative error
\[
 A_{t,s}=P(t+st^{-q})P(t)^{-1}\exp(-sN_*)\in\mathcal G
\]
therefore satisfies
$\norm{A_{t,s}-I_n}_{\mathrm{op}}\le
C M(P)T^{-1}(1+L)^{d+p}$.
For $A\in\mathcal G$ and $z\in\mathcal G/\Lambda$,
\begin{equation}\label{eq:leftshift}
 |F(Az)-F(z)|\le C\norm{A-I_n}_{\mathrm{op}}\norm F_{C^1}.
\end{equation}
Near the identity, the matrix logarithm of $A$ belongs to
$\Lie(\mathcal G)$; integrate along $\exp(u\log A)z$ and use
the fixed basis norm. Outside that neighborhood, use
$2\norm F_\infty$. This proves the estimate globally without
an injectivity-radius hypothesis.

The map $t\mapsto t+st^{-q}$ has derivative within
$qLT^{-q-1}<1/2$ of one and moves each endpoint by at most
$LT^{-q}$. Since $q\le3d-1$, its normalized endpoint and
Jacobian errors are at most $C_d\norm F_\infty LT^{-q-1}$.
Average over $I_L$, using \eqref{eq:leftshift} with
$z=\exp(sN_*)P(t)y$ for the Taylor error.
If $q=0$, the increment is exactly $\exp(sN_*)$ and the
endpoint error is $O(\norm F_\infty L/T)$.
This proves \eqref{eq:smoothing} for all three intervals.
Finally,
\[
 T^{-1}(1+L)^{d+p}\le2^{d+p}T^{-1+\alpha/2}
       \le2^{d+p}T^{-1/2},\qquad L/T\le T^{-3/4}.
\]
\end{proof}

\begin{lemma}[Dyadic summation]
\label{lem:dyadic}
Let $y:[0,\infty)\to\mathcal Y$ be a measurable trajectory, let
$\mu$ be a probability measure on $\mathcal Y$, and let $F$ be a
bounded measurable function with $\mathcal N(F)\ge\norm F_\infty$.
Suppose $D\ge0$, $T_*\ge1$, $0<\delta\le1/2$, and
\[
 \left|\frac1S\int_S^{2S}F(y(t))\dd t-\mu(F)\right|
 \le D\mathcal N(F)S^{-\delta}\qquad(S\ge T_*).
\]
Then
\begin{equation}\label{eq:dyadic}
 \left|\frac1T\int_0^T F(y(t))\dd t-\mu(F)\right|
 \le C(D+T_*^\delta)\mathcal N(F)T^{-\delta}\qquad(T\ge2),
\end{equation}
where $C$ is an absolute constant independent of $\delta$.
\end{lemma}

\begin{proof}
For $T\ge2T_*$, split $[0,T]$ into backward dyadic intervals
$[T/2^{j+1},T/2^j]$ with left endpoint at least $T_*$ and an initial
remainder of length less than $2T_*$. The dyadic errors are bounded
by $D\mathcal N(F)T^{-\delta}$ times a geometric series of ratio
$2^{-(1-\delta)}\le2^{-1/2}$. The remainder contributes at most
$4(T_*/T)\norm F_\infty\le4T_*^\delta T^{-\delta}\mathcal N(F)$.
For $T<2T_*$, the trivial discrepancy bound
$2\norm F_\infty$ gives the same conclusion with a fixed larger
constant.
\end{proof}

\section{Quantitative nondivergence and equidistribution theorems}
\label{sec:inputs}

This section collects the quantitative estimates used in the proof
of Theorem~\ref{thm:intrinsic}. The quantitative nondivergence theorem
of Kleinbock and Margulis \cite[Theorem~5.2]{KM1998}, recorded in
Theorem~\ref{prop:km}, controls exceptional parameter sets where
the relevant lattices have short vectors. The ambient estimates
of Yang \cite[Theorem~1.2]{Yang2025} and Lindenstrauss, Mohammadi,
Wang, and Yang \cite[Theorem~1.1]{LMWY2025}, stated in
Theorems~\ref{thm:yang} and~\ref{prop:regularinput}, will be applied
to singular and regular unipotent tangent averages, respectively.
On compact unipotent supports, the quantitative Leibman theorem
of Green and Tao \cite[Theorem~2.9]{GreenTao2012}, recorded in
Theorem~\ref{prop:green-tao}, provides a horizontal character
obstruction. Its continuous-time form in
Lemma~\ref{lem:continuousleibman} will be combined with the intrinsic
coefficient bounds of Lemma~\ref{lem:intrinsiccoefficients} to
exclude that obstruction.

\iffalse
We state the quantitative nondivergence theorem, Theorem~\ref{prop:km}, and the equidistribution results used in the proof: Theorem~\ref{thm:yang} and Theorems~\ref{prop:regularinput}
and~\ref{prop:green-tao}. The arithmetic bounds are collected in
Lemma~\ref{lem:arithmeticinputs}.
The horocycle estimate is Theorem~\ref{thm:coverinput} in
Section~\ref{subsec:congruencecovers}, and the affine block-flow
estimate is Theorem~\ref{thm:affineblockinput} in
Section~\ref{subsec:affineblock}.
The paragraphs following
the statements explain the specializations, normalization changes,
and deductions needed for the formulations used here.
The rationality assertion in Lemma~\ref{lem:arithmeticinputs}
also uses \cite[Proposition 2.1(2)]{Shah1994}, as explained there.
The polynomial, smoothing, lattice-geometric, and integral-tensor
arguments are proved here. The regular affine argument also uses
the rational finite-field bound and the derivative test stated
in Section~\ref{sec:regularaffine}. The support reduction in
Section~\ref{sec:supportstructure} also records the algebraic-group
structure results it uses.
\fi

\subsection{Quantitative nondivergence}

Use Euclidean norms on $\R^3$ and its exterior powers. For $y=g\Gamma$,
define the exterior systole $\extmin(y)$ by
\begin{equation}\label{eq:systole}
 \extmin(y)=\min\left\{1,
   \inf_{0\ne v\in\Z^3}\norm{gv},
   \inf_{0\ne w\in\wedge^2\Z^3}\norm{(\wedge^2g)w}\right\}>0.
\end{equation}
Here $\wedge^2\Z^3$ means the integer lattice in $\wedge^2\R^3$.
The two infima are the shortest nonzero vector lengths of $g\Z^3$
and its dual lattice $g^{-\mathsf T}\Z^3$, respectively.

%the Euclidean Hodge identification intertwines $\wedge^2g$ with $g^{-\mathsf T}$.

Thus $\extmin(y)$ is the minimum of these two lengths, capped at one.
The definition is independent of the representative $g$.

A subgroup $\Delta\le\Z^n$ of rank $k$ is \emph{primitive} when
$\Delta=\operatorname{span}_{\R}(\Delta)\cap\Z^n$. If
$v_1,\ldots,v_k$ is a lattice basis of $\Delta$, its exterior
representative $w_\Delta=v_1\wedge\cdots\wedge v_k$ is a primitive
integral vector, defined up to sign (Lemma~\ref{app:plucker}).
The quantity $\norm{(\wedge^kh)w_\Delta}$ is the Euclidean covolume
of $h\Delta$ in its real span (Section~\ref{app:tensors}).
The full-rank primitive subgroup is $\Z^n$,
whose covolume is one when $h\in\SL_n(\R)$. Testing all nonzero
integral exterior vectors, as our hypotheses do, in particular tests
these decomposable representatives.

Recall from \cite{KM1998} that a nonnegative function $F$ is called $(C,\gamma)$-\emph{good} on an interval
if every bounded subinterval $I$ and $0<a<1$ satisfy
\[
 \Leb\{t\in I:F(t)<a\sup_I F\}\le Ca^\gamma|I|.
\]
This says that a function cannot be much smaller than its local
maximum on a large fraction of the interval. The following elementary
fact makes the constants uniform over all polynomial coefficients.

\begin{lemma}[Polynomial sublevel bounds]\label{lem:polynomialgoodness}
Let $D\ge1$ be an integer and let $P$ be a polynomial of degree at most $D$ with
values in a finite-dimensional Euclidean space. Then $\norm{P(t)}$
is $(C_D,1/D)$-good on $\R$, independently of the dimension and
coefficients.
More generally, if $P$ is a vector polynomial of total degree at most
$D$ on $\R^n$ and $B$ is a bounded rectangular box of positive volume, then
\begin{equation}\label{eq:boxsublevel}
 \Leb\{z\in B:\norm{P(z)}<a\sup_B\norm P\}
 \le C_{n,D}a^{1/(nD)}\Leb(B)\qquad(0<a<1).
\end{equation}
\end{lemma}
\begin{proof}
The one-variable scalar estimate is due to Kleinbock and Margulis
\cite[Proposition 3.2]{KM1998}. We recall its interpolation argument
and then prove the vector-valued and rectangular-box extensions used here.
First let $P$ be scalar and fix a compact interval $I$. Write
$S=\sup_I|P|$ and $E=\{t\in I:|P(t)|<aS\}$. The cases $S=0$
or $|E|=0$ are immediate. Otherwise put
$b=|E|/[2(D+1)]$. One can choose $D+1$ points of $E$ separated
pairwise by at least $b$: after choosing at most $D$ points, the
union of their open $b$-neighborhoods has measure less than $|E|$.
Lagrange interpolation at these points gives
\[
 S\le(D+1)aS\,(|I|/b)^D.
\]
Consequently $|E|\le2(D+1)^{1+1/D}a^{1/D}|I|$.
For a vector polynomial with nonzero maximum, choose $t_0\in I$
where $\norm{P(t_0)}$ is maximal and take the scalar polynomial
$p(t)=\langle P(t),P(t_0)/\norm{P(t_0)}\rangle$. Its absolute
maximum equals that of $\norm P$, and the vector sublevel set is
contained in the scalar one. This proves the same bound. Endpoints
do not affect the measure, so the assertion holds for all bounded
intervals.

For the box assertion, an affine change in each coordinate reduces
to $B=[-1,1]^n$. First take a scalar polynomial of supremum one.
Induct on $n$, the preceding argument giving $n=1$. Write
$P(z',t)=\sum_{j=0}^D a_j(z')t^j$ and choose $j$ with
$\sup|a_j|\ge1/(D+1)$. Put $b=a^{(n-1)/n}$.
By induction, the set where $|a_j(z')|<b/(D+1)$ has measure
$O_{n,D}(b^{1/((n-1)D)})$. Outside it, equivalence of coefficient
and supremum norms for one-variable polynomials gives
$\sup_t|P(z',t)|\ge c_Db$. The one-variable estimate bounds the
relative length of $\{t:|P(z',t)|<a\}$ by
$O_D((a/b)^{1/D})$; if this bound exceeds one, the trivial bound
has the same consequence. Fubini gives \eqref{eq:boxsublevel},
since both exponents equal $1/(nD)$. For vector polynomials,
projection onto the direction of a maximizing value reduces to the
scalar case exactly as above.
\end{proof}

We use the following one-parameter form of the quantitative
nondivergence theorem of Kleinbock and Margulis.

\begin{theorem}[{\cite[Theorem~5.2]{KM1998}}]\label{prop:km}
Fix $n\in\{2,3,6\}$, $0<\gamma\le1$, and $C\ge1$.
Let $I$ be a bounded interval of positive length, and let $3^nI$
denote its concentric dilation by $3^n$. Suppose
$h:3^nI\to\SL_n(\R)$ is continuous. For every nonzero
primitive subgroup $\Delta\le\Z^n$, let $w_\Delta$ be a primitive
exterior vector representing it. Suppose the functions
$t\mapsto\norm{(\wedge^{\rank\Delta}h(t))w_\Delta}$ are
$(C,\gamma)$-good on $3^nI$, and have supremum at least $\varrho$ on $I$,
where $0<\varrho\le1/n$. Then
\begin{equation}\label{eq:km}
 \Leb\{t\in I:\inf_{v\in\Z^n\setminus\{0\}}\norm{h(t)v}<\epsilon\}
    \ll C(\epsilon/\varrho)^\gamma |I|,
    \qquad 0<\epsilon\le\varrho.
\end{equation}
The implicit constant depends only on $n$.
\end{theorem}

\begin{proof}[Deduction]
This is the one-parameter specialization of
\cite[Theorem 5.2]{KM1998}. To compare its coordinate maximum
norms with our Euclidean norms, choose $c_n\ge1$ so that
$\|v\|_\infty\le\|v\|\le c_n\|v\|_\infty$ in every exterior
power. Euclidean goodness with constants $(C,\gamma)$ implies
maximum-norm goodness with constants $(c_n^\gamma C,\gamma)$,
and the maximum-norm suprema are at least $\varrho/c_n$.
Apply the cited theorem with that lower-bound parameter when
$\epsilon\le\varrho/c_n$. For larger $\epsilon\le\varrho$,
the trivial bound $|I|$ gives the same conclusion after increasing
the constant. Since $\gamma\le1$, all changes are bounded in
terms of $n$ alone.
We use $\gamma$ for the goodness exponent and $\varrho$ for this
lower-bound parameter, to distinguish them from the arithmetic
exponent $\alpha=(\tau+1)^{-1}$ and representation symbol $\rho$.
Lemma~\ref{lem:polynomialgoodness} supplies the goodness hypotheses
in every application below; the nondivergence bound follows from the cited theorem.
\end{proof}

\begin{lemma}[Exterior-systole nondivergence]\label{lem:qndboth}
Let $h:\R\to G$ be such that all matrix entries of $h$ and
$\wedge^2h$ are polynomials of degree at most $D\ge1$. Suppose that on a
bounded interval $I$, for $k=1,2$ and every nonzero integer exterior vector $w$,
\begin{equation}\label{eq:qndlower}
 \sup_{t\in I}\norm{(\wedge^kh(t))w}\ge1.
\end{equation}
There are constants $C_D$ and $\epsilon_0>0$ such that
\begin{equation}\label{eq:qndboth}
 \Leb\{t\in I:\extmin(h(t)\Gamma)<\epsilon\}
       \le C_D\epsilon^{1/D}|I|
       \qquad(0<\epsilon\le\epsilon_0).
\end{equation}
\end{lemma}

\begin{proof}
All primitive covolume functions are good with exponent $1/D$;
the rank-three one is the constant $1$. Apply
Theorem~\ref{prop:km} with a fixed sufficiently small $\varrho$.
Polynomiality on all of $\R$ supplies the required goodness on $27I$.
This bounds the short-vector event for $h(t)\Z^3$.

To bound the exterior-square event, apply the same theorem to
$h^*(t)=h(t)^{-\mathsf T}$. The integral Hodge identification in
dimension three gives
\[
 \wedge^2h\simeq h^{-\mathsf T},\qquad
 \wedge^2(h^{-\mathsf T})\simeq h,
\]
through fixed signed orthogonal identifications. Consequently both
polynomial degrees and both lower bounds for $h^*$ are already among
those assumed for $h$. Its short-vector event is precisely the
exterior-square event for $h$. A union bound proves
\eqref{eq:qndboth}. Norm comparisons, if needed for the formulation of
Theorem~\ref{prop:km}, only change the fixed constants.
\end{proof}

The following lemma bounds the injectivity radius below in terms
of the shortest-vector length of the lattice.
\begin{lemma}
\label{lem:injectivity}
Let $n\in\{2,3\}$, $G_n=\SL_n(\R)$, $\Gamma_n=\SL_n(\Z)$,
$\Delta\le\Gamma_n$, and $b\in G_n$. Write
$\latticemin(b\Gamma_n)=\min_{0\ne v\in\Z^n}\norm{bv}$.
For a fixed right-invariant group metric, let $\inj_\Delta(b\Delta)$
be the local injectivity radius of the orbit map, capped at one.
Then
\[
 \inj_\Delta(b\Delta)\ge
 c_n\min\{1,\latticemin(b\Gamma_n)\}^{n(n-1)}.
\]
In particular,
\begin{align}
 \inj_\Delta(b\Delta)&\ge c\min\{1,\latticemin(b\Gamma_2)^2\}
                  &&(\Delta\le\Gamma_2),\label{eq:coverinjectivity}\\
 \inj(y)&\ge c_4\extmin(y)^6 &&(y\in X).
                                                   \label{eq:injfromsystole}
\end{align}
The constants depend only on the fixed metrics, not on $\Delta$.
\end{lemma}
\begin{proof}
Put $\varrho=\min\{1,\latticemin(b\Gamma_n)\}$. Choose a shortest
primitive lattice vector, then a shortest primitive vector in each
successive orthogonal projection, and lift them to a basis. Reduce
each Gram--Schmidt coefficient to $[-1/2,1/2]$ by subtracting earlier
basis vectors. If $a_1,\ldots,a_n$ are the Gram--Schmidt lengths,
minimality in each projected lattice gives
$a_{i+1}\ge(\sqrt3/2)a_i$. Indeed, in the projection orthogonal
to the first $i-1$ vectors, a reduced lift of the next projected
vector has length at most
$\sqrt{a_{i+1}^2+a_i^2/4}$. This length cannot be less than
the shortest length $a_i$ in that projected lattice.
Since $a_1\ge\varrho$ and
$\prod_i a_i=1$, every $a_i$ lies between
$c_n\varrho$ and $C_n\varrho^{-(n-1)}$.
The reduced coefficients bound all basis vectors by the latter
quantity. Changing a sign if needed gives $\widetilde b=b\gamma$,
$\gamma\in\Gamma_n$, with
\begin{equation}\label{eq:elementaryreduction}
 \norm{\widetilde b}_{\mathrm{op}}\ll\varrho^{-(n-1)},\qquad
 \norm{\widetilde b^{-1}}_{\mathrm{op}}
                  \ll\varrho^{-(n-1)^2}.
\end{equation}
The second bound follows from the cofactor formula.

For $\groupid\ne\delta\in\Delta$, the matrix
$k=\gamma^{-1}\delta\gamma$ is integral and nonidentity. Hence
$\norm{k-I_n}_{\mathrm{op}}\ge1$ and
\[
 \norm{b\delta b^{-1}-I_n}_{\mathrm{op}}
 \ge\bigl(\norm{\widetilde b}_{\mathrm{op}}
          \norm{\widetilde b^{-1}}_{\mathrm{op}}\bigr)^{-1}
 \gg\varrho^{n(n-1)}.
\]
If two elements in a sufficiently small identity ball have the same
image in $G_n/\Delta$, their quotient is such a stabilizer element.
Local comparison of matrix distance and the fixed metric proves the
injectivity bound. For $n=3$, use
$\extmin(y)\le\min\{1,\latticemin(y)\}$ to obtain the last assertion.
\end{proof}

\subsection{Singular unipotent segments}

Set
\begin{align}
 u_{\mathrm S}(s)&=I_3+sE_{23},\label{eq:u}\\
 a_{\mathrm S}(\sigma)&=\diag(1,e^{\sigma/2},e^{-\sigma/2}),\nonumber\\
 a_0(\ell)&=\diag(e^{\ell/3},e^{-\ell/6},e^{-\ell/6}),\nonumber\\
 D_{\sigma,\ell}&=a_{\mathrm S}(-\sigma)a_0(\ell),\nonumber\\
 U_L^{\mathrm S}f(y)&=\frac1L\int_{-L/2}^{L/2}f(u_{\mathrm S}(s)y)\dd s.\label{eq:unipotentaverage}
\end{align}

Yang's theorem \cite[Theorem~1.2]{Yang2025} asserts that at least one
of two alternatives holds:
an effective equidistribution estimate or a short standard or
exterior-square lattice vector after one of the specified diagonal
translations.
We record the sufficient condition obtained by excluding the second
alternative, with constants adjusted to our norm conventions.

\begin{theorem}[{Singular unipotent flow estimate;
  after \cite[Theorem~1.2]{Yang2025}}]
\label{thm:yang}
There are absolute constants $\etaY>0$, $\sigma_Y>1$, $C_Y>0$,
$b_Y\ge2$, and a fixed Sobolev norm $\SY$ such that the following holds.
Let $\sigma\ge\sigma_Y$, $L=e^\sigma$, and $y\in X$. If
\begin{equation}\label{eq:yangcondition}
 \extmin(D_{\sigma,\ell}y)\ge b_Ye^{-|\ell|/3}
 \quad\hbox{for every }\ell\in\R\hbox{ with }
        0.99\sigma\le|\ell|\le1.01\sigma,
\end{equation}
then
\begin{equation}\label{eq:yangestimate}
 \abs{U_L^{\mathrm S}f(y)-m_X(f)}\le C_YL^{-\etaY}\SY(f)
 \qquad(f\in C_c^\infty(X)).
\end{equation}
\end{theorem}

\begin{proof}[Deduction]
Take $\sigma_Y$ larger than the starting threshold in
\cite[Theorem~1.2]{Yang2025}. Choose $b_Y\ge2$ strictly larger
than the constants comparing our Euclidean norms on $\R^3$ and
$\wedge^2\R^3$ with the corresponding norms in Yang's theorem.
Under \eqref{eq:yangcondition}, every nonzero vector in either
translated lattice then has norm strictly greater than
$e^{-|\ell|/3}$ in Yang's conventions, for every specified $\ell$.
This excludes the second alternative of his theorem. The first
alternative, with $s=e^\sigma r$ and $r\in[-1/2,1/2]$, gives
\eqref{eq:yangestimate}. The exponent $\etaY$ and the Sobolev norm
$\SY$ are those supplied by Yang's theorem; $C_Y$ incorporates
the fixed normalization constants.
\end{proof}

\subsection{Regular expanding translates}

For the regular tangent, set
\begin{equation}\label{eq:regularnormalization}
 N=N_0=E_{12}+E_{23},\qquad
 v(s)=\exp(sN_0),\qquad a_s^{\mathrm R}=\diag(e^s,1,e^{-s}).
\end{equation}
Set
\begin{equation}\label{eq:regularform}
 Q_0=\begin{pmatrix}0&0&1\\0&-1&0\\1&0&0\end{pmatrix},\qquad
 H_{\mathrm{pr}}=\{h\in G:hQ_0h^{\mathsf T}=Q_0\}^{\circ}.
\end{equation}
Here $Q_0$ is a symmetric tensor, with action
$Q\mapsto hQh^{\mathsf T}$ (Section~\ref{app:tensors}).
By Lemma~\ref{app:symmetric-square}, $H_{\mathrm{pr}}$ is the image
of a principal $\SL_2(\R)$ and contains $v(s)$ and
$a_s^{\mathrm R}$. For example,
$N_0Q_0+Q_0N_0^{\mathsf T}=0$ and
$a_s^{\mathrm R}Q_0(a_s^{\mathrm R})^{\mathsf T}=Q_0$.
The associated $\mathfrak{sl}_2$ triple has raising operator $N_0$,
lowering operator $2(E_{21}+E_{32})$, and diagonal element
$2\diag(1,0,-1)$. Thus the normalization of $a_s^{\mathrm R}$ agrees
with the image of $\diag(e^{s/2},e^{-s/2})$ in Theorem~1.1 of
Lindenstrauss, Mohammadi, Wang, and Yang.
In particular,
\[
 a_s^{\mathrm R}v(r)a_{-s}^{\mathrm R}=v(e^sr).
\]
Use the one-sided average
\begin{equation}\label{eq:regularaverage}
 V_Lf(y)=\frac1L\int_0^L f(v(s)y)\dd s.
\end{equation}
The conjugation relation above gives, for every $L>0$ and $z\in X$,
\begin{equation}\label{eq:regularidentity}
 V_Lf(z)=\int_0^1 f\bigl(a_{\log L}^{\mathrm R}v(r)
                                   a_{-\log L}^{\mathrm R}z\bigr)\dd r.
\end{equation}
This identity relates the long-segment average to the expanding
average in Theorem~\ref{prop:regularinput}, with starting point
$y=a_{-\log L}^{\mathrm R}z$. The threshold
\eqref{eq:regularinputthreshold} must be verified at this point.
The singular average $U_L^{\mathrm S}$ is centered, whereas $V_L$ is one-sided.

Fix the right-invariant metric on $G$ and the quotient metric $d_X$
specified before \cite[Theorem 1.1]{LMWY2025}.
Use the capped local injectivity radius from
Lemma~\ref{lem:injectivity} for this metric. The constants below absorb
the choice of metric and the truncation of the injectivity radius at
one. The orbit volume
uses the fixed normalization of that paper.

We specialize the expanding-translate theorem of Lindenstrauss,
Mohammadi, Wang, and Yang to the principal subgroup fixed above.

\begin{theorem}[{\cite[Theorem~1.1]{LMWY2025}}]\label{prop:regularinput}
There are fixed constants $A\ge1$, $\kappa>0$, $m_0>0$, $C_0\ge1$,
and a Sobolev norm $\SR$ such that, if
\begin{equation}\label{eq:regularinputthreshold}
 R\ge C_0\max\{1,\inj(y)^{-m_0}\},\qquad L\ge R^A,
\end{equation}
then either
\begin{equation}\label{eq:regulargood}
 \left|\int_0^1 f(a_{\log L}^{\mathrm R}v(r)y)\dd r-m_X(f)\right|
       \le C_0\SR(f)R^{-\kappa}
       \quad\text{for every }f\in C_c^\infty(X),
\end{equation}
or there is $z\in X$ with
\begin{equation}\label{eq:regularbad}
 H_{\mathrm{pr}}z\text{ periodic},\qquad \vol(H_{\mathrm{pr}}z)\le R,\qquad
 d_X(y,z)\le\varepsilon,
 \quad\varepsilon=R^A(\log L)^A L^{-2}.
\end{equation}
\end{theorem}

\begin{proof}[Deduction]
This is \cite[Theorem 1.1 and (10.3)]{LMWY2025}, specialized to the
fixed $G,\Gamma,H_{\mathrm{pr}}$ above. Here the complementary
$H_{\mathrm{pr}}$-module in
$\mathfrak{sl}_3$ has dimension five, so the parameter $\mathsf m$
in that theorem equals $2$. We enlarge $A,C_0$ if needed.
The starting threshold is the explicit polynomial bound in
\eqref{eq:regularinputthreshold}.
\end{proof}

\subsection{Quantitative equidistribution on compact nilmanifolds}
\label{subsec:nilinput}

Let $H$ be a nontrivial connected closed subgroup of a conjugate
of the upper unitriangular group in $G$. Suppose that $Y=Hz$
is compact for some $z\in X$, and put
\[
 \Delta_z=\{h\in H:hz=z\}.
\]
The orbit map identifies $Y$ with the compact nilmanifold
$H/\Delta_z$. Write $\mu_Y$ for its $H$-invariant probability
measure, and use Haar probability in integrals over $H/\Delta_z$.

Lemma~\ref{app:malcev-existence} gives a Mal'cev basis
$B=(B_1,\ldots,B_m)$ adapted to the lower central series.
The general existence result is described after
\cite[Definition 2.1]{GreenTao2012}, with proofs in
\cite[Theorem 5.1.8(a), Corollary 5.2.2, and Proposition 5.3.2]{CorwinGreenleaf1990}.
The associated ordered exponential coordinates are
\[
 h=\exp(a_1B_1)\cdots\exp(a_mB_m),\qquad
 \chi_B(h)=(a_1,\ldots,a_m).
\]
They are unique and send $\Delta_z$ onto $\Z^m$. The terminal
spans of the basis are ideals, and the last $m-r$ vectors span
$[\Lie(H),\Lie(H)]$, where $1\le m\le3$ and
$r=\dim H/[H,H]$. Fix $Q_B\ge1$ bounding the numerator and
denominator of each rational structure constant in this basis.
The group, lattice, and basis are fixed in the theorem below.

By Lemma~\ref{app:horizontal}, the first $r$ coordinates define
the homomorphism $\chi_{\mathrm{ab}}:H\to\R^r$ and identify
$H/[H,H]\Delta_z$ with $\R^r/\Z^r$. Every horizontal character
is therefore
\[
 \eta_k(h)=k\cdot\chi_{\mathrm{ab}}(h)\bmod\Z,
 \qquad k\in\Z^r.
\]
These characters vanish on $\Delta_z$.

Use the Mal'cev metric $d_B$ of \cite[Definition 2.2]{GreenTao2012}:
on $H$ it is the infimum, over finite chains, of the sums of
\[
 \min\{\norm{\chi_B(h_i h_{i-1}^{-1})}_\infty,
          \norm{\chi_B(h_{i-1}h_i^{-1})}_\infty\};
\]
on $H/\Delta_z$ take the infimum over representatives. Set
$\norm F_{\mathrm{Lip},B}=\norm F_\infty+\operatorname{Lip}_{d_B}(F)$.
For $F(h\Delta_z)=f(hz)$, telescoping along ordered exponentials
and then along the chains defining $d_B$ gives
\begin{equation}\label{eq:nillipcontrol}
 \norm F_{\mathrm{Lip},B}\le(m+1)\mathcal N_{B,1,Y}(f).
\end{equation}
Indeed a left displacement with coordinates $(t_i)$ changes $f$
by at most $\sum_i|t_i|\max_i\norm{D_{B_i}f}_{\infty,Y}$.
The reverse displacement gives the other term in the minimum.

For a real polynomial $P(n)=\sum_{j=0}^{D}\alpha_j\binom nj$,
write
\[
 \norm{P\bmod\Z}_{C^\infty[N]}
     =\max_{1\le j\le D}N^j\operatorname{dist}(\alpha_j,\Z).
\]
For coordinate degree $D$, use the filtration
\begin{equation}\label{eq:nilfiltration}
 H_0=H,\qquad
 H_i=
 \begin{cases}
 H,&1\le i\le D,\\
 [H,H],&D<i\le2D,\\
 \{\groupid\},&i>2D.
 \end{cases}
\end{equation}
Its degree is at most $2D$ and $B$ is adapted to it.
Lemma~\ref{app:nil-sampling} shows that coordinate-polynomial
sequences, including $n\mapsto\Phi(sn)a$ for a fixed right
translate $a$ and real sampling step $s$, are polynomial for
this filtration. The matrix-to-coordinate degree bound is
Lemma~\ref{app:horizontal}.

We use the following specialization of the quantitative Leibman
theorem of Green and Tao.

\begin{theorem}[{\cite[Theorem~2.9]{GreenTao2012}}]\label{prop:green-tao}
For each integer $D\ge1$, there are constants
$a_D\ge\max\{2,D\}$ and $b_D\ge1$, depending only on $D$,
with the following property. Let $0<\epsilon\le(4Q_B)^{-1}$,
$N\ge2$ be an integer, and let $q:\Z\to H$ have polynomial ordered coordinates
of degree at most $D$. If for some Lipschitz $F$,
\[
 \left|\frac1N\sum_{n=1}^N F(q(n)\Delta_z)
                     -\int_{H/\Delta_z}F\right|
                 >\epsilon\norm F_{\mathrm{Lip},B},
\]
there is $0\ne k\in\Z^r$ such that
\begin{equation}\label{eq:gtobstruction}
 \norm k\le b_D\epsilon^{-a_D},\qquad
 \norm{\eta_k\circ q}_{C^\infty[N]}\le b_D\epsilon^{-a_D}.
\end{equation}
\end{theorem}

\begin{proof}[Deduction]
Lemma~\ref{app:nil-sampling} equips each such $q$ with the
filtration \eqref{eq:nilfiltration}, of degree at most $2D$,
to which the chosen basis is adapted. Apply
\cite[Theorem 2.9]{GreenTao2012} to this filtration.
The basis is $1/\epsilon$-rational since
$Q_B\le1/\epsilon$. Take maxima of the constants in Green and Tao's
Theorem~2.9 over $1\le m\le3$ and filtration degrees at most $2D$,
and enlarge them for Euclidean versus coordinate maximum norms.
The character and smoothness conventions are those of
\cite[Definitions 2.6 and 2.7]{GreenTao2012}

%\footnote{Only the one-parameter theorem is used; the authors' erratum explicitly leaves that theorem unaffected \cite[Section 1]{GreenTaoErratum}.}.

\end{proof}

\section{Reduction to model supports}
\label{sec:supportstructure}

Lemma~\ref{lem:orbitnormalization} supplies a connected rational
support group whose radical is unipotent. This section proves
the support classification and the coordinate and covering reductions
needed for the effective estimates. The general algebraic and lattice
facts used below are collected in Appendix~\ref{app:algebraic-background}.
All coordinate systems, parametrizations, and covering levels
below are fixed once the normalized support is fixed.

\subsection{The support types}

Define the two rational affine embeddings by
\begin{equation}\label{eq:affinegroups}
 \mathbf A=
 \left\{\begin{pmatrix}A&v\\0&1\end{pmatrix}:
           A\in\SL_2,\ v\in\mathbb G_a^2\right\},
 \qquad
 \mathbf A^\vee=\{b^{-\mathsf T}:b\in\mathbf A\}.
\end{equation}
Here $\mathbb G_a^2$ is the two-dimensional additive group.
The embeddings are isomorphic as algebraic groups, but are
listed separately as subgroups of $\SL_3$.
Write $A_{\mathrm{aff}}=\mathbf A(\mathbb R)$; its elements
are pairs $(A,v)$ with product
$(A,v)(B,w)=(AB,Aw+v)$.

\begin{proposition}[Classification of the normalized supports]
\label{prop:supportclassification}
Let $\mathbf H\le\SL_3$ be connected and defined over $\mathbb Q$,
and suppose $\mathbf H/R_u(\mathbf H)$ is semisimple.
Then $\mathbf H$ is one of the following:
\begin{enumerate}
\item a unipotent group;
\item $\SL_3$;
\item a rational conjugate of the block $\SL_2$;
\item $\SO(\qform)$ for a nondegenerate rational ternary quadratic form $\qform$;
\item a rational conjugate of $\mathbf A$ or $\mathbf A^\vee$.
\end{enumerate}
Rational conjugacy here means conjugacy by a matrix in
$\mathrm{GL}_3(\mathbb Q)$, and
$\SO(\qform)=\{b\in\SL_3:b^{\mathsf T}Qb=Q\}$ when
$\qform(v)=v^{\mathsf T}Qv$.
\end{proposition}
\begin{proof}
Let $V=\mathbb Q^3$ and put $\mathbf U=R_u(\mathbf H)$.
By the characteristic-zero
Levi decomposition in Proposition~\ref{app:levi-decomposition},
there is a semisimple $\mathbb Q$-subgroup $\mathbf S$ such that
\[
 \mathbf H=\mathbf S\ltimes\mathbf U.
\]
The semisimplicity of $\mathbf S$ follows from the hypothesis on
$\mathbf H/\mathbf U$; this quotient, and hence $\mathbf S$,
may be trivial. If $\mathbf S$ is trivial, then $\mathbf H=\mathbf U$
is unipotent, which is the first possibility in the statement.
We therefore assume $\mathbf S$ is nontrivial. We first classify
$\mathbf S$ and then determine which nontrivial radicals
$\mathbf U$ can occur.

\emph{The semisimple factor over $\mathbb C$.}
A maximal torus of $\mathbf S_{\mathbb C}$ is diagonal in a
suitable complex basis of $\mathbb C^3$ (Section~\ref{app:tori}).
It is therefore contained in the determinant-one diagonal torus
of $\SL_3$, so its rank is at most two. Since $\mathbf S$ is
nontrivial and semisimple, its rank is one or two.

Suppose the rank is two. The maximal torus is then the full
diagonal torus of $\SL_3$. The root spaces of
$\Lie(\mathbf S_{\mathbb C})$ are among the six lines
$\mathbb C E_{ij}$, $i\ne j$. By the semisimple root decomposition
\cite[Theorem~21.11 and Corollary~21.12]{Milne2017}, recorded in
Proposition~\ref{app:root-inputs}, the roots occur in opposite
pairs and span a two-dimensional space. Hence at least two
independent opposite pairs occur. After permuting the three
basis vectors, these pairs can be taken to be
$\mathbb C E_{12},\mathbb C E_{21}$ and
$\mathbb C E_{23},\mathbb C E_{32}$. The brackets
\[
 [E_{12},E_{23}]=E_{13},\qquad
 [E_{32},E_{21}]=E_{31}
\]
give the remaining two root spaces. Together with the diagonal
torus, they give all of $\mathfrak{sl}_3(\mathbb C)$.
Thus $\mathbf S_{\mathbb C}=\SL_{3,\mathbb C}$, and consequently
$\mathbf S=\SL_3$.

Suppose instead that the rank is one. By
\cite[Proposition~20.32]{Milne2017}, applied over $\mathbb C$
(Proposition~\ref{app:root-inputs}), there is a central isogeny
$\SL_{2,\mathbb C}\to\mathbf S_{\mathbb C}$.
Composing it with the inclusion in $\SL_{3,\mathbb C}$ gives a
three-dimensional representation with finite kernel.
Lemma~\ref{app:sl2-representations} gives two possibilities:
the standard two-dimensional representation plus a trivial line,
or the irreducible symmetric square. In either case,
$\dim\mathbf S=3$. We now descend these two possibilities to
$\mathbb Q$.

\emph{Descent in the block case.}
Over $\mathbb C$, the invariant-vector
space and the invariant-covector space are one-dimensional and
have nonzero pairing: the vector spans the trivial summand,
and the covector is projection onto that summand.
These fixed spaces are defined by the infinitesimal invariance
equations in Section~\ref{app:lie}. Since $\Lie(\mathbf S)$ has
a rational basis, the equations have rational coefficients.
Their solution spaces therefore have rational bases.
Choose rational generators $v$ and $\ell$ with $\ell(v)\ne0$.
Then
\[
 V=\mathbb Qv\oplus\ker\ell
\]
is a rational $\mathbf S$-invariant decomposition. The action
on $\mathbb Qv$ is trivial, and the action on $\ker\ell$ has
determinant one because $\mathbf S\le\SL_3$.
The latter action is faithful: an element acting trivially on
both summands acts trivially on $V$. Its image is therefore a
dimension-three algebraic subgroup of $\SL(\ker\ell)$
\cite[Theorem~3.34]{Milne2017}, which
also has dimension three, so it is the whole group.
A rational basis adapted to this decomposition conjugates
$\mathbf S$ to the block $\SL_2$.

\emph{Descent in the irreducible case.}
The symmetric-square representation can be realized on matrices
\[
 M=\begin{pmatrix}x&y\\y&z\end{pmatrix},\qquad
 M\longmapsto AMA^{\mathsf T}\quad(A\in\SL_2).
\]
Diagonal invariance restricts an invariant quadratic form to
$a xz+b y^2$. The upper unipotent matrix
$\left(\begin{smallmatrix}1&s\\0&1\end{smallmatrix}\right)$
sends $(x,y,z)$ to $(x+2sy+s^2z,y+sz,z)$; invariance for every
$s$ forces $b=-a$. Thus the invariant quadratic forms are exactly
the multiples of the determinant $xz-y^2$, and every nonzero
such form is nondegenerate.

In the original rational coordinates on $V$, the invariant
symmetric matrices $Q$ satisfy
\[
 X^{\mathsf T}Q+QX=0\qquad(X\in\Lie(\mathbf S)).
\]
Using a rational basis of $\Lie(\mathbf S)$ makes this a linear
system over $\mathbb Q$. Its solution space is one-dimensional,
as shown by the complex symmetric-square calculation above,
so it has a nonzero rational generator $Q$. The corresponding
form $\qform(v)=v^{\mathsf T}Qv$ is nondegenerate.
By connectedness and Lemma~\ref{app:lie-tests}, infinitesimal
invariance implies group invariance, and hence
$\mathbf S\le\SO(\qform)$.
The group $\SO(\qform)$ is connected
\cite[Example~2.43 and Exercise~2-9]{Milne2017} and has dimension
three by Lemma~\ref{app:symmetric-square}.
Since $\dim\mathbf S=3$, this inclusion is an equality:
$\mathbf S=\SO(\qform)$.

\emph{The remaining affine cases.}
If $\mathbf U$ is trivial, then $\mathbf H=\mathbf S$, and the
preceding classification gives the three semisimple possibilities.
It remains to consider $\mathbf S$ and $\mathbf U$ both nontrivial.
Proposition~\ref{app:unipotent-flag} gives a nonzero rational
fixed space $F=V^{\mathbf U}$. It is proper because $\mathbf U$
is nontrivial and its given action on $V$ is faithful.
Normality of $\mathbf U$ makes $F$ stable under $\mathbf S$:
for $v\in F$, $s\in\mathbf S$, and $u\in\mathbf U$,
\[
 u(sv)=s(s^{-1}us)v=sv.
\]
Thus $\mathbf S$ acts reducibly on $V$. Of the semisimple
possibilities just established, only the block $\SL_2$ is
reducible. After a rational change of basis, write
\[
 V=W\oplus L,\qquad \dim W=2,\quad\dim L=1,
\]
where $\mathbf S$ acts on $W$ by its standard representation
and acts trivially on $L$.

The standard module $W$ is irreducible and has no nonzero
invariant covector (Lemma~\ref{app:affine-algebra}).
Consequently the only nonzero proper $\mathbf S$-invariant
subspaces of $V$ are $W$ and $L$. To see this explicitly,
if $\dim F=1$, its projection to $W$ is an invariant subspace
of dimension at most one and must vanish, so $F=L$.
If $\dim F=2$, its projection to $W$ is an isomorphism; hence
$F$ is the graph of an $\mathbf S$-equivariant map $W\to L$.
Since $L$ is trivial and $W$ has no nonzero invariant covector,
this map is zero, so $F=W$.

If $F=W$, then $\mathbf U$ fixes $W$ pointwise. Its action on
the one-dimensional quotient $V/W$ is unipotent and therefore
trivial. Thus $\mathbf U$ is contained in the additive group
$\operatorname{Hom}(L,W)$, embedded in the chosen basis by
\[
 v\longmapsto\begin{pmatrix}I_2&v\\0&1\end{pmatrix}.
\]
Here a basis of $L$ identifies $\operatorname{Hom}(L,W)$ with $W$.

If $F=L$, then $\mathbf U$ fixes $L$ pointwise.
Its induced action on $V/L\simeq W$ is unipotent, so
Proposition~\ref{app:unipotent-flag} gives a nonzero fixed
subspace in this quotient. This subspace is $\mathbf S$-stable
by the same normality argument. Irreducibility of $W$ makes it
the entire quotient, so the quotient action is trivial.
Hence $\mathbf U$ is contained in the additive group
$\operatorname{Hom}(W,L)$, embedded by
\[
 w^{\mathsf T}\longmapsto
 \begin{pmatrix}I_2&0\\w^{\mathsf T}&1\end{pmatrix}.
\]

In either case $\Lie(\mathbf U)$ is a nonzero
$\mathbf S$-invariant subspace of the standard two-dimensional
module or its dual: conjugation by $\operatorname{diag}(A,1)$
acts respectively by $v\mapsto Av$ and
$w^{\mathsf T}\mapsto w^{\mathsf T}A^{-1}$.
These modules are irreducible by
Lemma~\ref{app:affine-algebra}, so $\Lie(\mathbf U)$ is the
entire two-dimensional translation space. Since $\mathbf U$
is connected, it equals that translation group.
The resulting semidirect products are precisely $\mathbf A$
and $\mathbf A^\vee$ from \eqref{eq:affinegroups}.
The changes of basis were rational, giving the stated rational
conjugacies.
\end{proof}

These possibilities lead to the unipotent, rank-one, ambient,
and affine estimates in Sections~\ref{sec:unipotentproof},
\ref{sec:modularproof}, \ref{sec:ambientproof}, and
\ref{sec:leviaffine}, respectively. Definite orthogonal groups
have only constant matrix-polynomial curves, as recorded below,
and therefore do not occur as supports of the nonconstant
trajectories in the main theorem.

\subsection{Coordinates and covers of the model supports}

Unipotent supports are treated in rational Mal'cev coordinates.
For proper semisimple supports, the reduction proceeds through
$\SL_2$ parametrizations. We establish the quotient geometry,
the polynomial lifts, and the congruence covers in separate lemmas.

\begin{lemma}
\label{lem:rationalnilcoordinates}
Let $\mathbf H\le\SL_3$ be a connected unipotent
$\mathbb Q$-group and put $H=\mathbf H(\mathbb R)^\circ$.
Then $H/(H\cap\Gamma)$ is compact and admits a Mal'cev basis
of rational matrices adapted to the lower central series.
Its ordered-coordinate map is polynomial over $\mathbb Q$
of degree at most two in the matrix entries. Consequently,
a matrix curve in $\mathbf H(F[t])$ of degree at most $d$
has ordered coordinates in $F[t]$ of degree at most $2d$,
for every subfield $F\subseteq\mathbb R$.

For $r=\dim H/[H,H]$, the horizontal projection
$\chi_{\mathrm{ab}}:H\to\mathbb R^r$ is a rational polynomial
map of matrix-entry degree at most two. Every horizontal character
trivial on $H\cap\Gamma$ is
$k\cdot\chi_{\mathrm{ab}}\bmod\mathbb Z$, $k\in\mathbb Z^r$.
\end{lemma}
\begin{proof}
Proposition~\ref{app:unipotent-flag} gives an invariant rational
full flag and hence a rational triangularization of $H$.
By Lemma~\ref{app:unipotent-explog}, its Lie algebra has dimension
at most three and nilpotency class at most two, and exponential
and logarithm are inverse polynomial maps.
Choose a rational basis $B_1,\ldots,B_m$ whose terminal spans
are ideals and whose last vectors span the commutator algebra.
In the abelian case any rational basis suffices. In the
nonabelian case the Lie algebra is three-dimensional with
one-dimensional center, and we take $B_3$ central.
After multiplication by a common positive integer, all
$\exp B_i$ are integral matrices, since
$\exp B_i=I_3+B_i+B_i^2/2$. They generate a discrete subgroup
$\Lambda_0\subseteq\Lambda:=H\cap\Gamma$.
In global ordered coordinates, right multiplication by
$\exp(-nB_i)$ subtracts $n$ from the $i$th coordinate and
leaves the earlier coordinates unchanged. Successive reduction
places every right coset modulo $\Lambda_0$ in the image of
$[0,1]^m$. This proves compactness of $H/\Lambda_0$ and
$H/\Lambda$.

Apply the lattice construction in Lemma~\ref{app:malcev-existence}
to $H/\Lambda$. Its adapted basis is formed from logarithms of
elements of $\Lambda$, using a lattice basis in the abelian case
and horizontal lifts and a central generator in the Heisenberg
case. For $\gamma\in\Lambda$,
\[
 \log\gamma=(\gamma-I_3)-\tfrac12(\gamma-I_3)^2
\]
has rational entries, so the resulting basis matrices are rational.
Lemma~\ref{app:horizontal} gives ordered coordinates over
$\mathbb Q$ of matrix-entry degree at most two, and the same
degree bound for the horizontal projection. Substitution of a
curve in $\mathbf H(F[t])$ gives coordinates in $F[t]$ of
degree at most $2d$. The character assertion is also
Lemma~\ref{app:horizontal}.
\end{proof}

The degree bound is relevant to the intrinsic condition: every
horizontal phase is a rational polynomial of matrix-entry degree
at most two. The factor two in the coordinate degree is needed
in general. Relative to $(E_{12},E_{23},E_{13})$, the curve
$I_3+tE_{12}+tE_{23}$ has ordered coordinates $(t,t,-t^2)$.
The sampling filtration used with the Green--Tao theorem is
specified in \eqref{eq:nilfiltration}; thus matrix degree $d$
allows coordinate degree $2d$ and filtration degree $4d$.

For proper semisimple supports, we first determine the available
parametrizations and the compactness properties of the orthogonal quotients.

\begin{lemma}
\label{lem:semisimpleparametrizations}
Let $\mathbf H\le\SL_3$ be a nontrivial proper connected
semisimple $\mathbb Q$-group, and put
$H=\mathbf H(\mathbb R)^\circ$.
For an orthogonal group $\mathbf H=\SO(\qform)$, where $\qform$
is a nondegenerate rational quadratic form in three variables,
the orbit $H\Gamma/\Gamma$ is closed.
If $\qform$ is anisotropic over $\mathbb Q$, this orbit and
$H/(H\cap\Gamma)$ are compact. If $\qform$ is definite over
$\mathbb R$, $H$ itself is compact and every matrix-polynomial
curve in $H$ is constant.

For such an orthogonal group, if $\qform$ is indefinite over
$\mathbb R$, a symmetric-square
parametrization $\SL_2\to\mathbf H$ exists over $\mathbb R$;
if $\qform$ is isotropic over $\mathbb Q$, it can be chosen
over $\mathbb Q$. A rational block group has a block
parametrization over $\mathbb Q$. Each of these parametrizations
has finite kernel and real image $H$.
\end{lemma}
\begin{proof}
By Proposition~\ref{prop:supportclassification}, the proper
semisimple possibilities are rational block groups and orthogonal
groups. In the block case, conjugating
$A\mapsto\operatorname{diag}(A,1)$ by the fixed rational
change of basis gives the required rational parametrization;
its kernel is trivial and its real image is $H$.
In the orthogonal case, Lemma~\ref{app:symmetric-square} gives
the symmetric-square parametrization over any field over which
$\qform$ is isotropic, with kernel $\{I_2,-I_2\}$ and real
image $H$. An indefinite real form is isotropic over $\mathbb R$,
so the real parametrization always exists in that case, and
rational isotropy gives a parametrization over $\mathbb Q$.

Now suppose $\mathbf H=\SO(\qform)$.

Scale $Q$ to be integral. If $h_n\gamma_n\to g$, the integral
matrices $\gamma_n^{\mathsf T}Q\gamma_n$ converge to
$g^{\mathsf T}Qg$ and are eventually constant. Fix an index in
this tail and denote its matrix by $\gamma_0$. Then
$\delta_n=\gamma_n\gamma_0^{-1}$ lies in
$\SO(\qform)(\mathbb R)\cap\Gamma$.
There are finitely many real components, so pass to a subsequence
in one component and fix $\delta_*$ from that subsequence.
Now $\delta_n\delta_*^{-1}\in H$ and
\[
 h_n\delta_n\delta_*^{-1}
 \longrightarrow g\gamma_0^{-1}\delta_*^{-1}\in H,
\]
because $H$ is closed. Thus $g\in H\delta_*\gamma_0\subset H\Gamma$,
proving closedness.
For completeness, finiteness of the real components follows
from polar decomposition after diagonalizing the form to signature
$(p,3-p)$: the orthogonal group retracts onto
$S(O(p)\times O(3-p))$, which has finitely many components.
The positive part is joined to $\groupid$ by its real powers.

In the anisotropic case, for $0\ne v\in\mathbb Z^3$,
\begin{equation}\label{eq:anisotropicsystole}
 1\le |\qform(v)|=|\qform(hv)|\le\norm Q_{\mathrm{op}}\norm{hv}^2.
\end{equation}
Thus these unimodular lattices have a uniform shortest-vector
lower bound. The elementary lattice reduction in
\eqref{eq:elementaryreduction} puts them in a compact part of
$X$. Closedness proves compactness of the orbit. The closed-orbit
identification with $H/(H\cap\Gamma)$ gives compactness of
that quotient as well.
For a definite form, changing real coordinates puts $H$ in the
usual compact orthogonal group. Polynomial matrix entries bounded
on $\mathbb R$ are constant.
\end{proof}

We next lift normalized polynomial curves through these parametrizations.

\begin{lemma}
\label{lem:polynomiallift}
Let $\iota:\SL_2\to\mathbf H\le\SL_3$ be a block or
symmetric-square parametrization defined over
$K\subseteq F\subseteq\mathbb R$, and put $H=\iota(G_2)$.
Every normalized $\Psi\in\mathbf H(F[t])$ of degree at most $d$
has a normalized lift $\psi\in\SL_2(F[t])$ satisfying
$\Psi=\iota\circ\psi$ and $\deg\psi\le d$.
In the symmetric-square case, $\deg\psi\le\lfloor d/2\rfloor$.
Every $h\in H$ has a real lift in $G_2$.
\end{lemma}
\begin{proof}
By Lemma~\ref{app:sl2-representations}, a change of basis over $K$
puts $\iota$ in block or symmetric-square form.
In the block case the polynomial lift is the two-by-two block.
For the symmetric square, a local real
analytic lift of $\Psi(t)$ near zero exists and starts at $\groupid$.
Denote its entries by $u,v,w,z$. The symmetric-square matrix
determines all their pairwise products in $F[t]$: it gives their
squares and $uv,uw,vz,wz,uz+vw$, while $uz-vw=1$ gives $uz$ and $vw$.

Put $A=u^2\in F[t]$, with $A(0)=1$. For any nonzero entry $q$,
$q^2=(uq)^2/A$. If an irreducible factor had odd multiplicity in
$A$, this identity would give odd multiplicity in every nonzero
polynomial $q^2$, hence multiplicity at least one. Both polynomial
products $uz$ and $vw$ would then be divisible by that factor:
their squares are products of two such squares. A zero product
causes no exception. This contradicts $uz-vw=1$.
Thus $A=cR^2$ for $c\in F^\times$ and $R\in F[t]$.
Since $A(0)=1$, the polynomial $u=R/R(0)$ is its square root with
value one at zero. The other entries $q=(uq)/u$ lie in $F(t)$.
Their squares lie in $F[t]$, so their reduced denominators must be
constant; hence they too lie in $F[t]$. These choices agree with
the local lift, and its identities extend polynomially to give a
normalized lift in $\SL_2(F[t])$. Each entry square has degree at
most $d$, so the lift actually has degree at most $\lfloor d/2\rfloor$
in the symmetric-square case.

Every $h\in H$ has a real lift in $G_2$ by the definition of $H$.
\end{proof}

For rational parametrizations, the lifted trajectory can be studied
on a fixed congruence quotient.

\begin{lemma}
\label{lem:semisimplecongruencecover}
Let $\iota:\SL_2\to\mathbf H\le\SL_3$ be defined over
$\mathbb Q$, with finite kernel, proper image, and real image
$H=\mathbf H(\mathbb R)^\circ$. Suppose $H\Gamma/\Gamma$
is closed. Then there is an integer $N\ge1$, depending only
on $\iota$, such that $\iota(\Gamma_2(N))\subseteq\Gamma$.
The map $G_2/\Gamma_2(N)\to H\Gamma/\Gamma$ is a finite
covering preserving Haar probability.
Here $\Gamma_2(N)=\{\gamma\in\Gamma_2:
\gamma\equiv I_2\pmod N\}$.
\end{lemma}
\begin{proof}
By Lemma~\ref{app:sl2-representations}, the entries of $\iota$ are rational polynomials
of degree one or two in the four matrix entries on $\SL_2$.
Choose $N$ clearing their denominators. By
Lemma~\ref{app:congruence}, $\iota(\groupid)=\groupid$ implies that
substitution of $\gamma\equiv I_2\pmod N$ gives an integral matrix
$\iota(\gamma)$ of determinant one.
The discrete stabilizer $\iota^{-1}(\Gamma)$ contains the lattice
$\Gamma_2(N)$; the discrete overgroup of this lattice has finite index:
a positive-volume neighborhood in its quotient cannot have
infinitely many disjoint lifts in the finite-volume congruence
quotient. By Lemma~\ref{app:lattice-covers}, the resulting finite
cover pushes Haar probability to Haar probability.
\end{proof}

%Thus the proper semisimple estimates use fixed congruence covers in the rationally split cases and fixed compact rank-one quotients in the rationally anisotropic indefinite case. The latter groups are noncompact over $\mathbb R$; it is their arithmetic quotients that are compact. The standard two-dimensional representation does not descend through the symmetric-square cover, whose kernel is $\{I_2,-I_2\}$. Lemma~\ref{lem:quadratictransfer} supplies the corresponding coefficient bound on the cover. The affine covering reduction is kept with its estimate in Lemma~\ref{lem:affinecover}.

\section{Polynomial estimates on unipotent supports}
\label{sec:unipotentproof}

A unipotent support is a homogeneous orbit $Y=Hz\simeq H/\Delta_z$,
where $H$ is unipotent (every matrix has all eigenvalues equal to one) and
$\Delta_z=\{h\in H:hz=z\}$. The normalized rational supports of this
type are compact nilmanifolds, with rational ordered exponential
coordinates supplied by Lemma~\ref{lem:rationalnilcoordinates}.
Their horizontal torus $H/[H,H]\Delta_z$ records the coordinates
after quotienting out commutators. This section converts the
Green--Tao theorem, Theorem~\ref{prop:green-tao}, into the
continuous-time horizontal obstruction of
Lemma~\ref{lem:continuousleibman}. Proposition~\ref{prop:separatedunipotent}
then gives a power rate of convergence to Haar probability when the
nonconstant coefficients of every nonzero horizontal character phase
satisfy \eqref{eq:toralseparation}; Section~\ref{sec:mainproof} derives
these bounds from the intrinsic condition.

Fix the compact orbit $Y=Hz\simeq H/\Delta_z$, basis $B$, and
constant $Q_B$ of Section~\ref{subsec:nilinput}. Identify
$b\Delta_z$ with $bz\in Y$. Let $\Phi:\R\to H$ be nonconstant,
with $\Phi(0)=\groupid$, and suppose its ordered coordinates
are polynomials of degree at most an integer $D\ge1$.
The horizontal projection is
\begin{equation}\label{eq:horizontalcurve}
 p(t)=\chi_{\mathrm{ab}}(\Phi(t))
     =\sum_{j=1}^D v_jt^j\in\R^r[t].
\end{equation}
The obstruction lemma below also applies when $p$ is constant.
For a matrix-degree bound $d$, Lemma~\ref{app:horizontal}
permits $D=2d$, with filtration degree at most $2D=4d$.

\begin{lemma}[Continuous-time horizontal obstruction]
\label{lem:continuousleibman}
After increasing $b_D$ in Theorem~\ref{prop:green-tao}, the
following holds for the same $a_D$. Let $T\ge1$,
$a\in H$, $x=a\Delta_z$, and $0<\epsilon\le(4Q_B)^{-1}$.
For a Lipschitz function $F$ on $H/\Delta_z$, if
\[
 \left|\frac1T\int_0^T F(\Phi(t)x)\dd t
                 -\int_{H/\Delta_z}F\right|
                  >2\epsilon\norm F_{\mathrm{Lip},B},
\]
then some $0\ne k\in\Z^r$ satisfies
\begin{equation}\label{eq:continuousobstruction}
 \norm k\le b_D\epsilon^{-a_D},\qquad
 \left(\sum_{j=1}^D T^{2j}|k\cdot v_j|^2\right)^{1/2}
                     \le b_D\epsilon^{-a_D}.
\end{equation}
The constants do not depend on $T,x$, or the coefficients of $\Phi$.
\end{lemma}
\begin{proof}
Fix $T,\epsilon,F$, and take arbitrarily large integers $N$.
The Riemann sums
\[
 \frac1N\sum_{n=1}^N F(\Phi(Tn/N)x)
           \longrightarrow\frac1T\int_0^T F(\Phi(t)x)\dd t
\]
therefore fail the $\epsilon$ discrepancy bound for every
sufficiently large $N$. The coordinate degree and filtration of
$q_N(n)=\Phi(Tn/N)a$ are independent of $N$, by
Lemma~\ref{app:nil-sampling}. Apply Theorem~\ref{prop:green-tao}
to each $q_N$. The vectors $k_N$ lie in a fixed finite set of
nonzero integral vectors. Along an infinite subsequence they equal
one fixed $k$.

Let $P(t)=k\cdot\chi_{\mathrm{ab}}(\Phi(t)a)
       =\beta_0+\sum_{j=1}^D\beta_jt^j$, where
$\beta_j=k\cdot v_j$ for $j\ge1$. Write the real polynomial
$P(Tn/N)=\sum_{j=0}^D\alpha_j(N)\binom nj$.
For each $j\ge1$, finite differences at $n=0$ give
\[
 \alpha_j(N)=j!\,\beta_j(T/N)^j+O_{P,T,D}(N^{-j-1}),
 \qquad N^j\alpha_j(N)\longrightarrow j!\,\beta_jT^j.
\]
This follows by applying the $j$th difference to each monomial:
it annihilates degrees below $j$ and equals $j!$ on $n^j$ at zero.
In particular $\alpha_j(N)\to0$. For large $N$ in the subsequence,
$\operatorname{dist}(\alpha_j(N),\Z)=|\alpha_j(N)|$.
The smoothness bound \eqref{eq:gtobstruction}, followed by the limit,
gives $j!\,T^j|\beta_j|\le b_D\epsilon^{-a_D}$ for every $j$.
Summing the squares and enlarging $b_D$ by a factor depending only
on $D$ proves \eqref{eq:continuousobstruction}.

The number of samples needed for convergence may depend on the fixed
data, but the resulting inequality does not: it is obtained by a limit
of the uniform bounds \eqref{eq:gtobstruction}. No integer-time
equidistribution assertion for the original curve is used.
\end{proof}

Assume that the horizontal coefficients satisfy, for some
$0<c\le1$ and $\tau\ge0$,
\begin{equation}\label{eq:toralseparation}
 \left(\sum_{j=1}^D|k\cdot v_j|^2\right)^{1/2}
       \ge c\norm k^{-\tau}\qquad(0\ne k\in\Z^r).
\end{equation}
These are the nonconstant coefficients of the horizontal character phase
$k\cdot p(t)$. Continuous time requires separation from zero in
$\R$, rather than separation from integer values.

For a change of horizontal coordinates
$p_U(t)=U^{-1}p(t)$, with $U\in\mathrm{GL}_r(\Z)$, the character
phases satisfy
\[
 k\cdot p_U(t)=(U^{-\mathsf T}k)\cdot p(t)
 \qquad(k\in\Z^r).
\]
The map $k\mapsto U^{-\mathsf T}k$ is a bijection of $\Z^r$.
Consequently, \eqref{eq:toralseparation} holds for $p_U$ with
the same exponent $\tau$ and with constant
$c\max\{1,\norm{U^{-\mathsf T}}_{\mathrm{op}}\}^{-\tau}$.
In the abelian case, this coordinate change is induced by the
lattice basis change $B'=BU$.

\begin{proposition}[Equidistribution on unipotent supports]
\label{prop:separatedunipotent}
Assume \eqref{eq:toralseparation} for the horizontal
projection $p=\chi_{\mathrm{ab}}\circ\Phi$, and put
$\gamma_{D,\tau}=1/[a_D(\tau+1)]$. Then for every
$f\in C_c^\infty(X)$ and $x\in Y$,
\begin{equation}\label{eq:separatedunipotent}
 |A_T^\Phi f(x)-\mu_Y(f)|
 \le C_D Q_B\,c^{-\gamma_{D,\tau}}
       \mathcal N_{B,1,Y}(f)T^{-\gamma_{D,\tau}}
       \qquad(T\ge1).
\end{equation}
The estimate is uniform over $x\in Y$. The constant $C_D$ is
independent of $\tau$, the polynomial coefficients, and the orbit
coordinates; the latter enter through $Q_B$, $c$, and the test norm.
\end{proposition}
\begin{proof}
Put $F(b\Delta_z)=f(bz)$. The coefficient bound and \eqref{eq:continuousobstruction}, using
$T^j\ge T$, would imply
\[
 Tc\norm k^{-\tau}\le b_D\epsilon^{-a_D},\qquad
 \norm k\le b_D\epsilon^{-a_D},
 \quad\text{hence}\quad
 cT\le b_D^{\tau+1}\epsilon^{-a_D(\tau+1)}.
\]
Set $\epsilon_T=(2b_D)^{1/a_D}(cT)^{-\gamma_{D,\tau}}$ and
$\epsilon_0=(4Q_B)^{-1}$. If $\epsilon_T\le\epsilon_0$,
the last inequality is impossible with $\epsilon=\epsilon_T$.
Lemma~\ref{lem:continuousleibman} bounds the discrepancy by
$2\epsilon_T\norm F_{\mathrm{Lip},B}$. If
$\epsilon_T>\epsilon_0$, the trivial discrepancy bound
$2\norm F_\infty$ is at most
$8Q_B\epsilon_T\norm F_{\mathrm{Lip},B}$.
In both cases apply \eqref{eq:nillipcontrol} and $m\le3$.
This proves the stated bound with, for example, a fixed multiple of
$32(2b_D)^{1/a_D}$ for $C_D$.
\end{proof}

\section{Polynomial estimates on rank-one supports}
\label{sec:modularproof}
\label{sec:modularinput}

The rank-one supports, whose maximal tori over $\mathbb C$ have
dimension one, are the proper semisimple cases in
Proposition~\ref{prop:supportclassification}: rational block $\SL_2$
groups and $\SO(\qform)$ for a nondegenerate rational ternary quadratic
form $\qform$. Lemma~\ref{lem:semisimpleparametrizations}
parametrizes their noncompact real identity components by
$G_2=\SL_2(\mathbb R)$; when $\qform$ is definite, polynomial curves
are constant. We therefore prove the estimates on quotients
$G_2/\Delta$: fixed congruence covers for block groups and forms
isotropic over $\mathbb Q$, and compact quotients for indefinite
forms anisotropic over $\mathbb Q$.
Lemma~\ref{lem:modulartransfer} combines tangent averaging and
quantitative nondivergence to transfer horocycle estimates to
polynomial curves. The resulting power bounds are
Proposition~\ref{prop:congruencemodular}, under coefficient separation,
and Proposition~\ref{prop:compactpolynomial}, without an arithmetic
condition on compact quotients. Lemma~\ref{lem:quadratictransfer}
supplies the coefficient bound needed on symmetric-square covers;
the transfer back to the original supports is carried out in
Section~\ref{sec:mainproof}.

For the rank-one models, $E_{ij}$ denotes a $2\times2$ matrix unit.
Set
\[
 u_2(s)=I_2+sE_{12},\qquad
 a_2(L)=\diag(L^{1/2},L^{-1/2}).
\]
For any lattice $\Delta\le G_2$, write $X_\Delta=G_2/\Delta$
and let $m_\Delta$ denote its Haar probability measure.
For $s\in\Z_{\ge0}$ and $F\in C_b^s(X_\Delta)$, let
$\mathcal D_s(F)$ be the sum of supremum norms of all ordered
left-action derivatives of orders at most $s$ in the basis
$E_{12},E_{21},\diag(1,-1)$ of $\mathfrak{sl}_2$, including
$\norm F_\infty$ as the term of order zero. These conventions
apply to both congruence covers and compact rank-one quotients.

\subsection{A common transfer from horocycles to polynomial curves}

Let $\Delta\le\Gamma_2$ have finite index. For $y\in X_\Delta$,
let $\bar y$ denote its image in $X_2$, and put
\[
 U_L^\varepsilon F(y)=\frac1L\int_0^L F(u_2(\varepsilon s)y)\dd s,
 \qquad\varepsilon\in\{1,-1\}.
\]
The next lemma deduces a polynomial estimate from a flow estimate
and a growth bound on the rank-one base. It applies to both the
modular and regular affine cases.

\begin{lemma}[Polynomial transfer through a rank-one base]
\label{lem:modulartransfer}
Let $\mathcal G\le\SL_n(\R)$, $n\in\{2,3\}$, be a fixed
connected closed matrix group with lattice $\Lambda$, and let
$\pi:\mathcal G\to G_2$ be a homomorphism with
$\pi(\Lambda)\subseteq\Gamma_2$. For $y=b\Lambda$ put
$\bar y=\pi(b)\Gamma_2$, and let $m$ be invariant probability
on $\mathcal X=\mathcal G/\Lambda$.
Fix a nonzero nilpotent $N_*\in\Lie(\mathcal G)$ and set
$p=1$ if $N_*^2=0$, $p=2$ otherwise.
Let $\mathcal U_L$ be one of the averages in
Lemma~\ref{lem:smoothing}, and let $\mathcal N_\ell$ be the
sum of supremum Lie-derivative norms through order $\ell\ge1$.

Suppose that, for fixed $0<\omega\le1/2$, $\zeta>0$, $L_0\ge2$,
and $C_0$, the estimate
\begin{equation}\label{eq:transferinput}
 \latticemin(a_2(L)^{-1}\bar y)\ge L^{-\omega}
 \quad\Longrightarrow\quad
 |\mathcal U_LF(y)-m(F)|
       \le C_0\mathcal N_\ell(F)L^{-\zeta}
       \qquad(L\ge L_0)
\end{equation}
holds for every $y$ and every test function in a fixed class
contained in $C_b^\ell(\mathcal X)$.
The class may consist of functions supported in a fixed compact
set, in which case $C_0,L_0$ may depend on that set.
Let $P(t)\in\mathcal G$ be normalized with leading tangent $N_*$,
and suppose both $P$ and $\pi\circ P$ have matrix degree at most $d$.
For $h\in\mathcal G$, assume that $0<\alpha\le1$, $0<K\le1$, and
\begin{equation}\label{eq:separatedmodulargrowth}
 \sup_{t\in[T,2T]}\norm{\pi(P(t)h)w}\ge KT^\alpha
       \qquad(T\ge1,\ 0\ne w\in\Z^2).
\end{equation}
Put
\[
 \theta=\frac{\alpha}{2(d+p)},\qquad
 \nu=\min\{\zeta,\omega/d,1/4\},\qquad
 \delta=\theta\nu,\qquad
 \beta=\frac{2\nu}{4(d+p)-1}.
\]
Then for $T\ge2$,
\begin{equation}\label{eq:modulartransfer}
 \left|\frac1T\int_0^T F(P(t)h\Lambda)\dd t-m(F)\right|
 \le C\bigl(1+M(P)+K^{-\beta}\bigr)\mathcal N_\ell(F)T^{-\delta}.
\end{equation}
The constant $C$ depends only on
$d,\mathcal G,\Lambda,\pi,N_*,\omega,\zeta,\ell,C_0,L_0$
and the chosen derivative norms; it is independent of
$\alpha,K,P,h$.
\end{lemma}
\begin{proof}
Set $L=T^\theta$. Lemma~\ref{lem:smoothing} gives
\begin{equation}\label{eq:separatedmodularsmoothing}
 \left|\frac1T\int_T^{2T}(F-\mathcal U_LF)(P(t)h\Lambda)\dd t\right|
 \le C_dM(P)\mathcal N_\ell(F)T^{-1/2}
 \qquad(T\ge T_d).
\end{equation}
The base-lattice vector suprema are at least
\[
 \sup_{[T,2T]}\norm{a_2(L)^{-1}\pi(P(t)h)w}
       \ge KT^{\alpha-\theta/2}.
\]
Take
\begin{equation}\label{eq:separatedmodularthreshold}
 T_*=\max\{T_d,L_0^{1/\theta},2^{1/(\omega\theta)},
                         K^{-1/(\alpha-\theta/2)}\}.
\end{equation}
For $T\ge T_*$ these suprema are at least one, the determinant
is one, and the vector polynomials are $(C_d,1/d)$-good on
$9[T,2T]$. Theorem~\ref{prop:km} in dimension two,
with $\varrho=1/2$, therefore gives
\begin{equation}\label{eq:separatedmodularbadset}
 \frac1T\Leb\{t\in[T,2T]:
 \latticemin(a_2(L)^{-1}\pi(P(t)h)\Z^2)<L^{-\omega}\}
       \le C_dL^{-\omega/d}.
\end{equation}
Apply \eqref{eq:transferinput} off this set and use
$2\norm F_\infty$ on it. Together with smoothing, this gives
\begin{equation}\label{eq:separatedmodulardyadic}
\begin{aligned}
 \left|\frac1T\int_T^{2T}F(P(t)h\Lambda)\dd t-m(F)\right|
 &\le C\mathcal N_\ell(F)
       \bigl(M(P)T^{-1/2}+L^{-\zeta}+L^{-\omega/d}\bigr)\\
 &\le C(1+M(P))\mathcal N_\ell(F)T^{-\delta}.
\end{aligned}
\end{equation}
Since $0<\delta\le1/16$, Lemma~\ref{lem:dyadic} applies.
Its initial-interval cost is uniform in $\alpha$, because
\[
 T_*^\delta\le T_d^{1/2}+L_0^\nu+2^{\nu/\omega}+K^{-\beta},
 \qquad \frac{\delta}{\alpha-\theta/2}=\beta.
\]
This proves \eqref{eq:modulartransfer}.
\end{proof}

\iffalse
For the modular model take $\mathcal G=G_2$, $\Lambda=\Delta$,
$\pi$ the identity, $N_*=\varepsilon E_{12}$, and
$\mathcal N_\ell=\mathcal D_\ell$.
The regular affine application instead uses the two-dimensional
block projection and $p=2$; the nondivergence argument is unchanged.
\fi

\subsection{Coefficient separation on symmetric-square covers}

For a symmetric-square parametrization, the standard
two-dimensional representation does not descend to the orthogonal
group: the kernel element $-I_2$ acts nontrivially on the standard
module (Section~\ref{app:quadratic}). The following coefficient
calculation recovers the test on the cover, with separation
exponent $2\tau+1$ in place of $\tau$.

\begin{lemma}
\label{lem:quadratictransfer}
Let $\psi\in\SL_2(\R[t])$ be normalized with $\deg\psi\le d$,
let $g\in G_2$, and let $\sigma(A)Q=AQA^{\mathsf T}$ on
the symmetric two-by-two matrices. Declare
$E_{11},E_{22},E_{12}+E_{21}$ orthonormal and use their integer span.
Write
\[
 \sigma(\psi(t)g)=\sigma(g)+\sum_{j=1}^{2d}t^j C_j.
\]
Suppose $c>0$, $\tau\ge0$, and
\[
 \left(\sum_{j=1}^{2d}\|C_j Q\|^2\right)^{1/2}
       \ge c\|Q\|^{-\tau}
 \quad\text{for every nonzero integral symmetric matrix }Q.
\]
For $0\ne w\in\Z^2$, write
\[
 u(t)=\psi(t)gw=u_0+\sum_{j=1}^d t^j u_j,
 \qquad b(w)=\left(\sum_{j=1}^d\|u_j\|^2\right)^{1/2}.
\]
Then there is $c'=c'_{\psi,g,d,c}>0$ such that
\begin{equation}\label{eq:quadratictransfer}
 b(w)\ge c'\|w\|^{-(2\tau+1)}
 \qquad(0\ne w\in\Z^2).
\end{equation}
\end{lemma}
\begin{proof}
The integral vector $ww^{\mathsf T}$ has coordinate norm at most
$\|w\|^2$. Thus the nonconstant coefficient norm of
$u(t)u(t)^{\mathsf T}=\sigma(\psi(t)g)(ww^{\mathsf T})$
is at least $c\|w\|^{-2\tau}$.
Set $u_j=0$ outside $0\le j\le d$. Its coefficient of $t^j$,
$j\ge1$, is
\[
 u_0u_j^{\mathsf T}+u_ju_0^{\mathsf T}
       +\sum_{\substack{i,k\ge1\\i+k=j}}u_iu_k^{\mathsf T}.
\]
The Frobenius norm dominates the declared coordinate norm.
The triangle inequality for the coefficient sequence and
$\sum_{j\ge1}\|u_j\|\le\sqrt d\,b(w)$ give the upper bound
\[
 2\|u_0\|b(w)+\sqrt d\,b(w)^2
       \le C_{\psi,g,d}\|w\|b(w).
\]
The last inequality holds because every coefficient of $u$ is a
fixed linear function of $w$. Comparing the bounds proves the claim.
\end{proof}

\subsection{Fixed congruence covers}
\label{subsec:congruencecovers}

The intrinsic theorem requires finite covers of $X_2$, since a
rational embedding need not carry all of $\Gamma_2$ into $\Gamma$
(Section~\ref{app:arithmetic}).
We first record the horocycle estimate for an arbitrary fixed
lattice $\Delta\le G_2$, so it also applies to the compact quotients
below. Let $\inj_\Delta$ be the local injectivity radius of
the quotient orbit map, capped at one.
Constants are allowed to depend on the fixed lattice.

The following fixed-lattice horocycle estimate follows from
Zheng's theorem after changing quotient and metric conventions.

\begin{theorem}[{\cite[Theorem~B.1]{Zheng2015}}]
\label{thm:coverinput}
For every lattice $\Delta\le G_2$, there are
$A_\Delta,B_\Delta>0$, an integer $\ell_\Delta\ge1$,
and $C_\Delta>0$ such that, for $L\ge2$, $\varepsilon\in\{1,-1\}$,
$y\in X_\Delta$, and $F\in C_b^{\ell_\Delta}(X_\Delta)$,
\begin{equation}\label{eq:coverhorocycle}
 \left|\frac1L\int_0^L F(u_2(\varepsilon s)y)\dd s-m_\Delta(F)\right|
 \le C_\Delta\mathcal D_{\ell_\Delta}(F)
 L^{-A_\Delta}\inj_\Delta(a_2(L)^{-1}y)^{-B_\Delta},
\end{equation}
where $a_2(L)=\diag(L^{1/2},L^{-1/2})$.
\end{theorem}

\begin{proof}[Deduction]
Apply \cite[Theorem B.1]{Zheng2015} to $G_2$ and its expanding
horospherical group. This theorem is stated for the right quotient
$\Delta\backslash G_2$ and an arbitrary lattice $\Delta$ in a
semisimple group, so it applies to $\SL_2(\R)$ even if $-I_2\notin\Delta$.
If $y=g\Delta$, inversion identifies it with
$x=\Delta g^{-1}$. Under this identification
$u_2(\varepsilon s)y$ becomes $xu_2(-\varepsilon s)$, and the
renormalized point $x a_{\log L}$ in the proof of the cited theorem
becomes $a_2(L)^{-1}y$. Here its diagonal $a_{\log L}$ is
$a_2(L)$, since
$a_2(L)u_2(s)a_2(L)^{-1}=u_2(Ls)$.
The one-dimensional averaging box in Zheng's Theorem~B.1 is $[0,L]$
after choosing the appropriate unit root vector; this covers both
signs of $\varepsilon$.
The injectivity radii on the two quotients are comparable under
inversion, up to constants depending on the fixed metrics.
Thus Theorem~B.1 yields the displayed estimate for smooth $F$,
with possibly renamed positive exponents and a fixed derivative order.
Convolution with a smooth approximate identity on $G_2$ extends
the smooth-function estimate to $C_b^{\ell_\Delta}$: the approximants
converge uniformly, and their derivatives through order
$\ell_\Delta$ are bounded by a fixed multiple of
$\mathcal D_{\ell_\Delta}(F)$. Pass to the limit in the integral
and in the invariant probability mean.
The constants $A_\Delta,B_\Delta,C_\Delta,\ell_\Delta$ are chosen
for each lattice $\Delta$; no bound uniform in $\Delta$ is used.
\end{proof}

For the remainder of this subsection, let $\Delta\le\Gamma_2$
have finite index, and retain the projection $y\mapsto\bar y$ to
$X_2$. For the coefficient test, write
\[
 \psi(t)=I_2+\sum_{j=1}^d t^j\psi_j\in\SL_2(\mathbb R[t]),
 \qquad g\in G_2,
\]
where $\psi$ is nonconstant, $\psi(0)=\groupid$, and $d\ge1$. Assume
\begin{equation}\label{eq:separatedmodularhypothesis}
 \left(\sum_{j=1}^d\norm{\psi_jgw}^2\right)^{1/2}
       \ge c\norm w^{-\tau}
       \qquad(0\ne w\in\mathbb Z^2),
\end{equation}
with $0<c\le1$ and $\tau\ge0$. For this quotient, a change of lift
is $g\mapsto g\gamma$ with $\gamma\in\Delta$. The calculation in
Remark~\ref{rem:coefficientlift} gives coefficient separation for
the new lift with the same exponent $\tau$.
The leading term of $\psi'\psi^{-1}$ is a nonzero square-zero matrix by
Lemma~\ref{lem:leading}. Choose $b_0\in G_2$ and
$\varepsilon\in\{1,-1\}$ as in Lemma~\ref{lem:normalization} and set
\begin{equation}\label{eq:separatedmodularnormalization}
 P=b_0^{-1}\psi b_0,\qquad h=b_0^{-1}g.
\end{equation}
For the constants below put
\begin{align*}
 \alpha&=(\tau+1)^{-1},
 &c_*&=c\min\{1,s_{\min}(b_0^{-1})\},\\
 m_h&=s_{\min}(h)>0,
 &K&=\min\{1,\kappa_d m_h^{1-\alpha}c_*^\alpha\}>0.
\end{align*}
Here $\kappa_d$ is fixed by \eqref{eq:vandermondeconstant}.

Set
\begin{equation}\label{eq:coverexponent}
 b_\Delta=\min\left\{\frac14,
                  \frac{A_\Delta}{4(B_\Delta+1)}\right\},\qquad
 \nu_{\Delta,d}=\min\left\{\frac{A_\Delta}{2},
                         \frac{b_\Delta}{d},\frac14\right\},\qquad
 \delta_{\Delta,d}=\frac{\nu_{\Delta,d}}{2(d+1)}.
\end{equation}

\begin{proposition}
\label{prop:congruencemodular}
Let $\psi\in\SL_2(\R[t])$ be normalized, nonconstant, and of degree
at most $d$, and let $g\in G_2$. Under coefficient separation
\eqref{eq:separatedmodularhypothesis}, for every
$F\in C_b^{\ell_\Delta}(X_\Delta)$,
\begin{equation}\label{eq:coverpolynomial}
 \left|\frac1T\int_0^T F(\psi(t)g\Delta)\dd t-m_\Delta(F)\right|
 \le C\mathcal D_{\ell_\Delta}(F_{b_0})
                 T^{-\delta_{\Delta,d}/(\tau+1)}\qquad(T\ge2).
\end{equation}
Here $b_0$ normalizes the leading tangent as in
\eqref{eq:separatedmodularnormalization}. More explicitly, with
$P,h,K$ as in that normalization, one may take
\[
 C=C_{\Delta,d}\bigl(1+M(P)+K^{-\beta_{\Delta,d}}\bigr),
 \qquad \beta_{\Delta,d}=\frac{2\nu_{\Delta,d}}{4(d+1)-1}.
\]
The constant $C_{\Delta,d}$ is independent of $\psi,g,c,\tau$.
\end{proposition}
\begin{proof}
The identity $\psi(t)g=b_0P(t)h$ shows that the average in
\eqref{eq:coverpolynomial} is the average of $F_{b_0}$ along
$P(t)h\Delta$. Left invariance gives
$m_\Delta(F_{b_0})=m_\Delta(F)$.
We establish the flow hypothesis for an arbitrary test function
and then apply the transfer lemma to $F_{b_0}$.
On the set $\latticemin(a_2(L)^{-1}\bar y)\ge L^{-b_\Delta}$,
Lemma~\ref{lem:injectivity} gives
$\inj_\Delta(a_2(L)^{-1}y)\ge cL^{-2b_\Delta}$. Therefore
Theorem~\ref{thm:coverinput} gives
\[
 |U_L^\varepsilon F(y)-m_\Delta(F)|
 \le C_\Delta\mathcal D_{\ell_\Delta}(F)
       L^{-A_\Delta+2B_\Delta b_\Delta}
 \le C_\Delta\mathcal D_{\ell_\Delta}(F)L^{-A_\Delta/2}.
\]
This is \eqref{eq:transferinput} with $\omega=b_\Delta$,
$\zeta=A_\Delta/2$, $\ell=\ell_\Delta$, and $L_0=2$.
For each nonzero integral $w$, the nonconstant coefficient
norm of $P(t)hw$ is at least $c_*\norm w^{-\tau}$, since
$P_jhw=b_0^{-1}\psi_jgw$. Lemma~\ref{lem:separatedgrowth}
therefore gives \eqref{eq:separatedmodulargrowth} with the
$\alpha,K$ above. Lemma~\ref{lem:modulartransfer}, with
$\omega=b_\Delta$, $\zeta=A_\Delta/2$,
$\ell=\ell_\Delta$, and $L_0=2$, gives the stated
estimate and its explicit constant.
No density assertion on the cover is used.
\end{proof}

\subsection{Compact rank-one quotients}\label{subsec:compactrankone}

The same horocycle estimate gives a simpler polynomial estimate when
the quotient is compact. In this subsection $\Delta$ is any fixed
cocompact lattice in $G_2$, and $m_\Delta$ is Haar probability on
$X_\Delta=G_2/\Delta$.

\begin{proposition}[Equidistribution on compact rank-one quotients]
\label{prop:compactpolynomial}
Let $\psi\in\SL_2(\mathbb R[t])$ be normalized, nonconstant,
and of degree at most $d\ge1$. Choose a matrix $b_0$ normalizing
its leading tangent as in Lemma~\ref{lem:normalization}, and put
$P=b_0^{-1}\psi b_0$.
With $A_\Delta,\ell_\Delta$ from Theorem~\ref{thm:coverinput},
set
\begin{equation}\label{eq:compactmodularexponent}
 \delta_{\Delta,d}^{\mathrm c}
       =\min\left\{\frac12,\frac{A_\Delta}{2(d+1)}\right\}.
\end{equation}
Then, uniformly in $g\in G_2$, for $T\ge2$ and
$F\in C^{\ell_\Delta}(X_\Delta)$,
\begin{equation}\label{eq:compactpolynomial}
 \left|\frac1T\int_0^T F(\psi(t)g\Delta)\dd t-m_\Delta(F)\right|
 \le C_{\Delta,d}(1+M(P))\,
       \mathcal D_{\ell_\Delta}(F_{b_0})
       T^{-\delta_{\Delta,d}^{\mathrm c}}.
\end{equation}
The constant depends on $\Delta,d$ (in particular, on the
injectivity radius of the compact space $X_\Delta$).
No arithmetic condition on $\psi$ or $g$ is imposed.
\end{proposition}
\begin{proof}
Compactness gives $\inf_{y\in X_\Delta}\inj_\Delta(y)>0$.
Theorem~\ref{thm:coverinput} therefore yields
\[
 |U_L^\varepsilon F(y)-m_\Delta(F)|
       \le C_\Delta\mathcal D_{\ell_\Delta}(F)L^{-A_\Delta}
\]
for every $y$, $L\ge2$, and either sign.
Write $h=b_0^{-1}g$ and replace $F$ by $F_{b_0}$.
For $L=T^{1/[2(d+1)]}$, Lemma~\ref{lem:smoothing} bounds the
discrepancy on $[T,2T]$ by
\[
 C_{\Delta,d}\mathcal D_{\ell_\Delta}(F_{b_0})
 \left(M(P)T^{-1/2}+T^{-A_\Delta/[2(d+1)]}\right),
\]
once $T\ge\max\{T_d,2^{2(d+1)}\}$. This is uniform in $h$.
Lemma~\ref{lem:dyadic} gives \eqref{eq:compactpolynomial}.
\end{proof}

\section{Polynomial estimates on the ambient quotient}
\label{sec:ambientproof}

The ambient case in Proposition~\ref{prop:supportclassification}
is the full support
$X=\SL_3(\mathbb R)/\SL_3(\mathbb Z)$, with limiting measure the
Haar probability $m_X$. This section proves power estimates for
polynomial trajectories on $X$ under the coefficient bounds
\eqref{eq:coefficientseparation} in the standard, exterior-square,
and symmetric-square representations. The nonzero leading tangent
$N$ of $\Phi'\Phi^{-1}$ is either singular ($N^2=0$) or regular
($N^2\ne0$), by Lemma~\ref{lem:leading}.
Propositions~\ref{prop:separatedsingular} and~\ref{prop:separatedregular}
treat these two cases by combining tangent averaging and quantitative
nondivergence with the flow estimates of
Theorems~\ref{thm:yang} and~\ref{prop:regularinput}.
The standard and exterior-square tests control excursions into the
cusp, where lattices have short vectors. In the regular case,
Proposition~\ref{prop:extraction} and the symmetric-square test also
control proximity to periodic orbits of the principal $\SL_2$
subgroup, excluding the obstruction in the regular flow estimate.

\subsection{Ambient representation tests}
\label{subsec:separation}

For the ambient quotient, use the representation family
\begin{equation}\label{eq:intermediatefamilies}
 \mathcal R_{\mathrm G}
 =\{\mathrm{Std}_3,\ \wedge^2\mathrm{Std}_3,\
                       \rho_{\mathrm{sym}}\}.
\end{equation}
Here $\mathrm{Std}_n$ is the standard action on $\R^n$.
We regard $Q$ as a symmetric tensor (Section~\ref{app:tensors});
the symmetric-square action is $\rho_{\mathrm{sym}}(a)Q=aQa^{\mathsf T}$.
For a representation $\rho$, write $V_\rho$ for its real vector
space and $V_{\rho,\Z}$ for the specified integral lattice. For a
subgroup $L$, write $V_\rho^L=\{v:\rho(\ell)v=v\text{ for all }\ell\in L\}$.
The action $\wedge^k\rho$ sends
$v_1\wedge\cdots\wedge v_k$ to
$\rho(g)v_1\wedge\cdots\wedge\rho(g)v_k$.
The standard and exterior lattices have their usual orthonormal
integral bases. For $V=\Sym^2(\R^3)$, take the lattice $V_\Z$ of
integral symmetric matrices and declare the basis
\begin{equation}\label{eq:symbasis}
 D_1=E_{11},\ D_2=E_{22},\ D_3=E_{33},\quad
 S_{12}=E_{12}+E_{21},\ S_{13}=E_{13}+E_{31},\ S_{23}=E_{23}+E_{32}
\end{equation}
orthonormal. For the exterior powers used in nondivergence, use
its induced exterior bases.
These lattices are preserved by the corresponding integral group.

Fix a nonconstant normalized ambient curve
$\Phi\in\SL_3(\mathbb R[t])$ of degree at most $d$ and a
starting matrix $g\in G$. Put $\Psi=\Phi$ and expand the
\emph{lifted trajectory}:
\begin{equation}\label{eq:liftedcoefficients}
 \rho(\Psi(t)g)=\rho(g)+\sum_{j=1}^{D_\rho}t^jA_{\rho,j},\qquad
 \mathcal A_{\rho,\Psi,g}w=(A_{\rho,1}w,\ldots,A_{\rho,D_\rho}w).
\end{equation}
Zero coefficients may be appended as needed. The degree bounds are $d$ for the standard
representation and $2d$ for both $\wedge^2\mathrm{Std}_3$ and
$\rho_{\mathrm{sym}}$. The coefficient-separation form is
\begin{equation}\label{eq:coefficientseparation}
 \norm{\mathcal A_{\rho,\Psi,g}w}\ge c\norm w^{-\tau}
 \qquad(\rho\in\mathcal R_{\mathrm G},\quad
                              0\ne w\in V_{\rho,\Z}).
\end{equation}

We call these inequalities $(c,\tau)$-coefficient separation.
They are auxiliary hypotheses for the estimates below;
Lemma~\ref{lem:intrinsictransfer} supplies them from the one
intrinsic condition in the main theorem. An integral vector is
primitive if its coordinates have greatest common divisor one.
It suffices to test primitive vectors, since coefficient norms
scale linearly under integer dilation.

The next two lemmas verify the hypotheses needed to apply
Lemma~\ref{lem:intrinsictransfer} to this family: absence of fixed
vectors and polynomial degree at most two.

\begin{lemma}
\label{lem:allwedges}
For the family in \eqref{eq:intermediatefamilies},
\begin{equation}\label{eq:allwedges}
 V_\rho^G=\{0\}\qquad(\rho\in\mathcal R_{\mathrm G}).
\end{equation}
\end{lemma}

\begin{proof}
For the positive diagonal torus, the weights of the standard
representation are $e_i$, those of its exterior square are
$e_i+e_j$ with $i<j$, and those of the symmetric square are
$2e_i$ and $e_i+e_j$ with $i<j$; see
Lemma~\ref{app:weights}. A weight vanishes on the determinant-one
torus exactly when its three coordinates are equal. None of the
listed weights has this property. Thus these representations have
no nonzero torus-fixed vectors, and hence no nonzero $G$-fixed vectors.
\end{proof}

\begin{lemma}[Degree-two polynomial tests]\label{lem:degreecutoff}
The three representations in \eqref{eq:intermediatefamilies} have
matrix entries represented by rational polynomials of degrees at most
$1,2,2$, respectively. Thus degree two suffices for all the
coefficient-separation hypotheses in this section.
\end{lemma}
\begin{proof}
The standard action has linear entries. The entries of its exterior
square are $2$ by $2$ minors, hence are quadratic. The entries of
$Q\mapsto AQA^{\mathsf T}$ in the basis \eqref{eq:symbasis} are also
quadratic (Lemma~\ref{app:tensor-degrees}).
\end{proof}

\iffalse
The higher exterior powers of $\rho_{\mathrm{sym}}$ still occur in
quantitative nondivergence. Their covolumes will be bounded below
using one point in a fixed thick part of the ambient quotient.
No coefficient-separation hypothesis is imposed on these exterior
powers.
\fi

\subsection{Normalization and growth constants}

Retain the normalized curve $\Phi$ of degree at most $d$
and the specified lift $x=g\Gamma$ from the preceding subsection.
Suppose the lifted coefficient maps satisfy
\eqref{eq:coefficientseparation} with $0<c\le1$ and $\tau\ge0$.
Write
\[
 \Phi'(t)\Phi(t)^{-1}=t^qN+\text{lower powers},\qquad
 \alpha=\frac1{\tau+1}.
\]
Lemma~\ref{lem:leading} gives $0\le q\le3d-1$ and $N^{d+1}=0$.
Choose $b_0\in G$ as in \eqref{eq:mainnormalizers}, and set
\begin{equation}\label{eq:separatedambientnormalization}
 \bar N=b_0^{-1}Nb_0,
 \qquad P(t)=b_0^{-1}\Phi(t)b_0,\qquad h=b_0^{-1}g.
\end{equation}
The letter $b_0$ distinguishes this matrix from the fixed positive
scalar $b$ in \eqref{eq:bchoice}.

Put
\[
 \mathcal F_{\mathrm{std}}=
       \{\mathrm{Std}_3,\wedge^2\mathrm{Std}_3\},\qquad
 \mathcal F_{\mathrm{aux}}=\{\rho_{\mathrm{sym}}\}.
\]
Use the bases and degree bounds $D_\rho$ from
Section~\ref{subsec:separation}. Lemma~\ref{lem:coefficientnormalization}
gives normalized separation with
\begin{equation}\label{eq:separatedambientc}
 c_*=c\min\left\{1,\min_{\rho\in\mathcal R_{\mathrm G}}
                     s_{\min}(\rho(b_0^{-1}))\right\}>0.
\end{equation}
For each $\rho$ define
\begin{align}
 m_\rho&=s_{\min}(\rho(h)),\qquad
 K_\rho=\min\{1,\kappa_{D_\rho}m_\rho^{1-\alpha}c_*^\alpha\},
                                             \label{eq:separatedambientK}\\
 K_{\mathrm{std}}&=\min_{\rho\in\mathcal F_{\mathrm{std}}}K_\rho,
 \qquad
 K_{\mathrm{aux}}=\min_{\rho\in\mathcal F_{\mathrm{aux}}}K_\rho.
 \nonumber
\end{align}
Here $\kappa_{D_\rho}$ is fixed by \eqref{eq:vandermondeconstant}.
Lemma~\ref{lem:separatedgrowth} gives
\begin{equation}\label{eq:separatedambientgrowth}
 \sup_{T\le t\le2T}\norm{\rho(P(t)h)w}\ge K_\rho T^\alpha
 \quad(T\ge1,\ \rho\in\mathcal R_{\mathrm G},\
                         0\ne w\in V_{\rho,\Z}).
\end{equation}
The vector may vary with $T$. The coefficient bound $M(P)$ is
\eqref{eq:incrementM}.

\subsection{The singular tangent}

Set
\begin{equation}\label{eq:separatedsingularrates}
 \nu_{\mathrm S}=\min\{\etaY,1/(16d)\},\qquad
 \theta_{\mathrm S}=\frac{\alpha}{2(d+1)},\qquad
 \delta_{\mathrm S}=\theta_{\mathrm S}\nu_{\mathrm S},\qquad
 \beta_{\mathrm S}=\frac{\nu_{\mathrm S}}{2d+1}.
\end{equation}
Define the singular-tangent exponent by
\begin{equation}\label{eq:exponent}
 \delta_d^{\mathrm{sing}}=\frac{\nu_{\mathrm S}}{2(d+1)}.
\end{equation}
Then $\delta_{\mathrm S}=\alpha\delta_d^{\mathrm{sing}}$ and
$\beta_{\mathrm S}=\delta_{\mathrm S}/(\alpha-\theta_{\mathrm S})$.
In particular, $\beta_{\mathrm S}$ does not depend on $\tau$.

\begin{proposition}[Ambient equidistribution for singular tangents]
\label{prop:separatedsingular}
If the normalized tangent is $E_{23}$, then
\begin{equation}\label{eq:separatedsingular}
 |A_T^P F(h\Gamma)-m_X(F)|
 \le C_d\bigl(1+M(P)+K_{\mathrm{std}}^{-\beta_{\mathrm S}}\bigr)
            \Q(F)T^{-\delta_{\mathrm S}}\qquad(T\ge2).
\end{equation}
Only separation in $\mathcal F_{\mathrm{std}}$ is needed here;
with that weaker hypothesis, take the minimum in
\eqref{eq:separatedambientc} over $\mathcal F_{\mathrm{std}}$.
The constant $C_d$ depends only on $d$, the norms in
\eqref{eq:allnorm}, and the constants in
Theorem~\ref{thm:yang}, Theorem~\ref{prop:km}, and
Lemma~\ref{lem:smoothing}; it is independent of $P,h,c,\tau$.
\end{proposition}

\begin{proof}
Fix $L=T^{\theta_{\mathrm S}}$ and write $\sigma=\log L$, keeping
this time parameter distinct from the arithmetic exponent $\tau$.
Choose a fixed $L_{\mathrm S}$ such that
\[
 L_{\mathrm S}\ge e^{\max\{2,\sigma_Y\}},\qquad
 e^{1/3}b_YL_{\mathrm S}^{-1/4}\le\epsilon_0,
\]
where $\epsilon_0<1$ is the constant in Lemma~\ref{lem:qndboth}
for degree bound $2d$. Use the explicit threshold
\begin{equation}\label{eq:separatedsingularthreshold}
 T_{\mathrm S,*}=\max\left\{T_d,
 L_{\mathrm S}^{1/\theta_{\mathrm S}},
 K_{\mathrm{std}}^{-1/(\alpha-\theta_{\mathrm S})}\right\}.
\end{equation}

For $|\ell|\le2\sigma$, the logarithms of the diagonal entries
of $D_{\sigma,\ell}$ are
$\ell/3,-\sigma/2-\ell/6,\sigma/2-\ell/6$; their absolute values
are at most $5\sigma/6$. Pairwise sums are their negatives. Thus
\begin{equation}\label{eq:diagonalnorms}
 \norm{\wedge^kD_{\sigma,\ell}}_{\mathrm{op}},\quad
 \norm{(\wedge^kD_{\sigma,\ell})^{-1}}_{\mathrm{op}}\le L
       \qquad(k=1,2).
\end{equation}
For $T\ge T_{\mathrm S,*}$, \eqref{eq:separatedambientgrowth}
therefore gives, for $k=1,2$,
\[
 \sup_{T\le t\le2T}
 \norm{(\wedge^k(D_{\sigma,\ell}P(t)h))w}
       \ge L^{-1}K_{\mathrm{std}}T^\alpha\ge1
       \quad(0\ne w\in\wedge^k\Z^3).
\]
These include every proper primitive-subgroup lower bound in
dimension three; the full-rank covolume is $1$. Entries in the two
exterior powers have degree at most $2d$. Lemma~\ref{lem:qndboth},
including its application to the dual lattice, therefore gives
\[
 \frac1T\Leb\{t\in[T,2T]:
 \extmin(D_{\sigma,\ell}P(t)h\Gamma)<\zeta\}
       \le C_d\zeta^{1/(2d)}\qquad(0<\zeta\le\epsilon_0).
\]

Let $\zeta_L=e^{1/3}b_YL^{-1/4}$ and remove the finite union
\[
 \mathcal B_{\mathrm S}=
 \bigcup_{\substack{j\in\Z\\ |j|\le2\sigma}}
 \{t\in[T,2T]:\extmin(D_{\sigma,j}P(t)h\Gamma)<\zeta_L\}.
\]
This is measurable. If \eqref{eq:yangcondition} fails at some
$0.99\sigma\le|\ell|\le1.01\sigma$, an integer $j$ with
$|j-\ell|\le1$ has $|j|\le2\sigma$, and the corresponding exterior
vector is enlarged by at most $e^{1/3}$. Such a vector exists because
the threshold is below $1$, so the cap in $\extmin$ cannot cause
failure of the condition. Its length becomes less
than $e^{1/3}b_YL^{-33/100}\le\zeta_L$. Thus this finite union
contains every parameter for which \eqref{eq:yangcondition} fails.
The preceding nondivergence bound gives
\[
 \frac{\Leb(\mathcal B_{\mathrm S})}{T}
   \le C_d(1+\log L)L^{-1/(8d)}\le C_dL^{-1/(16d)}.
\]
Consequently, for every
$t\in[T,2T]\setminus\mathcal B_{\mathrm S}$, the point
$y=P(t)h\Gamma$ satisfies \eqref{eq:yangcondition} with
$\sigma=\log L$. We therefore apply Theorem~\ref{thm:yang}
at this point, with averaging length $L=T^{\theta_{\mathrm S}}$.
On $\mathcal B_{\mathrm S}$, use the trivial bound
$2\norm F_\infty$. Integration and the centered version of
Lemma~\ref{lem:smoothing} now give
\begin{align*}
 \left|\frac1T\int_T^{2T}F(P(t)h\Gamma)\dd t-m_X(F)\right|
 &\le C_d\Q(F)
       \left(L^{-\etaY}+L^{-1/(16d)}+M(P)T^{-1/2}\right)\\
 &\le C_d(1+M(P))\Q(F)T^{-\delta_{\mathrm S}}.
\end{align*}
Here $0<\delta_{\mathrm S}\le1/64$, and the smoothing lemma
already includes $q=0$.

Apply Lemma~\ref{lem:dyadic}. Since
\[
 T_{\mathrm S,*}^{\delta_{\mathrm S}}
 \le T_d^{1/2}+L_{\mathrm S}^{\nu_{\mathrm S}}
                      +K_{\mathrm{std}}^{-\beta_{\mathrm S}},
\]
the fixed first two terms are absorbed in $C_d$. This proves
\eqref{eq:separatedsingular} and its asserted uniformity in $\tau$.
\end{proof}

\subsection{Arithmetic extraction for the regular tangent}
\label{subsec:arithmeticextraction}

Use the principal group $H_{\mathrm{pr}}$ in
\eqref{eq:regularform}, distinct from the support group $H=G$.
In the extraction argument write $\rho=\rho_{\mathrm{sym}}$ on
$V=\Sym^2(\R^3)$ with the basis and lattice of \eqref{eq:symbasis}.
This coordinate norm is uniformly equivalent to the Frobenius norm.
The lattice has covolume one, $\rho(\Gamma)$ preserves it, and
$\det\rho(g)=(\det g)^4=1$.

A real subspace is rational if it has a basis of rational vectors.
Fix an ordering $B_1,\ldots,B_8$ of the integral basis of
$\mathfrak{sl}_3(\Z)$ consisting of the six off-diagonal matrix
units and $E_{11}-E_{22},E_{22}-E_{33}$. Equip
$\mathfrak{sl}_3(\R)$ with the inner product for which this basis
is orthonormal, and use the induced Euclidean norms on its
exterior powers.
For a rational three-dimensional subspace $\mathfrak j$ of
$\mathfrak{sl}_3(\R)$, Lemma~\ref{app:plucker} gives a primitive
integral generator $p_{\mathfrak j}$ of its exterior line. Define
\begin{equation}\label{eq:heightdefinition}
 \Height(\mathfrak j)=\norm{p_{\mathfrak j}},\qquad
 p_{\mathfrak j}\text{ primitive integral in }
       \wedge^3\mathfrak j\cap\wedge^3\mathfrak{sl}_3(\Z).
\end{equation}
Either sign is allowed. For $I=\{i_1<i_2<i_3\}\subset\{1,\ldots,8\}$,
put $B_I=B_{i_1}\wedge B_{i_2}\wedge B_{i_3}$ and write
\[
 p_{\mathfrak j}=\sum_{|I|=3}p_I B_I.
\]
The integer coefficients $p_I$ are the Pl\"ucker coordinates of
$\mathfrak j$ in primitive integral normalization. They are
determined up to a common sign. Equivalently, they are the
corresponding $3$ by $3$ minors of a matrix whose columns are the
coordinates, in the basis $B_1,\ldots,B_8$, of a lattice basis of
$\mathfrak j\cap\mathfrak{sl}_3(\Z)$. Thus
$\Height(\mathfrak j)=(\sum_{|I|=3}p_I^2)^{1/2}$.
Different fixed norms only affect constants. Set
$\mathcal N(b)=\max\{1,\norm b_{\mathrm{op}},\norm{b^{-1}}_{\mathrm{op}}\}$.

The following reduction and volume--height bounds combine the
cited arithmetic estimates with the rationality of the stabilizer.

\begin{lemma}[{\cite[Lemmas~B.10 and~11.2]{LMWY2025};
\cite[Proposition~2.1(2)]{Shah1994}}]
\label{lem:arithmeticinputs}
There are fixed constants $C_1,D,E\ge1$ with these properties:
\begin{enumerate}
\item If $\inj(h\Gamma)\ge\eta$, $0<\eta<1$, there is
      $\gamma_0\in\Gamma$ with
      $\mathcal N(h\gamma_0)\le C_1\eta^{-D}$.
\item If $H_{\mathrm{pr}}b\Gamma$ is periodic, then
      $\mathfrak j_b=\Ad(b^{-1})\Lie(H_{\mathrm{pr}})$ is rational, and
      \begin{equation}\label{eq:volumeheightinput}
       \Height(\mathfrak j_b)
       \le C_1\bigl(\mathcal N(b)\max\{1,\vol(H_{\mathrm{pr}}b\Gamma)\}\bigr)^E.
      \end{equation}
\end{enumerate}
\end{lemma}

\begin{proof}[Deduction]
The first assertion is \cite[Lemma B.10]{LMWY2025}; passing to
$\mathcal N$ only enlarges $D$: for a determinant-one
three-by-three matrix, each inverse entry is a quadratic cofactor.
The injectivity radius used in Lemma~B.10 is capped at a different
fixed constant; the two capped radii are comparable, so the same
power bound applies.

For the second assertion, put $J_b=b^{-1}H_{\mathrm{pr}}b$.
First, \cite[Proposition 2.1(2)]{Shah1994} gives its rationality.
Indeed, apply that result to the regular algebraic map
\[
 (s,t)\longmapsto b^{-1}\exp(sN_0)
                       \exp\bigl(2t(E_{21}+E_{32})\bigr)b.
\]
Its image generates $J_b$, because the two nilpotents generate the
principal $\mathfrak{sl}_2$ Lie algebra. The orbit $J_b\Gamma$ is
closed, being the left translate by $b^{-1}$ of
$H_{\mathrm{pr}}b\Gamma$. The cited proposition implies that the
Zariski closure of $J_b\cap\Gamma$ contains $J_b$
(Proposition~\ref{app:shah-normalization}). Conversely it lies
in the real algebraic group whose identity component of real points
is $J_b$, so these groups have the same Lie algebra. The
Zariski closure of a subgroup of $\SL_3(\mathbb Z)$ is defined over
$\mathbb Q$: in each bounded degree its vanishing ideal is the
solution space of linear equations with rational coefficients
(Lemma~\ref{app:rational-envelope}). Thus
$\mathfrak j_b=\Lie(J_b)$ is rational.

Now use \cite[Lemma 11.2 and height definition (11.1)]{LMWY2025}.
In the present split integral group $G=\SL_3(\mathbb R)$,
$\Gamma=\SL_3(\mathbb Z)$, its auxiliary arithmetic group is
$\SL_3$ itself and its compact factor is trivial. Its first
comparison bounds the height of $J_b$ by a fixed power of the
volume of $J_b\Gamma/\Gamma$; its second bounds this volume by a
fixed power of $\mathcal N(b)$ times the volume of
$H_{\mathrm{pr}}b\Gamma/\Gamma$. Fixed norms and Haar
normalizations are comparable. Combining the two comparisons and
enlarging $C_1,E$ gives \eqref{eq:volumeheightinput}.
\end{proof}

\begin{lemma}
\label{lem:rationalkernel}
Suppose $\mathfrak j=\Ad(b^{-1})\Lie(H_{\mathrm{pr}})$ is rational and has height
$H_0\ge1$. Then there is a primitive nonzero $q\in V_{\Z}$ with
\begin{equation}\label{eq:kernelbound}
 \norm q\ll H_0^5,\qquad
 Xq+qX^{\mathsf T}=0\quad(X\in\mathfrak j).
\end{equation}
Consequently $\rho(b)q=c_qQ_0$ for some $c_q\ne0$ and
\begin{equation}\label{eq:scalarbound}
 |c_q|\ll\mathcal N(b)^2H_0^5.
\end{equation}
\end{lemma}

\begin{proof}
Choose a nonzero coordinate $p_I$ of the primitive Pl\"ucker vector
$p=p_{\mathfrak j}$. Projection onto those three coordinates gives an
isomorphism from $\mathfrak j$ onto $\R^3$. The basis whose projected
coordinates are the standard basis has entries equal, up to signs,
to $p_J/p_I$. Multiplication by $p_I$ therefore gives three integral
matrices $X_1,X_2,X_3$ spanning $\mathfrak j$, each with entries
$O(H_0)$. A lattice basis of $\mathfrak j\cap\mathfrak{sl}_3(\Z)$
is not needed here.

In the six coordinates \eqref{eq:symbasis}, combine the equations
$X_iQ+QX_i^{\mathsf T}=0$ into a single integral matrix system with
six columns and entries $O(H_0)$. Its kernel is one-dimensional.
Indeed, for $\Lie(H_{\mathrm{pr}})$ itself, invariance under
$\diag(1,0,-1)$ first restricts $Q$ to
$aS_{13}+bD_2$. Invariance under $N_0$ then gives $a+b=0$, leaving
exactly $\R Q_0$. Conversely $Q_0$ is invariant under all of $H_{\mathrm{pr}}$.
Conjugating by $b$ preserves this kernel dimension.
The equivalence of the Lie algebra equations with group
invariance is Lemma~\ref{app:lie-tests}, with the tensor
convention in Section~\ref{app:tensors}.

The equation matrix thus has rank five. Choose five independent rows.
Their six signed maximal minors form a nonzero integral kernel vector,
each coordinate bounded by $O(H_0^5)$. These rows span the full row
space, so this vector solves every equation. Divide by the greatest
common divisor of its coordinates to obtain $q$ satisfying
\eqref{eq:kernelbound}. Since the invariant space for $H_{\mathrm{pr}}$ is
$\R Q_0$, we have $\rho(b)q=c_qQ_0$. Invertibility of $\rho(b)$
ensures $c_q\ne0$, and $\norm{\rho(b)}_{\mathrm{op}}\ll\mathcal N(b)^2$
gives \eqref{eq:scalarbound}.
\end{proof}

\begin{proposition}[Integral witness for a nearby periodic principal orbit]
\label{prop:extraction}
There are fixed $B,C_2\ge1$ and $c_2>0$ such that the following holds.
Let $0<\eta<1$, $R\ge2$, $y=h\Gamma$, and $\inj(y)\ge\eta$.
Suppose $H_{\mathrm{pr}}z$ is periodic, $\vol(H_{\mathrm{pr}}z)\le R$, and
$d_X(y,z)\le\varepsilon\le c_2$. Then some nonzero $w\in V_{\Z}$
and some $c_w\ne0$ satisfy
\begin{equation}\label{eq:extractedwitness}
 \rho(h)w=c_wQ_0+E_w,\qquad
 |c_w|\le M,\qquad \norm{E_w}\le M\varepsilon,
 \quad M=C_2(R/\eta)^B.
\end{equation}
All constants are independent of the chosen representative $h$.
\end{proposition}

\begin{proof}
By lifting a short path for the quotient metric, choose $k\in G$
such that $z=ky$ and $\norm{k-I_3}\ll\varepsilon$.
The bound is uniform: right invariance gives
$d_G(kh,h)=d_G(k,\groupid)$. Decrease $c_2$ so that $k,k^{-1}$ stay in a
fixed identity neighborhood.

Use Lemma~\ref{lem:arithmeticinputs}(1) at $y$ to choose
$a=h\gamma_0$ with $\mathcal N(a)\ll\eta^{-D}$, and set $b=ka$.
Then $z=b\Gamma$ and $\mathcal N(b)\ll\eta^{-D}$.
This choice avoids needing to transfer an injectivity-radius bound
from $y$ to $z$. By \eqref{eq:volumeheightinput},
\[
 H_0:=\Height(\mathfrak j_b)\ll R^E\eta^{-DE}.
\]
Lemma~\ref{lem:rationalkernel} gives $q\in V_{\Z}\setminus\{0\}$ and
\[
 \rho(b)q=c_qQ_0,\qquad
 |c_q|\ll R^{5E}\eta^{-D(5E+2)}.
\]
Put $w=\rho(\gamma_0)q$. This vector is integral and nonzero because
$\rho(\gamma_0)$ is a lattice automorphism. Since $h\gamma_0=k^{-1}b$,
\[
 \rho(h)w=c_q\rho(k^{-1})Q_0
         =c_qQ_0+O(|c_q|\varepsilon).
\]
Take any fixed
\begin{equation}\label{eq:extractionB}
 B\ge\max\{1,5E,D(5E+2)\}
\end{equation}
and enlarge $C_2$ to get \eqref{eq:extractedwitness}.
No bound on $\gamma_0$ or on $\norm w$ is asserted or needed.
\end{proof}

\subsection{Cusp geometry and compression}

For $0<\varepsilon<1$, let $b_\varepsilon\in\SL(V)$ act by
$\varepsilon^{5/6}$ on $\R Q_0$ and by $\varepsilon^{-1/6}$ on its
orthogonal complement, using the norm \eqref{eq:symbasis}. Then
\[
 \det b_\varepsilon
       =\varepsilon^{5/6}(\varepsilon^{-1/6})^5=1.
\]
The fixed line is compressed more strongly than its five-dimensional
complement is expanded. This turns a vector close to that line into
a short vector in a unimodular six-dimensional lattice. Specifically,
\begin{align}
 s_{\min}(\wedge^kb_\varepsilon)
     &=\varepsilon^{(6-k)/6}\ge\varepsilon^{5/6}
                      &&(1\le k\le5),\label{eq:compressionweights}\\
 \norm{\wedge^k\rho(a_{\pm\log L}^{\mathrm R})}_{\mathrm{op}}
     &\le L^3 &&(1\le k\le5),\label{eq:diagonalaux}
\end{align}
where $s_{\min}$ denotes the smallest singular value.
For the second assertion, the six diagonal weights are
$L^2,L,1,1,L^{-1},L^{-2}$; the sum of their positive exponents is $3$.

Fix
\begin{equation}\label{eq:compressionconstant}
 C_3=\max\{1,C_2(\norm{Q_0}+1)\}.
\end{equation}
This constant is used in the compression estimate for the vector $w$
from Proposition~\ref{prop:extraction}. The proof below includes all
such short-vector conditions in one measurable exceptional set.

\subsection{The regular tangent}

Fix $A,\kappa,m_0,C_0$ from Theorem~\ref{prop:regularinput},
and $B$ as in \eqref{eq:extractionB}. The following $r,b$ are fixed
positive exponents; this $r$ is unrelated to the torus dimension.
Define
\begin{align}
 b&=\frac1{2(m_0+1)},\label{eq:bchoice}\\
 r&=\min\left\{\frac1{4(A+1)},\frac1{4(b+1)},
                    \frac1{48(B+1)(1+b)}\right\}.\label{eq:rchoice}
\end{align}
Then
\begin{equation}\label{eq:parameterslack}
 bm_0<\tfrac12,\qquad Ar<\tfrac14,\qquad br<\tfrac14,
 \qquad Br(1+b)<\tfrac1{48}.
\end{equation}
Set
\begin{align}
 \nu_{\mathrm R}&=\min\{r\kappa,rb/(12d),1/(80d)\},
 \qquad \theta_{\mathrm R}=\frac{\alpha}{2(d+2)},
 \qquad \delta_{\mathrm R}=\theta_{\mathrm R}\nu_{\mathrm R},
                                              \label{eq:separatedregularrates}\\
 \beta_{\mathrm{Rc}}&=\frac{\nu_{\mathrm R}}{2d+3},\qquad
 \beta_{\mathrm{Ra}}=\frac{\nu_{\mathrm R}}{2d+4-14/3}.
                                              \label{eq:separatedregularbetas}
\end{align}
Define
\begin{equation}\label{eq:regularexponent}
 \delta_d^{\mathrm{reg}}=\frac{\nu_{\mathrm R}}{2(d+2)},
\end{equation}
and the common ambient exponent
\begin{equation}\label{eq:bothdelta}
 \delta_d=\min\{\delta_d^{\mathrm{sing}},\delta_d^{\mathrm{reg}}\}>0.
\end{equation}
The subscripts $\mathrm{c}$ and $\mathrm{a}$ refer to the cusp and
auxiliary thresholds. Their denominators are positive for $d\ge1$.
The powers $\beta_{\mathrm{Rc}}$ and $\beta_{\mathrm{Ra}}$ are
independent of $\tau$, and
\[
 \delta_{\mathrm R}=\alpha\delta_d^{\mathrm{reg}},\qquad
 \beta_{\mathrm{Rc}}=\frac{\delta_{\mathrm R}}{\alpha-\theta_{\mathrm R}},
 \qquad
 \beta_{\mathrm{Ra}}=\frac{\delta_{\mathrm R}}
                                  {\alpha-14\theta_{\mathrm R}/3}.
\]

\begin{proposition}[Ambient equidistribution for regular tangents]
\label{prop:separatedregular}
If the normalized tangent is $N_0=E_{12}+E_{23}$, then
\begin{equation}\label{eq:separatedregular}
\begin{aligned}
 |A_T^P F(h\Gamma)-m_X(F)|
 &\le C_d\left(1+M(P)+K_{\mathrm{std}}^{-\beta_{\mathrm{Rc}}}
                    +K_{\mathrm{aux}}^{-\beta_{\mathrm{Ra}}}\right)\\
 &\qquad{}\cdot\bigl(\SR(F)+\norm F_{C_G^1}\bigr)T^{-\delta_{\mathrm R}}
       \qquad(T\ge2).
\end{aligned}
\end{equation}
The constant $C_d$ depends only on $d$, the norms in
\eqref{eq:allnorm}, and the constants in
Theorems~\ref{prop:regularinput} and~\ref{prop:km},
Lemma~\ref{lem:arithmeticinputs}, and Lemma~\ref{lem:smoothing}; it is
independent of $P,h,c,\tau$.
\end{proposition}

\begin{proof}
On $[T,2T]$, put
\begin{equation}\label{eq:separatedregularparameters}
 L=T^{\theta_{\mathrm R}},\qquad R=L^r,\qquad
 \eta=R^{-b},\qquad \varepsilon=R^A(\log L)^AL^{-2},\qquad
 y_t=a_{-\log L}^{\mathrm R}P(t)h\Gamma.
\end{equation}
We first specify the fixed constants and thresholds. Let $C_{2d}$
and $\epsilon_0$ be the constants in Lemma~\ref{lem:qndboth} for
degree bound $2d$. Choose $0<\zeta_d\le\min\{1,\epsilon_0\}$ with
$C_{2d}\zeta_d^{1/(2d)}\le1/2$. There is a fixed $0<c_d\le1/6$
such that
\[
 \extmin(a\Gamma)\ge\zeta_d\quad\Longrightarrow\quad
 \norm{(\wedge^k\rho_{\mathrm{sym}})(a)w}\ge c_d
 \quad(2\le k\le5,\ 0\ne w\in\wedge^kV_\Z).
\]
Indeed, the reduction bound \eqref{eq:elementaryreduction} gives
$\gamma\in\Gamma$ for which $a\gamma$ and its inverse have norms
bounded in terms of $\zeta_d$. The inverse exterior actions of
$\rho_{\mathrm{sym}}(a\gamma)$ therefore have bounded norms.
Moreover, $(\wedge^k\rho_{\mathrm{sym}})(\gamma^{-1})w$ is a nonzero
integral vector and has norm at least one. This proves the displayed
bound, uniformly in $a,k,w$.

Put $\chi=1/6-Br(1+b)>0$. A valid fixed choice of threshold is
\begin{equation}\label{eq:separatedregularL}
 L_{\mathrm R}=\max\left\{e,\ 2^{1/r},\ C_0^{1/[r(1-bm_0)]},\ 
       \exp\!\left(\left[\frac{2A}{1-Ar}\right]^2\right),\ 2/c_2,
       C_3^{1/\chi},\ (2C_3/c_d)^{48}\right\}.
\end{equation}
The following estimates remain valid if $L_{\mathrm R}$ is increased.
For $L\ge L_{\mathrm R}$, \eqref{eq:parameterslack} gives
\begin{gather}
 R\ge2,\quad L\ge R^A,\quad R\ge C_0\eta^{-m_0},\quad
 L^{-2}\le\varepsilon\le L^{-1}<\min\{1,c_2\},
                                             \label{eq:separatedregularfixed}\\
 \beta:=C_3(R/\eta)^B\varepsilon^{5/6}
       \le L^{-2/3}.\nonumber
\end{gather}
For the logarithmic inequality, with $u=\log L$ use
$\log u\le\sqrt u$ and
$\sqrt u\ge2A/(1-Ar)$ to get
$(\log L)^A\le L^{(1-Ar)/2}\le L^{1-Ar}$.
For the last line, first bound
$\beta\le C_3L^{Br(1+b)-5/6}$ and then use
$C_3\le L^{\chi}$. Thus $L\ge L_{\mathrm R}$ implies all the
conditions in \eqref{eq:separatedregularfixed}, with $L_{\mathrm R}$
independent of the curve and starting point.

The two data-dependent growth exponents are positive:
\[
 \alpha-\theta_{\mathrm R}>0,\qquad
 \alpha-14\theta_{\mathrm R}/3
   =\alpha\left(1-\frac7{3(d+2)}\right)\ge2\alpha/9.
\]
Take
\begin{equation}\label{eq:separatedregularthreshold}
 T_{\mathrm R,*}=\max\left\{T_d, L_{\mathrm R}^{1/\theta_{\mathrm R}}, K_{\mathrm{std}}^{-1/(\alpha-\theta_{\mathrm R})}, K_{\mathrm{aux}}^{-1/(\alpha-14\theta_{\mathrm R}/3)} \right\}.
\end{equation}
In the rest of the dyadic argument suppose $T\ge T_{\mathrm R,*}$.

\emph{Cusp control.}
Both proper exterior powers of $a_{\pm\log L}^{\mathrm R}$ have
operator norm at most $L$. Hence
\eqref{eq:separatedambientgrowth} gives, for $k=1,2$ and every
nonzero integer $k$-vector,
\[
 \sup_{T\le t\le2T}
 \norm{(\wedge^k(a_{-\log L}^{\mathrm R}P(t)h))w}
       \ge K_{\mathrm{std}}T^\alpha/L\ge1.
\]
Lemma~\ref{lem:qndboth} and
$\inj(y)\ge c_4\extmin(y)^6$ from
Lemma~\ref{lem:injectivity} therefore imply
\begin{equation}\label{eq:separatedregularcusp}
 \frac1T\Leb\{t\in[T,2T]:\inj(y_t)<\eta\}
       \le C_d\eta^{1/(12d)}=C_dL^{-rb/(12d)}.
\end{equation}
For values of $\eta$ outside the range obtained from \eqref{eq:qndboth},
the same statement follows by enlarging the fixed constant and using
$\Leb([T,2T])=T$.

\emph{One measurable set covering all periodic-orbit exceptions.}
Use the compression $b_\varepsilon\in\SL(V)$ defined before
\eqref{eq:compressionweights}, and set
\[
 h_6(t)=b_\varepsilon\rho_{\mathrm{sym}}
                     (a_{-\log L}^{\mathrm R}P(t)h),\qquad
 \mathcal B_{\mathrm R}=
 \{t\in[T,2T]:\inf_{0\ne w\in V_\Z}\norm{h_6(t)w}\le\beta\}.
\]
Let $t\in[T,2T]$ satisfy $\inj(y_t)\ge\eta$, and suppose there is
$z\in X$ such that $H_{\mathrm{pr}}z$ is periodic and
\[
 \vol(H_{\mathrm{pr}}z)\le R,\qquad d_X(y_t,z)\le\varepsilon.
\]
Proposition~\ref{prop:extraction}, applied with the representative
$a_{-\log L}^{\mathrm R}P(t)h$ of $y_t$, gives a nonzero $w\in V_\Z$ with
\begin{gather*}
 \rho_{\mathrm{sym}}(a_{-\log L}^{\mathrm R}P(t)h)w=c_wQ_0+E_w,
 \\
 \quad |c_w|\le C_2(R/\eta)^B,\quad
 \norm{E_w}\le C_2(R/\eta)^B\varepsilon.
\end{gather*}
Compression gives
\[
 \norm{h_6(t)w}
 \le |c_w|\norm{Q_0}\varepsilon^{5/6}
                  +\varepsilon^{-1/6}\norm{E_w}\le\beta.
\]
The choice \eqref{eq:compressionconstant} ensures the last inequality.
Thus every $t$ satisfying these conditions belongs to
$\mathcal B_{\mathrm R}$. This set is Borel: the integer
lattice is countable and each individual non-strict short-vector
condition is closed; the shortest length is attained at each $t$.

For $0\ne w\in V_\Z$, the singular-value and diagonal estimates
\eqref{eq:compressionweights}--\eqref{eq:diagonalaux}, together with
the growth bound for $\rho_{\mathrm{sym}}$, give
\begin{align*}
 \sup_{T\le t\le2T}\norm{h_6(t)w}
 &\ge\varepsilon^{5/6}L^{-3}
       \sup_{T\le t\le2T}\norm{\rho_{\mathrm{sym}}(P(t)h)w}\\
 &\ge\varepsilon^{5/6}L^{-3}K_{\mathrm{aux}}T^\alpha
  \ge K_{\mathrm{aux}}T^\alpha L^{-14/3}\ge1.
\end{align*}
Here $\varepsilon\ge L^{-2}$ gives the penultimate inequality.

For the higher exterior powers, use the standard growth bounds
already established in the cusp argument. Lemma~\ref{lem:qndboth}
and the choice of $\zeta_d$ give
\[
 \frac1T\Leb\{t\in[T,2T]:\extmin(y_t)<\zeta_d\}\le\frac12.
\]
Hence some $t_0\in[T,2T]$ satisfies $\extmin(y_{t_0})\ge\zeta_d$.
The defining property of $c_d$ and \eqref{eq:compressionweights}
imply, for every $2\le k\le5$ and $0\ne w\in\wedge^kV_\Z$,
\begin{align*}
 \sup_{T\le t\le2T}\norm{(\wedge^kh_6(t))w}
 &\ge\norm{(\wedge^kh_6(t_0))w}\\
 &\ge c_d\varepsilon^{(6-k)/6}
  \ge c_d\varepsilon^{2/3}.
\end{align*}
Thus all proper primitive-subgroup covolumes have supremum at least
\[
 \varrho_T=c_d\varepsilon^{2/3}\le1/6.
\]
The full-rank covolume is $1$, since $\det h_6(t)=1$. For
$1\le k\le5$, polynomial coordinates in the $k$th exterior power
have degree at most $2kd\le10d$, so all these covolume functions are
$(C_d,1/(10d))$-good. The curve is polynomial on the entire real
line, including the enlarged interval required by nondivergence.
The lower-bound parameter in Theorem~\ref{prop:km} is allowed
to depend on $T$. Its ratio with the short-vector threshold satisfies
\[
 \frac{2\beta}{\varrho_T}
 =\frac{2C_3}{c_d}(R/\eta)^B\varepsilon^{1/6}
 \le\frac{2C_3}{c_d}L^{Br(1+b)-1/6}
 \le L^{-1/8}.
\]
The last inequality uses $Br(1+b)<1/48$ and
$L\ge(2C_3/c_d)^{48}$; in particular, $2\beta\le\varrho_T$.
Theorem~\ref{prop:km}, with lower-bound parameter $\varrho_T$
and short-vector threshold $2\beta$, now gives
\begin{equation}\label{eq:separatedregularavoidance}
 \frac{\Leb(\mathcal B_{\mathrm R})}{T}
       \le C_d(2\beta/\varrho_T)^{1/(10d)}
       \le C_dL^{-1/(80d)}.
\end{equation}
Since $\mathcal B_{\mathrm R}$ contains every periodic-orbit exception
with $\inj(y_t)\ge\eta$, this bounds their total measure. The bound
does not require counting the periodic orbits or bounding $\norm w$.

\emph{The analytic estimate and its constant.}
Remove the cusp set and $\mathcal B_{\mathrm R}$. At every remaining
parameter, the starting-point threshold in
Theorem~\ref{prop:regularinput} follows from
$\inj(y_t)\ge\eta$ and \eqref{eq:separatedregularfixed}.
Its periodic-orbit alternative is impossible by the preceding
containment. The identity \eqref{eq:regularidentity}, with $z=P(t)h\Gamma$, yields
\[
 |V_LF(P(t)h\Gamma)-m_X(F)|\le C_0\SR(F)L^{-r\kappa}.
\]
Use $2\norm F_\infty$ on the two removed sets. By
\eqref{eq:separatedregularcusp}, \eqref{eq:separatedregularavoidance},
and the one-sided version of Lemma~\ref{lem:smoothing}, the dyadic discrepancy
is bounded by
\begin{align*}
 &C_d\bigl(\SR(F)+\norm F_{C_G^1}\bigr)
 \bigl(L^{-r\kappa}+L^{-rb/(12d)}+L^{-1/(80d)}
                          +M(P)T^{-1/2}\bigr)\\
 &\quad\le C_d(1+M(P))\bigl(\SR(F)+\norm F_{C_G^1}\bigr)
                          T^{-\delta_{\mathrm R}}.
\end{align*}
The argument includes $q=0$ and $0<\delta_{\mathrm R}\le1/90$.
Finally,
\[
 T_{\mathrm R,*}^{\delta_{\mathrm R}}
 \le T_d^{1/2}+L_{\mathrm R}^{\nu_{\mathrm R}}
       +K_{\mathrm{std}}^{-\beta_{\mathrm{Rc}}}
       +K_{\mathrm{aux}}^{-\beta_{\mathrm{Ra}}}.
\]
Lemma~\ref{lem:dyadic} proves
\eqref{eq:separatedregular}, absorbing only fixed quantities in
$C_d$. The contribution of the fixed lower threshold $L_{\mathrm R}$
to $C_d$ is independent of $\tau$.
\end{proof}

In both ambient tangent cases, the curve and starting lift enter the
bound only through $M(P)$, the displayed growth constants, and the
pulled-back test norm. The powers of these constants and
the ratio $\delta/\theta$ are independent of $\tau$.

\section{Polynomial estimates on affine supports}
\label{sec:leviaffine}

An affine support is a homogeneous orbit whose normalized group is
a rational conjugate of one of the embeddings $\mathbf A,\mathbf A^\vee$
in \eqref{eq:affinegroups}, as classified in
Proposition~\ref{prop:supportclassification}. The model group
$A_{\mathrm{aff}}=\SL_2(\mathbb R)\ltimes\mathbb R^2$ consists of
the affine transformations $w\mapsto Aw+v$ with $\det A=1$.
Its quotient by $\Gamma_2(N)\ltimes\mathbb Z^2$ is a torus bundle
over $\SL_2(\mathbb R)/\Gamma_2(N)$, with the fibers recording
the translation coordinate; Lemma~\ref{lem:affinequotient} establishes
this geometry and the invariant probability. We prove power rates
of equidistribution
for polynomial curves in each of the three leading tangent cases---
translation, block, and regular---identified in
Lemma~\ref{lem:regularaffinenormalization}.
The regular case uses Proposition~\ref{prop:regularaffineflow}, whose
Fourier proof is given in Section~\ref{sec:regularaffine}.
Finally, the rational cover of Lemma~\ref{lem:affinecover} transfers
these model estimates to arbitrary affine supports, yielding the
intrinsic estimate of Proposition~\ref{prop:affineintrinsic}.

\subsection{The three tangent types}

\begin{lemma}[Classification of nilpotent affine tangents]
\label{lem:regularaffinenormalization}
Write a nonzero nilpotent affine matrix as
$N=\left(\begin{smallmatrix}B&v\\0&0\end{smallmatrix}\right)$.
Write $N_{\mathrm r}=E_{12}+E_{23}$.
The three disjoint cases have the following normal forms under
conjugation by $A_{\mathrm{aff}}$:
\[
 \begin{array}{c|c}
 B=0 & E_{13}\\
 B\ne0,\ Bv=0 & \varepsilon E_{12},\quad\varepsilon\in\{1,-1\}\\
 Bv\ne0 & cN_{\mathrm r},\quad c\ne0 .
 \end{array}
\]
In the regular case, put
\[
 c=\sqrt[3]{\det(Bv,v)},\qquad
 S=(c^{-2}Bv,c^{-1}v)\in\SL_2(\mathbb R).
\]
Then $\diag(S,1)^{-1}N\diag(S,1)=cN_{\mathrm r}$.
\end{lemma}
\begin{proof}
If $B=0$, a determinant-one block matrix sends the nonzero
translation vector to $e_1$. Otherwise $B$ is a nonzero
nilpotent two-by-two matrix, so $B^2=0$ and
$\ker B=\operatorname{im}B$ is a line
(Section~\ref{app:unipotent}), and
$N^2=\left(\begin{smallmatrix}0&Bv\\0&0\end{smallmatrix}\right)$.
When $Bv=0$, one has $v\in\ker B=\operatorname{im}B$.
Choose $x$ with $Bx=-v$. Then $(I_2,x)^{-1}N(I_2,x)$ has
zero translation component. Lemma~\ref{lem:normalization}
puts its block in the form $\varepsilon E_{12}$.
When $Bv\ne0$, the vectors $Bv,v$ are independent. The displayed
matrix $S$ has determinant one, $S^{-1}BS=cE_{12}$, and
$S^{-1}v=ce_2$, proving the regular normal form.
\end{proof}

We treat these three cases in the same order as the table above.
Each estimate states explicitly which coefficient bounds it needs.

\subsection{Averaging in the torus fibers}

Fix an integer $N\ge1$ and, for $\Delta_2=\Gamma_2(N)$, put
\[
 \Delta_{\mathrm{aff}}=\Delta_2\ltimes\mathbb Z^2,
 \qquad X_{\mathrm{aff}}=A_{\mathrm{aff}}/\Delta_{\mathrm{aff}}.
\]

\begin{lemma}[Affine quotient and fiber integration]
\label{lem:affinequotient}
The group $\Delta_{\mathrm{aff}}$ is a lattice in $A_{\mathrm{aff}}$.
The projection $(A,v)\mapsto A$ induces a torus bundle
$X_{\mathrm{aff}}\to G_2/\Delta_2$, with fiber
$\mathbb R^2/A\mathbb Z^2$ over $A\Delta_2$.
If $m_{\Delta_2}$ denotes invariant probability on the base,
the invariant probability $m_{\mathrm{aff}}$ on $X_{\mathrm{aff}}$
is given, for $F\in C_c(X_{\mathrm{aff}})$, by
\[
 m_{\mathrm{aff}}(F)
 =\int_{G_2/\Delta_2}\int_{\mathbb R^2/\mathbb Z^2}
       F((A,A\eta)\Delta_{\mathrm{aff}})\dd\eta
       \dd m_{\Delta_2}(A\Delta_2),
\]
where $\dd\eta$ is Haar probability on $\mathbb R^2/\mathbb Z^2$.
\end{lemma}
\begin{proof}
The group law is $(A,v)(B,w)=(AB,v+Aw)$.
In particular, right multiplication by $(I_2,m)$ sends
$(A,v)$ to $(A,v+Am)$, giving the stated fiber.
Changing the representative $A$ to $A\gamma$, with
$\gamma\in\Delta_2$, replaces the fiber coordinate $\eta$
by $\gamma^{-1}\eta$ modulo $\mathbb Z^2$.
This change preserves Haar probability, so the displayed integral
is independent of the representative. Local sections of the base
quotient give the torus-bundle trivializations.

The base has finite invariant volume
(Section~\ref{app:arithmetic}). Each fiber has Euclidean covolume
one because $\det A=1$. Left multiplication by $(B,w)$ sends
the fiber over $A\Delta_2$ to that over $BA\Delta_2$ by
$v\mapsto Bv+w$. Since $\det B=1$, it preserves the fiber
probabilities. Together with invariance of $m_{\Delta_2}$,
this proves that the displayed integral defines an invariant
probability on $X_{\mathrm{aff}}$.
Finally, $\Delta_{\mathrm{aff}}$ is discrete, since its matrix
and translation entries are integral. It is therefore a lattice.
\end{proof}

Let $\mathcal D_s^{\mathrm{aff}}$ be the sum of supremum norms
of left-action derivatives up to order $s$ in the basis
$E_{12},E_{21},E_{11}-E_{22},E_{13},E_{23}$.
For a smooth function $F$ with bounded derivatives, define
\[
 \overline F(A\Delta_2)
   =\int_{\mathbb R^2/\mathbb Z^2}F((A,A\eta)\Delta_{\mathrm{aff}})
          \dd\eta,
 \qquad
 V_LF(y)=\frac1L\int_0^L F(\exp(sE_{13})y)\dd s.
\]
An integral change of representative in $\Delta_2$ preserves the
fiber integral, so $\overline F$ is well-defined. Left translation
by a block element commutes with fiber integration. Thus
$\mathcal D_s(\overline F)\le\mathcal D_s^{\mathrm{aff}}(F)$
and $m_{\Delta_2}(\overline F)=m_{\mathrm{aff}}(F)$.

\begin{lemma}\label{lem:affinefiber}
Let $L\ge4$ and $y=(A,v)\Delta_{\mathrm{aff}}$.
If
\begin{equation}\label{eq:affinefibergood}
 \latticemin(A^{-\mathsf T}\Gamma_2)\ge L^{-1/16},
 \qquad
 \latticemin(a_2(L)A^{-\mathsf T}\Gamma_2)\ge L^{-1/16},
\end{equation}
then
\begin{equation}\label{eq:affinefiberestimate}
 |V_LF(y)-\overline F(A\Delta_2)|
       \le C\mathcal D_6^{\mathrm{aff}}(F)L^{-1/8}.
\end{equation}
The constant is absolute and the estimate is uniform in $v$.
\end{lemma}
\begin{proof}
Expand $\eta\mapsto F((A,A\eta)\Delta_{\mathrm{aff}})$ in its
torus Fourier series. Index its modes by
$\xi=A^{-\mathsf T}k$, $k\in\mathbb Z^2$.
Integration by parts with $(1-D_{E_{13}}^2-D_{E_{23}}^2)^3$ gives
\[
 |\widehat F_A(k)|
       \le C\mathcal D_6^{\mathrm{aff}}(F)(1+\norm\xi)^{-6}.
\]
The $V_L$-average multiplies the $k$th mode by a phase times
$L^{-1}\int_0^L e^{2\pi i s\xi_1}\dd s$, whose absolute value is
at most $\min\{1,(\pi L|\xi_1|)^{-1}\}$.

Put $R=L^{1/16}$. The first condition in
\eqref{eq:affinefibergood} and disjoint disks give
\[
 \#\{\xi\in A^{-\mathsf T}\mathbb Z^2:\norm\xi\le r\}
       \le C L^{1/8}r^2\qquad(r\ge1).
\]
For a nonzero mode with $\norm\xi\le R$, the second condition gives
\[
 L\xi_1^2+L^{-1}\xi_2^2\ge L^{-1/8},
 \qquad |\xi_1|\ge 2^{-1/2}L^{-9/16}.
\]
Here $L^{-1}R^2\le\tfrac12L^{-1/8}$ for $L\ge4$.
There are at most $CL^{1/4}$ such modes, so their total contribution
is at most $C\mathcal D_6^{\mathrm{aff}}(F)L^{-3/16}$.
Summing the remaining modes in dyadic annuli gives
\[
 \sum_{\norm\xi>R}(1+\norm\xi)^{-6}
       \le C L^{1/8}R^{-4}=CL^{-1/8}.
\]
This also proves absolute convergence and justifies the
termwise integration. The zero mode is $\overline F$, which
proves the estimate.
\end{proof}

For a vector polynomial $Q(t)=\sum_{j=0}^D t^jQ_j$, write
$\operatorname{coeff}_{>0}Q=(Q_1,\ldots,Q_D)$, with the Euclidean
direct-sum norm. Appending zero coefficients does not change this norm.

\begin{proposition}
\label{prop:affinepolynomial}
Let $\psi(t)=(P_2(t),v(t))\in\mathbf A(\mathbb R[t])$ be normalized,
of degree at most $D\ge1$, with nonzero leading coefficient of
$\psi'(t)\psi(t)^{-1}$ in $\Lie(R_u(\mathbf A)(\mathbb R))$.
Let $g=(B,w)\in A_{\mathrm{aff}}$.
Suppose $c>0$, $\tau\ge0$, and
\[
 \norm{\operatorname{coeff}_{>0}(P_2(t)Bz)}
 \ge c\norm z^{-\tau}\qquad(0\ne z\in\mathbb Z^2).
\]
Set
\begin{equation}\label{eq:affineexponent}
 \eta_{\Delta_2,D}
    =\min\left\{\delta_{\Delta_2,D},\frac1{32D(D+1)}\right\},
 \qquad s=\max\{6,\ell_{\Delta_2}\}.
\end{equation}
Then, for every smooth $F$ on $X_{\mathrm{aff}}$ with
$\mathcal D_s^{\mathrm{aff}}(F)<\infty$,
\begin{equation}\label{eq:affinepolynomial}
 \left|\frac1T\int_0^T F(\psi(t)g\Delta_{\mathrm{aff}})\dd t
                       -m_{\mathrm{aff}}(F)\right|
 \le C\mathcal D_s^{\mathrm{aff}}(F)
                T^{-\eta_{\Delta_2,D}/(\tau+1)}\qquad(T\ge2).
\end{equation}
The constant depends on $N,D,\psi,B,c,\tau$ and can be chosen
independently of $w$.
\end{proposition}
\begin{proof}
A fixed block element $b_0\in\SL_2(\mathbb R)\le A_{\mathrm{aff}}$
conjugates the leading translation to $E_{13}$.
Replace $\psi,g,F$ by $b_0^{-1}\psi b_0,b_0^{-1}g,F_{b_0}$.
Coefficient separation is preserved, with a changed constant.
All derivative comparisons under this fixed translation are
absorbed into $C$. We continue to write $\psi,g,F$ for these
transformed objects.

Put $\alpha=(\tau+1)^{-1}$ and $A(t)=P_2(t)B$.
Lemma~\ref{lem:separatedgrowth} supplies $K>0$ such that
\[
 \sup_{t\in[T,2T]}\norm{A(t)z}\ge KT^\alpha
       \qquad(T\ge1,\ 0\ne z\in\mathbb Z^2).
\]
The same lower bound holds for $A(t)^{-\mathsf T}$, because
$A^{-\mathsf T}=J_2AJ_2^{-1}$ with
$J_2=\left(\begin{smallmatrix}0&-1\\1&0\end{smallmatrix}\right)$
integral and orthogonal.
Both matrices have polynomial entries of degree at most $D$.

Take $L=T^{\alpha/[2(D+1)]}$.
Lemma~\ref{lem:smoothing}, inside the affine group with tangent
$E_{13}$ and $p=1$, gives
\begin{equation}\label{eq:affinesmoothing}
 \left|\frac1T\int_T^{2T}(F-V_LF)(\psi(t)g\Delta_{\mathrm{aff}})\dd t\right|
       \le C_D M(\psi)\mathcal D_1^{\mathrm{aff}}(F)T^{-1/2}
\end{equation}
for $T\ge T_D$.

For the two lattice curves in \eqref{eq:affinefibergood}, every
nonzero integer vector has supremum at least
$KT^\alpha L^{-1/2}$ on $[T,2T]$, and the determinant is one.
For all sufficiently large $T$, this is at least one and
$L^{-1/16}\le1/2$. Their vector norms are
$(C_D,1/D)$-good on the required enlarged interval, by
Lemma~\ref{lem:polynomialgoodness}.
Apply Theorem~\ref{prop:km} in dimension two with
$\varrho=1/2$ to each curve.
The proportion of times when \eqref{eq:affinefibergood} fails is
at most $C_D L^{-1/(16D)}$.
On this exceptional set, use $2\norm F_\infty$; on its complement,
use Lemma~\ref{lem:affinefiber}. Together with
\eqref{eq:affinesmoothing}, this proves
\[
 \left|\frac1T\int_T^{2T}
  \bigl(F(\psi(t)g\Delta_{\mathrm{aff}})
                -\overline F(P_2(t)B\Delta_2)\bigr)\dd t\right|
 \le C\mathcal D_6^{\mathrm{aff}}(F)
       \left(T^{-1/2}+L^{-1/8}+L^{-1/(16D)}\right).
\]
All lower bounds on $T$ used here depend on $K,D,\alpha$ and
are independent of $w$.
Coefficient separation forces $P_2$ to be nonconstant.
Proposition~\ref{prop:congruencemodular} applies to the base curve
and $\overline F$. Its estimate on $[T,2T]$ follows by subtracting
the estimates on $[0,2T]$ and $[0,T]$.
The pullback by the matrix normalizing the base tangent changes
derivative norms only by a fixed factor, and the mean of $\overline F$ is
$m_{\mathrm{aff}}(F)$.
Combining the estimates gives the exponent in
\eqref{eq:affineexponent}. Lemma~\ref{lem:dyadic} handles the
initial threshold and proves \eqref{eq:affinepolynomial}.
\end{proof}

\subsection{The affine block tangent}
\label{subsec:affineblock}

Keep $X_{\mathrm{aff}}=A_{\mathrm{aff}}/\Delta_{\mathrm{aff}}$ with
$\Delta_{\mathrm{aff}}=\Gamma_2(N)\ltimes\mathbb Z^2$.
Write $S_0=\{(A,0):A\in\SL_2(\mathbb R)\}$,
$u(s)=\exp(sE_{12})$, and
$a(L)=\diag(L^{1/2},L^{-1/2},1)$.
Fix a right-invariant Riemannian metric on $A_{\mathrm{aff}}$ and
the resulting quotient distance $d_{\mathrm{aff}}$.
Write $\inj_{\mathrm{aff}}$ for the local injectivity radius,
capped at one. Volumes of $S_0$-orbits use one fixed Haar measure
on $S_0$.

The following affine block-flow estimate follows from Lin's theorem
on the congruence quotient constructed above.

\begin{theorem}[{\cite[Theorem~2.6 and Remark~1.2]{Lin2026}}]\label{thm:affineblockinput}
For each fixed $N$ there are constants $A_*,B_*\ge1$,
$\kappa_*>0$, $0<\chi_*\le1$, $C_*>1$, and a fixed finite-order
Sobolev norm $\mathcal S_N$ with the following property.
Let $z\in X_{\mathrm{aff}}$, $R\ge C_*\inj_{\mathrm{aff}}(z)^{-B_*}$,
and $L\ge\max\{2,R^{A_*}\}$.
At least one of the following holds:
\begin{enumerate}
\item For every $F\in C_c^\infty(X_{\mathrm{aff}})$,
\[
 \left|\frac1{2L}\int_{-L}^L F(u(s)a(L)z)\dd s
                          -m_{\mathrm{aff}}(F)\right|
       \le C_*\mathcal S_N(F)R^{-\kappa_*}.
\]
\item There is $z'$ on a closed $S_0$-orbit of volume at most
$C_*R$ with $d_{\mathrm{aff}}(z,z')\le C_*L^{-\chi_*}$.
\end{enumerate}
Fix an integer $\ell_N$ such that, for every compact
$K\subset X_{\mathrm{aff}}$,
$\mathcal S_N(F)\le C_{N,K}\mathcal D_{\ell_N}^{\mathrm{aff}}(F)$
when $\operatorname{supp}F\subset K$.
\end{theorem}

\begin{proof}[Deduction]
This is the specialization of \cite[Theorem 2.6, Remark 1.2,
and Example 1.9]{Lin2026}.
The polynomial injectivity threshold is from Remark~1.2;
the Sobolev norm is the fixed finite-order norm used in
\cite[Section 14.1]{Lin2026}. Its restriction to a fixed compact
set has the displayed derivative bound.
We now verify the group and lattice hypotheses, including the cover.

By Lemma~\ref{app:affine-algebra}, the group $\mathbf A$ is
perfect: $\SL_2$ is perfect and its commutators with
$\mathbb G_a^2$ span $\mathbb G_a^2$. The complement of $\Lie S_0$
is the standard irreducible two-dimensional module; its
$u(\mathbb R)$-fixed space is a line, by the same lemma.
This verifies Lin's \emph{Irreducible} and \emph{Schubert}
hypotheses; see \cite[Lemmas 2.3--2.4]{Lin2026}.

To realize the integral lattice required in Lin's Theorem~2.6,
use five-by-five matrix units and set
\[
 C_N=I_5+N^{-1}(E_{14}+E_{25}),\qquad
 \rho_N(b)=C_N\diag(b,I_2)C_N^{-1}\in\SL_5.
\]
This is a faithful rational representation. Its top-left block is
$b$. The first three entries of its two extra columns are
$N^{-1}(e_i-be_i)$, so its integral points are exactly
\[
 \rho_N(\mathbf A)(\mathbb R)\cap\SL_5(\mathbb Z)
       =\rho_N(\Gamma_2(N)\ltimes\mathbb Z^2).
\]
Here $e_i$ denotes the $i$th standard column in $\mathbb R^3$.
In Lin's normalization, the unit averaging box in $U$ has the form
$\{u(s):|s|\le r_N\}$, with $r_N>0$ fixed. Let $\lambda_N>0$
be the root exponent of Lin's unit-speed diagonal flow. At time
$\lambda_N^{-1}\log(L/r_N)$, this flow equals $a(L/r_N)$.
Apply Lin's Theorem~2.6 with parameter $T=(L/r_N)^{1/\lambda_N}$
to the starting point $a(r_N)z$. The identity
\[
 a(L/r_N)u(s)a(r_N)z=a(L)u(s/r_N)z
       \qquad(|s|\le r_N)
\]
shows that its normalized average is exactly the average in the
first alternative above. The point $a(r_N)z$ has injectivity
radius comparable to that of $z$; a nearby periodic $S_0$-orbit
can be translated back by $a(r_N)^{-1}\in S_0$ with comparable
distance and unchanged orbit volume. For the constants $A_3,A_4$
in Lin's Theorem~2.6, the threshold $T\ge R^{A_3}$ follows from
$L\ge R^{A_*}$ after increasing $A_*,C_*$. The distance bound is
$T^{-1/A_4}=r_N^{1/(\lambda_N A_4)}L^{-1/(\lambda_N A_4)}$.
Enlarge $B_*,C_*$ and choose
$0<\chi_*\le\min\{1,1/(\lambda_N A_4)\}$.
The pulled-back metric, orbit-volume normalization, and derivative
norms differ from the fixed ones by constants depending on $N$.
This proves the stated dichotomy for the exact congruence lattice.
\end{proof}

\begin{lemma}\label{lem:affineleviorbits}
If the orbit $S_0z'$ is closed and has volume at most $V$, then
every representative $g'=(A,v)$ of $z'$ has an integral vector
$w\ne0$ with
\[
 g'w=q e_3,\qquad 1\le q\le C_N V^{1/2}.
\]
Consequently, there is $\varepsilon_N>0$ such that if
$d_{\mathrm{aff}}(z,z')\le\varepsilon\le\varepsilon_N$,
every representative $g$ of $z$
has an integral vector $w\ne0$ for which
\begin{equation}\label{eq:affinelevivector}
 gw=(x_1,x_2,q)^{\mathsf T},\qquad
 1\le q\le C_N V^{1/2},\qquad
 \norm{(x_1,x_2)}\le C_N q\varepsilon.
\end{equation}
\end{lemma}
\begin{proof}
Put $\xi=A^{-1}v$. The stabilizer of $z'$ in $S_0$, conjugated
by $A^{-1}$, is
\[
 \Lambda_\xi=\{\gamma\in\Gamma_2(N):
                         (\gamma-I_2)\xi\in\mathbb Z^2\}.
\]
By Lemma~\ref{app:lattice-covers}, finite covolume makes
$\Lambda_\xi$ finite index in $\SL_2(\mathbb Z)$. It contains
positive powers of both elementary unipotents, forcing both
coordinates of $\xi$ to be rational (Lemma~\ref{app:primitive-vectors}).
Let $q\ge1$ be their least common denominator and put $p=q\xi$.
The orbit of $p$ under $\SL_2(\mathbb Z/q\mathbb Z)$ consists of
all primitive pairs modulo $q$. Elementary matrices and the Chinese
remainder theorem give transitivity; by Lemma~\ref{app:primitive-vectors},
their number is
\[
 q^2\prod_{\mathfrak p\mid q}(1-\mathfrak p^{-2})
       \ge q^2/\zeta(2).
\]
The product runs over prime divisors of $q$, and $\zeta(2)=\sum_{n\ge1}n^{-2}$.
Thus $[\SL_2(\mathbb Z):\Lambda_\xi]\ge q^2/\zeta(2)$.
Conjugation by $A$ preserves Haar measure on $\SL_2(\mathbb R)$,
so the orbit volume is at least $c_Nq^2$.
The vector $w=(-p_1,-p_2,q)^{\mathsf T}$ satisfies $g'w=q e_3$.
Changing representatives multiplies $w$ by an integral matrix.

For the last assertion, lift a quotient path of length at most
$2\varepsilon$ from $z$ to $z'$. Right invariance gives
$z'=kz$ with $\norm{k-I_3}\le C_N\varepsilon$.
Apply the first assertion to the representative $kg$ of $z'$.
Then $gw=qk^{-1}e_3$. An affine matrix has last row $(0,0,1)$,
and $k^{-1}e_3=e_3+O(\varepsilon)$, which gives
\eqref{eq:affinelevivector}.
\end{proof}

\begin{proposition}
\label{prop:affineblockpolynomial}
Let $\psi(t)=(P_2(t),v(t))\in\mathbf A(\mathbb R[t])$ be normalized
of degree at most $D\ge1$, and let $g=(B,w)\in A_{\mathrm{aff}}$.
Assume its leading tangent is $\varepsilon E_{12}$ with
$\varepsilon\in\{1,-1\}$. Suppose, with $c_2,c_3>0$ and
$\tau\ge0$, that
the nonconstant coefficient norms satisfy
\begin{equation}\label{eq:affinetwotests}
 \norm{\operatorname{coeff}_{>0}(P_2(t)Bz)}
       \ge c_2\norm z^{-\tau},\qquad
 \norm{\operatorname{coeff}_{>0}(\psi(t)gk)}
       \ge c_3\norm k^{-\tau}
\end{equation}
for all nonzero $z\in\mathbb Z^2$, $k\in\mathbb Z^3$.
There are $s_N\ge1$ and $\eta_{N,D}^{\mathrm b}>0$, independent
of $\tau$ and the curve coefficients, such that for every compact
$K\subset X_{\mathrm{aff}}$,
\begin{equation}\label{eq:affineblockpolynomial}
 \left|\frac1T\int_0^T F(\psi(t)g\Delta_{\mathrm{aff}})\dd t
                     -m_{\mathrm{aff}}(F)\right|
 \le C\mathcal D_{s_N}^{\mathrm{aff}}(F)
               T^{-\eta_{N,D}^{\mathrm b}/(\tau+1)}\qquad(T\ge2)
\end{equation}
for $F\in C_c^\infty(X_{\mathrm{aff}})$ supported in $K$.
The constant $C$ may depend on $N,D,K,\psi,g,c_2,c_3,\tau$.
\end{proposition}
\begin{proof}
Write $\alpha=(\tau+1)^{-1}$, $M(t)=\psi(t)g$, and
$A(t)=P_2(t)B$. Lemma~\ref{lem:separatedgrowth} gives constants
$K_2,K_3>0$ with
\begin{equation}\label{eq:affineblockgrowth}
 \sup_{[T,2T]}\norm{A(t)z}\ge K_2T^\alpha,
 \qquad \sup_{[T,2T]}\norm{M(t)k}\ge K_3T^\alpha
\end{equation}
for all nonzero integer vectors in the respective spaces.

Fix the constants from Theorem~\ref{thm:affineblockinput}
and choose
\begin{equation}\label{eq:affineblockparameters}
 r=\min\{(2A_*)^{-1},\chi_*/8\},\qquad
 b=\frac{r}{12(B_*+1)},\qquad
 \sigma=\chi_*/4,\qquad
 L=T^{\alpha/[2(D+1)]},\quad R=L^r.
\end{equation}
Increase the initial threshold in $T$ by constants determined by
$N,D,K,\psi,g,c_2,c_3,\tau$ as needed below.
For the moving point $y(t)=M(t)\Delta_{\mathrm{aff}}$, set
$z(t)=a(L)^{-1}y(t)$.
First exclude the times for which
\begin{equation}\label{eq:affineblockbasegood}
 \latticemin(a_2(L)^{-1}A(t)\Gamma_2)<L^{-b}.
\end{equation}
The vector suprema for this two-dimensional lattice curve are
at least $K_2T^\alpha L^{-1/2}$, and its determinant is one.
Theorem~\ref{prop:km} bounds the excluded proportion by
$C_D L^{-b/D}$.

At each remaining time, the three-dimensional lattice underlying
$z(t)$ has shortest-vector length at least $\min\{1,L^{-b}\}$:
an integer vector with nonzero last coordinate has image length
at least one, and the other vectors are governed by the base.
Lemma~\ref{lem:injectivity}, restricted to the affine group,
therefore gives $\inj_{\mathrm{aff}}(z(t))\ge c_NL^{-6b}$.
The choices in \eqref{eq:affineblockparameters} give
$6bB_*<r$ and $rA_*\le1/2$, so both starting conditions of
Theorem~\ref{thm:affineblockinput} hold for sufficiently large $T$.

If the second alternative gives a nearby closed orbit, apply
Lemma~\ref{lem:affineleviorbits} with $V=C_*R$ and
$\varepsilon=C_*L^{-\chi_*}$.
For some nonzero integer $k$ it gives
\[
 a(L)^{-1}M(t)k=(x_1,x_2,q)^{\mathsf T},\quad
 q\ll_N L^{r/2},\quad \norm{(x_1,x_2)}\ll_N qL^{-\chi_*}.
\]
Define the determinant-one matrix
\begin{equation}\label{eq:affineblockdilation}
 D_L=\diag(L^{\sigma-1/2},L^{\sigma+1/2},L^{-2\sigma}).
\end{equation}
The displayed bounds imply
\[
 \norm{D_LM(t)k}
 \ll_N L^{\sigma+r/2-\chi_*}+L^{-2\sigma+r/2}
       <L^{-\chi_*/4}
\]
after a fixed threshold: $r\le\chi_*/8$ makes both exponents
strictly smaller than $-\chi_*/4$.

We bound these times by three-dimensional nondivergence.
Since $\chi_*\le1$, $s_{\min}(D_L)\ge L^{-1}$, so the
standard-vector suprema are at least $K_3T^\alpha L^{-1}$.
For the exterior square, use the basis
$e_1\wedge e_2,e_1\wedge e_3,e_2\wedge e_3$ and let
$\pi$ discard the first coordinate. Affine matrix multiplication gives
\[
 (\wedge^2M(t))(e_1\wedge e_2)=e_1\wedge e_2,
 \qquad
 \pi((\wedge^2M(t))\omega)=A(t)\pi(\omega).
\]
If $\pi(\omega)\ne0$ for an integral exterior vector $\omega$,
the quotient coordinates after $\wedge^2D_L$ are scaled by
$L^{-\sigma-1/2}$ and $L^{-\sigma+1/2}$.
Their supremum is at least $K_2T^\alpha L^{-1}$.
If $\pi(\omega)=0$ and $\omega\ne0$, its norm after this map
is at least $L^{2\sigma}$. The determinant is one.
All vector and exterior-vector polynomials have degree at most $2D$,
so Lemma~\ref{lem:polynomialgoodness} supplies goodness with
exponent $1/(2D)$ on the required enlarged interval.
The preceding suprema exceed one for all sufficiently large $T$.
Theorem~\ref{prop:km}, with $n=3$ and $\varrho=1/3$, now bounds
the proportion of parameters for which the periodic-orbit alternative
holds by $C_D L^{-\chi_*/(8D)}$.
This treats the fixed exterior line explicitly; no Diophantine
condition on a representation with a fixed vector was invoked.

On the remaining set of parameters, the block-flow average differs
from $m_{\mathrm{aff}}(F)$ by at most
$C\mathcal S_N(F)L^{-r\kappa_*}$.
On the exceptional set, use $2\norm F_\infty$.
Lemma~\ref{lem:smoothing}, inside the affine group with
$p=1$ and $I_L=[-L,L]$, gives error at most
$C_DM(\psi)\mathcal D_1^{\mathrm{aff}}(F)T^{-1/2}$.
The sign $\varepsilon$ disappears by symmetry.
Thus the dyadic discrepancy is at most
\[
 C\mathcal D_{s_N}^{\mathrm{aff}}(F)
 \left(T^{-1/2}+L^{-b/D}+L^{-\chi_*/(8D)}+L^{-r\kappa_*}\right),
 \qquad s_N=\max\{1,\ell_N\}.
\]
One may take
\begin{equation}\label{eq:affineblockexponent}
 \eta_{N,D}^{\mathrm b}
 =\frac{\min\{1,b/D,\chi_*/(8D),r\kappa_*\}}{2(D+1)}.
\end{equation}
Lemma~\ref{lem:dyadic} absorbs this initial threshold and proves
\eqref{eq:affineblockpolynomial}.
\end{proof}

\subsection{Regular affine polynomial curves}
\label{subsec:regularaffinepolynomial}

For a congruence level $N\ge1$, set $\Delta_2=\Gamma_2(N)$ and
\begin{equation}\label{eq:regularaffineflowconstants}
 \epsilon_0=10^{-12},\qquad
 \omega_N=\min\{b_{\Delta_2},\epsilon_0/8\},\qquad
 \zeta_N=\min\{\epsilon_0/2,A_{\Delta_2}/2\},\qquad
 \ell_N^{\mathrm r}=\max\{64,\ell_{\Delta_2}\}.
\end{equation}
Here $b_{\Delta_2}$ is defined in \eqref{eq:coverexponent}, and
$A_{\Delta_2},\ell_{\Delta_2}$ are supplied by
Theorem~\ref{thm:coverinput}. Recall
$N_{\mathrm r}=E_{12}+E_{23}$ and
$a_2(L)=\diag(\sqrt L,1/\sqrt L)$.
The following estimate is uniform in the starting translation $v$
and allows $A$ to vary with $L$. It is exactly the flow estimate
needed to apply the common rank-one transfer lemma.

\begin{proposition}[Regular affine flow estimate]
\label{prop:regularaffineflow}
For each congruence level $N\ge1$, use the constants in
\eqref{eq:regularaffineflowconstants}.
For every compact $K\subset X_{\mathrm{aff}}$ there are
$C_{N,K},L_{N,K}\ge1$ such that
\begin{equation}\label{eq:regularaffineflowinput}
 \begin{split}
 &\latticemin(a_2(L)^{-1}A\mathbb Z^2)\ge L^{-\omega_N}
 \\[-2pt]
 &\quad\Longrightarrow\quad
 \left|\frac1L\int_{-L/2}^{L/2}
    F(\exp(sN_{\mathrm r})(A,v)\Delta_{\mathrm{aff}})\dd s
       -m_{\mathrm{aff}}(F)\right|
 \le C_{N,K}\mathcal D_{\ell_N^{\mathrm r}}^{\mathrm{aff}}(F)
                      L^{-\zeta_N}
 \end{split}
\end{equation}
for all $L\ge L_{N,K}$, $A\in G_2$, $v\in\R^2$, and smooth $F$ supported in $K$.
\end{proposition}
The proof is given in Section~\ref{sec:regularaffine}, after
the cancellation estimate for the nonzero fiber Fourier modes.
We first deduce its consequence for polynomial trajectories.

\begin{proposition}
\label{prop:regularaffinereduction}
Fix $N\ge1$ and
$\Delta_{\mathrm{aff}}=\Gamma_2(N)\ltimes\mathbb Z^2$.
Let $\psi(t)=(P_2(t),v(t))$ be normalized and polynomial of
degree at most $D\ge1$, with regular leading tangent, and let
$g=(B,w)\in A_{\mathrm{aff}}$.
Assume the two-dimensional coefficient test in
\eqref{eq:affinetwotests} with constants $c_2>0,\tau\ge0$.
For every compact $K\subset X_{\mathrm{aff}}$,
\[
 \left|\frac1T\int_0^T F(\psi(t)g\Delta_{\mathrm{aff}})\dd t
             -m_{\mathrm{aff}}(F)\right|
 \le C\mathcal D_{\ell_N^{\mathrm r}}^{\mathrm{aff}}(F)
             T^{-\eta_{N,D}^{\mathrm r}/(\tau+1)}
 \qquad(T\ge2)
\]
for smooth $F$ supported in $K$, where
\begin{equation}\label{eq:regularaffineexponent}
 \eta_{N,D}^{\mathrm r}
   =\frac{\min\{1,\zeta_N,\omega_N/D\}}{2(D+2)}>0.
\end{equation}
The constant can be chosen as
$C=C(N,D,K,\psi,B,c_2,\tau)>0$, independently of $w$ and $F$.
\end{proposition}
\begin{proof}
Use Lemma~\ref{lem:regularaffinenormalization} to put the leading
tangent in the form $cN_{\mathrm r}$. If its logarithmic-derivative
degree is $q$, replacing $\psi(t)$ by $\psi(\beta t)$ with
$\beta=|c|^{-1/(q+1)}$ puts it in the form
$\pm N_{\mathrm r}$. This preserves the degree and the separation
exponent, changing only the coefficient constants and the time
scale. Let $b_0\in A_{\mathrm{aff}}$ be the fixed conjugating element.
The base coefficient vectors are multiplied on the left by a fixed
invertible matrix, and for every fixed derivative order $\ell$,
\[
 \mathcal D_\ell^{\mathrm{aff}}(F_{b_0})
 \le C_{b_0,\ell}\mathcal D_\ell^{\mathrm{aff}}(F).
\]
Thus both changes preserve the claimed form of the estimate.
Centered flow averages are identical for these two signs.

Apply Lemma~\ref{lem:modulartransfer} with
\[
 \mathcal G=A_{\mathrm{aff}},\quad
 \Lambda=\Delta_{\mathrm{aff}},\quad
 \pi(A,v)=A,\quad p=2,\quad
 \mathcal N_\ell=\mathcal D_{\ell_N^{\mathrm r}}^{\mathrm{aff}}.
\]
Its flow hypothesis is Proposition~\ref{prop:regularaffineflow},
with $(\omega,\zeta)=(\omega_N,\zeta_N)$ and
$I_L=[-L/2,L/2]$. Its growth hypothesis follows from the base
coefficient test and Lemma~\ref{lem:separatedgrowth}, with
$\alpha=(\tau+1)^{-1}$ and some $0<K_2\le1$.
Both polynomial degrees are at most $D$.
Since $\zeta_N\le10^{-12}/2<1/4$, the exponent supplied by the
transfer lemma is exactly
$\eta_{N,D}^{\mathrm r}/(\tau+1)$ in
\eqref{eq:regularaffineexponent}.
The fixed normalization and time change affect only the constant.
\end{proof}

\subsection{Transfer to affine supports in arithmetic coordinates}

\begin{lemma}\label{lem:affinecover}
Suppose $\mathbf H$ is a rational conjugate of $\mathbf A$ or
$\mathbf A^\vee$, and $H\Gamma/\Gamma$ is closed.
There is a rational algebraic isomorphism
$\kappa:\mathbf A\to\mathbf H$ and an integer $N\ge1$ such that
\[
 \kappa(\Gamma_2(N)\ltimes\mathbb Z^2)\subseteq\Gamma.
\]
The induced map $A_{\mathrm{aff}}/(\Gamma_2(N)\ltimes\mathbb Z^2)
\to H\Gamma/\Gamma$ is a finite covering preserving Haar probability.
The inverse image under $\kappa$ of a normalized matrix-polynomial
curve of degree at most $d$ has degree at most $2d$.
The representation $\rho(b)$ obtained by taking the two-by-two
block of $\kappa^{-1}(b)$ has rational matrix entries of degree
at most one on $H$, and no nonzero fixed vector.
\end{lemma}
\begin{proof}
Start with $b\mapsto CbC^{-1}$ or
$b\mapsto Cb^{-\mathsf T}C^{-1}$, where
$C\in\mathrm{GL}_3(\mathbb Q)$ is fixed.
These maps and their inverses have matrix-entry degrees at most
two on $\SL_3$, since $b^{-1}=\operatorname{adj}(b)$ there.
By Lemma~\ref{app:congruence}, choose $N$ such that the chosen
map sends integral source matrices congruent to $I_3$ modulo $N$
to integral matrices.
Precompose with the rational automorphism $(A,v)\mapsto(A,Nv)$.
For $A\in\Gamma_2(N)$ and $v\in\mathbb Z^2$, the matrix
$\left(\begin{smallmatrix}A&Nv\\0&1\end{smallmatrix}\right)$
is integral and congruent to $I_3$ modulo $N$.
The resulting $\kappa$ therefore has the asserted inclusion;
the translation scaling and its inverse are linear, so the same
degree bounds hold.

The discrete group $\kappa^{-1}(\Gamma)$ contains
$\Gamma_2(N)\ltimes\mathbb Z^2$, which is a lattice by
Lemma~\ref{lem:affinequotient}. A discrete overgroup of
a lattice has finite index by Lemma~\ref{app:lattice-covers},
which also proves the covering and Haar-probability assertions.
In the dual case the top block of $C^{-1}bC$ is $A^{-\mathsf T}$;
recovering $A$ uses the two-by-two adjugate, which is linear.
This proves the assertion about $\rho$ in both cases.
Its image is $\SL_2$, so it has no nonzero fixed vector by
Lemma~\ref{app:affine-algebra}.
\end{proof}

The canonical estimates now imply a single intrinsic estimate
on every affine support. The only arithmetic step is the application
of Lemma~\ref{lem:intrinsictransfer} to two representations.

\begin{proposition}[Intrinsic estimate on an affine support]
\label{prop:affineintrinsic}
Let $\mathbf H$ be a rational conjugate of $\mathbf A$ or
$\mathbf A^\vee$, put $H=\mathbf H(\R)^\circ$, and suppose
$H\Gamma/\Gamma$ is closed with finite invariant volume.
Fix $a\in G$, $h\in H$, and a nonconstant normalized curve
$\Psi\in\mathbf H(\R[t])$ of degree at most $d$.
Put $J=\overline{\langle\Psi(\R)\rangle}$ and $Y=aH\Gamma/\Gamma$,
and denote by $\mu_Y$ the $aHa^{-1}$-invariant probability on $Y$.
If $(J,h)$ has intrinsic type $(c_0,\tau)$, there are
$s=s(\mathbf H,d)\ge1$ and $\vartheta_{\mathbf H,d}>0$ such that
\begin{equation}\label{eq:affineintrinsicestimate}
 \left|\frac1T\int_0^T f(a\Psi(t)h\Gamma)\dd t-\mu_Y(f)\right|
 \le C\norm f_{C_G^s}T^{-\vartheta_{\mathbf H,d}/(\tau+1)}
 \qquad(T\ge2).
\end{equation}
For each compact $\Omega\subset X$, the constant may be chosen as
\[
 C=C(\mathbf H,a,\Psi,h,c_0,\tau,\Omega,\mathscr L_{\mathbf H})>0,
\]
uniformly for $T\ge2$ and $f\in C_c^\infty(X)$ supported in $\Omega$.
\end{proposition}
\begin{proof}
Use Lemma~\ref{lem:affinecover} and write $\Psi=\kappa\circ\psi$,
$h=\kappa(g)$, with $\deg\psi\le2d$.
The differential $d\kappa$ is a conjugation or a conjugation
composed with $X\mapsto-X^{\mathsf T}$; the translation scaling
in that lemma is also a conjugation. Thus the leading coefficient
$N_\psi$ of $\psi'\psi^{-1}$ is nonzero and nilpotent of the same singular or regular type.

The base representation $\rho_2$ from Lemma~\ref{lem:affinecover}
has degree one. The three-dimensional representation
$\rho_3(b)=\kappa^{-1}(b)$ has rational matrix entries of degree
at most two. Neither representation has a nonzero fixed vector:
the base image is $\SL_2$, while the affine standard representation
has none because the translations move $e_3$ and $\SL_2$ fixes
no vector in its first two coordinates.
Lemma~\ref{lem:intrinsictransfer} applied to both representations
gives exactly \eqref{eq:affinetwotests}, with the original $\tau$.
These tests therefore impose no new arithmetic hypothesis.

Lemma~\ref{lem:regularaffinenormalization} gives exactly three
tangent cases. Proposition~\ref{prop:regularaffinereduction}
treats the regular case using the base test alone, and
Proposition~\ref{prop:affinepolynomial} treats the translation case.
In the remaining case, a fixed affine conjugation puts the tangent
in the form $\varepsilon E_{12}$.
By Lemma~\ref{lem:coefficientnormalization}, both coefficient
tests hold after this conjugation with the same exponent $\tau$.
Proposition~\ref{prop:affineblockpolynomial} therefore applies
with $D=2d$.

In all three cases use
$F(b\Delta_{\mathrm{aff}})=f(a\kappa(b)\Gamma)$.
The finite covering of the closed support is proper, so fixed
compact support for $f$ gives compact support for $F$.
Its derivative norms, including after the fixed tangent
normalization, are bounded by a constant times $\norm f_{C_G^s}$.
Haar probability is preserved by the covering.
With $N$ now denoting its fixed congruence level, one may take
\[
 \vartheta_{\mathbf H,d}
   =\min\{\eta_{\Gamma_2(N),2d},\eta_{N,2d}^{\mathrm b},\eta_{N,2d}^{\mathrm r}\},
 \qquad
 s=\max\{6,\ell_{\Gamma_2(N)},s_N,\ell_N^{\mathrm r}\}.
\]
This proves the assertion. The normalization affects the constant,
not this exponent or derivative order.
\end{proof}

\section{Proof of the regular affine flow estimate}
\label{sec:regularaffine}
\label{subsec:generalpointwisecancellation}

We prove Proposition~\ref{prop:regularaffineflow}.
The zero fiber Fourier mode is the base horocycle average and is
handled by Theorem~\ref{thm:coverinput}. It remains to bound the
sum of the nonzero fiber modes uniformly in the translation $v\in\R^2$.

Retain $X_{\mathrm{aff}}=A_{\mathrm{aff}}/\Delta_{\mathrm{aff}}$,
$\Delta_2=\Gamma_2(N)$, $\overline F$, and
$\mathcal D_s^{\mathrm{aff}}$ from Section~\ref{sec:leviaffine}.
Put $N_{\mathrm r}=E_{12}+E_{23}$,
$a_2(L)=\diag(\sqrt L,1/\sqrt L)$, and
$\mathbb T^2=\mathbb R^2/\mathbb Z^2$ with Haar probability.
Throughout this section $F$ is smooth and compactly supported.
Constants in the flow and chart estimates may depend on the fixed
congruence level $N$; the arithmetic lemmas state their dependence
separately. We use two Fourier indices. The original fiber frequency
is $k\in\Z^2$, which we write as $m$ times a primitive frequency,
with $m\ge1$. The second index $n\in\Z$ comes from a periodic
coordinate in a congruence chart. In particular, $n=0$ is part of
the contribution from $k\ne0$.

Frequency localization first restricts the original Fourier sum.
The congruence charts then express each retained term using a real
phase and a modular inverse. For $n\ne0$, four differences reduce
the arithmetic estimate to complete rational exponential sums
(Section~\ref{subsec:nonzeroinverse}). For $n=0$, derivative estimates
and a quadratic estimate averaged over the moduli give the required
cancellation (Section~\ref{subsec:zeroinverse}). The normalization
and summation of these bounds are carried out in
Section~\ref{subsec:regularflowassembly}.

\subsection{Fourier reduction and the cancellation estimate}

Write $e(x)=\exp(2\pi i x)$ and
\[
 \widehat F_A(k)=\int_{\mathbb T^2}
        F((A,A\eta)\Delta_{\mathrm{aff}})e(-k\cdot\eta)\dd\eta.
\]
For $\xi=A^{-\mathsf T}k$ put
\begin{equation}\label{eq:regularfouriercoefficient}
 c_k(L,A;F)=\frac1L\int_{-L/2}^{L/2}
   \widehat F_{u_2(s)A}(k)
       e\left(-\frac{\xi_1s^2}{2}+\xi_2s\right)\dd s.
\end{equation}
Indeed,
\[
 \exp(sN_{\mathrm r})=(u_2(s),(s^2/2,s)),\qquad
 (u_2(s)A)^{-1}\bigl(u_2(s)v+(s^2/2,s)\bigr)
    =A^{-1}v+A^{-1}(-s^2/2,s).
\]
Consequently the regular-flow average minus its zero fiber mode is
$\mathcal R_{L,A}F(A^{-1}v)$, where
\begin{equation}\label{eq:fiberremainder}
 \mathcal R_{L,A}F(\eta)=\sum_{k\ne0}c_k(L,A;F)e(k\cdot\eta).
\end{equation}
The omitted mode is
$L^{-1}\int_{-L/2}^{L/2}\overline F(u_2(s)A\Delta_2)\dd s$.
Integration by parts with translation derivatives gives, for
every nonnegative integer $h$,
\begin{equation}\label{eq:fiberfourierdecay}
 |\widehat F_A(k)|
 \le C_h\mathcal D_h^{\mathrm{aff}}(F)
                  (1+\|A^{-\mathsf T}k\|)^{-h}.
\end{equation}
For example, even orders follow from powers of
$1-D_{E_{13}}^2-D_{E_{23}}^2$; directional integration by parts
gives every integer order. For fixed $L,A$ this proves absolute
convergence uniformly on the compact averaging interval.

For $L\ge2$ and $A\in G_2$, define the normalized cusp height
\begin{equation}\label{eq:regularaffineheight}
 \mathcal H_L(A)=
 \max\{1,\latticemin(a_2(L)^{-1}A\mathbb Z^2)^{-1}\}.
\end{equation}
The required cancellation estimate is as follows.

\begin{proposition}
\label{prop:generalregularpointwise}
Use $\epsilon_0$ from \eqref{eq:regularaffineflowconstants}.
For each congruence level $N$
and compact $K\subset X_{\mathrm{aff}}$, there are
constants $C_{N,K},L_{N,K}\ge1$ such that
\[
 \mathcal H_L(A)\le L^{\epsilon_0/8},\qquad L\ge L_{N,K},
 \qquad \operatorname{supp}F\subset K
\]
imply
\[
 \sup_{\eta\in\mathbb T^2}
       |\mathcal R_{L,A}F(\eta)|
 \le C_{N,K}\mathcal D_{64}^{\mathrm{aff}}(F)
                         L^{-\epsilon_0/2}.
\]
The derivative order and decay exponent are independent of $A,L,K,F$.
\end{proposition}
Its proof is completed in Section~\ref{subsec:regularflowassembly}.
The intervening lemmas provide the localization and arithmetic bounds
used there.

\subsection{Frequency localization}

\begin{lemma}\label{lem:shearmajorant}
Let $L\ge2$, $A\in G_2$, and put $H=\mathcal H_L(A)$
as in \eqref{eq:regularaffineheight}.
For $0\ne k\in\Z^2$, put
$\xi=A^{-\mathsf T}k$ and
$w=(p,q)=(\sqrt L\,\xi_1,\xi_2/\sqrt L)$, and define
\[
 b(w)=\int_{-1/2}^{1/2}
       \left(1+\sqrt{p^2/L+L(q-u p)^2}\right)^{-4}\dd u.
\]
For the coefficient $c_k(L,A;F)$ of
\eqref{eq:regularfouriercoefficient},
$|c_k|\le C\mathcal D_4^{\mathrm{aff}}(F)b(w)$, and
\begin{equation}\label{eq:shearweightbound}
 b(w)\le C(1+\sqrt L\,\|w\|)^{-1}
                       (1+\|w\|/\sqrt L)^{-3},
 \qquad \sum_{k\ne0}b(w_k)\le CH^2.
\end{equation}
Here $C>0$ is absolute.
\end{lemma}
\begin{proof}
The coefficient bound follows from \eqref{eq:fiberfourierdecay}.
The vectors $w_k$ form the nonzero vectors of
$\Lambda=(a_2(L)^{-1}A)^{-\mathsf T}\mathbb Z^2$.
For a unimodular planar lattice the dual is a rotation by a right
angle. Thus $\lambda=\latticemin(\Lambda)$ also equals
$\latticemin(a_2(L)^{-1}A\Z^2)$.
Disjoint disks of radius $\lambda/2$ give
\[
 \#\{w\in\Lambda:0<\|w\|\le R\}\le C(R/\lambda)^2
 \quad(R\ge\lambda),\qquad \lambda^{-2}\le H^2.
\]
Write $r=\|w\|$. If $|q|\le|p|$, then $r\asymp|p|$ and the
substitution $x=\sqrt L(q-up)$ gives
\[
 b(w)\le \frac{C}{\sqrt L\,|p|}
                         (1+|p|/\sqrt L)^{-3}.
\]
For $\sqrt L\,|p|<1$ use $b(w)\le1$.
If $|q|>|p|$, then $|q-up|\ge|q|/2$ for $|u|\le1/2$, and
$b(w)\le C(1+\sqrt L\,r)^{-4}$.
These estimates prove the radial bound.
Summing it over dyadic shells, the ranges
$r\le L^{-1/2}$, $L^{-1/2}<r\le L^{1/2}$, and
$r>L^{1/2}$ contribute respectively at most
$CH^2/L$, $CH^2$, and $CH^2$.
\end{proof}

\begin{lemma}
\label{lem:regularfrequencywindow}
Keep $L,A,H$ as in Lemma~\ref{lem:shearmajorant}, and let
$c_k=c_k(L,A;F)$ be defined by
\eqref{eq:regularfouriercoefficient}.  For $0<\varepsilon<1/4$
and $w_k=(p_k,q_k)=(\sqrt L\,\xi_1,\xi_2/\sqrt L)$,
$\xi=A^{-\mathsf T}k$, put
\[
 \mathcal W_{L,\varepsilon}(A)=
 \left\{k\ne0:
 L^{1/2-\varepsilon}<\|w_k\|<L^{1/2+\varepsilon},\quad
 |q_k|\le(1/2+L^{-\varepsilon})|p_k|\right\}.
\]
Then, uniformly in $\eta\in\mathbb T^2$,
\begin{equation}\label{eq:regularfrequencywindow}
 \begin{split}
 &\left|\mathcal R_{L,A}F(\eta)
       -\sum_{k\in\mathcal W_{L,\varepsilon}(A)}c_k e(k\cdot\eta)\right|\\
 &\qquad\le C\mathcal D_4^{\mathrm{aff}}(F)
       \left(H^2L^{-\varepsilon}+H^2L^{-2\varepsilon}
                         +H^4L^{-2+4\varepsilon}\right).
 \end{split}
\end{equation}
The constant is absolute.  Every retained mode satisfies
\[
 \tfrac12 L^{-\varepsilon}<|\xi_1|<L^\varepsilon,
 \qquad
 \left|\frac{\xi_2}{\xi_1}\right|
                       \le L/2+L^{1-\varepsilon}.
\]
Thus its quadratic phase has its stationary point in the displayed
enlargement of the averaging interval.
\end{lemma}
\begin{proof}
The proof of Lemma~\ref{lem:shearmajorant} gives
$|c_k|\le C\mathcal D_4^{\mathrm{aff}}(F)b(w_k)$.
Use the lattice $\Lambda$ from that proof and put $\lambda=\latticemin(\Lambda)$.
Dyadic shells, using $\lambda^{-2}\le H^2$, show that for
$0<R\le\sqrt L$,
\[
 \sum_{0<\|w\|\le R}b(w)
 \le C H^2\sum_{r\le R}\frac{r^2}{1+\sqrt L\,r}
 \le C H^2 R/\sqrt L.
\]
Here and below $r$ ranges over geometric shell radii, so the last
sum is bounded by a constant times its largest summand.  If
$R<\lambda$ the sum is empty.  For $R\ge\sqrt L$, the radial
bound gives $b(w)\le CL\|w\|^{-4}$, hence
\[
 \sum_{\|w\|\ge R}b(w)\le C H^2 L R^{-2}.
\]
These prove the first two error terms with
$R=L^{1/2-\varepsilon}$ and $R=L^{1/2+\varepsilon}$, respectively.

For $0<\delta\le1$, if $|q|>(1/2+\delta)|p|$, then for every
$|u|\le1/2$ one has
$|q-up|\ge c\delta\|(p,q)\|$ with an absolute $c>0$.
Therefore
\[
 b(w)\le C\delta^{-4}(1+\sqrt L\,\|w\|)^{-4}.
\]
Summing from the shortest vector in geometric shells gives
\[
 \sum_{|q|>(1/2+\delta)|p|}b(w)
 \le C\delta^{-4}L^{-2}\lambda^{-4}
 \le C\delta^{-4}H^4L^{-2}.
\]
Take $\delta=L^{-\varepsilon}$ and use the triangle inequality.
Finally the cone condition implies $\|w\|<2|p|$ and
$|q/p|\le1/2+L^{-\varepsilon}$; together with the radial window
these give the last two assertions.
\end{proof}

\subsection{Congruence charts and the inverse phase}

\begin{lemma}
\label{lem:congruenceprimitivecharts}
Fix $N\ge1$ and put
$\Delta_N=\Gamma_2(N)\ltimes\mathbb Z^2$ and
$J_0=\left(\begin{smallmatrix}0&1\\-1&0\end{smallmatrix}\right)$.
The nonzero fiber contribution on $A_{\mathrm{aff}}/\Delta_N$
is a sum of finitely many expressions of the following form:
\begin{equation}\label{eq:congruencechartsum}
 \sum_{m\ge1}\ \sum_{\substack{(q,c)\in\mathbb Z^2\ \mathrm{primitive}\\
                        q\equiv1\ (N),\ c\equiv0\ (N)}}
 \chi(q,c)\,c^{G}_{m(c,-q)}(L,B)\,
                           e\bigl(m(c\theta_1-q\theta_2)\bigr).
\end{equation}
Here $B=A\tau$, $\theta=\tau^{-1}\eta$, and
$G=F\circ R_{(\tau^{-1},0)}$ for a fixed
$\tau\in\SL_2(\mathbb Z)$ depending only on the chart and $N$.
The function $\chi$ is smooth away from the origin, homogeneous
of degree zero, with derivatives controlled in terms of $N$.
It is supported where $|q|\le |c|$, so $c\ne0$.
Here $c_k^G(L,B)=c_k(L,B;G)$ is the coefficient defined in
\eqref{eq:regularfouriercoefficient}.
The number of charts and all chart constants depend only on $N$.
Right deck maps commute with left-action derivatives, so
$\mathcal D_s^{\mathrm{aff}}(G)=\mathcal D_s^{\mathrm{aff}}(F)$.
Their supports range over finitely many fixed compact sets when
$\operatorname{supp}F$ is fixed.  Moreover
$\mathcal H_L(B)=\mathcal H_L(A)$.
\end{lemma}
\begin{proof}
Right multiplication by $(\tau,0)$ is well-defined on the quotient:
$\Gamma_2(N)$ is normal in $\SL_2(\mathbb Z)$ and the integer
translation lattice is preserved by $\tau$.  It depends only on
$\tau\bmod N$ and commutes with left multiplication.
Directly in fiber coordinates,
\[
 \widehat{(F\circ R_{(\tau^{-1},0)})}_{A\tau}(k)
                     =\widehat F_A(\tau^{-\mathsf T}k).
\]
The same identity holds for the averaged coefficients $c_k$,
since their physical dual frequency is unchanged.  Also
$\tau^{-\mathsf T}J_0=J_0\tau$, and therefore the translation
phase transforms by $\eta\mapsto\tau^{-1}\eta$.

Write the primitive column indexing an original nonzero mode as
$z$, so that $k=mJ_0z$.  Split the $z$ by their unimodular
residue classes $a\bmod N$.  For each such class choose a fixed
$\tau_a\in\SL_2(\mathbb Z)$ whose first column is congruent
to $a$. Such a choice exists: complete the unimodular residue
column to a matrix in $\SL_2(\mathbb Z/N\mathbb Z)$ and lift
that matrix using the surjectivity proved in
Lemma~\ref{app:primitive-vectors}. Writing $z=\tau_aw$
now gives $w\equiv e_1\pmod N$.

To ensure that the second coordinate has absolute value comparable
to the vector norm, put
$v_N=\left(\begin{smallmatrix}1&0\\N&1\end{smallmatrix}\right)$.
For $w=(q,c)$,
\[
 \max\{|c|,|c-Nq|\}\ge\|w\|/\sqrt5.
\]
Choose a smooth homogeneous partition $\chi_0(w)+\chi_1(w)=1$
such that $\chi_0$ is supported where
$|c|\ge\|w\|/(2\sqrt5)$ and $\chi_1$ where
$|c-Nq|\ge\|w\|/(2\sqrt5)$.
For the first part keep $w$ and $\tau_a$.
For the second part set $w'=v_N^{-1}w$ and use
$\tau_av_N$, with weight $\chi_1(v_Nw')$.
Since $v_N\equiv I_2\pmod N$, both new columns remain congruent
to $e_1$.  In both charts the first coordinate has absolute
value at most $2\sqrt5$ times the second.
Finally use a finite smooth partition of unity in $q/c$, with
the $j$th part supported where $|q/c-j|\le1$ and $|j|\le5$.
On that part substitute $q=q'+jc$, equivalently
$w=u_2(j)w'$, and replace the current $\tau$ by $\tau u_2(j)$.
The column $w'$ remains congruent to $e_1\pmod N$ because
$c\equiv0\pmod N$, and it now satisfies $|q'|\le|c|$.
The weights remain smooth and homogeneous, with controlled
derivatives.  After renaming $q'$, this gives the asserted support.
The covariance already established proves the exact decomposition
\eqref{eq:congruencechartsum}.  Norm, support, and height claims
follow from the finite deck action and $B\mathbb Z^2=A\mathbb Z^2$.
\end{proof}

\begin{lemma}
\label{lem:congruenceinversephase}
Consider a chart from Lemma~\ref{lem:congruenceprimitivecharts},
and write
$B=\left(\begin{smallmatrix}\alpha&\beta\\\gamma_0&\delta\end{smallmatrix}\right)$.
For $q\equiv1\pmod N$, $c\equiv0\pmod N$, $(q,c)=1$, $c\ne0$,
choose an inverse $\overline q$ of $q$ modulo $N|c|$ and set
\[
 r\equiv-\overline q\pmod{N|c|},\qquad
 p=(1+qr)/c,\qquad
 \gamma=\begin{pmatrix}p&q\\r&c\end{pmatrix}.
\]
Then $p\in\mathbb Z$, $p\equiv0\pmod N$, and
$\gamma\equiv J_0\pmod N$.
Put $G^\sharp=G\circ R_{(J_0^{-1},0)}$ and
$f_m(D)=\widehat{G^\sharp}_D(me_1)$. For $d\ne0$, write
\[
 D(x,d,t)=\begin{pmatrix}d^{-1}+xt&x\\dt&d\end{pmatrix},
 \qquad \Phi_m(x,d,t)=f_m(D(x,d,t)).
\]
The function $\Phi_m$ is periodic in $t$ with period $N$. Write
\[
 \Phi_m(x,d,t)=\sum_{n\in\mathbb Z}a_{m,n}(x,d)e(nt/N).
\]
With $b=\alpha q+\beta c$, $d=\gamma_0q+\delta c$, and $d\ne0$, the
$m,n,q,c$ term in the unfolded coefficient has the following phase
factor outside the integral
\begin{equation}\label{eq:congruenceinversephase}
 e\left(\frac{m b^2}{2d}
       -\frac{n\overline q}{Nc}
       +\frac{n\gamma_0}{Ncd}
       +m(c\theta_1-q\theta_2)\right)
\end{equation}
and the corresponding integral
\[
 \frac1{Ld}\int_{b-Ld/2}^{b+Ld/2}
             a_{m,n}(x,d)e(-mx^2/(2d))\,\mathrm dx.
\]
The integral is oriented for $d<0$.  The formulas are independent
of the chosen inverse representative. Since $\xi_1=md$, the case
$d=0$ is excluded by Lemma~\ref{lem:regularfrequencywindow}.
For all nonnegative integers $j,h$,
\begin{equation}\label{eq:congruenceamplitudes}
 \begin{split}
 |a_{m,n}(x,d)|
 &\le C_{N,j,h}\mathcal D_{j+h}^{\mathrm{aff}}(F)
                         (1+\sqrt{x^2+d^2})^{2j}\\
 &\qquad\times(1+m\sqrt{x^2+d^2})^{-h}(1+|n|)^{-j}.
 \end{split}
\end{equation}
\end{lemma}
\begin{proof}
Since $q\equiv1\pmod N$ and $(q,c)=1$, the integer $q$ is
invertible modulo $N|c|$.  The congruence defining $r$ makes
$1+qr$ divisible by $Nc$, proving the assertions about $p$.
Reducing the determinant equation modulo $N$ gives
$r\equiv-1\pmod N$, and hence $\gamma\equiv J_0\pmod N$.
Right deck maps depend only on their residue modulo $N$, so
\[
 \widehat G_{u_2(s)B}(m(c,-q))
                     =f_m(u_2(s)B\gamma).
\]
For $j\in\Z$, put
$\ell_j=\left(\begin{smallmatrix}1&0\\Nj&1\end{smallmatrix}\right)
\in\Gamma_2(N)$. Since
\[
 D(x,d,t+Nj)=D(x,d,t)\ell_j,\qquad
 \ell_j^{-\mathsf T}e_1=e_1,
\]
the Fourier covariance from Lemma~\ref{lem:congruenceprimitivecharts}
gives the asserted period.
Write $B\gamma=\left(\begin{smallmatrix}a&b\\c_B&d\end{smallmatrix}\right)$.
Its physical dual frequency is $m(d,-b)$, so the flow phase is
$-m(ds^2/2+bs)$. The substitution $x=b+sd$ changes this to
$mb^2/(2d)-mx^2/(2d)$ and gives the stated oriented integral.
The lower-left entry $c_B$ of $B\gamma$ satisfies
\[
 \frac{c_B}{d}=\frac rc+\frac{\gamma_0}{cd}.
\]
Inserting the period-$N$ Fourier expansion gives
\eqref{eq:congruenceinversephase}.  Changing $r$ by $N|c|$
changes $c_B/d$ by an integer multiple of $N$, as required.
It remains to prove the amplitude estimate. Differentiation in $t$
is the left-action derivative associated with
\[
 Y_{x,d}=\begin{pmatrix}xd&-x^2\\d^2&-xd\end{pmatrix}.
\]
This matrix is independent of $t$, and hence
\[
 \partial_t^j\Phi_m(x,d,t)
 =\widehat{D_{Y_{x,d}}^jG^\sharp}_{D(x,d,t)}(me_1).
\]
Put $r=(x^2+d^2)^{1/2}$. The physical frequency is
$D(x,d,t)^{-\mathsf T}(me_1)=m(d,-x)$, and the coefficients of
$Y_{x,d}$ are $O(r^2)$. Applying
\eqref{eq:fiberfourierdecay} after the $j$ derivatives in $t$ gives
\[
 |\partial_t^j\Phi_m(x,d,t)|
 \le C_{j,h}\mathcal D_{j+h}^{\mathrm{aff}}(F)
                            (1+r)^{2j}(1+mr)^{-h}.
\]
Integration by parts in the period-$N$ Fourier coefficient gives
\eqref{eq:congruenceamplitudes} for $n\ne0$. For $n=0$, use
$j=0$ in the last estimate and enlarge the derivative norm.
\end{proof}

\paragraph{Reduction to arithmetic sums.}
For $c>0$ the inverse modulus is $M=Nc$.  On the set
$q\equiv1\pmod N$, primitivity is equivalent to $(q,M)=1$.
The remaining residue restriction is
\[
 1_{q\equiv1\ (N)}=
       \frac1N\sum_{j=0}^{N-1}e\bigl(j(q-1)/N\bigr).
\]
It changes only the real linear coefficient in the exponential
sum.  Thus arithmetic bounds uniform in that linear coefficient
apply at modulus $M=Nc$, with a bounded number of terms and
constants depending only on $N$.  A chart restricts $q$ to an
interval of length $O_N(|c|)$; it can be split into $O_N(1)$
intervals of length at most $M$.  Smooth chart weights are
handled by partial summation, their derivatives being controlled
in the ratio $q/c$. For $c<0$, the same reduction applies with
the sign of the inverse coefficient reversed. The extra real term is
$n\gamma_0/(Ncd)$, the correction required by the determinant
identity. For $n=0$ it disappears entirely.
The estimates below sum nonnegative bounds over the moduli.
Restricting these sums to $c\equiv0\pmod N$ can only decrease them.
This proves the reduction at fixed congruence level used below.

\subsection{Nonzero inverse coefficients}
\label{subsec:nonzeroinverse}

For $n\ne0$, four differencing steps give a rational exponential sum
multiplied by a controlled real-phase weight.
The following estimate and its corollary give the bound used in the
final assembly; the last lemma controls the exceptional moduli.

Write $\tau(n)$ for the number of positive divisors of $n$.
For every $\epsilon>0$, we use the elementary bound
\begin{equation}\label{eq:elementarydivisorbound}
 \tau(n)\le C_\epsilon n^\epsilon\qquad(n\ge1).
\end{equation}
For completeness, if $p\ge2^{1/\epsilon}$ and $a\ge1$, then
$a+1\le2^a\le p^{\epsilon a}$. For each of the finitely many
smaller primes the supremum of $(a+1)p^{-\epsilon a}$ over
$a\ge0$ is finite. Multiplication over
the prime powers dividing $n$ proves
\eqref{eq:elementarydivisorbound}.

\begin{lemma}[Fourfold differencing for inverse phases]
\label{lem:fourdiffinverse}
Let $c\ge1$ be an integer and write $c=uv$, where
$u=\prod_{\operatorname{ord}_{\mathfrak p}(c)=1}\mathfrak p$ and
$v=\prod_{\operatorname{ord}_{\mathfrak p}(c)\ge2}
\mathfrak p^{\operatorname{ord}_{\mathfrak p}(c)}$.
Thus $u$ is square-free, $v$ is square-full, and $(u,v)=1$;
we call $v=v(c)$ the square-full part of $c$.
(An integer is square-full if every prime dividing it has exponent
at least two; the empty product is one.) Let $n\in\mathbb Z\setminus\{0\}$,
let $I$ be a half-open real interval of length at most $c$, and let
$f$ be a real $C^6$ function on its closure. For $(q,c)=1$, write
$\overline q$ for the inverse of $q$ modulo $c$. Suppose
\[
 Z\ge1,\qquad \sup_I|f^{(6)}|\le Zc^{-2}.
\]
For each integer $1\le H\le c$ and each $\epsilon>0$,
\begin{equation}\label{eq:fourdiffinverse}
 \begin{split}
 &\frac1c\left|\sum_{\substack{q\in I\cap\mathbb Z\\(q,c)=1}}
                   e\left(f(q)+\frac{n\overline q}{c}\right)\right|\\
 &\qquad\le C_\epsilon\left\{H^{-1}
       +Z|n|^{1/2}H^{64}c^{-1/2+\epsilon}v^{1/2}\right\}^{1/16}.
 \end{split}
\end{equation}
The constant is independent of the position of $I$ and of all
derivatives of $f$ of order at most five.
\end{lemma}
\begin{proof}
Extend the summand by zero outside $I$ and outside the units
modulo $c$. Four applications of the elementary van der Corput
inequality, using Cauchy--Schwarz between applications, give
\[
 \left|\frac1c\sum_q a_q\right|^{16}
 \le C\left\{\frac1H+
 \frac1{H^4}\sum_{1\le h_1,\ldots,h_4\le H}
 \left|\frac1c\sum_q
       \prod_{\nu\in\{0,1\}^4}
          \mathcal C^{|\nu|}a_{q+\nu\cdot h}\right|\right\},
\]
where $\mathcal C$ denotes complex conjugation. To obtain it, start from
\[
 \left|\frac1c\sum_q a_q\right|^2
 \le C\left(\frac1H+\frac1H\sum_{h=1}^H
               \left|\frac1c\sum_q a_{q+h}\overline{a_q}\right|\right).
\]
This follows by averaging $H$ translates and applying
Cauchy--Schwarz; the enlarged support has length at most $2c$.
Apply it successively to the correlation sequences, which still
have modulus at most one and support in an interval of at most
$c$ integers. Jensen's inequality at each squaring averages the
previous shift variables. This gives the displayed fourfold bound,
including the total $O(H^{-1})$ diagonal contribution.
For a fixed tuple $h=(h_1,\ldots,h_4)$, the condition that all
$q+\nu\cdot h$ lie in $I$ defines the interval
\[
 J=I\cap\bigl(I-h_1-h_2-h_3-h_4\bigr).
\]
An empty interval contributes zero. The coprimality conditions on
the shifted integers remain part of the periodic arithmetic factor.

There are $O(H^3)$ tuples $h$ for which the sixteen subset sums
$\sigma_\nu=\nu\cdot h$ are not distinct as integers. Indeed,
each collision is a nonzero equation
$\sum_i\varepsilon_i h_i=0$ with
$\varepsilon_i\in\{-1,0,1\}$; after fixing three variables the
fourth is determined. These tuples contribute $O(H^{-1})$.
For every other tuple define the nonzero integer
\[
 \Delta(h)=\prod_{\nu<\nu'}(\sigma_\nu-\sigma_{\nu'}),
 \qquad |\Delta(h)|\le(4H)^{120}.
\]
Here the product uses any fixed ordering of $\{0,1\}^4$.
The rational part of the correlation is
\[
 R_h(X)=n\sum_{\nu\in\{0,1\}^4}
                    \frac{(-1)^{|\nu|}}{X+\sigma_\nu}.
\]
The corresponding exponential is evaluated using inverses modulo
$c$ and is set to zero unless every $q+\sigma_\nu$ is a unit.
This defines the $c$-periodic factor whose complete Fourier
coefficients we now estimate.
For a prime $p\mid u$ with $p\nmid n\Delta(h)$, $R_h$ has
sixteen distinct simple poles with nonzero residues. Adding a
linear polynomial $kX$ does not change them. The rational Weil
bound, in the pure additive case of
\cite[Theorem 1.1]{CochranePinner2006}, therefore gives
\[
 \left|\sum_{\substack{q\bmod p\\
                   p\nmid\prod_\nu(q+\sigma_\nu)}}
       e\left(\frac{R_h(q)+kq}{p}\right)\right|\le32\sqrt p.
\]
Here rational expressions inside an additive character are
evaluated in the finite field. The sixteen simple finite poles
and a polynomial part of degree at most one satisfy the hypotheses
of the cited theorem: in its notation, $L=32$ and $M\le1$.
Thus $M+L-1\le32$. In particular, a simple pole precludes an
Artin--Schreier coboundary. Use the trivial bound at primes dividing
$n\Delta(h)$ and on the entire square-full factor $v$.
The Chinese remainder
theorem then bounds every complete Fourier coefficient of the
periodic rational correlation by
\[
 32^{\omega(u)}u^{1/2}v\,(u,n\Delta(h))^{1/2}
 \le C_\epsilon c^{1/2+\epsilon}v^{1/2}|n|^{1/2}H^{60}.
\]
The factors produced by the Chinese remainder theorem are units,
so the exceptional primes still divide $n\Delta(h)$. We have also used
$32^{\omega(u)}\ll_\epsilon c^\epsilon$, where $\omega(u)$ is the
number of distinct prime divisors of $u$. This follows by absorbing
$32$ into $p^\epsilon$ for all sufficiently large primes and
bounding the product over the finitely many remaining primes.

It remains to control the real phase
$\phi_h(q)=\sum_\nu(-1)^{|\nu|}f(q+\sigma_\nu)$.
The integral formula for a finite difference gives
\[
 \sup_J|\phi_h''|\le h_1h_2h_3h_4\sup_I|f^{(6)}|
                         \le ZH^4c^{-2}.
\]
Subtract the first-order Taylor polynomial of $\phi_h$ at a point of $J$.
The derivative of the remainder is at most $ZH^4/c$, so its
exponential has total variation at most $CZH^4$ on $J$.
Here $\operatorname{Var}_J(W)$ denotes total variation on $J$;
for a $C^1$ weight it is $\int_J|W'|$.

To bound the factor introduced by completion, for any real $\theta$
and any interval $J'$ of at most $c$ consecutive integers, the
finite geometric-series bound gives
\[
 \frac1c\sum_{j=0}^{c-1}
 \left|\sum_{q\in J'}e((\theta-j/c)q)\right|
 \le C\log(2c).
\]
Indeed the summands are bounded by
$\min\{c,C\|\theta-j/c\|_{\R/\Z}^{-1}\}$; ordering the
distances to the $c$ equally spaced points gives the harmonic
sum $C\sum_{j\le c}j^{-1}$ after division by $c$.
This is the normalized discrete Fourier $\ell^1$ norm of
the restricted linear phase. Multiplying it by the largest
complete Fourier coefficient of the periodic rational factor,
and then using partial summation for the remaining weight,
completes the correlation with a factor at most
$CZH^4\log(2c)$, uniformly in all lower derivatives of $f$.
Each correlation with distinct subset sums is consequently at most
\[
 C_\epsilon Z|n|^{1/2}H^{64}
               c^{1/2+\epsilon}v^{1/2}.
\]
Absorb the logarithm into an initially smaller epsilon and
substitute in the fourfold inequality.
\end{proof}

\begin{corollary}
\label{cor:fourdiffuniformrange}
Let $S\ge2$ and $0<a\le1/100$. Suppose
\[
 S^{1-a}\le c\le S^{1+a},\quad
 0<|n|\le S^a,\quad Z\le S^a,\quad v\le S^{1/100}.
\]
Under the hypotheses of Lemma~\ref{lem:fourdiffinverse},
\[
 \left|\sum_{\substack{q\in I\cap\mathbb Z\\(q,c)=1}}
       e\left(f(q)+\frac{n\overline q}{c}\right)\right|
       \le C c S^{-1/16000}.
\]
The same bound multiplied by
$\|W\|_\infty+\operatorname{Var}_I(W)$ holds after inserting
any complex weight $W$ of bounded variation. Here $C>0$ is
absolute.
\end{corollary}
\begin{proof}
Take $H=\lfloor S^{1/1000}\rfloor$ and use
\eqref{eq:fourdiffinverse} with $\epsilon=1/1000$.
The second term inside braces is at most
\[
 C S^{a+a/2+64/1000+1/200-(1-a)(1/2-1/1000)}
       \le C S^{-2/5}.
\]
The first is $O(S^{-1/1000})$. Taking the sixteenth root
proves the claim; summation by parts proves the weighted version,
since the unweighted estimate applies to every initial subinterval.
\end{proof}

\begin{lemma}
\label{lem:squarefulldiscard}
For $X,Y\ge1$, the number of positive integers $c\le X$ whose
square-full factor $v(c)$ exceeds $Y$ is at most $CXY^{-1/2}$.
In particular, among $c\le S^{1+a}$, the condition
$v(c)>S^{1/100}$ excludes at most $CS^{1+a-1/200}$ moduli.
Here $C>0$ is absolute.
\end{lemma}
\begin{proof}
Every square-full integer can be written uniquely as $r^2t^3$
with $t$ square-free, so the number up to $R$ is at most
$\sqrt R\sum_{t\ge1}t^{-3/2}\le C\sqrt R$.
Dyadic summation gives
$\sum_{v>Y,\ v\ {\rm square\text{-}full}}v^{-1}\le CY^{-1/2}$.
Since $v(c)$ divides $c$, the asserted count is at most
$\sum_{v>Y}\lfloor X/v\rfloor$ over square-full $v$.
\end{proof}

\subsection{Zero inverse coefficient}
\label{subsec:zeroinverse}

For $n=0$, the rational inverse term is absent. We first prove an
estimate for real quadratic sums averaged over the modulus, then
combine it with derivative estimates to obtain a bound uniform in
the base matrices specified in Lemma~\ref{lem:movingzeroinverse}.
These remain nonzero fiber Fourier modes.

\begin{lemma}[Mean bounds for quadratic sums with real coefficients]
\label{lem:averagedrealquadratic}
For an integer $c\ge1$ and a real number $\kappa$, put
\[
 M_c(\kappa)=\sup_{\substack{\beta\in\mathbb R\\
                         I\text{ half-open},\ |I|\le c}}
 \left|\sum_{\substack{q\in I\cap\mathbb Z\\(q,c)=1}}
                         e(\kappa q^2/c+\beta q)\right|.
\]
For every $\epsilon>0$ there is $C_\epsilon>0$ such
that, whenever $R\ge1$ and $|\kappa|\ge8R$,
\begin{align}
 \sum_{\substack{c\in\mathbb Z\\R\le c<2R}}M_c(\kappa)^2
     &\le C_\epsilon(|\kappa|R)^\epsilon R^2,
                       \label{eq:averagedquadraticsquare}\\
 \sum_{\substack{c\in\mathbb Z\\R\le c<2R}}M_c(\kappa)
     &\le C_\epsilon(|\kappa|R)^\epsilon R^{3/2}.
                       \label{eq:averagedquadraticabsolute}
\end{align}
The same estimates hold without the restriction $(q,c)=1$.
In particular the interval and the real linear coefficient may
vary independently with the modulus.
\end{lemma}
\begin{proof}
First omit primitivity, and write $U_c(\kappa)$ for the resulting
supremum. For any of its defining sums, expansion by the difference
of the two indices and the finite geometric-series estimate give
\[
 |Q_c|^2\le Cc+C\sum_{1\le h<c}
      \min\{c,\|2\kappa h/c\|_{\mathbb R/\mathbb Z}^{-1}\}.
\]
The reciprocal is interpreted as infinity when its denominator
vanishes. The bound is independent of the interval's position and
of $\beta$, so it also bounds $U_c(\kappa)^2$.
Choose a power of two $K_*\in[2R,4R)$.
Dyadic decomposition bounds the reciprocal sum by a constant times
\[
 \sum_{\substack{K=2^j\\1\le K\le K_*}}K\,
 \#\{(c,h):R\le c<2R,\ 1\le h<c,
                       \|2\kappa h/c\|_{\mathbb R/\mathbb Z}\le K^{-1}\}.
\]
For each counted pair choose an integer $a$ realizing this
distance and put $\ell=ac$. Then
\[
 c\mid\ell,\qquad
 |\ell-2\kappa h|\le2R/K,\qquad
 0<|\ell|\le C|\kappa|R.
\]
The strict lower bound follows from $h\ge1$ and
$|\kappa|\ge8R$; it is the reason for this hypothesis.
For fixed $h<2R$, the displayed interval contains at most
$C(1+R/K)$ integers. Each nonzero integer in it admits at most
$C_\epsilon(|\kappa|R)^\epsilon$ possible divisors $c$.
Consequently the number of resonant pairs is at most
\[
 C_\epsilon(|\kappa|R)^\epsilon R(1+R/K).
\]
Summing the dyadic weights gives
\[
 \sum_{R\le c<2R}U_c(\kappa)^2
 \le C_\epsilon(|\kappa|R)^\epsilon R^2\log(2R).
\]
Use a smaller epsilon in the divisor bound and absorb the logarithm
to obtain \eqref{eq:averagedquadraticsquare} for $U_c$.

For primitivity, M\"obius inversion followed by $c=dm$ and $q=dr$
changes each inner sum to
\[
 \sum_{r\in (I/d)\cap\mathbb Z}
        e\left(\frac{\kappa d\,r^2}{m}+\beta d\,r\right),
 \qquad |I/d|\le m.
\]
Thus Cauchy--Schwarz over the divisors of $c$ gives
\[
 M_c(\kappa)^2\le\tau(c)
                  \sum_{d\mid c}U_{c/d}(\kappa d)^2.
\]
Fix a small $\epsilon'>0$ and sum over $R\le c<2R$.
For $d\le R$, apply the unrestricted estimate at scale $R/d$
and quadratic parameter $\kappa d$. Its size hypothesis holds
because $|\kappa|d\ge8R/d$, and the product of the two parameters
is $|\kappa|R$. Therefore
\[
 \sum_{R\le c<2R}M_c(\kappa)^2
 \le C_{\epsilon'}R^{\epsilon'}
 \left((|\kappa|R)^{\epsilon'}
           \sum_{1\le d\le R}(R/d)^2+R\right).
\]
The last $R$ accounts for $R<d<2R$, when the only possible
quotient $m$ is one and $U_1\le1$.
Since $\sum_{d\ge1}d^{-2}<\infty$, choosing $\epsilon'$ smaller
than $\epsilon/2$ proves \eqref{eq:averagedquadraticsquare}.
Finally Cauchy--Schwarz over the $O(R)$ moduli proves
\eqref{eq:averagedquadraticabsolute}, using $\epsilon'<\epsilon/2$
as above.
\end{proof}

\begin{lemma}[Cancellation without the modular-inverse term]
\label{lem:movingzeroinverse}
There is an absolute constant $C>0$ with the following
property. Let $S,P\ge2$ be real, let $m$ be an integer with
$1\le m\le P$, and let
\[
 B=\begin{pmatrix}\alpha&\beta\\\gamma&\delta\end{pmatrix}
       \in\SL_2(\mathbb R),\qquad
 \max\{|\alpha|,|\beta|,|\gamma|,|\delta|\}\le P.
\]
For an integer $c>0$ put $D_c(q)=\gamma q+\delta c$ and
\[
 \Theta_{c,x}(q)=\frac{mS^3}{2}
              \frac{(\alpha q+\beta c)^2}{D_c(q)}
                -\frac{mSx^2}{2D_c(q)}.
\]
Let $Z_c$ be the supremum of
\[
 \left|\sum_{\substack{q\in J\cap\mathbb Z\\(q,c)=1}}
                  e\bigl(\Theta_{c,x}(q)+\xi q\bigr)\right|
\]
over $|x|\le P$, $\xi\in\mathbb R$, and half-open intervals
$J\subset[-5c,5c]$ on which
\begin{equation}\label{eq:movingzerodenominator}
                    S/P\le |D_c(q)|\le PS.
\end{equation}
Set $Z_c=0$ if there is no nonempty admissible interval.
Then
\begin{equation}\label{eq:movingzeroestimate}
 \sum_{S/P\le c\le PS}Z_c
                      \le C P^{1000}S^{2-1/1000}.
\end{equation}
The same bound holds if, independently for each $c$, the sum is
restricted to one residue class modulo a positive integer.
For a choice of admissible intervals $J_c$ and weights $w_c$ of
bounded variation on $J_c$, inserting $w_c(q)$ in each sum
multiplies the right side by at most
\[
 \sup_c\bigl(\|w_c\|_\infty+\operatorname{Var}_{J_c}(w_c)\bigr).
\]
The intervals, residue classes, real linear coefficients, and
parameters $x$ may all vary with $c$.
\end{lemma}
\begin{proof}
We use the following derivative estimate of
\cite[Theorem 11, equation (15), and Table 1]{AriasReyna2024}.
Let $k\in\{3,4,5\}$ and let $f$ be a real $C^k$ function
on a half-open interval $I$ of length $Y$. If
$0<\lambda\le |f^{(k)}|\le\Lambda$ on $I$ and
$\lfloor Y\rfloor>k$, then
\begin{equation}\label{eq:explicitderivativeinput}
 \frac1Y\left|\sum_{n\in I\cap\mathbb Z}e(f(n))\right|
 \le 11\max\left\{
  \left(\frac{\Lambda}{\lambda Y}\right)^{2/2^k},
  \left(\frac{\Lambda^2}{\lambda}\right)^{1/(2^k-2)},
  (\lambda Y^k)^{-2/2^k}\right\}.
\end{equation}
The derivative has constant sign by continuity, and changing
$f$ to $-f$ covers the negative sign. The three constants in
the cited estimate are all less than $11$ for these orders.
All other estimates below are elementary or are supplied by
Lemma~\ref{lem:averagedrealquadratic}.

The trivial estimate $Z_c\le10c+1$ gives
$\sum_{S/P\le c\le PS}Z_c\ll P^2S^2$, which proves the result
when $S<(10P)^{10}$, after increasing $C$. Assume henceforth that
$S\ge(10P)^{10}$ and restrict attention to $S/P\le c\le PS$.
All powers of $P$ arising in the estimates below have exponents
at most $1000$.
We distinguish the range where a derivative of order three, four,
or five gives cancellation from the range where a quadratic
approximation can be averaged over $c$.

\medskip
\noindent\emph{The range $|\gamma|\ge S^{-15/4}$.}
For $r\ge2$ direct differentiation, using
$\alpha\delta-\beta\gamma=1$, gives
\begin{equation}\label{eq:movingzeroderivatives}
 \Theta_{c,x}^{(r)}(q)=
 \frac{(-1)^r r!mS^3c^2\gamma^{r-2}}{2D_c(q)^{r+1}}
               \left(1-\frac{x^2\gamma^2}{S^2c^2}\right).
\end{equation}
The factor in parentheses lies between $1/2$ and $1$, since
\[
 0\le\frac{x^2\gamma^2}{S^2c^2}\le P^6S^{-4}\le\frac12.
\]
For each of the following ranges choose the indicated derivative:
\[
\begin{array}{c|c|c}
 \text{range of }|\gamma|&r&
          \text{bounds for }|\Theta_{c,x}^{(r)}|\\ \hline
 [S^{-2/3},P]&5&
          P^{-20}S^{-3}\ \ll |\Theta^{(5)}|
                                  \ll P^{20}S^{-1}\\
 {[S^{-3/2},S^{-2/3}]}&4&
          P^{-20}S^{-3}\ \ll |\Theta^{(4)}|
                                  \ll P^{20}S^{-4/3}\\
 {[S^{-15/4},S^{-3/2}]}&3&
          P^{-20}S^{-11/4}\ \ll |\Theta^{(3)}|
                                  \ll P^{20}S^{-1/2}.
\end{array}
\]
The implicit constants are absolute. For fixed $c,x$ and the chosen
derivative order $r$,
\[
 \frac{\sup_{q\in J}|\Theta_{c,x}^{(r)}(q)|}
      {\inf_{q\in J}|\Theta_{c,x}^{(r)}(q)|}\le3P^{12},
\]
since only the factor $D_c(q)^{-r-1}$ varies. The derivative has a
constant sign on $J$.

If $|J|\le cS^{-1/100}$, the trivial bound is sufficient.
Otherwise remove primitivity by M\"obius inversion. For a divisor
$d\mid c$, the inner sum is over $q=du$, with interval
$I_d=J/d$ of length $Y_d=|J|/d$.
Divisors $d>S^{1/100}$ contribute at most
\[
 \sum_{\substack{d\mid c\\d>S^{1/100}}}(10c/d+1)
                       \le11cS^{-1/100}\tau(c).
\]
For $d\le S^{1/100}$, set
$f_d(u)=\Theta_{c,x}(du)+\xi du$ and define
\[
 \lambda_d=\inf_{u\in I_d}|f_d^{(r)}(u)|,\qquad
 \Lambda_d=\sup_{u\in I_d}|f_d^{(r)}(u)|.
\]
Then
\[
 Y_d\ge P^{-1}S^{98/100},\qquad
 f_d^{(r)}(u)=d^r\Theta_{c,x}^{(r)}(du),\qquad
 \frac{\Lambda_d}{\lambda_d}\le3P^{12}.
\]
The upper derivative bounds give
\[
 \Lambda_d\le CP^{20}
 \begin{cases}
 S^{-95/100},&r=5,\\
 S^{-4/3+4/100},&r=4,\\
 S^{-47/100},&r=3.
 \end{cases}
\]
In the product $\lambda_dY_d^r$ the factors $d^r$ cancel, so
\[
 \lambda_dY_d^r
 =\left(\inf_{q\in J}|\Theta_{c,x}^{(r)}(q)|\right)|J|^r
 \gg P^{-25}
 \begin{cases}
 S^{195/100},&r=5,\\
 S^{96/100},&r=4,\\
 S^{22/100},&r=3.
 \end{cases}
\]
The bounds on $Y_d$ and $S$ ensure $\lfloor Y_d\rfloor>r$,
so \eqref{eq:explicitderivativeinput} applies to $f_d$.
For its middle term we use
\[
 \frac{\Lambda_d^2}{\lambda_d}
 =\frac{\Lambda_d}{\lambda_d}\Lambda_d
 \le3P^{12}\Lambda_d.
\]
Each term on the right side of \eqref{eq:explicitderivativeinput}
is bounded by $CP^{100}$ times the corresponding entry below:
\[
\begin{array}{c|c|c|c}
 r&
 \left(\dfrac{\Lambda_d}{\lambda_dY_d}\right)^{2/2^r}&
 \left(\dfrac{\Lambda_d^2}{\lambda_d}\right)^{1/(2^r-2)}&
 (\lambda_dY_d^r)^{-2/2^r}\\ \hline
 5&S^{-49/800}&S^{-19/600}&S^{-39/320}\\
 4&S^{-49/400}&S^{-97/1050}&S^{-3/25}\\
 3&S^{-49/200}&S^{-47/600}&S^{-11/200}
\end{array}
\]
Every saving exponent in the table is at least $19/600>1/40$.
Since $Y_d\le10c/d$, it follows that
\[
 \left|\sum_{u\in I_d\cap\mathbb Z}e(f_d(u))\right|
 \le CP^{100}\frac cd S^{-1/40}.
\]
The elementary estimate
$\tau(c)\ll c^{1/1000}$ now gives, including the large divisors
and the initially short intervals,
\begin{equation}\label{eq:movingzeronontiny}
                  Z_c\le CP^{101}cS^{-1/200}.
\end{equation}
Since $\sum_{c\le PS}c\ll P^2S^2$, this proves
\eqref{eq:movingzeroestimate} in the present range.

\medskip
\noindent\emph{The range $|\gamma|<S^{-15/4}$.}
Since $5|\gamma|c\le5PS^{-11/4}\le S/(4P)$, any admissible
pair $(c,J)$ satisfies $|\delta|c\ge S/(2P)$. Consequently
\[
 (2P^2)^{-1}\le|\delta|\le P,\qquad
 |D_c(q)|\ge S/(4P)\quad (|q|\le5c).
\]
Thus the phase is smooth on the whole interval $[-5c,5c]$,
even when $0\notin J$. Taylor expansion at $q=0$ gives
\begin{equation}\label{eq:tinylowerleftquadratic}
 \Theta_{c,x}(q)=C_{c,x}+B_{c,x}q+\frac{\kappa_0q^2}{c}
                                        +R_{c,x}(q),
 \qquad \kappa_0=\frac{mS^3}{2\delta^3}.
\end{equation}
Here one may take the constant and linear terms from the Taylor
polynomial of the first rational expression in $\Theta_{c,x}$,
and include the constant term of the second expression in
$C_{c,x}$.  Taylor's formula applied to the first expression,
and direct differentiation of the second, give
\[
 \operatorname{Var}_{[-5c,5c]}(R_{c,x})
    \le CP^{30}(S^4|\gamma|+|\gamma|)
    \le CP^{30}S^{1/4}.
\]
For clarity, the third derivative of the first rational term is bounded by
$CP^{20}S|\gamma|$, so its quadratic Taylor remainder has
variation at most $CP^{20}S|\gamma|c^3$; the second expression
has variation at most $CP^{10}|\gamma|$.
The parameter $\kappa_0$ in \eqref{eq:tinylowerleftquadratic} is
independent of $c,J,x$, so Lemma~\ref{lem:averagedrealquadratic}
can be applied after summing over the moduli.

Partial summation, splitting $[-5c,5c]$ into ten intervals of
length $c$ when necessary, bounds $Z_c$ by
\[
 CP^{30}S^{1/4}M_c(\kappa_0).
\]
Moreover $|\kappa_0|\ge S^3/(2P^3)\ge8PS$ and
$|\kappa_0|\le4P^7S^3$.  On the dyadic blocks meeting
$[S/P,PS]$, Lemma~\ref{lem:averagedrealquadratic}, with its
epsilon chosen as $1/100$, therefore gives
\[
\begin{aligned}
 \sum_{S/P\le c\le PS}M_c(\kappa_0)
 &\le C|\kappa_0|^{1/100}(PS)^{3/2+1/100}\\
 &\le CP^{79/50}S^{77/50}
 \le CP^{20}S^{3/2+1/20}.
\end{aligned}
\]
The first inequality sums a geometric series over the dyadic
scales; the second uses $|\kappa_0|\le4P^7S^3$.
Consequently the left side of \eqref{eq:movingzeroestimate} is
at most $CP^{50}S^{9/5}\le CP^{1000}S^{2-1/1000}$, as required.

Finally, the identity
\[
 1_{q\equiv a\, (\mathrm{mod}\, M)}
       =\frac1M\sum_{j=0}^{M-1}e\bigl(j(q-a)/M\bigr)
\]
introduces only real linear phases and proves the assertion about
residue classes. Every initial subinterval of an admissible interval
is again admissible, so partial summation proves the weighted
assertion.
\end{proof}

\subsection{Assembly of the flow estimate}
\label{subsec:regularflowassembly}

\begin{proof}[Proof of Proposition~\ref{prop:generalregularpointwise}]
Put
\[
 S=\sqrt L,\qquad H=\mathcal H_L(A)\le S^{\epsilon_0/4},\qquad
 Q=S^{\epsilon_0},\qquad P=S^{10\epsilon_0}.
\]
The parameter $Q$ sets the truncation bounds for the Fourier
indices and the integration variable; $P$ bounds the remaining
chart parameters.
We enlarge the lower threshold for $S$, depending only on $N,K$,
whenever necessary. In particular, we assume $P\ge2$.

\medskip\noindent\emph{Lattice reduction and frequency localization.}
Choose a shortest primitive vector of
$a_2(L)^{-1}A\Z^2$, of length $\lambda$, and complete it to
a positively oriented lattice basis. Reducing the second vector
modulo the first bounds its parallel component by $\lambda/2$;
its perpendicular component has length $\lambda^{-1}$.
The disk of radius $\lambda/2$ embeds in the quotient by this
unimodular lattice, so $\pi\lambda^2/4\le1$ and
$\lambda\le2/\sqrt\pi$. Together with $\lambda^{-1}\le H$,
the reduced basis bounds therefore give
$\gamma_{\mathrm{red}}\in\SL_2(\Z)$ such that
\[
 B_0=a_2(L)^{-1}A\gamma_{\mathrm{red}},\qquad
 \|B_0\|+\|B_0^{-1}\|\le CH.
\]
Replace $A,\eta,F$ by
$A\gamma_{\mathrm{red}},\gamma_{\mathrm{red}}^{-1}\eta,
F\circ R_{(\gamma_{\mathrm{red}}^{-1},0)}$.
This preserves the averaged Fourier sum exactly by the covariance
in Lemma~\ref{lem:congruenceprimitivecharts}.
The right deck map depends only on
$\gamma_{\mathrm{red}}\bmod N$. It preserves the left-action
derivative norms, and the support of the pullback is contained in
one of finitely many fixed compact sets.
The same lemma now gives finitely many charts with $|q|\le|c|$.
In each chart the actual base matrix is $a_2(L)B$, where
$B=B_0\tau$ for a fixed integral chart matrix $\tau$; hence
$\|B\|+\|B^{-1}\|\ll_N H$.

Use Lemma~\ref{lem:regularfrequencywindow} with
$\varepsilon=\epsilon_0$.  Its discarded contribution is
\[
 O\left(\mathcal D_4^{\mathrm{aff}}(F)
       (H^2S^{-2\epsilon_0}
        +H^2S^{-4\epsilon_0}
        +H^4S^{-4+8\epsilon_0})\right)
       =O\left(\mathcal D_4^{\mathrm{aff}}(F)S^{-3\epsilon_0/2}\right).
\]
\medskip\noindent\emph{Fourier truncation.}
The localized frequency set is finite. Moreover, its modes have
$\xi_1\ne0$, so the corresponding chart coordinate $d=\xi_1/m$
is nonzero. We may therefore use
Lemma~\ref{lem:congruenceinversephase} and truncate
$m,|n|,|x|\le Q$.

Compact support supplies
$r=(x^2+d^2)^{1/2}\ge\rho_{N,K}>0$ whenever an amplitude is
nonzero: the second column $(x,d)$ is a nonzero lattice vector
on the compact projection of one of the finitely many deck
translates of $K$. If $r<\rho_{N,K}$, the chart function vanishes
for every $t$, so all of its coefficients $a_{m,n}(x,d)$ vanish.
With $j=12,h=40$, \eqref{eq:congruenceamplitudes} reads
\[
 |a_{m,n}(x,d)|\ll_{N,K}\mathcal D_{52}^{\mathrm{aff}}(F)
 (1+mr)^{-4}(1+|n|)^{-12}
 \frac{(1+r)^{24}}{(1+mr)^{36}}.
\]
The final ratio is bounded on $r\ge\rho_{N,K}$.
It is $O_{N,K}(Q^{-10})$ if $m>Q$ or $|x|>Q$:
in either case extract ten powers of $1+mr$, and bound the
remaining ratio using $m\ge1$ and $r\ge\rho_{N,K}$.
Also $\sum_{|n|>Q}(1+|n|)^{-12}\ll Q^{-11}$.
In the original flow variable the retained factor
$(1+mr)^{-4}$ is precisely the integrand of the shear majorant,
whose sum is at most $CH^2$ by Lemma~\ref{lem:shearmajorant}.
The factors $(1+|n|)^{-12}$ are summable, and the change of
variables $x=b+sd$ has $\dd s/L=\mathrm dx/(L|d|)$ when
the integral is written with its endpoints in increasing order. Thus all three
tails together are
\[
 O_{N,K}\left(\mathcal D_{52}^{\mathrm{aff}}(F)
                       H^2S^{-10\epsilon_0}\right).
\]

\medskip\noindent\emph{The retained range and its weights.}
Enlarge the threshold in $S$ so that all fixed chart constants
are absorbed by the powers of $S$ below.
For $B=\left(\begin{smallmatrix}\alpha&\beta\\
\gamma&\delta\end{smallmatrix}\right)$ write
$X(q)=\alpha q+\beta c$ and $D(q)=\gamma q+\delta c$.
In this normalization $b=SX$, $d=D/S$, and the normalized physical
frequency is $w=m(D,-X)$. The frequency bounds therefore become
\[
 S^{1-2\epsilon_0}<m\sqrt{D^2+X^2}<S^{1+2\epsilon_0},
 \qquad |X|\le(1/2+S^{-2\epsilon_0})|D|.
\]
The cone condition gives $\sqrt{D^2+X^2}\asymp|D|$.
Also $B(q,c)^{\mathsf T}=(X,D)^{\mathsf T}$ and
$|q|\le|c|$. Combining these facts with
$m\le Q$ and $\|B\|+\|B^{-1}\|\ll_N H$ gives
\[
 \begin{gathered}
 |D|\gg S^{1-3\epsilon_0},\qquad |D|\le S^{1+2\epsilon_0},\\
 S^{1-13\epsilon_0/4}\ll_N|c|
                  \ll_N S^{1+9\epsilon_0/4}.
 \end{gathered}
\]
In particular, all retained terms satisfy
\[
 |m|,|n|,|x|,\|B\|,|d|,|d|^{-1}\le P,\qquad
 S/P\le |c|,|D|\le PS .
\]
Reversing the integral when $d<0$ and including the chart factor
gives the weight
\[
 W(q)=\frac{\chi(q,c)a_{m,n}(x,D(q)/S)}{S|D(q)|}.
\]
The integral endpoint condition is
$|x/S-X(q)|\le |D(q)|/2$.
The window, cone, integral endpoint, and $|q|\le|c|$ restrictions
are inequalities in polynomials of degree at most two in $q$.
Their intersection is therefore a union of intervals and isolated
points whose total number is bounded independently of all parameters.
On each interval the denominator has constant sign and is bounded
away from zero. The smooth homogeneous factor $\chi(q,c)$ has
bounded supremum and total variation, with bounds depending only
on $N$. We claim that
\[
 \|W\|_\infty+\operatorname{Var}(W)
 \ll_N\mathcal D_1^{\mathrm{aff}}(F)P^6S^{-2}.
\]
To verify it directly, differentiating the chart matrix
in $d$ at fixed $x,t$ gives the left-action derivative
\[
 -d^{-1}\begin{pmatrix}1&0\\0&-1\end{pmatrix}
                       +xd^{-2}E_{12}.
\]
Consequently
$|\partial_d a_{m,n}|\ll P^3\mathcal D_1^{\mathrm{aff}}(F)$.
Since $d'(q)=\gamma/S$ and
$|\partial_q\chi(q,c)|\ll_N|c|^{-1}$, we obtain
\[
 \|W\|_\infty\ll P S^{-2}\mathcal D_0^{\mathrm{aff}}(F),
 \qquad
 \|W'\|_\infty\ll P^5S^{-3}\mathcal D_1^{\mathrm{aff}}(F).
\]
The interval length is at most $2PS$, proving the variation bound.

We may replace each interval by one with half-open endpoints,
accounting separately for $O(1)$ integers per component. Isolated
points are treated in the same way. There are $O(PS)$ possible
moduli, while the sums over $m,n$ and the integration over $x$
contribute a factor $O(P^3)$. The resulting endpoint contribution is at most
\[
 C_N\mathcal D_0^{\mathrm{aff}}(F)P^5S^{-1}.
\]
It will be absorbed by the exceptional-modulus bound below.

\medskip\noindent\emph{Nonzero inverse coefficients.}
For $n\ne0$, the inverse modulus is $M=N|c|$.
The coefficient of $\overline q/M$ is $-\operatorname{sgn}(c)n$,
so its absolute value is $|n|$ for either sign of $c$.
The congruence restriction $q\equiv1\pmod N$ is removed
by the finite additive-character identity in
Lemma~\ref{lem:congruenceinversephase}; it only changes
the real linear coefficient. Apart from constant and real linear terms,
the real phase is
\[
 f(q)=\frac{mS^3(\alpha q+\beta c)^2-mSx^2}{2D(q)}
                       +\frac{n\gamma}{NcD(q)}.
\]
The determinant identity gives, for $\gamma\ne0$,
a linear polynomial plus a constant multiple of $D(q)^{-1}$.
Direct differentiation, also valid by continuity at
$\gamma=0$, yields
\[
 |f^{(6)}(q)|
 \le C_N\left(
   \frac{mS^3c^2|\gamma|^4}{|D|^7}
  +\frac{mSx^2|\gamma|^6}{|D|^7}
  +\frac{|n||\gamma|^7}{N|c||D|^7}\right)
 \le C_NP^{16}S^{-2}.
\]
Set $Z=\max\{1,M^2\sup|f^{(6)}|\}$. Then
$Z\le C_NP^{18}$. Since $\epsilon_0=10^{-12}$, increasing the
threshold for $S$ ensures
\[
 S^{1-1/100}\le M\le S^{1+1/100},\qquad
 0<|n|\le P\le S^{1/100},\qquad Z\le S^{1/100}.
\]
Indeed $M=N|c|$, $S/P\le|c|\le PS$, and
$P^{18}=S^{180\epsilon_0}$. Thus all size and derivative
hypotheses of Corollary~\ref{cor:fourdiffuniformrange} hold
with $a=1/100$.
Split intervals if necessary into a bounded number of pieces
of length at most $M$.
For moduli $M$ whose square-full part is at most $S^{1/100}$,
the corollary gives the per-modulus estimate
$O_N(|c|S^{-1/16000})$.
Lemma~\ref{lem:squarefulldiscard}, applied to $M=N|c|$,
shows that there are
$O_N(PS^{1-1/200})$ exceptional values of $c$.
For each such modulus there are $O(PS)$ possible $q$.
The supremum weight bound $O(P S^{-2}\mathcal D_0^{\mathrm{aff}}(F))$,
together with the sums over $m,n$ and the integration over $x$
(contributing a factor $O(P^3)$), therefore gives an exceptional contribution
$O_N(P^6S^{-1/200}\mathcal D_0^{\mathrm{aff}}(F))$.
This also absorbs the endpoint contribution, since
$P^5S^{-1}\le P^6S^{-1/200}$.

For the remaining nonzero-$n$ moduli, partial summation uses the
variation bound $O(P^6S^{-2}\mathcal D_1^{\mathrm{aff}}(F))$ and
$\sum_{|c|\le PS}|c|\ll P^2S^2$; the $m,n$ sums and
$x$ integral add $O(P^3)$.  This gives
$O_N(P^{11}S^{-1/16000}\mathcal D_1^{\mathrm{aff}}(F))$.

\medskip\noindent\emph{The zero inverse coefficient and final summation.}
For $n=0$, the phase is exactly $\Theta_{c,x}(q)$ from
Lemma~\ref{lem:movingzeroinverse}, up to constant and real linear
terms. The displayed retained range verifies its hypotheses, and
the lemma allows the congruence restrictions and the weights $W$.
For negative $c$, replace $(q,c)$ by $(-q,-c)$ and conjugate the
exponential; the same estimate applies, with a conjugated weight.

Partial summation multiplies the averaged arithmetic
bound by $O(P^6S^{-2}\mathcal D_1^{\mathrm{aff}}(F))$, while only the
$m$ sum and $x$ integral remain, contributing a factor $O(P^2)$.
Thus the retained contribution is at most
\[
 C_{N,K}\mathcal D_1^{\mathrm{aff}}(F)
 \left(P^{11}S^{-1/16000}
       +P^6S^{-1/200}
       +P^{1008}S^{-1/1000}\right)
 \ll_{N,K}\mathcal D_1^{\mathrm{aff}}(F)S^{-\epsilon_0}.
\]
For the last inequality, substitute $P=S^{10\epsilon_0}$ and use
\[
 111\epsilon_0<\frac1{16000},\qquad
 61\epsilon_0<\frac1{200},\qquad
 10081\epsilon_0<\frac1{1000}.
\]
The discarded-window error is $O(S^{-3\epsilon_0/2})$ and the
Fourier tails are $O(H^2S^{-10\epsilon_0})$ with their displayed
derivative norms. Both are smaller. All derivative norms used
have order at most $52$ and are therefore bounded by
$\mathcal D_{64}^{\mathrm{aff}}(F)$.
Since $S^{-\epsilon_0}=L^{-\epsilon_0/2}$, the proposition
follows, with the stated common derivative order $64$.
\end{proof}

\begin{proof}[Proof of Proposition~\ref{prop:regularaffineflow}]
The shortest-vector hypothesis and $\omega_N\le\epsilon_0/8$
give $\mathcal H_L(A)\le L^{\epsilon_0/8}$.
Proposition~\ref{prop:generalregularpointwise} therefore bounds the
nonzero fiber contribution by
$C_{N,K}\mathcal D_{64}^{\mathrm{aff}}(F)L^{-\epsilon_0/2}$,
uniformly in $A^{-1}v$.

For the zero fiber mode, put $T=L/2$ and enlarge the initial
threshold so that $T\ge2$. Since
\[
 a_2(T)^{-1}A=\diag(\sqrt2,1/\sqrt2)\,a_2(L)^{-1}A,
\]
the shortest-vector length at scale $T$ is at least
$2^{-1/2}L^{-\omega_N}$. Lemma~\ref{lem:injectivity} gives
\[
 \inj_{\Delta_2}(a_2(T)^{-1}A\Delta_2)
                              \gg L^{-2\omega_N}.
\]
Apply Theorem~\ref{thm:coverinput} separately to the positive
and negative halves of the time interval. Their normalized average
is the zero fiber contribution, and its error is at most
\[
 \begin{aligned}
 &C_N\mathcal D_{\ell_{\Delta_2}}(\overline F)
       L^{-A_{\Delta_2}+2\omega_NB_{\Delta_2}}\\
 &\qquad\le C_N\mathcal D_{\ell_{\Delta_2}}^{\mathrm{aff}}(F)
                                      L^{-A_{\Delta_2}/2}.
 \end{aligned}
\]
Here $\omega_N\le b_{\Delta_2}$ and
\eqref{eq:coverexponent} give the exponent inequality.
Fiber integration gives
$\mathcal D_s(\overline F)\le\mathcal D_s^{\mathrm{aff}}(F)$ and
$m_{\Delta_2}(\overline F)=m_{\mathrm{aff}}(F)$, as in
Section~\ref{sec:leviaffine}. Adding the nonzero-mode estimate
and using the definitions of $\zeta_N$ and $\ell_N^{\mathrm r}$
in \eqref{eq:regularaffineflowconstants} proves the proposition.
\end{proof}

\section{Proof of the main theorem}
\label{sec:mainproof}

All support estimates are now available. We apply the intrinsic
coefficient bounds, select the estimate for the classified support,
and compare derivative norms. The converse is proved separately in
Section~\ref{subsec:intrinsicnecessity}; it uses no effective flow estimate.

\subsection{Assembly of the proof}
\label{subsec:mainassembly}

We prove an estimate on a fixed closed support before applying it
to the support in Theorem~\ref{thm:intrinsic}. Let $\mathbf H\le\SL_3$
be a connected $\mathbb Q$-subgroup such that
$\mathbf H/R_u(\mathbf H)$ is semisimple and
$H\Gamma/\Gamma$ is closed with finite invariant volume, where
$H=\mathbf H(\R)^\circ$. Fix $a\in G$, $h\in H$, and a nonconstant
$\Psi\in\mathbf H(\R[t])$ with $\Psi(0)=\groupid$ and $\deg\Psi\le d$.
Write $J=\overline{\langle\Psi(\R)\rangle}$ and
$Y=aH\Gamma/\Gamma$, and let $\mu_Y$ be the $aHa^{-1}$-invariant
probability on $Y$. We call $(\mathbf H,a,\Psi,h)$ the
\emph{prescribed-support data}. The closed orbit $Y$ contains the
trajectory; the estimate below will identify its actual closure.

For $(J,h)$ of intrinsic type $(c_0,\tau)$, we establish
\begin{equation}\label{eq:prescribedestimate}
 \left|\frac1T\int_0^T f(a\Psi(t)h\Gamma)\dd t-\mu_Y(f)\right|
 \le C\norm f_{C_G^s}T^{-\vartheta_{\mathbf H,d}/(\tau+1)}
 \qquad(T\ge2).
\end{equation}
Here $s=s(\mathbf H,d)\ge1$ and $\vartheta_{\mathbf H,d}>0$;
for each compact $\Omega\subset X$, the bound holds for
$f\in C_c^\infty(X)$ supported in $\Omega$, with
$C=C(\mathbf H,a,\Psi,h,c_0,\tau,\Omega,
\mathscr L_{\mathbf H})>0$ independent of $f,T$.

\begin{proof}[Proof of Theorem~\ref{thm:intrinsic}]
Assume $(J,h)$ has intrinsic type $(c_0,\tau)$ for the data above.
For the main theorem, take $(a,\Psi,h)=(g,\Psi_g,\groupid)$.
Lemma~\ref{lem:orbitnormalization} supplies the structural
hypotheses on $\mathbf H_g$, without any arithmetic assumption
on $\Phi$ or $x$.

Write $f_a(y)=f(ay)$. Proposition~\ref{prop:supportclassification}
reduces the support to the semisimple, unipotent, and affine cases.
Proposition~\ref{prop:affineintrinsic} treats every affine tangent type.
It remains to treat the semisimple and unipotent cases. All coefficient constants below
may be decreased to at most one.

\medskip\noindent\emph{Unipotent supports.}
If $\mathbf H$ is unipotent, Lemma~\ref{lem:rationalnilcoordinates}
provides rational adapted coordinates of degree at most $2d$ on
the compact quotient, with a basis as in
Lemma~\ref{app:malcev-existence}. Let
$\chi_{\mathrm{ab}}:H\to\R^r$ be the horizontal coordinate
homomorphism of Lemma~\ref{app:horizontal}. It is a rational linear
function of $\log b=(b-I_3)-(b-I_3)^2/2$, since projection to the
abelianization annihilates every commutator term. For $0\ne k\in\Z^r$,
the function $p_k=k\cdot\chi_{\mathrm{ab}}$ is therefore a nonzero
element of $\mathscr P_{\mathbf H}$. A fixed integer $Q\ge1$
makes $Qp_k\in\mathscr L_{\mathbf H}$ and
$\norm{p_k}\le C\norm k$. Additivity gives
\[
 p_k(b\Psi(t)h)-p_k(bh)=p_k(\Psi(t)).
\]
Thus \eqref{eq:intrinsiccoefficientbound} is precisely a lower
bound $c\norm k^{-\tau}$ for the horizontal coefficient norm,
up to fixed factors. Proposition~\ref{prop:separatedunipotent},
with coordinate-degree bound $D=2d$, gives exponent
$1/[a_{2d}(\tau+1)]$.

\medskip\noindent\emph{The ambient quotient.}
If $\mathbf H=\SL_3$, apply Lemma~\ref{lem:intrinsictransfer}
to the three representations in \eqref{eq:intermediatefamilies}.
Their matrix-entry degree bounds are $1,2,2$ by
Lemma~\ref{lem:degreecutoff}, and none has a nonzero $H$-fixed
vector by Lemma~\ref{lem:allwedges}. Thus the degree-two intrinsic
condition gives all ambient coefficient bounds with exponent $\tau$.
In the regular case, the higher exterior-power lower bounds follow
from a uniformly thick point and the compression estimates in
Proposition~\ref{prop:separatedregular}; they require no further
arithmetic tests. Propositions~\ref{prop:separatedsingular}
and~\ref{prop:separatedregular} give rate $T^{-\delta_d/(\tau+1)}$.

\medskip\noindent\emph{Proper semisimple supports.}
Proposition~\ref{prop:supportclassification} gives a rational block group
or $\SO(\qform)$. If $\qform$ is definite, polynomial curves are constant,
so this case is excluded by the hypothesis.

Suppose first that $\qform$ is indefinite over $\mathbb R$ and
anisotropic over $\mathbb Q$ (Section~\ref{app:quadratic}).
Lemma~\ref{lem:semisimpleparametrizations} gives a compact quotient and a
real symmetric-square parametrization $\iota$.
Lemma~\ref{lem:polynomiallift}, with $K=F=\mathbb R$,
gives $\Psi=\iota\circ\psi$, $\deg\psi\le d$, and $h=\iota(g)$.
Put $\Delta=\iota^{-1}(\Gamma)$.
This is a discrete subgroup, and
$G_2/\Delta\simeq H/(H\cap\Gamma)$ is compact.
The identification preserves Haar probability.
For $F_0(b\Delta)=f(a\iota(b)\Gamma)$,
Proposition~\ref{prop:compactpolynomial} gives rate
$T^{-\delta_{\Delta,d}^{\mathrm c}}$, uniformly in $h$,
without using the Diophantine condition.
Its injectivity-radius constant is controlled by
\eqref{eq:anisotropicsystole}, Lemma~\ref{lem:injectivity},
and the fixed local coordinate comparison under $\iota$.
Since $\tau\ge0$, this also gives the required weaker rate with
exponent $\delta_{\Delta,d}^{\mathrm c}/(\tau+1)$.

The other proper semisimple supports are split over $\mathbb Q$,
by Proposition~\ref{prop:supportclassification} and
Lemma~\ref{lem:semisimpleparametrizations}.
Fix their rational map $\iota$ and the congruence subgroup
$\Delta=\Gamma_2(N)$ of
Lemma~\ref{lem:semisimplecongruencecover}. Lemma~\ref{lem:polynomiallift},
applied over $\R$, gives $\Psi=\iota\circ\psi$ with $\deg\psi\le d$ and
$h=\iota(g)$ for $g\in G_2$.
For a block embedding, projection to the two-by-two block after
the fixed rational change of basis gives a representation of $H$
with degree-one matrix entries and no fixed vector. Lemma~
\ref{lem:intrinsictransfer} supplies coefficient separation for
$\psi(t)g$ with exponent $\tau$.

For the symmetric-square embedding, the standard two-dimensional
representation does not descend to $H$, since the kernel element
$-I_2$ acts nontrivially on it (Section~\ref{app:quadratic}).
Use its three-dimensional symmetric square, which does descend and
whose entries are rational linear functions of the ambient matrix.
It has no fixed vector: diagonal invariance leaves only a multiple
of $E_{12}+E_{21}$, and upper-unipotent invariance forces its
coefficient to vanish. Lemma~\ref{lem:intrinsictransfer} gives its coefficient
separation with exponent $\tau$ in the basis
$E_{11},E_{22},E_{12}+E_{21}$ of symmetric matrices.

Lemma~\ref{lem:quadratictransfer} now gives the standard
two-dimensional coefficient bound with exponent $2\tau+1$.
No arithmetic hypothesis on the real lift $g$ is needed.

Apply Proposition~\ref{prop:congruencemodular} to
$F_0(b\Delta)=f(a\iota(b)\Gamma)$. By
Lemma~\ref{app:lattice-covers}, the covering pushes Haar
probability to the required probability on the support.
The exponent is $\delta_{\Delta,d}/(\tau+1)$ for a block
embedding and $\delta_{\Delta,d}/[2(\tau+1)]$ for a principal
embedding, by \eqref{eq:quadratictransfer}.

\medskip\noindent\emph{Exponents and derivative norms.}
For the semisimple and unipotent cases, one may therefore choose
\[
 \vartheta_{\mathbf H,d}=
 \begin{cases}
 a_{2d}^{-1},&\mathbf H\text{ unipotent},\\
 \delta_d,&\mathbf H=\SL_3,\\
 \delta_{\Delta,d},&\mathbf H\text{ a rational block group},\\
 \delta_{\Delta,d}/2,&\mathbf H\text{ a split principal group},\\
 \delta_{\Delta,d}^{\mathrm c},&
       \mathbf H=\SO(\qform),\ \qform\text{ indefinite and }\mathbb Q\text{-anisotropic}.
 \end{cases}
\]
All coordinates, parametrizations, and induced lattices are fixed
once $\mathbf H$ is fixed.

To compare norms, derivatives of each pullback are fixed linear
combinations of the $D_{Z_i}$. This also applies after normalizing
the tangent by $b_0$, including to $F_0(b_0b\Delta)$ on either
rank-one quotient.
In the ambient case the pullback is $f_{ab_0}$, with support in
$(ab_0)^{-1}\Omega$; \eqref{eq:localsobolev} bounds its Sobolev
norms by supremum derivatives. Take $s$ to be the maximum of
$s_*$ in \eqref{eq:localsobolev} and the derivative orders in
the model estimates invoked above. Then every test norm is at most
$C_{\Omega,a,\Psi,h,\mathbf H}\norm f_{C_G^s}$, independently
of $\tau$. The coefficient bounds above use only $c_0,\tau$,
finite-dimensional operator norms and positive singular values.
Together with the displayed model estimates this proves
\eqref{eq:prescribedestimate} with the stated dependencies.

\medskip\noindent\emph{Identification of the support.}
The estimate also identifies a prescribed support: if
$\overline{\{\Psi(t)h\Gamma:t\ge0\}}$ were proper in
$H\Gamma/\Gamma$, a nonnegative compactly supported smooth
function could vanish on its translate by $a$ and have positive
integral on $Y$. Its trajectory averages would vanish,
contradicting \eqref{eq:prescribedestimate} as $T\to\infty$.
Thus the polynomial trajectory, and hence $Jh\Gamma/\Gamma$,
is dense in the prescribed support whenever intrinsic type holds.

Finally, substituting $(a,\Psi,h)=(g,\Psi_g,\groupid)$ gives
$a\Psi(t)h\Gamma=\Phi(t)x$ and exactly the measure in
\eqref{eq:actualsupport}. Fixed conjugation preserves the degree
bound. This gives \eqref{eq:intrinsicmain}, with no loss in the
exponent or derivative order from the support normalization.
\end{proof}

\subsection{Necessity of finite intrinsic type}
\label{subsec:intrinsicnecessity}

\begin{proof}[Proof of Proposition~\ref{prop:intrinsicnecessity}]
Fix $m\ge1$ and write $\mathbf H=\mathbf H_g$, $H=H_g$,
$\Lambda=H\cap\Gamma$,
$\Psi=\Psi_g$, and
$\mathcal Y=H/\Lambda$, with Haar probability $\mu$.
The map $b\Lambda\mapsto gb\Gamma$ identifies $\mathcal Y$ with
the closed embedded orbit $Y$. On any fixed compact part of
$\mathcal Y$, smooth functions admit a fixed linear extension to
$C_c^\infty(X)$, supported in one compact neighborhood, with
$C_G^s$ norm bounded by a constant times their $C_H^s$ norm.
For a fixed basis $W_1,\ldots,W_n$ of $\Lie H$, this norm is
\[
 \norm F_{C_H^s}
 =\sum_{j=0}^s\ \sum_{i_1,\ldots,i_j=1}^n
       \norm{D_{W_{i_1}}\cdots D_{W_{i_j}}F}_{\infty,\mathcal Y}.
\]
The term $j=0$ is $\norm F_{\infty,\mathcal Y}$.
For the extension, choose finitely many charts straightening the
embedded orbit, extend a partition of unity and each localized
function constantly in the normal coordinates, and multiply by
fixed smooth cutoff functions there. Coordinate and Lie-derivative norms are
equivalent on these compact sets. Thus \eqref{eq:conversepower}
holds on $\mathcal Y$ for tests in any one such compact set,
with the same $s,\delta$.

Put $W=\R1\oplus\mathscr P_{\mathbf H}^{(m)}$, with the constant
function of norm one orthogonal to $\mathscr P_{\mathbf H}^{(m)}$,
and let $R(b)p(z)=p(zb)$. Its entries have degree at most $m$.
The integer span $M$ of the matrix-entry monomials of degree at
most $m$, viewed as functions on $H$, is a full lattice in $W$,
preserved by $R(\Lambda)$: both a matrix in $\Lambda$ and its
inverse have integral entries.
This is the invariant monomial lattice of
Lemma~\ref{lem:polynomial-functions}.
Clearing a fixed denominator, it suffices to prove
\eqref{eq:converseintrinsic} for nonzero $p$ in a fixed integer
multiple of $\mathscr L_{\mathbf H}^{(m)}$ contained in $M$.
For such $p$ write
\[
 P=\norm p\ge1,\qquad
 \Delta_p=\operatorname{dist}(p,\mathscr K_{J_g,\groupid}^{(m)}).
\]
Subtracting the orthogonal projection onto this invariant subspace
and using the degree of $R(\Psi(t))$ gives
\begin{equation}\label{eq:nearfixedduration}
 \sup_{0\le t\le T}\norm{R(\Psi(t))p-p}
 \le C_\Psi T^{md}\Delta_p\qquad(T\ge2).
\end{equation}

\emph{A fixed polynomial chart.}
For the supports classified in Proposition~\ref{prop:supportclassification}, Lemma~\ref{app:nilpotent-charts}
gives a basis $N_1,\ldots,N_n$ of $\Lie H$ consisting of
nilpotent matrices.
For $\mathfrak{sl}_2$ and $\mathfrak{sl}_3$, use the off-diagonal
matrix units and the additional matrices
$E_{ii}-E_{jj}+E_{ij}-E_{ji}$ for adjacent indices.
The block and principal embeddings preserve nilpotence; the
unipotent and affine radicals add nilpotent directions.
The definite orthogonal case has no nonconstant polynomial curve
and does not occur. Hence, on a sufficiently small box
$B=[-a,a]^n$, the map
\[
 b(z)=\exp(z_1N_1)\cdots\exp(z_nN_n)
\]
is a polynomial coordinate chart whose projection to $\mathcal Y$
is injective. Its Haar density is bounded above and below.
Choose $0\le\chi\in C_c^\infty(\mathcal Y)$ supported in the
interior of this chart, with $\mu(\chi)>0$.

For every fixed $C_1$, compactness gives a constant $c_1>0$ such that
\begin{equation}\label{eq:chartseparation}
 \sup_{z\in B}\norm{R(b(z))v-p}\ge c_1P
 \quad(p\in\mathscr P_{\mathbf H}^{(m)},\ \norm p=P>0,
                                      \norm v\le C_1P).
\end{equation}
Indeed, divide by $P$ and consider the compact parameter set
$\norm p=1$, $p\in\mathscr P_{\mathbf H}^{(m)}$, $\norm v\le C_1$.
If the supremum vanished, evaluation at $z=0$ would give $v=p$;
differentiation in every $N_i$ would make $p$ right-$H$-invariant.
Such a function is constant, and $p(\groupid)=0$ would force $p=0$.
Since each $N_i^3=0$, the vector polynomial in
\eqref{eq:chartseparation} has degree at most $D=2mn$.
Lemma~\ref{lem:polynomialgoodness} therefore gives, for
$\beta=1/(nD)>0$,
\begin{equation}\label{eq:chartsmalltarget}
 \Leb\{z\in B:\norm{R(b(z))v-p}\le\varepsilon\}
 \ll (\varepsilon/P)^\beta
 \qquad(0<\varepsilon\le1,\ \norm v\le C_1P).
\end{equation}
The non-strict inequality follows by increasing the threshold by
an arbitrarily small amount; larger thresholds are covered by the
trivial bound.

\emph{Smooth tests for an almost fixed lattice vector.}
Fix $\varphi\in C_c^\infty(W)$ with $0\le\varphi\le1$, equal
to one on the ball of radius $1/2$ and supported in the unit ball.
For $0<\varepsilon\le\varepsilon_0$, define
\begin{equation}\label{eq:latticedetectortest}
 F_{p,\varepsilon}(b\Lambda)
 =\chi(b\Lambda)\sum_{v\in M}
       \varphi\left(\frac{R(b)v-p}{\varepsilon}\right).
\end{equation}
The sum is locally finite and is independent of the lift $b$,
since $R(\Lambda)M=M$. On the fixed chart, the lattices $R(b)M$
have shortest-vector length bounded below. Choose
$\varepsilon_0$ so small that at most one summand is nonzero on
$\operatorname{supp}\chi$. Every relevant $v$ then has
$\norm v\le C_1P$ for a fixed $C_1$, by boundedness of $R(b)^{-1}$
on the chart. If $r=\dim W$, there are $O(P^r)$ such lattice vectors.
Equation~\eqref{eq:chartsmalltarget} and differentiation of
\eqref{eq:latticedetectortest} give
\begin{equation}\label{eq:detectormeanandnorm}
 0\le\mu(F_{p,\varepsilon})\ll P^r(\varepsilon/P)^\beta,
 \qquad
 \norm{F_{p,\varepsilon}}_{C_H^s}\ll(P/\varepsilon)^s.
\end{equation}
For the derivative bound, each derivative of $R(b)v$ is $O_s(P)$
on the chart. The separation of lattice vectors ensures that at most
one term contributes at each point, so no multiplicity factor occurs.

Choose a fixed $B_0>0$ with $\beta(B_0+1)>r$, and put
$\varepsilon=\varepsilon_*P^{-B_0}$. Decreasing the fixed
$\varepsilon_*>0$ gives
\[
 \mu(F_{p,\varepsilon})\le\tfrac14\mu(\chi),\qquad
 \norm{F_{p,\varepsilon}}_{C_H^s}\ll P^{s(B_0+1)}
\]
for all these $p$. If
$C_\Psi T^{md}\Delta_p\le\varepsilon/2$, then
\eqref{eq:nearfixedduration} and the summand $v=p$ show that
$F_{p,\varepsilon}(\Psi(t)\Lambda)=\chi(\Psi(t)\Lambda)$
for every $0\le t\le T$. At points in the support, all other
summands vanish; outside it both functions are zero.
Applying \eqref{eq:conversepower} to the fixed extensions of
$F_{p,\varepsilon}$ and $\chi$ would therefore give
\[
 \tfrac34\mu(\chi)
 \le\mu(\chi)-\mu(F_{p,\varepsilon})
 \ll P^{s(B_0+1)}T^{-\delta}.
\]
Take $T=K P^{s(B_0+1)/\delta}$ with fixed $K\ge2$ sufficiently
large. This is impossible, so
\[
 \Delta_p>\frac{\varepsilon}{2C_\Psi T^{md}}
 \gg P^{-B_0-md\,s(B_0+1)/\delta}.
\]
Put $\tau_m=B_0+md\,s(B_0+1)/\delta$. If the earlier lattice
multiple is $Q$, apply the estimate to $Qp$ for
$0\ne p\in\mathscr L_{\mathbf H}^{(m)}$ and use
$\Delta_{Qp}=Q\Delta_p$ and $\|Qp\|=Q\|p\|$.
The resulting constant is multiplied by $Q^{-(\tau_m+1)}$.
This proves \eqref{eq:converseintrinsic}.
For $m=2$ it gives necessity; sufficiency is
Theorem~\ref{thm:intrinsic}.
\end{proof}

\iffalse
The converse concerns a bound in one fixed finite-order derivative
norm, uniform over test functions. Its proof allows the test function to
depend on $p$ and $T$, while keeping its support in a single compact
set. It gives no sharp value of $\tau_m$ and does not assert
necessity for qualitative equidistribution.
\fi

\section{Applications and arithmetic limits}
\label{sec:applications}

This section emphasizes two consequences of the main theorem:
power estimates for polynomial curves generating the full affine
or ambient group, and the arithmetic characterization of a power
rate. The classical toral calculation tests the intrinsic condition
and the degree cutoff. Effective estimates on unipotent quotients
and compact rank-one quotients come from established results
\cite[Theorem 2.9]{GreenTao2012} and
\cite[Appendix B, Theorem B.1]{Zheng2015}; the common intrinsic
condition and the regular affine argument connect these estimates to
the general polynomial theorem. None of the applications is used
in its proof.

\subsection{Polynomial curves generating their full support}

\begin{corollary}[Equidistribution under full support generation]
\label{cor:intrinsicfullgeneration}
Use the notation of Theorem~\ref{thm:intrinsic}, without initially
assuming intrinsic type. If $J_g=H_g$, the intrinsic condition
holds with $(c_0,\tau)=(1,0)$ and
\eqref{eq:intrinsicmain} has rate $T^{-\vartheta_{\mathbf H_g,d}}$.
For the prescribed-support data of
Section~\ref{subsec:mainassembly}, $J=H$ likewise gives
\eqref{eq:prescribedestimate} with $\tau=0$ for every $h\in H$.
\end{corollary}
\begin{proof}
For every $h\in H$, right invariance under $h^{-1}Hh=H$
makes each function in $\mathscr K_{H,h}$ constant. The
normalization at $\groupid$ then gives $\mathscr K_{H,h}=\{0\}$.
This elementary invariant-function argument is recorded after
Lemma~\ref{lem:polynomial-functions}.
Every nonzero element of the chosen orthonormal integer lattice
has norm at least one, so the intrinsic condition holds with $(c_0,\tau)=(1,0)$.
The main theorem and its prescribed-support estimate give the
conclusions.
\end{proof}

The following examples give explicit curves to which this
corollary applies at every real starting point on the support.
For a normalized curve $\Psi$ and
$J=\overline{\langle\Psi(\R)\rangle}$, differentiating
$\Psi(t+s)\Psi(t)^{-1}\in J$ at $s=0$ shows that
$\Psi'(t)\Psi(t)^{-1}$ belongs to $\Lie J$. Polynomial interpolation
therefore implies that all its coefficient matrices belong to $\Lie J$.
It suffices to check that these coefficients generate $\Lie H$;
equality of the Lie algebras gives $J=H$ because $H$ is connected
(Lemma~\ref{app:lie-tests}).

\subsubsection{A regular affine curve}
\label{subsec:regularaffineexample}

Let $\mathbf H=\mathbf A$ be the affine group in
\eqref{eq:affinegroups}, and put $H=A_{\mathrm{aff}}$.
The orbit $H\Gamma/\Gamma$ is closed: $H$ stabilizes the row
covector $(0,0,1)$, whose $\Gamma$-orbit is discrete.
Moreover
\[
 H\cap\Gamma=\SL_2(\Z)\ltimes\Z^2,
\]
and Lemma~\ref{lem:affinequotient} with $N=1$ gives finite
invariant volume.

Consider the degree-six curve
\[
 \Psi_{\mathrm r}(t)=
 \begin{pmatrix}
 1+t^2&t&t^4/12-t^6/20\\
 t&1&t^3/3-t^5/20\\
 0&0&1
 \end{pmatrix}.
\]
Its determinant is one, it equals $\groupid$ at zero, and
\begin{equation}\label{eq:regularaffineexamplelog}
 \Psi_{\mathrm r}'\Psi_{\mathrm r}^{-1}
   =(E_{12}+E_{21})+t(E_{11}-E_{22})
          +t^2(-E_{12}+E_{23}).
\end{equation}
The leading tangent $N=-E_{12}+E_{23}$ satisfies
$N^2=-E_{13}\ne0$ and $N^3=0$, so it is regular.
Bracketing $E_{12}+E_{21}$ with $E_{11}-E_{22}$ separates
$E_{12}$ and $E_{21}$; the last coefficient supplies $E_{23}$,
and $[E_{12},E_{23}]=E_{13}$. The coefficients thus generate
the five-dimensional affine Lie algebra, giving $J=H$.

In particular, for every $a\in G$, $h\in H$, and compact
$\Omega\subset X$, there is $C=C(a,h,\Omega)>0$ such that
\begin{equation}\label{eq:regularaffineapplication}
 \left|\frac1T\int_0^T f(a\Psi_{\mathrm r}(t)h\Gamma)\dd t
           -\mu_{aH\Gamma/\Gamma}(f)\right|
 \le C\norm f_{C_G^s}T^{-\vartheta_{\mathbf A,6}}
 \qquad(T\ge2)
\end{equation}
for all $f\in C_c^\infty(X)$ supported in $\Omega$.
Here $s=s(\mathbf A,6)$, and
$\mu_{aH\Gamma/\Gamma}$ is the $aHa^{-1}$-invariant probability.
Full generation verifies the intrinsic condition even when
$a$ and $h$ have arbitrary real entries.

This curve generates a five-dimensional group and is therefore
not contained in any one-parameter subgroup.
The conclusion concerns a fixed polynomial curve and arbitrary
specified starting points. The related regular affine results
\cite[Theorems 1.1 and 1.2]{BV2016} and
\cite[Theorem 2]{Vinogradov2016} concern expanding translates
and nonlinear horocycle lifts, respectively, as discussed in the
introduction. The additional estimate used here is
Proposition~\ref{prop:regularaffineflow}, whose uniformity allows
the moving initial points arising from polynomial smoothing.

\subsubsection{A degree-three curve generating the ambient group}
\label{sec:regularexample}

Put $M=E_{31}$ and $N_0=E_{12}+E_{23}$. Consider
\begin{equation}\label{eq:regularexample}
 \Phi_{\mathrm{reg}}(t)=(I_3+tM)\exp(tN_0)
 =\begin{pmatrix}
  1&t&t^2/2\\0&1&t\\t&t^2&1+t^3/2
 \end{pmatrix}.
\end{equation}
It has degree three, determinant one, and value $\groupid$ at zero.
For $A=N_0+M$ and $C=E_{32}-E_{21}$, the identities
$M^2=0$ and $MN_0M=0$ give
\begin{equation}\label{eq:regularexamplederivative}
 \Phi_{\mathrm{reg}}'\Phi_{\mathrm{reg}}^{-1}
  =M+(I_3+tM)N_0(I_3-tM)=A+tC.
\end{equation}
Since $C^2=-E_{31}\ne0$ and $C^3=0$, the leading tangent is
regular. To verify full generation, compute
\begin{align*}
 D:=[C,A]&=\diag(1,-2,1),\\
 [D,[D,A]]&=9N_0,\\
 [D,N_0]&=3(E_{12}-E_{23}).
\end{align*}
Thus $E_{12},E_{23}$ and $M=A-N_0=E_{31}$ belong to $\Lie J$.
Their brackets give $E_{13},E_{21},E_{32}$ and the two diagonal
directions, so $\Lie J=\mathfrak{sl}_3(\R)$ and $J=G$.

For every $x\in X$ and compact $\Omega\subset X$, the corollary
therefore gives $C=C(x,\Omega)>0$ such that
\begin{equation}\label{eq:ambientapplication}
 \left|\frac1T\int_0^T f(\Phi_{\mathrm{reg}}(t)x)\dd t-m_X(f)\right|
 \le C\norm f_{C_G^s}T^{-\vartheta_{\SL_3,3}}
 \qquad(T\ge2)
\end{equation}
for all $f\in C_c^\infty(X)$ supported in $\Omega$.
Here $s=s(\SL_3,3)$ and $m_X$ is Haar probability on $X$.
Conjugation by any lift of $x$ preserves full generation,
so the conclusion applies to every starting lattice without
an algebraicity or Diophantine assumption on that lattice.

\subsubsection{The two singular affine tangents}
\label{subsec:affineexample}

For comparison with the regular affine application, the same
support admits the two degree-three curves
\[
 \begin{aligned}
 \Psi_{\mathrm{tr}}(t)&=
       \exp(t^3E_{13})\exp(tE_{12})\exp(tE_{21}),\\
 \Psi_{\mathrm b}(t)&=
       \exp(tE_{13})\exp(t^2E_{12})\exp(tE_{21}).
 \end{aligned}
\]
Both equal $\groupid$ at zero. Their logarithmic derivatives are
\[
 \begin{aligned}
 \Psi_{\mathrm{tr}}'\Psi_{\mathrm{tr}}^{-1}
  &=(E_{12}+E_{21})+t(E_{11}-E_{22})
       +t^2(-E_{12}+3E_{13})-t^3E_{23}-t^4E_{13},\\
 \Psi_{\mathrm b}'\Psi_{\mathrm b}^{-1}
  &=(E_{21}+E_{13})+t(2E_{12}-E_{23})
       +t^2(E_{11}-E_{22})-t^3E_{13}-t^4E_{12}.
 \end{aligned}
\]
In each line the coefficient matrices span the full affine Lie
algebra, so $J=H$ and type $(1,0)$ holds for every $h\in H$.
The leading tangents are $-E_{13}$ and $-E_{12}$, respectively.
Both have square zero; the first lies in the unipotent radical
and the second lies outside it. Thus the same intrinsic
criterion gives \eqref{eq:prescribedestimate} with rate
$T^{-\vartheta_{\mathbf A,3}}$ for both remaining affine tangent cases,
on every real translate and at every specified starting point.

\subsection{An algebraic criterion for intrinsic type}

\begin{corollary}[Equidistribution for algebraic normalized data]\label{cor:intrinsicalgebraic}
Let $\Phi,x,g,\Psi_g,J_g,\mathbf H_g$ be as in
Theorem~\ref{thm:intrinsic}, without initially assuming intrinsic type.
If $F\subset\R$ is a number field of degree $e$ and
$\Psi_g\in\mathbf H_g(F[t])$, then $(J_g,\groupid)$ has intrinsic type
with $\tau=e-1$ and some
$c_0=c_0(F,\Psi_g,\mathscr L_{\mathbf H_g})>0$. Consequently
\eqref{eq:intrinsicmain} has rate $T^{-\vartheta_{\mathbf H_g,d}/e}$.
In particular, this applies when the entries of the coefficients
of $\Phi$ and of a lift $g$ lie in $F$.
\end{corollary}

The algebraic criterion uses the following two elementary facts.
They separate the qualitative consequence of density from the
number-field estimate that supplies a power bound.

\begin{lemma}
\label{lem:intrinsicdensity}
Let $\mathbf H\le\SL_3$ be a connected $\mathbb Q$-group,
$H=\mathbf H(\R)^\circ$, $J\le H$ a closed subgroup, and $h\in H$.
Assume $H\Gamma/\Gamma$ is closed.
If $\overline{Jh\Gamma/\Gamma}=H\Gamma/\Gamma$, then
$\mathscr K_{J,h}\cap\mathscr P_{\mathbf H}(\mathbb Q)=\{0\}$,
where rationality refers to the chosen rational basis.
\end{lemma}
\begin{proof}
Let $\Lambda=H\cap\Gamma$. The integer span $M$ of the restrictions
of all matrix-entry monomials of degree at most $2$ is a full lattice
in $\mathscr W_{\mathbf H}$: these finitely many vectors span the
space and have rational coordinates. Integral matrices and their
integral inverses preserve this monomial span under right translation,
so $R(\Lambda)$ preserves $M$.
Lemma~\ref{lem:polynomial-functions} supplies this rational
polynomial-lattice construction.

If $p\in\mathscr K_{J,h}$ is rational, multiply it by an integer
to place it in $M$. The kernel identity in
Lemma~\ref{lem:intrinsiccoefficients} says $R(h)p\in\mathscr W_{\mathbf H}^J$.
Put $v=R(h)p$ and consider
\[
 Z_v=\{b\Lambda:R(b^{-1})v\in M\}.
\]
This is well defined because $R(\Lambda)$ preserves $M$.
It is closed: on each neighborhood admitting a continuous lift
to $H$, it is the inverse image of the closed lattice $M$
under the continuous map $b\mapsto R(b^{-1})v$.
For $b=jh$, with $j\in J$, one has
$R(b^{-1})v=R(h^{-1})R(j^{-1})R(h)p=p\in M$.
Thus it contains $Jh\Lambda/\Lambda$.
Density gives $Z_v=H/\Lambda$. Thus the continuous map
$b\mapsto R(b^{-1})v$ sends the connected group $H$ into the
discrete lattice $M$ and is constant. Hence $v$ is $H$-fixed,
and $p=R(h^{-1})v=v$ is right-$H$-invariant. It is therefore
constant; the normalization $p(\groupid)=0$ makes it zero.
\end{proof}

\begin{lemma}
\label{lem:algebraicseparation}
Let $A$ be a real matrix whose entries lie in a number field
$K\subset\R$ of degree $e$, and suppose
$\ker A\cap\mathbb Q^n=\{0\}$. Then
\[
 \norm{Aw}\ge c_A\norm w^{-(e-1)}\qquad(0\ne w\in\Z^n)
\]
for an explicit $c_A>0$ determined by denominators and conjugates
of the entries.
\end{lemma}
\begin{proof}
Choose an integer $Q\ge1$ so that $QA$ has algebraic-integer entries.
For each embedding $\sigma:K\to\mathbb C$ other than the specified
real embedding, let
$M_\sigma=\max\{1,\max_i\norm{\sigma((QA)_{i,*})}\}$.
For $w\ne0$, some coordinate $\ell_i(w)$ of $Aw$ is nonzero.
The algebraic integer $Q\ell_i(w)$ has nonzero integer field norm,
of absolute value at least one. All its other conjugates have
absolute value at most $M_\sigma\norm w$. Consequently one may take
\[
 c_A=\left(Q\prod_{\sigma\ne\mathrm{id}}M_\sigma\right)^{-1}.
\]
Use a common field and the minimum constant for a finite family.
\end{proof}

\begin{proof}[Proof of Corollary~\ref{cor:intrinsicalgebraic}]
Lemma~\ref{lem:orbitnormalization} supplies density in $H_g\Gamma/\Gamma$.
In the algebraic case, Lemma~\ref{lem:intrinsicdensity}, with $h=\groupid$,
therefore excludes nonzero rational vectors in
$\ker\mathscr C_{\Psi_g,\groupid}$.
Lemma~\ref{lem:intrinsiccoefficients} places its matrix entries
in $F$, without passing to a field containing a modular lift.
Lemma~\ref{lem:algebraicseparation} gives the coefficient bound
with exponent $e-1$. The upper singular-value bound in
Lemma~\ref{lem:intrinsiccoefficients} converts it into
\eqref{eq:intrinsicdiophantine}, with a positive constant
reduced to at most one. The main theorem gives the rate.
\end{proof}

\subsection{Necessity and polynomial test degree}
\label{subsec:degreecutoffs}

\subsubsection{Necessity of finite intrinsic type: a toral example}
\label{subsec:liouvilleobstruction}

The following example recalls the classical Diophantine obstruction
to effective equidistribution for linear flows on tori. For
background on toral examples and Diophantine error estimates, see
\cite[Sections~2.2--2.3]{Einsiedler2010}; related quantitative
estimates for discrete linear sequences on tori are given in
\cite[Section~3, Proposition~3.1]{GreenTao2012}.
We express the toral criterion in the matrix setting of this paper
and identify it with finite intrinsic type, illustrating
Proposition~\ref{prop:intrinsicnecessity}. The Fourier calculation
below makes explicit why Liouville slopes admit no power estimate
uniform over smooth tests in a fixed derivative norm of finite order.

Let $\mathbf U\le\SL_3$ be the $\mathbb Q$-algebraic subgroup
whose real points are
\[
 U=\{I_3+uE_{12}+vE_{13}:u,v\in\R\}.
\]
The coordinates $(u,v)$ identify $U$ with the additive group
$\R^2$ and $\Lambda:=U\cap\Gamma$ with $\Z^2$. Thus
\[
 Y:=U\Gamma/\Gamma\simeq U/\Lambda\simeq\R^2/\Z^2
\]
is a compact embedded orbit: the natural map $U/\Lambda\to X$
is an injective immersion with compact domain, hence an embedding.
Write $\mu_Y$ for its Haar probability measure. For an irrational
$\alpha$, put
\[
 \Psi_\alpha(t)=I_3+t(E_{12}+\alpha E_{13}),\qquad
 J=\overline{\langle\Psi_\alpha(\R)\rangle}
   =\Psi_\alpha(\R).
\]
Here $J$ is the closed line $v=\alpha u$ in $U$.
Fix $h=I_3+u_0E_{12}+v_0E_{13}\in U$. The trajectory
$\Psi_\alpha(t)h\Gamma$ has torus coordinates
$(u_0+t,v_0+\alpha t)$.

For $(m,n)\in\Z^2\setminus\{0\}$, the character
$\chi_{m,n}(u,v)=e^{2\pi\mathrm i(mu+nv)}$ has nonzero
frequency $\omega_{m,n}=m+\alpha n$ along this trajectory, and
\[
 \frac1T\int_0^T\chi_{m,n}(\Psi_\alpha(t)h\Gamma)\dd t
 =e^{2\pi\mathrm i(mu_0+nv_0)}
   \frac{e^{2\pi\mathrm i T\omega_{m,n}}-1}
        {2\pi\mathrm i T\omega_{m,n}}
 \longrightarrow0.
\]
The constant character has average one. Uniform approximation
by trigonometric polynomials therefore proves convergence to
$\mu_Y$ for every continuous function on $Y$. Every nonempty
open subset of $Y$ supports a nonnegative continuous function
with positive $\mu_Y$ integral, so the convergence also shows
that the trajectory is dense in $Y$. In particular, the
curve-generated group $J$ is strictly smaller than its
orbit-closure group $U$.

\begin{proposition}[Diophantine criterion for the linear torus flow]
\label{prop:toralintrinsictype}
For the above irrational slope $\alpha$ and any $h\in U$, the
following conditions are equivalent:
\begin{enumerate}
\item There exist $c>0$ and $\tau\ge0$ such that
\[
 |m+\alpha n|\ge c(m^2+n^2)^{-\tau/2}
 \qquad((m,n)\in\Z^2\setminus\{0\}).
\]
\item The pair $(J,h)$ has finite intrinsic type in
$\mathscr P_{\mathbf U}$.
\item There exist an integer $s\ge1$ and $\delta>0$ such that,
for every compact $\Omega\subset X$, some $C_\Omega>0$ satisfies
\[
 \left|\frac1T\int_0^T f(\Psi_\alpha(t)h\Gamma)\dd t
                    -\mu_Y(f)\right|
 \le C_\Omega\norm f_{C_G^s}T^{-\delta}\qquad(T\ge2)
\]
for all $f\in C_c^\infty(X)$ supported in $\Omega$.
\end{enumerate}
Conditions (1) and (2) hold with the same admissible exponent
$\tau$, although their positive constants may differ.
\end{proposition}

\begin{proof}
\emph{Equivalence of (1) and (2).}
In the coordinates $(u,v)$, restriction of matrix-entry
polynomials gives
\[
 \mathscr P_{\mathbf U}
 =\operatorname{span}_{\R}\{u,v,u^2,uv,v^2\}.
\]
Equip this space with the inner product for which the displayed
monomials are orthonormal, and take their integer span as
$\mathscr L_{\mathbf U}$. This lattice consists of the polynomials
with integral coefficients and zero constant term. This choice does not change
the admissible intrinsic exponents, by
Lemma~\ref{lem:intrinsicconjugacy}.
Since $U$ is abelian, $h^{-1}Jh=J$, so
$\mathscr K_{J,h}$ is independent of $h$. Put
\[
 D_\alpha=\partial_u+\alpha\partial_v:
 \mathscr P_{\mathbf U}\longrightarrow\mathscr W_{\mathbf U},
 \qquad
 \mathscr W_{\mathbf U}=\R1\oplus\mathscr P_{\mathbf U},
\]
with $1$ of norm one and orthogonal to
$\mathscr P_{\mathbf U}$; in particular, the constant coefficient
of $D_\alpha p$ is retained. Differentiating
$p(u+t,v+\alpha t)$ shows that right-$J$-invariance is
equivalent to $D_\alpha p=0$. In the linear coordinates
$(u,w)=(u,v-\alpha u)$, this differential operator is
$\partial_u$, with $w$ held fixed. Hence
\[
 \mathscr K_{J,h}=\ker D_\alpha
   =\operatorname{span}_{\R}
       \{v-\alpha u,(v-\alpha u)^2\}.
\]
The linear and quadratic homogeneous subspaces are orthogonal
in the chosen norm. Orthogonal projection onto the linear
kernel therefore gives
\[
 \operatorname{dist}(mu+nv,\mathscr K_{J,h})
 =\frac{|m+\alpha n|}{\sqrt{1+\alpha^2}}.
\]
Applying the intrinsic condition to $mu+nv$ proves (1) from
(2), with the same exponent.

Conversely, assume (1) and write a nonzero lattice polynomial as
$p=a_1u+a_2v+a_3u^2+a_4uv+a_5v^2$, with $a_i\in\Z$.
Then
\[
 D_\alpha p=(a_1+\alpha a_2)
             +(2a_3+\alpha a_4)u
             +(a_4+2\alpha a_5)v.
\]
At least one of the integer pairs
$(a_1,a_2)$, $(2a_3,a_4)$, $(a_4,2a_5)$ is nonzero, and
each has norm at most $2\norm p$. Condition (1), applied to
such a pair, gives
\[
 \norm{D_\alpha p}\ge c\,2^{-\tau}\norm p^{-\tau}.
\]
For every $k\in\ker D_\alpha$, one has
$D_\alpha p=D_\alpha(p-k)$. Taking the infimum of the
operator-norm bound over all such $k$ yields
\[
 \norm{D_\alpha p}
 \le\norm{D_\alpha}_{\mathrm{op}}
            \operatorname{dist}(p,\mathscr K_{J,h}).
\]
These two inequalities prove (2) with the same exponent,
after reducing the resulting positive constant to at most one.

\emph{From (2) to (3).}
The group $\mathbf U$ is connected, rational, and unipotent;
its quotient by $R_u(\mathbf U)$ is therefore trivial and hence
semisimple. Also, $Y$ is compact. Thus the prescribed-support estimate
\eqref{eq:prescribedestimate}, established in the proof of
Theorem~\ref{thm:intrinsic}, applies with
$(\mathbf H,a,\Psi,h)=(\mathbf U,\groupid,\Psi_\alpha,h)$
and degree $d=1$. It gives (3) with
$\delta=\vartheta_{\mathbf U,1}/(\tau+1)>0$.
For this compact unipotent support, the estimate comes from
quantitative polynomial equidistribution on compact nilmanifolds
\cite[Theorem 2.9]{GreenTao2012}, through its continuous-time
consequence in Lemma~\ref{lem:continuousleibman}.

\emph{From (3) to (1).}
For $(m,n)\in\Z^2\setminus\{0\}$, define the torus function
\[
 F_{m,n}(u,v)
 =\cos\bigl(2\pi[m(u-u_0)+n(v-v_0)]\bigr).
\]
It is well defined modulo $\Z^2$ and has Haar mean zero.
Along the trajectory it equals
$\cos(2\pi\omega_{m,n}t)$. At
$T_{m,n}=(4|\omega_{m,n}|)^{-1}$, direct integration gives
\[
 \frac1{T_{m,n}}\int_0^{T_{m,n}}
       F_{m,n}(\Psi_\alpha(t)h\Gamma)\dd t
 =\frac{\sin(2\pi\omega_{m,n}T_{m,n})}
        {2\pi\omega_{m,n}T_{m,n}}
 =\frac2\pi.
\]
Because $Y$ is compact and embedded, there is a fixed linear
smooth extension operator from $Y$ to $X$, with image supported
in one compact set $\Omega_0$. To construct it, take finitely
many charts straightening $Y$ and a partition of unity on $Y$,
extend each localized function constantly in the normal
coordinates, and multiply by fixed smooth cutoff functions. On these compact
charts, coordinate derivative norms and the Lie derivative
norms are comparable. This is the extension construction used
in the proof of Proposition~\ref{prop:intrinsicnecessity}.
In particular, it gives functions
$\widetilde F_{m,n}\in C_c^\infty(X)$ with
$\widetilde F_{m,n}|_Y=F_{m,n}$, all supported in $\Omega_0$,
and
\[
 \norm{\widetilde F_{m,n}}_{C_G^s}
 \le C_s(1+|m|+|n|)^s.
\]
Indeed, a derivative of order $j\le s$ of the cosine is
$O_j((1+|m|+|n|)^j)$, and the extension uses only fixed
charts and smooth cutoff functions. The constants are independent of $(m,n)$.

If $T_{m,n}\ge2$, apply (3) to this extension in the fixed
compact set $\Omega_0$. Since its mean on $Y$ is zero,
\[
 \frac2\pi
 \le C_{\Omega_0}C_s(1+|m|+|n|)^s
                 (4|m+\alpha n|)^\delta.
\]
Rearrangement gives
\[
 |m+\alpha n|\ge c_1(1+|m|+|n|)^{-s/\delta}.
\]
If $T_{m,n}<2$, then $|m+\alpha n|>1/8$, so the same
bound holds after decreasing $c_1$. Finally,
$1+|m|+|n|\le(1+\sqrt2)(m^2+n^2)^{1/2}$ for nonzero
integer pairs. This proves (1) with $\tau=s/\delta$.
\end{proof}

For completeness, we give the standard Fourier argument for the
Liouville obstruction in the derivative norm used here.
Assume here that $\alpha$ is Liouville. Fix any integer $s\ge1$
and any exponent $\delta>0$, and choose
$K>s/\delta$. There are integers $p_j$ and positive integers
$q_j\to\infty$ such that
\[
 0<\varepsilon_j:=|p_j-q_j\alpha|<q_j^{-K}.
\]
Set $M_j=1+|p_j|+q_j$ and $T_j=(4\varepsilon_j)^{-1}$.
Then $M_j\asymp_\alpha q_j$ and $T_j>q_j^K/4\to\infty$.
Using the cosine tests and their extensions constructed in the
proof, define
\[
 f_j=\frac{\widetilde F_{p_j,-q_j}}{C_sM_j^s}.
\]
All these smooth functions are supported in the same compact set
$\Omega_0\subset X$, and they satisfy
\[
 \norm{f_j}_{C_G^s}\le1,\qquad \mu_Y(f_j)=0.
\]
Their frequency along the trajectory is $p_j-\alpha q_j$.
The calculation above, evaluated at $T_j$, therefore gives
\[
 \left|\frac1{T_j}\int_0^{T_j}
          f_j(\Psi_\alpha(t)h\Gamma)\dd t-\mu_Y(f_j)\right|
 =\frac{2}{\pi C_sM_j^s}.
\]
Consequently, for a constant $c_{\alpha,s,\delta}>0$ independent
of $j$,
\[
 T_j^\delta
 \left|\frac1{T_j}\int_0^{T_j}
          f_j(\Psi_\alpha(t)h\Gamma)\dd t-\mu_Y(f_j)\right|
 \ge c_{\alpha,s,\delta}q_j^{K\delta-s}
 \longrightarrow\infty.
\]
If an estimate as in condition (3) held for this $s$ and $\delta$,
its left-hand discrepancy at $f=f_j$ and $T=T_j$ would instead
be at most $C_{\Omega_0}T_j^{-\delta}$ for all sufficiently
large $j$. The preceding divergence rules out every such
constant $C_{\Omega_0}$. Since $s$ and $\delta$ were arbitrary,
no power estimate in any fixed derivative norm of finite order
is possible, even for tests supported in this one compact set.

The qualitative convergence proved above concerns each fixed
continuous test function. The functions $f_j$ vary with $j$,
but they all lie in the unit ball of the same derivative norm;
a uniform power estimate would have to apply to all of them.
Thus the trajectory is equidistributed although its averaging
errors admit no such power bound. By the proposition, finite
intrinsic type also fails. For slopes of finite ordinary
Diophantine type, the power estimate holds without requiring
$\alpha$ to be algebraic. The next subsection uses the same
Fourier obstruction in a different matrix embedding to show
why the common degree cutoff for the intrinsic condition must be $2$.

\subsubsection{Optimality of the degree cutoff}
\label{subsec:degreeonecounterexample}

Theorem~\ref{thm:intrinsic} gives a power estimate under the
intrinsic condition with degree cutoff $2$ on every
support. The following example shows that degree one does not
suffice as a common degree cutoff: for a Liouville slope, the degree-one
intrinsic condition holds, but no power estimate uniform over test
functions in a fixed finite-order derivative norm is possible.
The construction re-embeds the classical toral flow of
Section~\ref{subsec:liouvilleobstruction} in a rational unipotent subgroup.

Let $N=E_{12}+E_{23}$ and let $\mathbf H$ be the rational
subgroup with real points
\[
 H=\{I_3+uN+vN^2:u,v\in\R\},\qquad
 \Psi_\alpha(t)=I_3+tN+
       \left(\frac{t^2}{2}+\left(\alpha-\frac12\right)t\right)N^2,
\]
where $\alpha$ is irrational, and put
$J=\overline{\langle\Psi_\alpha(\R)\rangle}$. The group law is
$(u,v)(u',v')=(u+u',v+v'+uu')$, and $H\cap\Gamma$ consists
exactly of the pairs with integral $u,v$. The coordinates
\[
 (u,z)=\left(u,v-\frac{u(u-1)}2\right)
\]
are additive and send this lattice onto $\Z^2$, so they are
lattice-adapted coordinates in the sense of Section~\ref{app:malcev}.
Thus $H/(H\cap\Gamma)\simeq\mathbb T^2$, and $\Psi_\alpha(t)$ is the
linear toral trajectory $(t,\alpha t)$. Its actual support is all
of this torus. If $\alpha$ is Liouville, the Fourier tests in
Section~\ref{subsec:liouvilleobstruction} exclude every power rate
in a fixed finite-order norm, also after this smooth coordinate change.

On the other hand, every normalized degree-one polynomial on $H$
is $p(u,v)=au+bv$. If it is invariant under the subgroup generated
by $\Psi_\alpha$, then evaluating along the trajectory gives
\[
 0=p(\Psi_\alpha(t))
   =at+b\left(\frac{t^2}{2}+\left(\alpha-\frac12\right)t\right),
\]
and hence $a=b=0$. Therefore $\mathscr K_{J,\groupid}^{(1)}=\{0\}$,
so the degree-one condition holds with $(c_0,\tau)=(1,0)$ in the
basis $u,v$, including for Liouville $\alpha$.
The degree-two invariant $z-\alpha u$ detects rational approximations to $\alpha$:
for $(p,q)\in\Z^2$,
\[
 \operatorname{dist}(pu-qz,\mathscr K_{J,\groupid}^{(2)})
 \le |p-q\alpha|\,\norm u.
\]
Here a fixed integer multiple of $pu-qz$ belongs to any chosen
rational test lattice, and $\norm{pu-qz}\asymp |p|+|q|$.
Liouville approximation therefore violates every finite degree-two
type. Together with Theorem~\ref{thm:intrinsic}, this proves that
two is the smallest common degree cutoff for the intrinsic hypothesis.

\subsection{Compact supports and horocycle estimates}

The following corollary is a consequence of the established
compact horocycle estimate
\cite[Appendix B, Theorem B.1]{Zheng2015}, through
Proposition~\ref{prop:compactpolynomial}. It records that no
arithmetic hypothesis is needed in the polynomial theorem
on these supports.

\begin{corollary}[Equidistribution on compact semisimple quotients]
\label{cor:compactsemisimple}
Use the data of Section~\ref{subsec:mainassembly}, and suppose
$\mathbf H$ is semisimple and $H\Gamma/\Gamma$ is compact.
Then for every nonconstant normalized
$\Psi\in\mathbf H(\mathbb R[t])$ of degree at most $d$
and every $h\in H$, the left-hand side of
\eqref{eq:prescribedestimate} is at most
$C\norm f_{C_G^s}T^{-\vartheta_{\mathbf H,d}}$ for $T\ge2$,
without an intrinsic Diophantine assumption. For every compact
$\Omega\subset X$, the constant may be chosen as
$C=C(\mathbf H,a,\Psi,\Omega)>0$, uniformly for $h\in H$
and $f\in C_c^\infty(X)$ supported in $\Omega$.
\end{corollary}

\begin{proof}
The ambient support $\mathbf H=\SL_3$ is excluded because $X$
is noncompact. By Lemma~\ref{lem:semisimplecongruencecover}, a rational
block or split principal support has a finite covering by a
congruence quotient. This quotient is noncompact
(Section~\ref{app:arithmetic}), so the support cannot be compact
by Lemma~\ref{app:lattice-covers}. A definite orthogonal group
has no nonconstant polynomial curve. The remaining case is an
anisotropic indefinite orthogonal support. Its real symmetric-square
parametrization and compact quotient are given by
Lemma~\ref{lem:semisimpleparametrizations}; applying
Proposition~\ref{prop:compactpolynomial} as in the compact case
of Section~\ref{subsec:mainassembly} proves the stated estimate,
uniformly in $h$ and without an intrinsic hypothesis.
\end{proof}

For each fixed curve and starting point,
Proposition~\ref{prop:intrinsicnecessity} also gives finite
intrinsic type after normalization at that point. Thus the
arithmetic condition is automatic on these compact supports.

\paragraph{A nonsplit compact support.}
\label{subsec:nonsplitexample}
For example, take $\qform(x,y,z)=x^2+y^2-3z^2$,
$\mathbf H=\SO(\qform)$, and $H=\mathbf H(\R)^\circ$.
The form is indefinite over $\R$ and
anisotropic over $\mathbb Q$: an integral solution of
$x^2+y^2=3z^2$ has $3\mid x,y$ by reduction modulo $3$, then
$3\mid z$, and infinite descent excludes a nonzero rational
solution. Lemma~\ref{lem:semisimpleparametrizations} therefore gives a
compact quotient $H\Gamma/\Gamma$ although $H$ is noncompact.
The corollary records the polynomial and intrinsic consequences
of the compact rank-one estimate on this support.

\appendix
\section{Preliminaries on algebraic groups and auxiliary results}
\label{app:algebraic-background}

This appendix supplies the general algebraic group and lattice facts
used in the classification and reduction arguments of
Section~\ref{sec:supportstructure} and elsewhere in the paper.
Throughout, we work with matrix groups in characteristic zero, mainly over
$\mathbb Q$, $\mathbb R$, and $\mathbb C$.

\iffalse
\begin{center}
\renewcommand{\arraystretch}{1.15}
\begin{tabular}{@{}p{0.43\textwidth}p{0.52\textwidth}@{}}
\hline
Topic & Main uses\\
\hline
Algebraic groups, closure, and rationality
& Normalization and intrinsic functions, Section~\ref{sec:refinement}.\\
Lie algebras and invariant vectors
& Coefficient tests, subgroup classification, and generation arguments.\\
Unipotent groups and exponential coordinates
& Leading tangents and the compact nilmanifold model.\\
Radicals, Levi subgroups, tori, and roots
& Proof of Proposition~\ref{prop:supportclassification}.\\
Ternary forms and the symmetric-square cover
& Proper semisimple supports and the principal group.\\
Mal'cev bases and horizontal characters
& Sections~\ref{subsec:nilinput} and~\ref{sec:unipotentproof}.\\
Lattices, rational envelopes, and finite covers
& Normalization, arithmetic coordinates, and affine quotients.\\
Tensor representations and rational heights
& Sections~\ref{subsec:separation} and~\ref{subsec:arithmeticextraction}.\\
\hline
\end{tabular}
\end{center}
\fi

\subsection{Matrix algebraic groups, connectedness, and rationality}
\label{app:basic}

A \emph{matrix algebraic group over a field $K$} is a subgroup of
$\mathrm{GL}_n$ defined by polynomial equations over $K$.
In $\mathrm{GL}_n$, one allows $\det^{-1}$ as a coordinate; equivalently,
one adds a variable $z$ with $z\det A=1$.
In $\SL_n$ the inverse is the adjugate, so all its entries are
polynomials. The notation $\mathbf H(F)$ means solutions of the
defining equations with entries in an extension field $F$ of $K$.
A homomorphism or representation is \emph{defined over $K$} when
its coordinate functions have coefficients in $K$.
The maps are required to be regular: on a subgroup of $\SL_n$,
their coordinates are restrictions of polynomials in matrix
entries. An algebraic representation is also often called a
\emph{rational representation}; its field of definition is a
separate condition. The representations called rational in our
arithmetic arguments are defined over $\mathbb Q$.

Examples are the additive group $\mathbb G_a$, the multiplicative
group $\mathbb G_m$, $\SL_n$, and the upper unitriangular group
\[
 \mathbf U_3=
 \left\{u(x,y,z)=
 \begin{pmatrix}1&x&z\\0&1&y\\0&0&1\end{pmatrix}\right\}.
\]
For example, $\mathbb G_a(\mathbb R)=(\mathbb R,+)$ and
$\mathbb G_m(\mathbb R)=(\mathbb R^\times,\cdot)$.

The \emph{Zariski closure} of a set of matrices is the zero set of
all polynomials vanishing on it, with the coefficient field
specified. It differs from closure in the usual real topology.
We use bold letters for algebraic groups and usually put
$H=\mathbf H(\mathbb R)^\circ$, where $\circ$ on a real Lie group
means the connected component containing the identity.
Algebraic connectedness is connectedness in the Zariski topology
after passage to an algebraic closure. An algebraically connected
group can have disconnected real points: $\mathbb G_m$ is connected,
whereas $\mathbb R^\times$ has two components.

\begin{lemma}
\label{app:components}
The algebraic identity component of a matrix algebraic group in
characteristic zero is defined over its field of definition.
If $\mathbf H$ is a connected real algebraic group, then
$\mathbf H(\mathbb R)^\circ$ is Zariski dense in $\mathbf H$.
\end{lemma}
\begin{proof}
The assertion about the algebraic identity component is
\cite[Proposition 1.34]{Milne2017}. Matrix algebraic groups in
characteristic zero are smooth by \cite[Theorem 3.23]{Milne2017};
their dimension therefore equals the dimension of the tangent space
at the identity. Let $\mathbf Z$ be the real Zariski closure of
$\mathbf H(\mathbb R)^\circ$. The closure of a subgroup is an
algebraic subgroup: multiplication and inversion preserve the
polynomial identities defining its closure. Its tangent space
contains the tangent space of $\mathbf H(\mathbb R)^\circ$,
which has dimension $\dim\mathbf H$. Thus $\mathbf Z$ and
$\mathbf H$ have the same dimension. A connected smooth algebraic
group is irreducible, and a proper closed subvariety has smaller
dimension, so $\mathbf Z=\mathbf H$. Here irreducibility follows
also from smoothness: the irreducible components are disjoint and
their translates are the connected components.
\end{proof}

\begin{lemma}
\label{app:rational-envelope}
The real Zariski closure of a subgroup of $\SL_n(\mathbb Q)$ is
defined over $\mathbb Q$. Every subgroup $J\le\SL_n(\mathbb R)$
is contained in a smallest $\mathbb Q$-algebraic subgroup. If $J$
is connected in the real topology, this smallest group is
algebraically connected.
\end{lemma}
\begin{proof}
In any bounded degree, vanishing on a set of rational matrices
is a system of linear equations with rational coefficients for
the polynomial coefficients. Even if there are infinitely many
equations, finitely many span their row space. Gaussian elimination
therefore gives a rational basis for its solution space. Applying
this in every degree proves the first assertion.

For the second assertion, intersect all $\mathbb Q$-algebraic
subgroups containing $J$. The intersection is an algebraic subgroup
over $\mathbb Q$: its equations generate an ideal in a polynomial
ring, which is finitely generated. If $J$ is connected, it lies
in the real points of the algebraic identity component of this
intersection. That component is defined over $\mathbb Q$ by
Lemma~\ref{app:components}. Minimality therefore makes the
intersection connected.
\end{proof}

The smallest rational group is sometimes called a \emph{rational
envelope}. It need not equal the real Zariski closure.
For example, in the additive group $\mathbb G_a^2$, the line
$J=\{(t,\sqrt2t):t\in\mathbb R\}$ is real Zariski closed, but
the smallest $\mathbb Q$-algebraic subgroup containing $J$ is $\mathbb G_a^2$.
Indeed, algebraic subgroups of an additive group in characteristic
zero are linear subspaces, as follows from the exponential
description below, and no rational line contains $J$.
The rational envelope is defined by minimality among
$\mathbb Q$-algebraic subgroups. The common zero set of rational
polynomials vanishing on $J$ need not be a subgroup. In this example,
$y^2-2x^2$ vanishes on $J$, but its zero set is a union of two
lines and is not an additive subgroup.

\subsection{Lie algebras, representations, and rational invariant spaces}
\label{app:lie}

For a matrix group $H$, its Lie algebra $\mathfrak h=\Lie(H)$ is
its tangent space at $\groupid$, with bracket $[X,Y]=XY-YX$.
The adjoint action is $\Ad(g)X=gXg^{-1}$.
For a differentiable representation $\rho:H\to\mathrm{GL}(V)$,
its differential $\mathrm d\rho$ satisfies
\[
 \rho(\exp(tX))=\exp(t\,\mathrm d\rho(X)).
\]
This identity follows by differentiating the one-parameter group
on the left and solving the resulting linear differential equation.
The fixed space $V^H$ consists of the vectors fixed by every
element of $H$. A representation on a nonzero vector space $V$
is irreducible if $V$ has no invariant subspaces other than $\{0\}$
and $V$; it is faithful if its kernel is trivial.

\begin{lemma}
\label{app:lie-tests}
Let $H$ be a connected real matrix Lie group.
A vector is fixed by $H$ if and only if it is annihilated by
$\mathrm d\rho(X)$ for every $X\in\Lie(H)$.
A closed subgroup $J\le H$ with $\Lie(J)=\Lie(H)$ equals $H$.
If a differentiable curve $\Psi(t)$ lies in $J$, then
$\Psi'(t)\Psi(t)^{-1}\in\Lie(J)$.
For a polynomial logarithmic derivative, each of its coefficient
matrices also belongs to $\Lie(J)$.
\end{lemma}
\begin{proof}
The exponential map gives a neighborhood of $\groupid$ from a neighborhood
of zero in the Lie algebra. A connected Lie group is generated by
any identity neighborhood: the generated subgroup is open, and
its other cosets would disconnect the group. The displayed
exponential identity proves the fixed-vector assertion. The same
neighborhood argument proves the assertion about $J$.
Differentiating $\Psi(t+s)\Psi(t)^{-1}$ at $s=0$ gives the
logarithmic derivative. Finally, its coefficients are linear
combinations of finitely many of its values, by polynomial
interpolation.
\end{proof}

If $\mathbf H$ is defined over $\mathbb Q$, its Lie algebra is
the common kernel of the differentials at $\groupid$ of its
defining equations over $\mathbb Q$. Thus it has a rational basis.
For a representation defined over $\mathbb Q$, the equations
$\mathrm d\rho(X)v=0$, with $X$ ranging over such a basis, are
rational linear equations. Their common kernel has a rational basis.
By Lemmas~\ref{app:components}
and~\ref{app:lie-tests}, this kernel is the invariant space for
the connected algebraic group. The same reasoning applies to
covectors, bilinear forms, and tensor representations.
It is the rational descent argument used in
Proposition~\ref{prop:supportclassification} and
Lemma~\ref{lem:arithmeticinputs}.

\subsection{Unipotent groups, nilpotent matrices, and polynomial coordinates}
\label{app:unipotent}

An $n$ by $n$ matrix $u$ is \emph{unipotent} if $u-I_n$ is nilpotent, or
equivalently all its eigenvalues are one. A matrix $X$ is
\emph{nilpotent} if $X^r=0$ for some $r$.
A unipotent algebraic group is a matrix algebraic group whose
elements over an algebraic closure are unipotent.
For $X^3=0$ and $(u-I_n)^3=0$,
\begin{equation}\label{eq:appendix-explog}
 \exp X=I_n+X+\tfrac12X^2,\qquad
 \log u=(u-I_n)-\tfrac12(u-I_n)^2.
\end{equation}
These formulas are inverse polynomial operations.

For unipotent groups we use the following triangularization result.

\begin{proposition}[{\cite[Proposition~14.3]{Milne2017}}]
\label{app:unipotent-flag}
Every representation over $K$ of a unipotent algebraic $K$-group
has a full $K$-rational invariant flag on whose successive
one-dimensional quotients the action is trivial. In particular,
a nonzero representation space has a nonzero fixed vector.
A unipotent $K$-subgroup of $\SL_3$ is conjugate over $K$ into
$\mathbf U_3$.
\end{proposition}
\begin{proof}
For the final assertion, choose a basis adapted to the invariant
flag provided by the cited result. The action is then upper
triangular with diagonal entries one. Applying this to the given
three-dimensional representation gives the required conjugation.
If a group $S$ normalizes the unipotent group $U$, then $V^U$ is
$S$-stable: for $v\in V^U$, one has
$u(sv)=s(s^{-1}us)v=sv$.
This last observation is used in the affine classification.
\end{proof}

Here a full invariant flag is a chain
$0=V_0\subset V_1\subset\cdots\subset V_n=V$ with
$\dim V_j=j$, each preserved by the group. Trivial action on
the quotients $V_j/V_{j-1}$ means that each induced matrix is $1$.

\begin{lemma}
\label{app:unipotent-explog}
For a connected unipotent subgroup of $\SL_3$, exponential and
logarithm identify its Lie algebra and its real points by polynomial
maps. The real group is connected and simply connected.
Its Lie algebra has dimension at most three, and its commutator
algebra is central. A nonabelian such Lie algebra is the full
three-dimensional Heisenberg algebra after conjugation.
\end{lemma}
\begin{proof}
Use Proposition~\ref{app:unipotent-flag} to work inside
$\mathbf U_3$. The exponential of a Lie algebra element lies
in the Lie group. Conversely, for $u$ in the algebraic group,
apply any defining polynomial to $\exp(t\log u)$. This is a
polynomial in $t$ vanishing at every integer, since
$\exp(n\log u)=u^n$. It therefore vanishes identically.
Differentiation at zero gives $\log u\in\Lie(H)$.
Formula~\eqref{eq:appendix-explog} now gives inverse maps between
a real vector space and the group, including the asserted
connectedness and simple connectedness.

The strictly upper triangular algebra has basis
$E_{12},E_{23},E_{13}$ and only one nonzero basic bracket,
$[E_{12},E_{23}]=E_{13}$. Two elements with nonzero bracket
have independent images modulo $\mathbb RE_{13}$ and generate
this whole algebra. This proves the last assertions.
\end{proof}

For these groups the Baker--Campbell--Hausdorff formula is the
finite identity
\[
 \log(\exp X\exp Y)=X+Y+\tfrac12[X,Y].
\]
It can be checked by multiplying \eqref{eq:appendix-explog};
all products of three strictly upper triangular matrices vanish.
The same exponential and logarithm description applies to a
connected closed Lie subgroup of $\mathbf U_3(\mathbb R)$.
Indeed, the formula makes the exponential of its Lie algebra a
subgroup. It is closed, since logarithm is a homeomorphism,
and it contains an identity neighborhood in the given group.
Connectedness therefore makes the two groups equal.
In entry coordinates the multiplication and logarithm are
\[
 u(x,y,z)u(x',y',z')
   =u(x+x',y+y',z+z'+xy'),\qquad
 \log u=xE_{12}+yE_{23}+(z-xy/2)E_{13}.
\]
These formulas underlie the degree bounds in
Lemma~\ref{app:horizontal}.

\begin{lemma}
\label{app:nilpotent-charts}
A nonzero nilpotent $3$ by $3$ matrix has either $X^2=0$ and
rank one, or $X^2\ne0$ and a single Jordan block of size three.
The real Lie algebra of every support occurring in this paper
has a basis of nilpotent matrices. For such a basis $X_1,\ldots,X_m$,
the map
\[
 (t_1,\ldots,t_m)\longmapsto
 \exp(t_1X_1)\cdots\exp(t_mX_m)
\]
is polynomial and is a local coordinate chart at zero.
\end{lemma}
\begin{proof}
The first assertion follows from the possible nilpotent Jordan
blocks. The determinant-one conjugacy normalizations, including
the sign choices in dimension two, are established in
Lemma~\ref{lem:normalization}. For $\mathfrak{sl}_n(\mathbb R)$,
$n=2,3$, use
the off-diagonal matrix units together with
$E_{ii}-E_{i+1,i+1}+E_{i,i+1}-E_{i+1,i}$ for adjacent indices.
The latter matrices square to zero, and their diagonal parts span
the traceless diagonal matrices. The differentials of the real
block and symmetric-square representations send nilpotent elements
to nilpotent endomorphisms, as follows from their polynomial
one-parameter subgroups. The Lie algebras of the unipotent and
affine radicals are spanned by nilpotent matrices. The real indefinite
orthogonal groups use the
symmetric-square description in Lemma~\ref{app:symmetric-square}.
Definite orthogonal groups have no nonconstant polynomial curve,
so do not occur as supports here. Finally, the derivative of the
displayed map at zero sends the standard basis to $X_1,\ldots,X_m$.
The inverse function theorem proves the local chart assertion.
\end{proof}

\subsection{Radicals, reductive groups, and Levi decomposition}
\label{app:levi}

The derived subgroup is generated algebraically by commutators
$aba^{-1}b^{-1}$. A group is \emph{solvable} if its repeated
derived groups reach the identity. The lower central series is
$H_1=H$, $H_{j+1}=[H,H_j]$; a group is \emph{nilpotent} if
this series reaches the identity.
The group $\mathbf U_3$ is nilpotent of class two, whereas an
upper triangular group with nontrivial diagonal entries is
solvable but need not be nilpotent.

For a connected algebraic group $\mathbf H$, its \emph{radical}
$R(\mathbf H)$ is its largest connected normal solvable subgroup,
and its \emph{unipotent radical} $R_u(\mathbf H)$ is its largest
connected normal unipotent subgroup. Over a characteristic-zero
field these are defined over that field.
The group is \emph{semisimple} if $R(\mathbf H)=1$ and
\emph{reductive} if $R_u(\mathbf H)=1$.
These definitions and the existence of the radicals are recorded in
\cite[6.46 and Sections 19a--19b]{Milne2017}.
Thus a torus is reductive but, unless trivial, is not semisimple.
The groups $\SL_2$ and $\SL_3$ are semisimple.
Here is a direct verification. Over an algebraic closure, a
nonzero ideal of $\mathfrak{sl}_n$, $n=2,3$, contains a matrix
unit $E_{ij}$ with $i\ne j$: projections to the weight spaces
for diagonal commutators give such a unit unless the ideal
contains a nonzero diagonal matrix, whose commutator with some
$E_{ij}$ does so. From $E_{ij}$, taking commutators with
$E_{ji}$ and the other matrix units gives all of
$\mathfrak{sl}_n$. Thus this Lie algebra is simple and is not
solvable. The Lie algebra of a connected normal solvable subgroup
is a solvable ideal, so it is zero. A connected algebraic group
of dimension zero is trivial, proving the claim.

Mostow's characteristic-zero Levi theorem yields the following
statement, including the consequences for the two radicals.

\begin{proposition}[{\cite[25.48--25.49]{Milne2017}}]
\label{app:levi-decomposition}
For a connected algebraic group $\mathbf H$ over a field $K$ of
characteristic zero, there is a reductive $K$-subgroup $\mathbf S$
mapping isomorphically onto $\mathbf H/R_u(\mathbf H)$.
Consequently
\[
 \mathbf H=\mathbf S\ltimes R_u(\mathbf H).
\]
If $R(\mathbf H)$ is unipotent, then
$R(\mathbf H)=R_u(\mathbf H)$ and $\mathbf S$ is semisimple.
Conversely, a semisimple quotient $\mathbf H/R_u(\mathbf H)$
forces $R(\mathbf H)=R_u(\mathbf H)$.
\end{proposition}
\begin{proof}
The cited Levi theorem gives the existence of $\mathbf S$.
The inclusions between the two radicals give their equality
when the radical is unipotent. The image of the radical in
$\mathbf H/R_u(\mathbf H)$ is its connected solvable radical;
this follows by taking images and preimages of connected normal
solvable groups. Thus that quotient is semisimple exactly when
the two radicals agree. The isomorphism defining $\mathbf S$
gives the remaining assertion.
\end{proof}

The notation $\mathbf S\ltimes\mathbf U$ means that $\mathbf U$
is normal and every element has a unique factorization into an
element of $\mathbf S$ and one of $\mathbf U$. The factors need
not commute. The basic example in this paper is
\[
 \mathbf A=\SL_2\ltimes\mathbb G_a^2,
 \qquad (A,v)(B,w)=(AB,Aw+v).
\]
Its unipotent radical is the translation group and its Levi
subgroup is $\{(A,0)\}$. To check the radical assertion, its
quotient by translations is the semisimple group $\SL_2$.
The image of any connected normal solvable subgroup in that
quotient is trivial. Since translations themselves form a
connected normal unipotent subgroup, they are both radicals.
The two embeddings
$\mathbf A$ and $\mathbf A^\vee$ in \eqref{eq:affinegroups}
have these same abstract group properties. The map
$b\mapsto b^{-\mathsf T}$ is an algebraic automorphism of $\SL_3$
whose coordinate degree is at most two.

\begin{lemma}
\label{app:affine-algebra}
The group $\mathbf A$ is perfect, meaning that its derived group
is itself. The standard two-dimensional $\SL_2$-module and its
dual are irreducible and have no nonzero fixed vector.
The fixed space of the upper unipotent subgroup in the standard
module is a line.
\end{lemma}
\begin{proof}
Upper and lower elementary unipotent matrices generate $\SL_2(K)$
by Gaussian elimination. Each is a commutator: conjugation by
$\diag(2,1/2)$ multiplies its upper parameter by $4$ and its
lower parameter by $1/4$, so the resulting commutator parameters
range over $K$. For translations, the commutator with this same
diagonal matrix is translation by $(\diag(2,1/2)-I_2)v$, which
ranges over $K^2$. Thus all of $\mathbf A$ is generated by
commutators. A line invariant under both upper and lower
unipotents in $K^2$ would have to be both coordinate axes,
which is impossible. Direct multiplication gives their respective
fixed axes and proves the assertions about invariant vectors.
The dual has the same properties by inverse transpose.
\end{proof}

\subsection{Tori, roots, and three-dimensional representations}
\label{app:tori}

A \emph{torus} over $K$ is an algebraic group that becomes
$(\mathbb G_m)^r$ over an algebraic closure. It is \emph{split over
$K$} if this isomorphism is already defined over $K$.
The diagonal determinant-one matrices form a split torus in
$\SL_n$. The circle $x^2+y^2=1$, with complex multiplication,
is a real torus that is not split over $\mathbb R$.
Here ``torus'' is an algebraic notion; the compact horizontal torus
$\mathbb R^r/\mathbb Z^r$ in a nilmanifold is a quotient Lie group.

An algebraic character of $(\mathbb G_m)^r$ is an algebraic
group homomorphism to $\mathbb G_m$. Every such character has the
form $t\mapsto t_1^{a_1}\cdots t_r^{a_r}$, with $a_i\in\mathbb Z$.
A \emph{weight space} is a subspace on which a torus acts by one
such character. A maximal torus is a torus not contained in a
larger one. The absolute rank is its dimension over an algebraic
closure; the $K$-split rank is the largest dimension of a torus
split over $K$. Maximal tori over an algebraic closure are
conjugate \cite[Theorem 17.10]{Milne2017}, so the absolute rank
does not depend on the choice of maximal torus.

\begin{lemma}
\label{app:weights}
A representation over $K$ of $(\mathbb G_m)^r$ is a direct sum
of weight spaces defined over $K$. Weights add under tensor
products and exterior products. On the determinant-one diagonal
torus, exponent vectors differing by a multiple of $(1,\ldots,1)$
define the same character.
\end{lemma}
\begin{proof}
Let $V$ be the representation space. The coordinate functions
are Laurent polynomials, so write $\rho(t)=\sum_a t^aP_a$.
Comparing coefficients in $\rho(t)\rho(s)=\rho(ts)$ gives
$P_aP_b=0$ for $a\ne b$, $P_a^2=P_a$, and $\sum_aP_a=\operatorname{id}_V$.
Their images are the required weight spaces. Tensor and exterior
products multiply the characters. On determinant-one matrices
$\prod_i t_i=1$; conversely, putting $t_n=(t_1\cdots t_{n-1})^{-1}$
shows that a trivial character has all its exponents equal.
\end{proof}

The roots of a semisimple group relative to a maximal torus are
the nonzero weights of its adjoint representation on its Lie
algebra. For $\SL_3$ these are particularly concrete:
\[
 \diag(t_1,t_2,t_3)E_{ij}\diag(t_1,t_2,t_3)^{-1}
      =(t_i/t_j)E_{ij}.
\]
Thus its root spaces are the six lines $KE_{ij}$, $i\ne j$.

We use the following root-decomposition and rank-one structure results.

\begin{proposition}[{\cite[Theorem~21.11, Corollary~21.12,
and Proposition~20.32]{Milne2017}}]
\label{app:root-inputs}
Over an algebraically closed characteristic-zero field, a
semisimple group has a root decomposition relative to a maximal
torus. Roots occur in opposite pairs, span the character space
over $\mathbb Q$, and have one-dimensional root spaces.
A connected semisimple group of rank one admits a central
isogeny from $\SL_2$.
\end{proposition}
The root assertions are the semisimple case of
\cite[Theorem 21.11 and Corollary 21.12]{Milne2017}.
The spanning assertion also follows from semisimplicity: a
positive-dimensional common kernel of the roots would be a
central torus. The rank-one assertion is
\cite[Proposition 20.32]{Milne2017}.
An isogeny is a surjective algebraic homomorphism with finite
kernel; it is central if its kernel lies in the center.

\begin{lemma}
\label{app:sl2-representations}
A three-dimensional algebraic representation of $\SL_2$ over
a characteristic-zero field $K$, with finite kernel, is isomorphic
to either the direct sum of the standard representation and a
one-dimensional trivial representation, or the symmetric square
of the standard representation. The latter is called the principal
representation here.
\end{lemma}
\begin{proof}
Algebraicity gives a finite Laurent expansion
(the weight decomposition of Lemma~\ref{app:weights})
\[
 a(\xi):=\iota\bigl(\diag(\xi,\xi^{-1})\bigr)
      =\sum_j\xi^j\Pi_j\qquad(\xi\in K^\times).
\]
Comparing coefficients in $a(\xi)a(\zeta)=a(\xi\zeta)$ shows that
$\Pi_j\Pi_k=0$ for $j\ne k$, $\Pi_j^2=\Pi_j$, and
$\sum_j\Pi_j=I_3$. Thus $K^3=\bigoplus_jV_j$, where
$V_j=\operatorname{im}\Pi_j$ and $a(\xi)$ acts on $V_j$ by $\xi^j$. Put
\[
 E=\mathrm d\iota(E_{12}),\qquad F=\mathrm d\iota(E_{21}),\qquad
 D=\mathrm d\iota(\diag(1,-1)).
\]
The relations $[D,E]=2E$, $[D,F]=-2F$, and $[E,F]=D$ imply
$EV_j\subseteq V_{j+2}$ and $FV_j\subseteq V_{j-2}$.
Let $m$ be the largest index with $V_m\ne0$, and choose $0\ne v\in V_m$.
Then $Ev=0$ and
induction on $k$ gives
\[
 EF^kv=k(m-k+1)F^{k-1}v\qquad(k\ge1).
\]
If $F^rv\ne0$ and $F^{r+1}v=0$, this identity gives $m=r$.
The chain $v,Fv,\ldots,F^mv$ has $m+1\le3$ distinct weights.
The case $m=0$ would force all weights to be zero, since
$\operatorname{tr}D=0$; then $E=F=0$, contradicting finite kernel.
If $m=2$, the chain is a basis and its displayed action is the
symmetric square. If $m=1$, the two-dimensional chain is standard
and the remaining weight is zero. The operators $E,F$ annihilate that
weight space because weights $2,-2$ are absent, so the remaining
line is trivial. All weight projections and chain vectors are
defined over $K$, so this gives a change of basis in $\mathrm{GL}_3(K)$.
The resulting representations agree on the upper and lower
unipotent subgroups: their images are the exponentials of the
displayed nilpotent operators, using the representation identity
$\iota(\exp X)=\exp(\mathrm d\iota(X))$ from Section~\ref{app:lie}.
By Gaussian elimination these subgroups generate $\SL_2(K)$,
and hence determine the algebraic representation.
\end{proof}

\subsection{Ternary quadratic forms and symmetric-square covers}
\label{app:quadratic}

A rational ternary quadratic form is $\qform(v)=v^{\mathsf T}Qv$
for a symmetric $3$ by $3$ matrix $Q$ with rational entries.
It is nondegenerate if $\det Q\ne0$. It is isotropic over a field
$K$ if $\qform(v)=0$
for some $0\ne v\in K^3$, and anisotropic otherwise.
Definiteness and indefiniteness refer to the signature of the
associated real quadratic form.
An indefinite real form can still be anisotropic over $\mathbb Q$.
Multiplying $Q$ by a nonzero scalar does not change
\[
 \SO(\qform)=\{g\in\SL_3:g^{\mathsf T}Qg=Q\}.
\]

\begin{lemma}[Symmetric-square parametrization]
\label{app:symmetric-square}
If a nondegenerate ternary form over $K\subseteq\mathbb R$ is
isotropic over $K$, it is, after a $K$-linear change of variables
and a scalar multiplication, the determinant form on symmetric
$2$ by $2$ matrices. The action $S\mapsto ASA^{\mathsf T}$ gives
a central isogeny $\SL_2\to\SO(\qform)$ with kernel $\{I_2,-I_2\}$.
On real points its image is $\SO(\qform)(\mathbb R)^\circ$.
For every nondegenerate ternary form, whether isotropic over $K$
or not, the algebraic group $\SO(\qform)$ is connected and has
dimension three.
\end{lemma}
\begin{proof}
Let $e$ be isotropic. For the associated bilinear form
$B(v,w)=v^{\mathsf T}Qw$, choose $f$ with $B(e,f)=1$ and
replace it by $f-\qform(f)e/2$. Now $\qform(f)=0$. Their orthogonal
complement is a nondegenerate line, so the form is $2xy+az^2$,
$a\ne0$. This equals $-a\det\left(\begin{smallmatrix}x&z\\z&-2y/a\end{smallmatrix}\right)$.

The stated action preserves determinant. If it fixes every
rank-one symmetric matrix $vv^{\mathsf T}$, then $Av=\pm v$
for every $v$, which forces $A=\pm I_2$. Its differential is
$X:S\mapsto XS+SX^{\mathsf T}$ and is injective: applying a
zero differential first to $S=I_2$ makes $X$ skew-symmetric,
and then to $S=\operatorname{diag}(1,0)$ makes $X=0$. The Lie
algebra of $\SO(\qform)$ is given by $X^{\mathsf T}Q+QX=0$ and has
dimension three. Connectedness of the algebraic orthogonal
group in determinant one is \cite[2.43 and Exercise 2-9]{Milne2017}.
The image of an algebraic group homomorphism is an algebraic
subgroup (\cite[Theorem 3.34]{Milne2017}); equality of dimensions
now proves surjectivity. On real points the differential is an
isomorphism, so the image contains an identity neighborhood.
Since $\SL_2(\mathbb R)$ is connected, that image is exactly
the real identity component.
\end{proof}

An indefinite real ternary form is isotropic over $\mathbb R$,
so the real covering exists in that case. Let $T$ be a nontrivial
$\mathbb Q$-split torus in $\SO(\qform)$. Its standard representation
has a nonzero rational vector $v$ of nontrivial weight $\chi$.
Invariance of $\qform$ gives
\[
 \chi(t)^2\qform(v)=\qform(v)\qquad(t\in T(\mathbb Q)),
\]
hence $\qform(v)=0$. Thus an anisotropic rational ternary form
has no nontrivial $\mathbb Q$-split torus. In the isotropic case,
the diagonal torus from the displayed $\SL_2$ supplies one.
This explains the field dependence of ``split'' in the main proof.

Orthogonal orbit closedness and quotient compactness are proved in
Lemma~\ref{lem:semisimpleparametrizations}. The distinction matters in the example
$\qform=x^2+y^2-3z^2$: the group $H$ is noncompact, but its arithmetic
quotient is compact. The rational anisotropy is proved by descent
modulo $3$ in Section~\ref{subsec:nonsplitexample}.
A representation descends through an isogeny precisely when the
kernel of the isogeny acts trivially. The standard representation
of $\SL_2$ does not descend through the symmetric-square isogeny,
since $-I_2$ acts as $-I_2$ on $\mathbb R^2$; its symmetric square does.
This is the reason for Lemma~\ref{lem:quadratictransfer}.

\subsection{Mal'cev bases, horizontal tori, and polynomial sequences}
\label{app:malcev}

A \emph{lattice} in a Lie group is a discrete subgroup whose
quotient has finite invariant volume. It is cocompact if the
quotient is compact. A nilmanifold is a compact quotient
$H/\Lambda$ with $H$ connected, simply connected, and nilpotent.
The unipotent groups used here have these properties by
Lemma~\ref{app:unipotent-explog}.

A Mal'cev basis $B_1,\ldots,B_m$ for $H/\Lambda$ is a Lie
algebra basis whose terminal spans are ideals, for which the
ordered coordinates
\[
 (a_1,\ldots,a_m)\longmapsto
 \exp(a_1B_1)\cdots\exp(a_mB_m)
\]
are unique and identify $\Lambda$ with $\mathbb Z^m$.
Being adapted to a filtration means that each subgroup in the
filtration has Lie algebra spanned by terminal basis vectors.
For the lower central series in our groups, this requires the
commutator algebra to be spanned by the last vectors.
These are the conventions of \cite[Definition 2.1]{GreenTao2012}.
The general existence result is described immediately after that
definition, with sources
\cite[Theorem 5.1.8(a), Corollary 5.2.2, and Proposition 5.3.2]{CorwinGreenleaf1990}.
Here an elementary proof suffices in every dimension that occurs.

\begin{lemma}
\label{app:malcev-existence}
Let $H$ be a connected closed Lie subgroup of a real conjugate
of $\mathbf U_3(\mathbb R)$, and let $\Lambda\le H$ be a
discrete subgroup with compact quotient.
There is a Mal'cev basis adapted to its lower central series.
\end{lemma}
\begin{proof}
In the abelian case, logarithm identifies $H$ with a vector space
and $\Lambda$ with a Euclidean lattice, so use a lattice basis.
Otherwise $H$ is a Heisenberg group by
Lemma~\ref{app:unipotent-explog}. Write $Z=[H,H]\simeq\mathbb R$
and let $\pi:H\to H/Z\simeq\mathbb R^2$ be the quotient.
Compactness implies that $\pi(\Lambda)$ spans $\mathbb R^2$:
otherwise a nonzero linear functional annihilating its span would
give an unbounded continuous image of $H/\Lambda$.
Two lattice elements with independent projections have a nontrivial
central commutator. Thus $\Lambda\cap Z$ is a nontrivial discrete
subgroup of $Z$; choose its generator $\gamma_3$.
The group $\pi(\Lambda)$ is discrete. Indeed, lift a sequence
approaching zero and reduce its central coordinate modulo
$\gamma_3$ to a bounded interval. The lifts then stay in a compact
set, which contains only finitely many elements of $\Lambda$.
Its quotient is compact as an image of $H/\Lambda$, so it is a
full lattice. Lift a basis to $\gamma_1,\gamma_2\in\Lambda$.
Every element of $\Lambda$ has a unique expression
$\gamma_1^{n_1}\gamma_2^{n_2}\gamma_3^{n_3}$, by first matching
its horizontal projection. The logarithms of these three elements
give the desired basis. Its only possible nonzero structure
constant is an integer, since
$[\gamma_1,\gamma_2]=\gamma_3^k$ for an integer $k$.
\end{proof}

A rational basis of the Lie algebra need not be a Mal'cev basis
relative to a prescribed lattice: the condition that the lattice
have exactly integer ordered coordinates is additional.
For the standard Heisenberg lattice,
$(E_{12},E_{23},E_{13})$ is a Mal'cev basis and
\[
 \exp(aE_{12})\exp(bE_{23})\exp(cE_{13})=u(a,b,ab+c).
\]
Thus ordered coordinates are $(x,y,z-xy)$, whereas logarithmic
coordinates are $(x,y,z-xy/2)$. In particular, these two types
of coordinates must be distinguished.

\begin{lemma}[Horizontal characters and polynomial coordinates]
\label{app:horizontal}
Let a Mal'cev basis be as above and put $r=\dim H/[H,H]$.
The first $r$ ordered coordinates define a homomorphism
$\chi_{\mathrm{ab}}:H\to\mathbb R^r$ sending $\Lambda$ to
$\mathbb Z^r$. The quotient
$H/[H,H]\Lambda$ is $\mathbb R^r/\mathbb Z^r$.
Every continuous homomorphism $H\to\mathbb R/\mathbb Z$
vanishing on $\Lambda$ is
$h\mapsto k\cdot\chi_{\mathrm{ab}}(h)\bmod\mathbb Z$ for
some $k\in\mathbb Z^r$.
The ordered-coordinate map is polynomial of degree at most two
in matrix entries. Thus a matrix curve of degree $d$ has ordered
coordinates of degree at most $2d$. Conversely, polynomial ordered
coordinates define a matrix-polynomial curve.
\end{lemma}
\begin{proof}
Commutators lie in the final coordinates, so the first coordinates
add. The adapted lattice basis identifies their lattice with
$\mathbb Z^r$. A continuous homomorphism to $\mathbb R/\mathbb Z$
vanishes on commutators and therefore factors through
$H/[H,H]\simeq\mathbb R^r$. The induced homomorphism is a real
linear functional modulo $\mathbb Z$. Vanishing on $\mathbb Z^r$
makes the coefficients of this functional integral. These are the
horizontal characters of
\cite[Definition 2.6]{GreenTao2012}.
For the degree assertion, logarithm has degree two by
\eqref{eq:appendix-explog}. In the abelian case its coefficients
are the ordered coordinates. In the Heisenberg case, the first
two coordinates are linear in the matrix entries, and the final
coordinate is linear plus a quadratic expression in the first
two, as in the displayed example. Rational basis matrices give
rational polynomial formulas. Conversely, each ordered exponential is
quadratic in its coordinate, so their product has polynomial
matrix entries whenever the coordinates are polynomial.
\end{proof}

A filtration $H_\bullet$ is a descending family with
$[H_i,H_j]\subseteq H_{i+j}$. A sequence $q:\mathbb Z\to H$
is polynomial for this filtration if its iterated multiplicative
differences of order $j$ lie in $H_j$, where
$\partial_aq(n)=q(n+a)q(n)^{-1}$.

\begin{lemma}
\label{app:nil-sampling}
Let $H/\Lambda$ and its adapted Mal'cev basis be as in
Lemma~\ref{app:malcev-existence}. For every integer $D\ge1$,
the groups in \eqref{eq:nilfiltration} form a filtration of
degree at most $2D$ to which the basis is adapted.
Every sequence $q:\mathbb Z\to H$ whose ordered coordinates
are polynomials of degree at most $D$ is polynomial for this
filtration. If $\Phi:\mathbb R\to H$ has polynomial ordered
coordinates of degree at most $D$, then
$n\mapsto\Phi(sn)a$ is polynomial for the same filtration,
for every $a\in H$ and every real $s$.
\end{lemma}
\begin{proof}
Since $[H,H]$ is central,
$[H_i,H_j]\subseteq H_{i+j}$ for the displayed filtration.
The last basis vectors span the commutator algebra, so the
basis is adapted. The polynomial-coordinate characterization
\cite[Lemma 6.7, one-parameter case]{GreenTao2012} applies:
only terms $\binom nj$ with $0\le j\le D$ occur in each
coordinate, and $H_j=H$ at these indices.

Write the ordered coordinates as $(u,v)$, where $u$ is
horizontal and $v$ is central. Their multiplication law is
\begin{equation}\label{eq:nilgrouplaw}
 (u,v)(u',v')=(u+u',v+v'+\mathcal B(u,u')),
\end{equation}
with $\mathcal B$ bilinear; the second component is absent
in the abelian case. Indeed the finite Baker--Campbell--Hausdorff
identity of Section~\ref{app:unipotent}, converted from
logarithmic to ordered coordinates, cancels the purely
quadratic terms in $u$ and $u'$. Right multiplication by
a fixed $a$ therefore preserves coordinate degree, as does
the substitution $t=sn$. This proves the sampling assertion.
\end{proof}

\subsection{Arithmetic orbits, congruence covers, and normalization}
\label{app:arithmetic}

For $x=g\Gamma$, the stabilizer of the action of $H$ on $x$ is
$H\cap g\Gamma g^{-1}$. The orbit is modeled by the quotient
by this stabilizer. A closed orbit of finite invariant volume
is called periodic in this paper. In particular, ``periodic''
does not necessarily mean the orbit of a periodic one-parameter
flow. We normalize the invariant measure on each periodic orbit
to have total mass one.
For a closed orbit the usual homogeneous-space identification
is the one used for its smooth and measure structures.

\begin{lemma}
\label{app:lattice-covers}
Let $L$ be a Lie group and let $\Lambda\subseteq\Lambda'$ be
discrete subgroups such that $L/\Lambda$ and $L/\Lambda'$ admit
$L$-invariant measures.
If $\Lambda$ is a lattice, then $[\Lambda':\Lambda]<\infty$.
If $\Lambda'$ is a lattice, then $\Lambda$ is a lattice exactly
when this index is finite. In the finite-index case the natural
quotient map is a finite covering and sends invariant probability
to invariant probability.
\end{lemma}
\begin{proof}
Choose a small open set of positive volume in the quotient by
$\Lambda'$ that lifts injectively to the group. Its distinct
lifts to the quotient by $\Lambda$ are disjoint, are indexed by
$\Lambda'/\Lambda$, and all have the same positive volume.
An infinite index is incompatible with finite total volume.
Conversely, finitely many translates of a fundamental domain
give a finite-volume domain. The quotient measures differ by
the index before normalization, proving the probability assertion.
The same local lifts give the covering property. In particular,
a finite cover is compact if and only if its base is compact.
\end{proof}

The group $\SL_2(\mathbb Z)$ is a lattice in $\SL_2(\mathbb R)$.
One elementary verification uses the upper-half-plane domain
$|x|\le1/2$, $x^2+y^2\ge1$, $y>0$: its hyperbolic area is
\[
 \int_{-1/2}^{1/2}\int_{\sqrt{1-x^2}}^\infty y^{-2}\,\dd y\,\dd x
 =\pi/3.
\]
Translations and inversion reduce any point to this domain;
equivalently choose a shortest primitive vector in the associated
two-dimensional lattice, complete it to a basis, and reduce the
second vector. The compact rotation factor gives finite Haar
volume in the group. Its quotient is noncompact, as seen from
$\diag(\varepsilon,\varepsilon^{-1})\mathbb Z^2$ with
$\varepsilon\to0$. The principal congruence subgroup
$\Gamma_2(N)$ has finite index since reduction modulo $N$ has
finite image. Lemma~\ref{app:lattice-covers} therefore gives a
finite-volume, noncompact congruence quotient.

\begin{lemma}
\label{app:congruence}
For a rational matrix-polynomial homomorphism $\iota$ with
$\iota(\groupid)=\groupid$, there is an integer $N\ge1$ such that
$\iota(\gamma)$ is integral for every integral matrix $\gamma$ in its
source satisfying $\gamma\equiv I_n\pmod N$, where $n$ is the size
of the source matrices.
\end{lemma}
\begin{proof}
Choose $N$ divisible by all denominators of the coordinate
polynomials. If an integral $n$ by $n$ matrix $\gamma$ satisfies
$\gamma\equiv I_n\pmod N$, then $M(\gamma)-M(I_n)$ is divisible
by $N$ for every monomial $M$ in the matrix entries. Subtracting
$\iota(\groupid)$ proves integrality.
\end{proof}

For a finite-kernel surjective Lie homomorphism, the inverse image
of a discrete subgroup is discrete: the map is locally invertible
at the identity. Its quotient by that inverse image identifies
with the target quotient.

\begin{lemma}
\label{app:primitive-vectors}
A finite-index subgroup of $\SL_2(\mathbb Z)$ contains a positive
power of each elementary unipotent. The action of
$\SL_2(\mathbb Z/q\mathbb Z)$ on primitive pairs modulo $q$
is transitive, and the number of these pairs is
$q^2\prod_{p\mid q}(1-p^{-2})$.
\end{lemma}
\begin{proof}
Among the finitely many cosets, two powers of a fixed element
have the same coset. Their quotient is a positive power in the
subgroup. A primitive pair $(a,b)$ modulo $q$ means
$\gcd(a,b,q)=1$. Bezout's identity gives $d,c$ with
$ad-bc=1\pmod q$, so it is the first column of a determinant-one
matrix. This proves transitivity. Over each $\mathbb Z/p^r\mathbb Z$
the excluded pairs are those with both entries divisible by $p$,
giving $p^{2r}-p^{2r-2}$ possibilities; the Chinese remainder
theorem gives the formula. Reduction from $\SL_2(\mathbb Z)$
is surjective: over $\mathbb Z/p^r\mathbb Z$ an invertible entry
allows elementary row reduction, and the Chinese remainder
theorem combines the elementary matrices. This also proves the
orbit count used in Lemma~\ref{lem:affineleviorbits}.
\end{proof}

The following qualitative results identify the rational group
associated with a polynomial orbit.

\begin{proposition}[{\cite[Proposition~2.1]{Shah1994};
\cite[Proposition~3.2]{Shah1991}}]
\label{app:shah-normalization}
For a regular algebraic map $\Theta:\mathbb R^k\to\mathbf G(\mathbb R)$
with $\Theta(0)=\groupid$, the closed subgroup generated by its image is generated
by algebraic unipotent one-parameter subgroups. In the quotient
$\SL_n(\R)/\SL_n(\Z)$, if the closure of such a subgroup orbit is a
connected homogeneous orbit, its group is the real identity
component of the smallest $\mathbb Q$-algebraic subgroup containing
the generated subgroup; the radical of that rational group is unipotent.
\end{proposition}
The first assertion is \cite[Proposition 2.1(1)]{Shah1994}.
The arithmetic identification is \cite[Proposition 3.2]{Shah1991}.
For the curves considered here, the results of
\cite[Theorem 1.1 and Proposition 2.1(2)]{Shah1994} imply that the
orbit closure is homogeneous and has finite invariant volume.
Apply these statements to $g^{-1}\Phi(t)g$ at $\Gamma$;
left translation by $g$ then gives \eqref{eq:actualsupport}.
Lemma~\ref{lem:orbitnormalization} writes out this deduction.
For a periodic principal orbit,
\cite[Proposition 2.1(2)]{Shah1994} also gives the Zariski-density
statement used in Lemma~\ref{lem:arithmeticinputs}.
Lemma~\ref{app:rational-envelope} then gives rationality of
its Lie algebra from its integral stabilizer.
The conclusions above require Shah's orbit-closure theorems;
finite invariant volume alone does not imply them for arbitrary
real algebraic orbits.

\subsection{Tensor representations, fixed forms, and rational heights}
\label{app:tensors}

For a representation on $V$, the dual action on a covector is
$\ell\mapsto\ell\circ g^{-1}$, and the exterior action is
\[
 g(v_1\wedge\cdots\wedge v_k)=gv_1\wedge\cdots\wedge gv_k.
\]
Its matrix entries are minors. Relative to a volume form,
$\wedge^{n-1}V$ identifies with $V^*$ for determinant-one
matrices, giving the inverse-transpose action. This explains
the dual-lattice tests and the cofactor degree bounds.

The symmetric-square representation can be realized on symmetric
matrices by $Q\mapsto gQg^{\mathsf T}$. A quadratic form on
vectors is instead a symmetric covariant tensor, with invariance
equation $g^{\mathsf T}Qg=Q$. For an invertible $Q$,
\[
 gQg^{\mathsf T}=Q
 \quad\Longleftrightarrow\quad
 g^{\mathsf T}Q^{-1}g=Q^{-1}.
\]
This relates the convention for $H_{\mathrm{pr}}$ in
\eqref{eq:regularform} to the convention for $\SO(\qform)$.
Differentiation gives respectively
$XQ+QX^{\mathsf T}=0$ and $X^{\mathsf T}Q+QX=0$.
By Lemma~\ref{app:lie-tests}, these equations characterize
invariant tensors for a connected group.

\begin{lemma}
\label{app:tensor-degrees}
For $S(A):Q\mapsto AQA^{\mathsf T}$ on symmetric $3$ by $3$
matrices, $\det S(A)=(\det A)^4$. On $\SL_3$, the entries of
$\wedge^kS(A)$ have polynomial representatives of degrees
$2,4,6,5,4$ for $k=1,\ldots,5$. All the indicated standard,
symmetric, and exterior integral lattices are preserved by
$\SL_3(\mathbb Z)$.
\end{lemma}
\begin{proof}
For diagonal $A$, the six eigenvalues are
$a_1^2,a_2^2,a_3^2,a_1a_2,a_1a_3,a_2a_3$, whose product is
$(a_1a_2a_3)^4$. Matrices with distinct complex eigenvalues form
a dense open set, so polynomial identity proves the formula for
all $A$. Entries of $S(A)$ are quadratic, and its $k$ by $k$
minors have degree $2k$.

If $\det A=0$, then $S(A)$ has image in
$\Sym^2(\operatorname{im}A)$, of dimension at most three.
Every $4$ by $4$ minor therefore vanishes on $\det A=0$ and
is divisible by $\det A$, leaving degree five. To prove the
divisibility, write the determinant as $c x_{33}+d$ in the
polynomial ring over the other eight entries.
Its cofactor $c=x_{11}x_{22}-x_{12}x_{21}$ is irreducible and does
not divide $d$: a permutation matrix can have $c=0$ and determinant
$-1$. Gauss's lemma therefore makes $c x_{33}+d$ irreducible.
Divide a vanishing minor by this linear polynomial over the
fraction field of the other entries. Substitution
$x_{33}=-d/c$ shows that the remainder is identically zero.
Clearing denominators and using irreducibility gives divisibility
in the polynomial ring, with rational quotient. On $\det A=1$
this quotient represents the original minor.

Complementary minors identify the $5$ by $5$ minors with entries
of $\det S(A)S(A)^{-1}$, up to fixed signs and transposition.
Since $A^{-1}=\operatorname{adj}(A)/\det A$,
\[
 \det S(A)S(A)^{-1}=(\det A)^2S(\operatorname{adj}(A)),
\]
which has degree four on $\SL_3$.
Finally, integral matrices and their integral inverses act
integrally on all these tensor bases, so preserve the lattices.
The quadratic degree bound is used in
Lemma~\ref{lem:degreecutoff}. The higher exterior powers occur in
the polynomial-goodness argument in
Proposition~\ref{prop:separatedregular}, where their unreduced
degree bounds $2k$ suffice.
\end{proof}

Fixed vectors can be computed from the zero-weight space for the
diagonal torus and the infinitesimal equations of
Lemma~\ref{app:lie-tests}. For a representation $\rho$ of
$G=\SL_3(\mathbb R)$ on $V$,
\[
 V^G=\bigcap_{i\ne j}\ker\bigl(\mathrm d\rho(E_{ij})\bigr),
\]
since the matrices $E_{ij}$, $i\ne j$, generate $\mathfrak{sl}_3$
as a Lie algebra. One can first restrict these equations to the
zero-weight space in Lemma~\ref{app:weights}. The three ambient
representations used for coefficient separation have no fixed
vectors by Lemma~\ref{lem:allwedges}. For the principal $\mathfrak{sl}_2$
acting on symmetric $3$ by $3$ matrices, the fixed space is
$\mathbb RQ_0$ by Lemma~\ref{lem:rationalkernel}.

\begin{lemma}
\label{app:plucker}
If $W\le\mathbb R^n$ has dimension $r$ and a rational basis,
then $W\cap\mathbb Z^n$ is a full lattice in $W$.
The line $\wedge^rW$ has a primitive integral generator,
unique up to sign. The wedge of any basis of $W\cap\mathbb Z^n$
is such a generator, and its Euclidean norm is the covolume of
this lattice in $W$. If its coordinates have absolute value
at most $H_0$, then $W$ has an integral spanning set of $r$
vectors with coordinates bounded by $H_0$.
\end{lemma}
\begin{proof}
Clearing denominators of a rational basis gives independent
integral vectors; discreteness follows from $\mathbb Z^n$.
The lattice $W\cap\mathbb Z^n$ is saturated: if an integer
multiple of an integer vector lies in it, so does the vector.
Integer row and column reduction, obtained from the Euclidean
algorithm, shows that a basis of a saturated lattice extends
to a basis of $\mathbb Z^n$. Its wedge is therefore primitive,
since an invertible integer change of basis takes it to
$e_1\wedge\cdots\wedge e_r$. Call this wedge $p$.
Its squared norm is the Gram determinant of the lattice basis,
which is the squared covolume. Every integral vector on this
line is an integer multiple of $p$, proving uniqueness up to
sign of the primitive generator.
Choose a nonzero coordinate $p_I$. Projection onto the $r$
coordinates in $I$ identifies $W$ with $\mathbb R^r$.
The basis projecting to the standard basis has coordinates
$p_J/p_I$, up to signs, by Cramer's rule. Multiplying these
vectors by $p_I$ proves the bound. They need not be a lattice
basis.
\end{proof}

\end{document}